\documentclass[11pt,a4paper,twoside]{article}
\usepackage{titling}
\usepackage{titletoc}
\usepackage{url}

\usepackage[utf8]{inputenc}
\usepackage[T1]{fontenc}
\usepackage{amsmath}
\usepackage{amsfonts}
\usepackage{amssymb}
\usepackage{amsthm}

\usepackage{mathrsfs}
\usepackage{mathtools}
\usepackage{stmaryrd}
\usepackage[all]{xy}
\usepackage{makeidx}

\usepackage[french,english]{babel}\frenchsetup{GlobalLayoutFrench=false}
\usepackage{xspace}
\usepackage{float}

\usepackage{etoolbox}

\RequirePackage{amsbsy}

\makeatletter\def\th@plain{\slshape}\makeatother
\makeatletter\patchcmd{\th@remark}{\itshape}{\slshape}{}{}\makeatother

\newcounter{bidon}

\newcommand{\rdb}{\refstepcounter{bidon}}

\newcommand{\eop}{\hfill \mbox{$\Box$}}

\newcommand \spa {\hspace*{1cm}}

\newcommand \CC{\mathbb {C}} 
\newcommand \DD{\mathbb {D}}

\newcommand \KK{\mathbb {K}} 
 
\newcommand \NN{\mathbb {N}} 
\newcommand \N{\NN} 
\newcommand \ZZ{\mathbb {Z}}

\newcommand \QQ{\mathbb {Q}} 
\newcommand \RR{\mathbb {R}} 
\newcommand \R{\RR} 

\newcommand \cit{\CC}

\newcommand \wi {\widetilde}

\newcommand \thref[1] {Theorem~\ref{#1}}
\newcommand \lemref[1] {Lemma~\ref{#1}}

\newcommand \flo[1] {\lfloor #1 \rfloor}
\newcommand \Flo[1] {\llfloor #1 \rrfloor}
\newcommand \abs[1]{\left|{#1}\right|}
\newcommand \abS[1]{\big|{#1}\big|}
\newcommand \Abs[1]{\Big|{#1}\Big|}
\newcommand \norme[1]{\lVert #1 \rVert}
\newcommand \Norme[1]{\big\lVert #1 \big\rVert}
\newcommand \NOrme[1]{\Big\lVert #1 \Big\rVert}

\newcommand \mni {\medskip\noindent }
\newcommand \sni {\smallskip\noindent }
\newcommand \bni {\bigskip\noindent }
\newcommand \snic[1] {\sni\centerline{$#1$}\smallskip}

\newcommand \sC{{{\mathscr C}}}
\newcommand \sE{{{\mathscr E}}}

\newcommand \ca{{{\cal C}}}
\newcommand \M{{{\cal M}}}
\newcommand \p{{\cal P}}
\newcommand \np{{\cal NP}}
\newcommand \cnp{{\cal CO}{\rm-}{\cal NP}}

\newcommand\Ev{{\rm Ev}}
\newcommand\Lin{{\rm Lin}}
\newcommand\QLin{{\rm QLin}}
\newcommand\Poly{{\rm Poly}}

\newcommand\csl{{\sC_{\rm slc}}}
\newcommand\cbo{{\sC_{\rm bc}}}
\newcommand\ckf{{\sC_{\rm KF}}}
\newcommand\caf{{\sC_{\rm ac}}}
\newcommand\capo{{\sC_{\rm pac}}}
\newcommand\cw{{\sC_{\rm W}}}
\newcommand\crf{{\sC_{\rm rff}}}
\newcommand\csr{{\sC_{\rm srff}}}
\newcommand\csp{{\sC_{\rm sp}}}

\newcommand\DTI{{\rm DTIME}}
\newcommand\DSPA{{\rm DSPACE}}
\newcommand\DSRT{{\rm DSRT}}

\newcommand\DSR{{\rm DSR}}
\newcommand\DRT{{\rm DRT}}
\newcommand\QL{{\rm QL}}
\newcommand\LINT{{\rm LINTIME}}
\newcommand\LINS{{\rm LINSPACE}}
\newcommand\PSP{{\rm PSPACE}}
\newcommand\Rec{{\rm Rec}}
\newcommand\Fnc{{\rm Fnc}}
\newcommand\Prim{{\rm Prim}}
\newcommand\lst{{\rm lst}}
\newcommand\Exp{{\rm Exp}}

\newcommand\Exec{{\rm Exec}}

\newcommand\Y{{\bf Y}}
\newcommand\U{{\bf U}}
\newcommand\CB{{\rm BC}}
\newcommand\CA{{\rm AC}}
\newcommand\Gv{{\bf Gev}}
\newcommand\ysl{{\Y_{\rm slc}}}
\newcommand\ybc{{\Y_{\rm bc}}}
\newcommand\ykf{{\Y_{\rm KF}}}
\newcommand\yaf{{\Y_{\rm ac}}}
\newcommand\yap{{\Y_{\rm pac}}}
\newcommand\yw{{\Y_{\rm W}}}
\newcommand\yrf{{\Y_{\rm rff}}}
\newcommand\ysr{{\Y_{\rm srff}}}
\newcommand\ysp{{\Y_{\rm sp}}}

\newcommand \cE {{\cal E}}

\newcommand \loca {locally\xspace}
\newcommand \uni {uniformly\xspace}
\newcommand \mcu {uniform continuity modulus\xspace}
\newcommand \unico {uniformly continuous}

\newcommand \com {complexity\xspace}
\newcommand \equiva {equivalent\xspace}
\newcommand \etpo {in polynomial time\xspace}

\newcommand \rp{rational presentation\xspace}
\newcommand \rps{rational presentations\xspace}
\newcommand \rapr{rationally presented\xspace}

\newcommand \num {{n$^{{\rm o}}$}}

\newcommand\C{{\bf C}}
\newcommand\Ci{\C^{(\infty)}}
\newcommand\czu{{\C[0,1]}}
\newcommand\cab{{\C[a,b]}}
\newcommand\cuu{{\C[-1,1]}}
\newcommand\po{{\bf P}}

\newcommand \Oo {{\rm O}}
\newcommand \sz {{\rm sz}}
\newcommand \Tch {{\rm T}}
\newcommand \Id {{\rm Id}}
\newcommand \depth {{\rm depth}}
\newcommand \Mag {{\rm mag}}
\newcommand \Lg {{\rm lg}}
\newcommand \F {{\rm F}}

\def\.@{\char'76}

\usepackage[bookmarksopen=false,breaklinks=true,%
      backref=page,pagebackref=true,plainpages=false,%
      hyperindex=true,pdfstartview=FitH,colorlinks=true,%
      pdfpagelabels=true,linkcolor=blue,%
      citecolor=red,urlcolor=red,
      ]%
   {hyperref}

\begin{document}
\selectlanguage{english}

\thispagestyle{empty}

\begin{center} 
{\LARGE \bf Rationally presented metric spaces and complexity, the case of the space of uniformly continuous real functions on a compact interval}\vspace{3cm}
\end{center}
 
\vspace{-2cm}
\noindent In this file you find the English version starting on the page  numbered \pageref{beginenglish}.


\begin{abstract}
We define the notion of {\em rational
presentation of a complete metric space} in order to study metric spaces from the algorithmic complexity point of view. In this setting, we study some presentations of the space $\czu$ of uniformly continuous real functions over [0,1] with the usual norm:
$\norme{f}_{\infty} = {\bf Sup} \{ \abs{f(x)} ; \;
 0 \leq x \leq 1\}.$
This allows us to have a comparison of a global kind between complexity notions attached to these presentations. In particular, we get a generalisation of Hoover's results concerning the {\sl Weierstrass approximation theorem in polynomial time}. We get also a generalisation of previous results on analytic functions which are computable in polynomial time.
\end{abstract}

\bigskip This is the English translation of the French paper 
``Espaces métriques rationnellement présentés et complexité, le cas de l'espace des fonctions réelles uniformément continues sur un intervalle compact'', \emph{Theoretical Computer Science} {\bf 250}, \num 1-2, 265--332, (2001). (Received April 1997; revised March 1999.) We have fixed some minor typos.

\bigskip\noindent   {\large \bf Authors}  

\smallskip \noindent Salah Labhalla,
 Dépt. de Mathématiques,
 Univ. de Marrakech, Maroc\\
{\tt labhalla@ucam.ac.ma} 

\smallskip \noindent Henri Lombardi, Université Marie et Louis Pasteur, F-25030 Besançon Cedex, France, \\
email: {\tt henri.lombardi@univ-fcomte.fr}

\smallskip \noindent 
E. Moutai,
Dépt. de Mathématiques,
Univ. de Marrakech, Maroc

\bigskip \noindent  
Then the French version begins on the page numbered \pageref{beginfrench}.

\begin{center} \label{center}

   {\Large \bf Espaces métriques rationnellement présentés et complexité, le cas de l'espace des fonctions réelles uniformément continues sur un intervalle compact}

\end{center}

\begin{abstract}
 Nous définissons la notion de {\em présentation rationnelle d'un
espace métrique complet} comme moyen d'étude des espaces métriques et des
fonctions continues du point de vue de la complexité algorithmique. Nous
étudions dans ce cadre différentes manières de présenter l'espace $\czu$ 
des fonctions réelles uniformément continues sur l'intervalle $[0,1]$, muni de la norme 
usuelle:
 $\norme{f}_{\infty} = {\bf Sup} \{ \abs{f(x)}  \;
 0 \leq x \leq 1\}.$ Ceci nous permet de
faire une comparaison de nature globale entre les notions de complexité
attachées à ces présentations. En particulier, nous obtenons une
généralisation des résultats de Hoover concernant le {\em théorème 
d'approximation de Weierstrass en temps polynomial}. Nous obtenons également une
généralisation des résultats de Ker-I. Ko, H. Friedman et N. Müller 
concernant les fonctions analytiques calculables en temps polynomial.
\end{abstract}

\normalsize
\newpage
\thispagestyle{empty}

~

\pagestyle{headings}
\patchcmd{\sectionmark}{\MakeUppercase}{}{}{}
\setcounter{page}{0}
\renewcommand\thepage{E\arabic{page}}

\begingroup
\def\proofname{\textsl{Proof}}



\theoremstyle{plain}
\newtheorem{theorem}{Theorem}[subsection]
\newtheorem{proposition}[theorem]{Proposition}
\newtheorem{propdef}[theorem]{Proposition and Definition}
\newtheorem{lemma}[theorem]{Lemma}
\newtheorem{corollary}[theorem]{Corollary}

\theoremstyle{definition}
\newtheorem{definition}[theorem]{Definition}
\newtheorem{notation}[theorem]{Notation}
\newtheorem{example}[theorem]{Example}
\newtheorem{examples}[theorem]{Examples}

\theoremstyle{remark}
\newtheorem{remark}[theorem]{Remark}
\newtheorem{remarks}[theorem]{Remarks}

\title{Rationally presented metric spaces and complexity, the case 
of the space of uniformly continuous real functions on a compact interval}

\author{S. Labhalla\\
 Dépt. de Mathématiques\\
 Univ. de Marrakech, Maroc\\
{\tt labhalla@ucam.ac.ma} 
\and H. Lombardi\\ Laboratoire de Mathématiques de Besançon \\
Université Marie et Louis Pasteur, France \\
{\tt henri.lombardi@univ-fcomte.fr}
\and E. Moutai\\
Dépt. de Mathématiques\\
Univ. de Marrakech, Maroc}

\maketitle
\rdb
\label{beginenglish}

This is the English translation of the French paper 
``Espaces métriques rationnellement présentés et complexité, le cas de l'espace des fonctions réelles uniformément continues sur un intervalle compact'', \emph{Theoretical Computer Science} {\bf 250}, \num 1-2, 265--332, (2001). (Received April 1997; revised March 1999.) We have fixed some minor typos.

\begin{abstract} \label{abstract}
We define the notion of {\em rational
presentation of a complete metric space} in order to study metric spaces from the algorithmic complexity point of view. In this setting, we study some presentations of the space $\czu$ of uniformly continuous real functions over [0,1] with the usual norm:
$\norme{f}_{\infty} = {\bf Sup} \{ \abs{f(x)} ; \;
 0 \leq x \leq 1\}.$
This allows us to have a comparison of a global kind between complexity notions attached to these presentations. In particular, we get a generalisation of Hoover's results concerning the {\sl Weierstrass approximation theorem in polynomial time}. We get also a generalisation of previous results on analytic functions which are computable in polynomial time.
\end{abstract}


\newpage
\setcounter{tocdepth}{4}
\markboth{Contents}{Contents}
\startcontents[english]

\printcontents[english]{}{1}{}
\normalsize

\section*{Introduction}\label{sec0}
\addcontentsline{toc}{section}{Introduction}
\markboth{Introduction}{Introduction}
Let $\czu$ be the space of uniformly continuous real functions on the interval 
$[0,1]$.

\smallskip In \cite{KF82}, Ker-I.\ Ko and Friedman introduced and studied the notion of complexity of real functions defined via an oracle Turing machine. 
In the article \cite{Ho90}, Hoover studied presentations of the space 
$\czu$  via boolean circuits and arithmetic circuits.

In these two articles the complexity in the $\czu$ space is studied ``in a pointwise meaning''. In other words, they study sentences such as: 
\begin{itemize}

\item 
$f$ is a $\p$-point in the sense of Ker-I.\ Ko and Friedman.

\item 
$f$ is a $\p$-point in the sense of boolean circuits.

\item 
$f$ is a $\p$-point in the sense of arithmetic circuits.
\end{itemize}

Ker-I.\ Ko and Friedman studied the properties of $\p$-points (in their sense). 
Hoover showed that $\p$-points in the sense of arithmetic circuits are the same as $\p$-points in the sense of Ker-I.\ Ko and Friedman.

\medskip In this paper, we study several presentations of the space $\czu$ based on the notion of {\em rational presentation of a complete metric space}. 
This allows us to make a global comparison between the notions of complexity attached to different presentations.
 
\medskip After some preliminaries in Section~\ref{sec1}, Section \ref{sec2} essentially contains an exposition of what seems to us to be a natural approach concerning complexity questions related to separable complete metric spaces. 

\smallskip 
Usually a metric space contains objects of infinite nature (the paradigm being real numbers defined à la Cauchy), which rules out a direct computer presentation (i.e.\ encoded on a finite alphabet) of these objects.

To get round this difficulty, we proceed as  usual for the space $\RR$. 
Let us consider a dense part~$Y$ of the metric space $X$ under consideration, which is simple enough that 
\begin{itemize}

\item 
its elements can be encoded as (certain) words on a fixed finite alphabet;

\item 
the distance function restricted to $Y$ is computable, i.e.\ given by a computable function:

\snic{\delta: Y\times Y\times \NN_1\to \QQ \; \qquad {\rm { with }}\qquad \; \abs{d_X(x,y) - \delta(x,y,n)}\; \leq 1/2^n.
}
\end{itemize}

The space $X$ then appears as the completion of $Y$.

We will say that the proposed encoding of $Y$ and the proposed description of the distance function constitute a ``rational presentation'' of the metric space $X$.
 
It should be noted that we do not intend to deal with non-complete metric spaces, for which the encoding difficulties seem insurmountable, and for which it is practically impossible to demonstrate anything serious in constructive analysis.

Similarly, the metric spaces treated are ``separable'' (also known as ``second countable'').

\smallskip The metric spaces studied in constructive analysis (cf.\ \cite{BB}) are very often defined via a rational presentation of this type, or at least easy to define according to this scheme. 
The problem that arises is generally to define constructively a dense countable part of the space under consideration. This is obviously impossible for classical non-separable Banach spaces such as $L^\infty$, but it is precisely these spaces that cannot be treated in their classical form by constructive methods. 
The problem is more delicate for spaces that are classically separable but for which there is no natural constructive method that gives a dense countable part. 
For example, this is the case for an arbitrary closed subspace of a complete separable space, and it is still the case for certain spaces of continuous functions.

\smallskip The unary presentation and the binary presentation of $\ZZ$ are not equivalent from the point of view of polynomial time complexity. 
The unary presentation is a natural presentation of $\ZZ$ as a group whereas the binary presentation is a natural presentation of $\ZZ$ as a ring. Similar questions can be asked about the usual classical metric spaces.

\smallskip The first question that arises is the comparison of the different presentations of a usual metric space. 
Note that the usual presentation of $\RR$ is  to be the only natural one and that it is from this presentation of $\RR$ that the complexity of a presentation $(Y_1, \delta_1)$ of $X$ or the complexity of the map $\Id_X$ from $(Y_1, \delta_1)$ to $(Y_2, \delta_2)$ is defined when comparing two distinct presentations of~$X$.

\smallskip The next question is that of the complexity of computable  continuous maps between metric spaces. 
While there is a natural notion of the complexity of functions in the case of compact spaces, the question is much trickier in the general case, since it refers to the complexity of arbitrary type-2 functionals.
 
\smallskip Another question is that of the complexity of objects linked in a natural way to the metric space under study. 
For example, in the case of the reals, and considering only the algebraic structure of $\RR$, we are interested in the fact that the complexity of addition, multiplication or finding the complex roots of a polynomial with real coefficients are all ``low level''. 
This kind of result legitimises a posteriori the choice that is usually made to present $\RR$, and results of the opposite kind disqualify other presentations (cf.\ \cite{LL}), which are less efficient than the presentation via Cauchy sequences of rationals written in binary.\footnote{Nevertheless, the sign test is undecidable for the reals presented à la Cauchy. 
We are satisfied with the constructive test in polynomial time: $x + y \geq 1/2^n \Rightarrow ( x \geq 1/2^{n+2} \;\hbox{ or }y \geq 1/2^{n+2} )$.}
 
Similarly, a usual metric space is generally provided with a richer structure than just the metric structure, and it is then a question of studying, for each presentation, the complexity of these ``natural elements of structure''.

\medskip 
In Section \ref{sec3}, we introduce the space $\czu$ of uniformly continuous real functions on the interval $[0,1]$ from the point of view of its rational presentations.
 
The evaluation function $(f, x) \mapsto f(x)$ is not locally uniformly continuous. 
Given its importance, we discuss what the complexity of this function means when we have chosen a rational presentation of the space $\czu$. 

We then give two significant examples of such rational presentations:
\begin{itemize}

\item 
A binary semilinear circuit presentation, $\csl$, where the rational points are exactly the binary semilinear maps. Such a map can be defined by a binary semilinear circuit (cf.\ Definition~\ref{321}) and encoded by an evaluation program corresponding to the circuit.

\item 
A presentation, denoted by $\crf$, via rational functions suitably controlled and given in presentation by formula (cf.\ Definition~\ref{325}).
\end{itemize}
Finally we establish complexity results related to Newman's approximation theorem. 
Briefly, Newman's theorem has polynomial complexity and consequently the piecewise linear functions are, each individually, points of complexity $\p$ in the rational presentation $\crf$.

\medskip In Section~\ref{sec4} we define and study several ``natural'' rational presentations of $\czu$, equivalent from the point of view of polynomial time complexity.
\begin{itemize}
\item 

The so-called Ko-Friedman presentation, denoted by $\ckf$, for which a rational point is given by a quadruplet $(Pr, n, m, T)$ where $Pr$ is a Turing machine program. The integers $n$, $m$, $T$ are control parameters. We make the link with the notion of complexity introduced by Ker-I.\ Ko and Friedman. We show a universal property that characterises this rational presentation from the point of view of polynomial-time complexity.

\item 

The boolean circuit presentation which is denoted by $\cbo$ and for which a rational point is given by a quadruplet $(C,n,m,k)$ where $C$ encodes a boolean circuit and $n$, $m$, $k$ are control parameters.

\item 

The presentation by arithmetic circuits (with magnitude) noted $\caf$. A rational point of this presentation is given by a pair $(C,M)$ where $C$ is the code of an arithmetic circuit, and $M$ is a control parameter.

\item 
The presentation by polynomial arithmetic circuits (with magnitude) noted $\capo$ analogous to the preceding one, but here the circuit is polynomial, i.e.\ it does not contain the ``pass to inverse'' gates). 
\end{itemize}
We show in Section \ref{subsec42} that the presentations $\ckf$, $\cbo$, $\csl$, $\caf$ and $\capo$ are equivalent in polynomial time. 

This result generalises and clarifies Hoover's results. Not only are the $\p$-points of $\ckf$ ``the same'' as the $\p$-points of $\capo$, but even better, the bijections 
\[
\ckf \rightarrow \capo \;\; \; {\rm and} \; \; \; \capo \rightarrow \ckf
\]
which represent the identity of $\czu$ are computable in polynomial time. 
Thus we obtain a fully controlled formulation of the Weierstrass approximation theorem from an algorithmic point of view. 

\smallskip In Section~\ref{subsec43}, we show that these presentations are not of class $\p$ by establishing the infeasibility of computing the norm function (if $\p\neq \np$). 
More precisely, we properly define ``the norm problem'' and show that it is an $\np$-complete problem for the  presentations we have introduced. 

As for the membership test of the set of rational (codes of) points, it is a $\cnp$-complete problem for the presentations $\ckf$ and $\cbo$, while it is in linear time for the presentation $\csl$. The latter is therefore slightly more satisfactory.

\medskip In Section \ref{sec5} we define other presentations of the space $\czu$.
\begin{itemize}

\item 
The presentation denoted by $\cw$ (like Weierstrass) for which the set of rational points is the set of (one-variable) polynomials with rational coefficients given in dense presentation.

\item 
The presentation denoted by $\csp$ for which the set of rational points is the set of piecewise polynomial functions, each polynomial being given as for $\cw$.

\item 
The presentation $\csr$ is obtained from $\crf$ in the same way as $\csp$ is obtained from $\cw$.
\end{itemize}
We try to see to what extent the $\p$ class character of these presentations provides a suitable framework for numerical analysis.

We characterise the $\p$-points of $\cw$ by establishing the equivalence between the following properties (cf.\ \thref{527}).

\begin{itemize}

\item [a)]
$f$ is a $\p$-point of $\cw$;

\item [b)]
the sequence $A_n(f)$ (which gives the expansion of $f$ in a Chebyshev series) is a $\p$-sequence in $\RR$ and verifies the inequality 

\snic{\abS{A_n(f)}\; \leq Mr^{n^{\gamma}} \;{\rm 
with }\;M > 0,~\gamma > 0,~0< r < 1;
}

\item [c)] 
$f$ is a $\p$-point of $\ckf$ and is in the Gevrey class. 
\end{itemize}
We deduce an analogous equivalence between the following properties (cf.\ \thref{528}).
\begin{itemize}

\item [a)] 
$f$ is an analytic function and is a $\p$-point of $\cw$;

\item [b)] 
$f$ is an analytic function and is a $\p$-point of $\ckf$.
\end{itemize}

The usual calculations on $\p$-points of $\cw$ (calculation of the norm, the maximum and the integral of a $\p$-point of $\cw$) are in polynomial time (cf.\ Proposition \ref{526}). 

Moreover a fairly good behaviour of  derivation with respect to complexity is obtained by showing that for any $\p$-point (or any $\p$-sequence) of $\cw$ the sequence (or double sequence) of its derivatives is a $\p$-sequence of $\cw$ (cf.\ \thref{5210} and Corollary \ref{5211}).

\smallskip In \thref{5213} and its Corollary \ref{5214}, we give a more uniform version of the previous results. Combined with Proposition \ref{526}
we get a significant improvement of the results in \cite{Mu87} and \cite{KF88}. And our proofs are more conceptual. 

\smallskip Moreover, these results, combined with Newman's theorem (in its version given in Section~\ref{subsec33}), show that the transition from the presentation by rational functions $\crf$ to the presentation by (dense) polynomials $\cw$ is not in polynomial time. 
Newman's theorem also shows that the presentations $\crf$ and $\csr$ are polynomially equivalent.

\bigskip Summarising the results of sections \ref{sec4} and \ref{sec5}, we obtain that the identity of $\czu$ is uniformly of class $\p$ in the following cases: 
\[
\cw \rightarrow \csp \rightarrow \crf \equiv \csr 
\rightarrow \ckf \equiv \cbo \equiv \csl \equiv \caf 
\equiv \capo.
\] 
and none of the arrows in the line above is a $\p$-equivalence except perhaps $\csp \to \crf$ and also perhaps $\crf \to \ckf$ (but that would imply $\p = \np$).

\medskip We end this introduction with a few remarks on constructivism.

The work presented here is written in the style of constructive mathematics à la Bishop as developed in particular in the books by Bishop \cite{Bi}, Bishop \& Bridges \cite{BB}, Mines, Richman \& Ruitenburg \cite{mrr}. 
This is a kind of minimal body of constructive mathematics.\footnote{A pioneering work in a purely formal style is that of Goodstein \cite{Go1,Go}.}
All the theorems proved in this way are in fact equally valid for classical mathematicians, the advantage being that here the theorems and proofs always have an algorithmic content. 
The classical reader can therefore read this work as an extension of the work on recursive analysis by Kleene and Turing and later by Ker-I.\ Ko  \& H.\ Friedman (\cite{KF82}, \cite{Ko91}), N.\ Th.\ Müller \cite{Mu86,Mu87}, Pour El \& Richards \cite{PR}, who themselves work within a framework of classical mathematics.

On the other hand, theorems and their proofs in Bishop's style are also acceptable within the framework of other variants of constructivism, such as Brouwer's intuitionism or the Russian constructive school of A.A.\ Markov, G.S.\ Ceitin, N.A.\ Shanin, B.A.\ Kushner and their pupils. 
An enlightening discussion of these different variants of constructivism can be found in Bridges \& Richman's book \cite{BR}. 

Beeson's comprehensive book \cite{Be} also gives thorough discussions of these views and their variants, including a remarkable survey of the work of logicians who have attempted to formalise constructive mathematics. 

Russian constructive mathematics can be discovered through the books by Kushner \cite{Ku} and O.\ Aberth \cite{Ab}. 
A very relevant historical article on the subject is written by M.\ Margenstern~\cite{Ma}. 
Chapter IV of \cite[Beeson]{Be} is also very instructive. 
Beeson states on page 58: ``We hope to show that this mathematical universe is an extremely entertaining place, full of surprises (like any foreign country), but by no means too chaotic or unlivable''. 
Russian constructive mathematics restricts its objects of study to ```recursive beings'', and in this respect is similar to certain works of classical recursive analysis developed by Grzegorczyk, Kreisel, Lacombe, Schoenfield and Specker, with whom it shares many results.
 One of the principles of Russian constructive mathematics, for example, is that there are only recursive reals given by algorithms (which calculate sequences of rationals with a controlled speed of convergence). 
And a real function defined on the reals is an algorithm $F$ which takes as input an algorithm $x$ and gives as output, under the condition that $x$ is an algorithm producing a recursive real, an algorithm $y$ producing a recursive real. 
This leads, for example, to Ceitin's theorem, which is classically false, according to which any real function defined on the reals is continuous at any point.
Classical recursive analysis states: if an algorithm terminates each time the input is the code of a recursive real and then gives the code of a recursive real as output, and if it defines a function (i.e.\ two codes of the same recursive real as input lead to two codes of the same recursive real as output) then it defines a function that is continuous at any recursive real point (Kreisel-Lacombe-Schoenfield-Ceitin theorem). 
Beeson analysed the Kreisel-Lacombe-Schoenfield proof to clarify its non-constructive aspects. 
Ceitin's proof, which is more constructive than Kreisel-Lacombe-Schoenfield's, is analysed in \cite{BR}.
 
In Bishop's style, since the notion of effectivity is  to be a primitive notion that does not necessarily reduce to recursivity, and since recursivity, on the contrary, cannot be defined without this primitive notion of effectivity, real numbers and real functions are objects that retain a greater degree of ``freedom'' and are therefore closer to real numbers and real functions as intuitively conceived by classical mathematicians. 
And the preceeding ``Russian'' theorem is therefore unprovable in Bishop style. Symmetrically, theorems from classical mathematics that directly contradict results from Russian constructivism (such as the possibility of unambiguously defining real discontinuous functions) cannot be proved in Bishop-style mathematics. 

In conclusion, Russian constructivism is of undeniable historical interest, develops a very coherent mathematical philosophy and can undoubtedly be revived by theoretical computer science concerns. 
It would also be interesting to study this mathematics from the point of view of algorithmic complexity. 

\section{Preliminaries}\label{sec1}
\subsection{Notations}\label{subsec11}
\begin{description}\itemsep.5pt
\item [$\NN_1$] 
set of unary natural numbers.

\item [$\N$]  
set of natural integers in binary. From the point of view of complexity, $\NN_1$ is isomorphic to the part of $\NN$ formed by the powers of $2$.

\item [$\ZZ$] 
set of integers in binary.

\item [$\QQ$] 
set of rational numbers presented in the form of a fraction with numerator and denominator in binary.

\item [$\QQ^{\NN_1}$] 
set of sequences of rationals, where the index is in unary.

\item [$\DD$] 
set of numbers of the form $k/2^n$ with $(k,n) \in \ZZ \times \NN_1$.

\item [$\DD_{[0,1]}$] 
$\DD \cap [0,1]$.

\item [$\DD_n$] 
set of numbers of the form $k/2^n $ where $k \in \ZZ$.

\item [$\DD_{n,[0,1]}$] 
$\DD_n \cap [0,1]$.

\item [${\DD[X]}$] 
set of polynomials with coefficients in $\DD$ given in dense presentation.

\item [${\DD[X]_f}$] 
set of polynomials with coefficients in $\DD$ given in formula presentation.

\item [$\RR$] 
set of real numbers presented by Cauchy sequences in $\QQ$.

\item [$\CB$] 
set of codes of boolean circuits.

\item [$\CA$] 
set of codes of arithmetic circuits.

\item [{\rm TM}] 
Turing machine.

\item [{\rm OTM}] 
Turing machine with oracle.

\item [$\mu$] 
modulus of uniform continuity (cf.\ Definition \ref{221}).

\item [$\Lg(a)$] 
the length of the binary coding of the dyadic $\vert a \vert$ (for $a \in \DD$).

\item [$\sz(C)$] 
size of circuit $C$ (the number of its gates).

\item [$\depth(C)$] 
depth of circuit $C$.

\item [$\Mag(C)$] 
magnitude of the arithmetic circuit $C$ (cf.\ Definition \ref{4110}).

\item [$\wi{f}$] 
continuous function encoded by $f$.

\item [$\M(n)$] 
complexity of calculating the multiplication of two integers in binary representation (in fast multiplication $\M(n) = \Oo(n \log(n) \log\log(n)$).

\item[$\log(n)$] for a natural integer $n$ is the length of its binary code.

\item [$\flo d$]	
the length (of the discrete encoding) of the object $d$ (the encoding of $d$ is a word on a fixed finite alphabet).

\item[$\Flo{x}$] 
the largest integer $m$ such that $m\leq x$.  
 
\end{description}

\paragraph{Presentations of the space $\czu$}~

\begin{itemize}\itemsep.5pt

\item [$\ckf$] 
presentation by Turing Machine, à la Ko-Friedman\\
$\ykf$ is the set of (codes of) rational points  (cf.\ Definition \ref{411})

\item [$\caf$] 
presentation by arithmetic circuits\\
$\yaf$ is the set of (codes of) rational points (cf.\ Definition \ref{4110})

\item [$\capo$] 
presentation by divisionfree (or polynomial) arithmetic circuits\\
$\yap$ is the set of (codes of) rational points (cf.\ Section \ref{subsubsec414})

\item [$\cbo$] 
presentation by boolean circuits\\
$\ybc$ is the set of (codes of) rational points (cf.\ Definition \ref{418})

\item [$\csl$] 
presentation by binary semilinear circuits\\
$\ysl$ is the set of (codes of) rational points (cf.\ Definition \ref{321})

\item [$\crf$] 
presentation by rational functions in a formula encoding\\
$\yrf$ is the set of (codes of) rational points (cf.\ Definition \ref{325})

\item [$\cw$] 
presentation à la Weierstrass by polynomials in a dense encoding\\
$\yw$ is the set of (codes of) rational points (cf.\ Section \ref{subsec51})

\item [$\csp$] 
presentation by piecewise polynomial functions in a dense encoding\\
$\ysp$ is the set of (codes of) rational points (cf.\ Section \ref{subsubsec512})

\item [$\csr$] 
presentation by piecewise rational functions in a formula encoding\\
$\ysr$ is the set of (codes of) rational points (cf.\ Section \ref{subsubsec513})
\end{itemize}

\subsection{Interesting classes of discrete functions}\label{subsec12} 

We shall consider classes of discrete functions $\ca$ (a discrete function is a function from $A^{\star}$ to $B^{\star}$, where~$A$ and $B$ are two finite alphabets)\footnote{If we want to make this sound more mathematical and less computational, we can replace $A^{\star}$ and $B^{\star}$ by sets~$\NN^k$.} with the following basic stability properties:
\begin{itemize}

\item 
 
$\ca$ contains the usual arithmetic functions and the comparison test in $\ZZ$ (for $\ZZ$ encoded in binary);

\item 
$\ca$ contains functions that can be calculated in linear time (i.e.\ $\LINT \subset \ca$);

\item 
$\ca$ is stable by composition;

\item 
$\ca$ is stable by lists: if $f:A^{\star} \to B^{\star}$ is in $\ca$ then $\lst(f): \lst(A^{\star}) \to \lst(B^{\star})$ is also in $\ca$ ($\lst(f)[x_1, x_2,\ldots, x_n] = [f(x_1), f(x_2),\ldots, f(x_n)]$).
\end{itemize}
A class which verifies the above stability properties is said to be \emph{elementarily stable}.
In practice, we will be particularly interested in the following elementarily stable classes:
\begin{itemize}

\item 
 
$\Rec:$ the class of recursive functions.

\item 
 
$\Fnc:$ the class of constructively defined functions (in constructive mathematics, this concept is a primitive concept which does not coincide with the preceding one, and it is necessary to have it beforehand in order to be able to define the preceding one, recursivity being interpreted as a purely mechanical construction in its unfolding as a process of calculation).

\item $\Prim:$ the class of primitive recursive functions.

\item $\p :$ the class of functions computable in polynomial time.

\item  $\cE:$ the class of elementary recursive functions, i.e.\ functions that can still be calculated in time bounded by a compound of exponentials.

\item $\PSP:$ the class of functions computable in polynomial space (with an output polynomially bounded in size).

\item $\LINS:$ the class of functions computable in linear space (with an output linearly bounded in size).

\item $\DSRT(\Lin, \Lin, \Poly):$ the class of functions computable in linear space, in polynomial time and with an output linearly bounded in size.

\item $\DSRT(\Lin, \Lin, \Oo(n^k)):$ the class of functions computable in linear space, in $\Oo(n^k)$ time (with $k > 1$) and with an output linearly bounded in size.

\item $\DSRT(\Lin, \Lin, \Exp):$ the class of functions computable in linear space, in time $\exp(\Oo(n))$ and with an output linearly bounded in size.

\item $\DRT(\Lin, \Oo(n^k)):$ the class of functions computable in time $\Oo(n^k)$, (with $k > 1$) and with an output linearly bounded in size.

\item $\DSR(\Poly, \Lin):$ the class of functions computable in polynomial space with an output linearly bounded in size.

\item $\DSR(\Oo(n^k), \Lin):$ the class of functions computable in $\Oo(n^k)$ space, (with $k > 1$) with an output linearly bounded in size.

\item $\QL:$ the class of functions computable in quasilinear time (cf.\ Schnorr \cite{Sc}.)  \\($\QL: = \cup_b \DTI (\Oo(n.\log^b(n))) = \DTI (\QLin)$).
\end{itemize}

\smallskip Note that, apart from the $\Rec$ and $\Fnc$ classes, all the classes we have defined are complexity classes (in Blum's sense). However, there is no need for this, as the examples of $\Rec$ and $\Fnc$ show. When we consider only the $\Fnc$ class, we develop a chapter of abstract constructive mathematics.
 Note that $\DTI (\Oo(n^k))$ for $k > 1$ is not stable by composition. But, more often than not, the $\Oo(n^k)$ time calculations that we will have to consider are in the class $\DRT(\Lin, \Oo(n^k))$ or even $\DSRT(\Lin, \Lin, \Oo(n^k))$ and these classes have the right stability properties.

\subsection{Complexity of a Universal Turing Machine}\label{subsec13} 
In the following we will need to use a Universal Turing Machine and estimate its algorithmic complexity. 
The following result, for which we have not found a reference, seems to be part of folklore, and was pointed out to us by Mr Margenstern.

\begin{lemma} \label{131} 
There is a universal TM $MU$ which does the following work.

\noindent 
It takes as input:

\noindent 
--- the code (whose size is $p$) of a TM $M_0$ (operating on a fixed alphabet, with one input band, one output band, several work bands) assumed to be of complexity in time $T$ and space $S$ (with $S(n) \geq n$);

\noindent 
--- an input $x$ (of size $n$) for $M_0$.

\noindent 
It outputs the result of the calculation performed by $M_0$ for input $x$.

\noindent 
It executes this task in time $\Oo(T(n)(S(n)+p))$ and using a space $\Oo(S(n)+p)$.
\end{lemma} 

\begin{proof} 
Machine $MU$ uses a work tape to write, at each elementary step of machine $M_0$, which it simulates, the list of the contents of each of the variables of $M_0.$ To simulate a step of $M_0$, machine $MU$ needs $\Oo(S(n)+p)$ elementary steps, using a space of the same order of magnitude. 
\end{proof}

\subsection{Circuits and evaluation programs}\label{subsec14} 

Circuit families are interesting computational models, particularly from the point of view of parallelism. 
Boolean circuit families are to some extent an alternative to the standard TM model. Small arithmetic circuits are able to calculate polynomials of very high degree. 
This is the case, for example, for an arithmetic circuit that simulates a Newton iteration for a map given by a rational function.
 
In all cases, the problem arises of knowing what coding to adopt for a circuit. We will always choose to code a circuit with one of the evaluation programs (or straight-line programs) which execute the same task as it does.\footnote{Different evaluation programs may correspond to the same circuit, depending on the order in which the instructions to be executed are written.}
 
Furthermore, w.r.t.\ arithmetic circuits, which represent polynomials or rational functions, we will consider them not from the point of view of exact calculation (which would be too costly), but from the point of view of approximate calculation. The question then arises as to how to evaluate their execution time when we want to guarantee a given precision on the result, for entries in $\DD$ which are themselves given with a certain precision. 
No reasonable bound in execution time can be obtained using general-purpose arguments unless the depth of the circuit is very small, as the degrees obtained are too large and the calculated numbers are likely to be too large. As it is not always possible to limit the depth of circuits to very small depths, it is essential to give a control parameter, called magnitude, which ensures that, despite a possible very large degree, the size of all the numbers calculated by the circuit (which will be evaluated with a precision that is itself limited) is not too large when the input represents a real number varying within a compact interval.

\section{Rationally presented complete metric spaces}\label{sec2}

\subsection[Rational presentation of a metric space, complexity of points \ldots]{Rational presentation of a metric space, complexity of points and families of points}\label{subsec21}

Unless explicitly stated otherwise, the classes of discrete functions we consider will be elementarily stable classes.

\begin{definition}[rational presentation of class $\ca$ for a complete metric space] \label{211}

\noindent 
A complete metric space $(X, d_X)$ is {\em given in a rational presentation of class $\ca$} in the following way. We give a triple $(Y, \delta, \eta)$ where
\begin{itemize}

\item $Y$ is a $\ca$-part of a language $A^{\star}$

\item  $\eta$ is a map from $Y$ to $X$

\item $\delta: Y\times Y\times \NN_1\to
\DD$ is a function in the class $\ca$ and verifying (for $n\in\NN_1$ and $x , y , z \in Y$):

\noindent 
--- $\abs{d_X(x,y)-\delta(\eta(x),\eta(y))}\; \leq 1/2^n$

\noindent 
--- $\delta(x,y,n) \in \DD_n $

\noindent 
--- $\delta(x,y,n) = \delta(y,x,n)$

\noindent 
--- $\delta(x,x,n) = $0

\noindent 
--- $\abs{\delta(x,y,n+1)-\delta(x,y,n)}\; \leq 1/2^{n+1}$

\noindent 
--- $\delta(x,z,n) \leq \delta(x,y,n) + \delta(y,z,n) + 2/2^n$ 
\end{itemize}
If we define the distance $d_Y(x,y)$ as the limit of $\delta(x,y,n)$ when $n$ tends to infinity, the map $\eta$ from~$Y$ to~$X$ identifies $(X,d_X)$ to the completion of $(Y,d_Y)$.
The set $\eta(Y)$ is called {\em the set of rational points of $X$ for the presentation under consideration}. If $y \in Y$ we will often note~$\wi{y}$ for~$\eta(y)$, and $y$ is called the code of~$\wi{y}$.
\end{definition}

We sometimes abbreviate ``$(Y,\delta,\eta)$ is a rational presentation of class $\ca$ for $(X,d_X)$'' to ``$(Y,\delta)$ is a $\ca$-presentation of $(X,d_X)$''.

\begin{remarks} \label{212}
1) We can replace $\DD$ by $\QQ$ in the definition above without modifying the notion that is defined. Choosing $\QQ$ is nicer mathematically, whereas choosing $\DD$ is more natural from a computing point of view. In addition, the constraint $\delta(x,y,n) \in \DD_n$ satisfies the natural request not to use more space than necessary to represent an approximation to within $1/2^n$ of a real number.

\noindent 
2) When we have given a rational presentation for an abstract metric space $(X,d_X)$, we will say that {\em we have provided $(X,d_X)$ with a calculability structure}. This is essentially the same as defining a language $Y \subset A^{\star}$ and then a  map $\eta$ from $Y$ to $X$ whose image is a dense part of $X$. The presentation is completely defined only when we have also given a map $\delta\colon Y \times Y \times \NN_1 \to \DD$ verifying the requirements of Definition \ref{211}. 
{\em In the following, we will nevertheless allow ourselves to say that an encoding of a dense part $Y$ of $X$ defines a presentation of class $\ca$ of $X$ when we show that the requests of Definition \ref{211} can be satisfied}. 

\noindent 
3) Understood in the constructive sense, the sentence included in the definition ``we give a  map $\eta$ from $Y$ to $X$'' requires that we have ``a certificate that $\eta(y)$ is indeed an element of $X$ for any $y$ in $Y$''. 	
It may be that this certificate of membership itself implies a natural notion of complexity. When this is the case, for example for the space $\czu$, it will be inevitable to take this complexity into account. {\em Thus, Definition \ref{211} is considered as being incomplete and needing to be specified on a case-by-case basis}. This is undoubtedly a pity from the point of view of the formal elegance of general definitions. But it is a de facto situation that seems very difficult to get round. 
\end{remarks}

\begin{examples}\label{213}~

\noindent 
--- The space $\RR$ is usually defined in the presentation where the set of rational points is $\QQ$. Equivalently, and this is what we will do in the following, we can consider as the set of rational points of $\RR$ the set $\DD$ of dyadic numbers.\footnote{This corresponds to the notation $\RR_{\rm conv}$ in \cite{LL} and $\RR_{\rm con}$ in \cite{Ko83}.}

\noindent 
--- We can easily define the product of two presentations for two metric spaces.

\noindent 
--- Any {\em discrete set} $Z$ ($Z$ is given as a part of a language $A^{\star}$ with an equivalence relation which defines equality in $Z$ and which is testable in the class $\ca$) gives rise to a {\em discrete metric space}, i.e.\ in which the distance of two distinct points is equal to $1$. 

\noindent 
--- The complete metric spaces of constructive mathematics in the style of Bishop \cite{BB} generally admit a rational presentation of class $\Fnc$. 
In each concrete case the presentation turns out to be a presentation of class $\Prim$ or even $\p$. 
\end{examples}

\begin{definition}[complexity of a point in a rationally presented metric space] \label{214} ~

\noindent 
Consider a complete metric space $X$ given in a presentation $(Y,\delta,\eta)$ of class $\ca'$. A point $x$ of $X$ is said to be of class $\ca$ when we know a sequence ($n \mapsto y_n$) of class $\ca$ (as a function $\NN_1 \to Y$) with $d_X(x,\eta(y_n)) \leq 1/2^n$.
We also say that $x$ is a \emph{$\ca$-point in $X$}.
\end{definition}

\begin{example} \label{215}
When $X = \RR$, the above definition corresponds to the usual notion of a real number of class $\ca$ (in the Cauchy sense).
\end{example}

\begin{remark} \label{216}
It would seem natural to require that the class $\ca'$ contain the class $\ca$, but this is not completely essential. This remark applies to almost all the definitions that follow.
\end{remark}

\mni
{\bf Definition \ref{214}bis}
{\em (complexity of a point in a rationally presented metric space, more explicit version)} 
 Consider a complete metric space $X$ given in a presentation of class $\ca'$: $(Y, \delta,\eta)$. 
\\ 
A point $x$ of class $\ca$ in $X$ is given by a sequence ($n \mapsto y_n$) of class $\ca$ (as a map $\NN_1 \to Y$) satisfying the  conditions
$
\delta(y_n,y_{n+1},n+1) \leq 1/2^n
$ 
and $x=\lim_{n\to \infty}\eta(y_n)$.

\medskip
We leave it to the reader to check that the two definitions \ref{214} and \ref{214}bis are equivalent.
We now turn to the definition of complexity for a family of points (with a discrete set of indices)

\begin{definition}[complexity of a family of points in a rationally presented metric space] \label{217} ~

\noindent 
Consider a discrete preset $Z$ (this is the set of indices of the family, it is given as a part of a language $A^{\star}$ over a finite alphabet $A$) and a complete metric space $X$ given in a presentation of class $\ca':~(Y, \delta,\eta)$. 
A function (or family) $f\colon Z \to X$ is said to be of class $\ca$ when we know a map $\varphi\colon Z \times \NN_1 \to Y$ which is of class $\ca$ and which verifies:
\[
d_X(f(z),\eta(\varphi(z,n))) \leq 1/2^n \;{\rm for\, all}\; z \in Z.
\] 
We then say that $\varphi$ is {\em a presentation of class $\ca$ of the family $f(z)_{z\in Z}$ of points of $X$.}
\end{definition}

\begin{examples} \label{218}~

\noindent 
--- When $Z = \NN_1$ and $X = \RR$ the above definition corresponds to the usual notion of ``sequence of real numbers of class $\ca$'' (we also say: ``a $\ca$-sequence in $\RR$''). 

\noindent 
--- The definition of a space rationally presented in class $\ca$ can be reread in the following way. We give: 

\noindent 
\spa -- a $\ca$-part $Y$ of a language $A^{\star}$. 

\noindent 
\spa -- a function $\varphi\colon Y \to X$ such that $\varphi(Y)$ is dense in $X$ and such that the family of real numbers $\; \; Y \times Y\to \RR$, $\; \; (y_1, y_2) \mapsto d_X(\varphi(y_1), \varphi(y_2))$ is of class $\ca$. 
\end{examples}
 
\begin{proposition} \label{219}
Let $(x_n)$ be a sequence of class $\ca$ (the index set is $\NN_1$) in a rationally presented metric space $(X,d)$. If the sequence is explicitly Cauchy with the inequality $d(x_n, x_{n+1}) \leq 1/2^n$ then the limit of the sequence is a point of class $\ca$ in $X$.
\end{proposition}

\begin{proof} Let $(Y, \delta, \eta)$ be a rational representation of $(X,d)$ and $x$ be the limit of the sequence $(x_n)$.
The sequence $(x_n)$ is of class $\ca$ in $X$, so there exists a function $\psi\colon \NN_1 \times \NN_1 \to Y$ of class $\ca$ such that: 
\[
d(x_n,\psi(n,m)) \leq 1/2^m \; \;{\rm for\; all}\; n, m \in \NN_1.
\]
Let $z_n = \psi(n+1,n+1)$,  this is a sequence of class $\ca$ and
\[
d(x,z_n) \leq d(x,x_{n+1}) + d(x_{n+1},\psi(n+1,n+1)) \leq 1/2^{n+1} +
1/2^{n+1} = 1/2^n.
\]
So $x$ is a $\ca$-point in $X$. 
\end{proof}

\begin{remark} \label{2110}
If we slow down the speed of convergence of the Cauchy sequence sufficiently, we can obtain as the limit point of a sequence of class $\DTI (n^2)$ an arbitrary recursive point of $X$ (cf.\ \cite{KF82} for the space $[0,1]$).
\end{remark} 

\mni
{\bf Definition \ref{217}bis }
{\em (complexity of a family of points in a metric space, more explicit version) }

\noindent 
Consider a discrete preset $Z$ (a part of a language $A^{\star}$) and a complete metric space $X$ given in a presentation of class $\ca': (Y,\delta,\eta).$ 
A family $ f\colon Z \to X$ is said to be of class $\ca$ if we have a function $ \varphi\colon Z \times \NN_1 \to Y$ of class $\ca$ which verifies
$
\delta(\varphi(z,n),\varphi(z,n+1),n+1) \leq 1/2^n 
$
and $f(z)=\lim_{n\to \infty}\eta(\varphi(z,n))$  for all  $z \in Z$.

\subsection{Complexity of uniformly continuous maps} \label{subsec22}

We begin by giving a ``reasonable'' definition, which will be justified by the examples and propositions that follow.

\begin{definition}[complexity of uniformly continuous maps between rationally presented metric spaces] \label{221}
Consider two complete metric spaces $X_1$ and $X_2$ given in presentations of class $\ca'$: 
$(Y_1,\delta_1,\eta_1)$ and $(Y_2,\delta_2,\eta_2)$. 
Let $f\colon X_1 \to X_2$ be a uniformly continuous map and $\mu\colon \NN_1 \to \NN_1$ be a sequence of integers.

\noindent 
We will say that $\mu$ is {\em a modulus of uniform continuity for the map $f$} if we have: 
\[
d_{X_1}(x,y) \leq 1/2^{\mu(n)} \Rightarrow d_{X_2}(f(x),f(y)) \leq
1/2^n \;{\rm for\, all}\; x, y \in X_1 \;{\rm and}\; n \in \NN_1
\]
We will say that the function $f$ is {\em uniformly of class} $\ca$ (for the considered presentations) when

\noindent 
--- it has a modulus of uniform continuity $\mu\colon \NN_1 \to \NN_1$ in the class $\ca$.

\noindent 
--- the restriction of $f$ to $Y_1$ is in the class $\ca$ in the sense of Definition \ref{217}, i.e.\ it is presented by a function $\varphi\colon\NN_1 \times Y_1 \to Y_2 $ which is in the class $\ca$ and which verifies:
\[
d_{X_2}(f(y),\eta_2(\varphi(y,n))) \leq 1/2^n \; {\rm for\, all}\;  y \in 
Y_1.
\]	
When $X_1 = X_2 =X$ the map $\Id_X$ has $\Id_{\NN_1}$ as its modulus of uniform continuity. If the two maps~$\Id_X$ (from $X_1$ to $X_2$ and from $X_2$ to $X_1$) are in the class $\ca$ we will say that {\em the two presentations are (uniformly) $\ca$-equivalent}.
\end{definition}

\begin{remarks} \label{222}~

\noindent 
1) We are not asking that, for $n$ fixed, the function $y \mapsto\varphi(y,n)$ be uniformly continuous on $Y_1$ or even that it be continuous at every point of $Y_1$.

\noindent 
2) Note that the definition here gives a modulus of uniform continuity corresponding (very nearly) to the definition given in  the classical approximation theory. A modulus of uniform continuity is there a nondecreasing sequence of integers 
$\NN\to\ZZ: n \mapsto \nu(n)$ which tends to $+\infty$ and verifies the following implication.
\[
d_{X_1}(x,y) \leq 1/2^n \Rightarrow d_{X_2}(f(x),f(y)) \leq 1/2^{\nu(n)}
\]
The function $\nu$ is a kind of ``reciprocal'' of the function $\mu$.
\end{remarks}

\begin{examples}\label{223}~

\noindent 
1) When $X_1 = [0,1]$, $X_2 = \RR$ and $\ca$ is an \emph{elementarily stable} class of complexity in time or space, the above definition is equivalent to the usual notion of a real function computable in the class $\ca$ (cf.\ \cite{KF82}, \cite{Ko91}), as we will prove in detail in Proposition~\ref{415}. 

\noindent 
2) When $X_1$ is a discrete metric space, Definition~\ref{221} gives back  Definition~\ref{217}. 
\end{examples}

Definition \ref{221} makes the notion of complexity accessible for uniformly continuous maps. It therefore deals with all continuous maps in the case where the initial space is compact. 
Remember that in constructive analysis, in the case of a compact space, continuity is defined as meaning uniform continuity (cf.\ \cite{Bi}) and not pointwise continuity. 

\begin{remark}\label{224} 
The control of continuity given by the modulus of uniform continuity is essential in Definition \ref{221}. For example, we can define a function $\varphi\colon [0,1] \to \RR$ which is continuous in the classical sense, whose restriction to $\DD_{[0,1]} = \DD\cap [0,1]$ is $\LINT$ but which does not have a recursive modulus of uniform continuity. For this we consider a function $\theta\colon \NN \to \NN$ of class $\LINT$ that is injective and has non-recursive range. For each $m \in \NN_1$ we consider $n = \theta(m)\; , a_n = 1/(3.2^n)$ and we define $\psi_m\colon [0,1] \to\RR$ everywhere zero except on an interval centred in $a_n$ on which the graph of $\psi_m$ makes a peak of height $1/2^n$ with a slope equal to $1/2^m$. 
Finally, we define $\varphi$ as the sum of the series $\sum_m \psi_m$. Although the sequence $\varphi(a_n)$ is not a recursive sequence of real numbers (which implies that $\varphi$ cannot have a recursive modulus of uniform continuity), the restriction of $\varphi$ to dyadic numbers is very easy to calculate (the disputed real numbers $a_n$ are not dyadic). 
This example clearly shows that {\em the classical definition of continuity (continuity at any point) does not give access to the calculation of the values of the function from its restriction to a dense part of the initial space}.
\end{remark}

\begin{remark}\label{225}
In \cite{KF82} it is ``proved'' that if a functional defined via an Oracle Turing Machine (OTM) computes a function from $[0,1]$ to $\RR$, then the function has a recursive modulus of uniform continuity. If we avoid all recourse to non-constructive principles (hidden in the Heine-Borel theorem), the proof of \cite{KF82} can be easily transformed into a constructive proof of the following theorem: 
\emph{if a function $f:[0,1] \to \RR$ is uniformly continuous and computable by an OTM then its modulus of uniform continuity is recursive}.

\noindent 
On the other hand, if we make the (by no means implausible) assumption that every oracle in~a TM is provided by a mechanical (but unknown) process, it is possible to define ``pathological'' functions from $[0,1]$ into $\RR$ by means of OTMs: these are functions that are continuous at any recursive real point but not uniformly continuous. 
This is based on the ``singular tree of Kleene'': an infinite recursive binary tree which has no infinite recursive branch. Cf.\ \cite{Be}, theorem from Section 7 of Chapter 4, page 70 where we are given a function $t(x)$ which is continuous at any recursive real point on $[0,1]$, but which is not bounded (hence not uniformly continuous) on this interval.

\noindent 
It seems strange that in a proof concerning computability issues, one can use without even mentioning it the hypothesis that the oracles of an OTM behave antagonistically with Church's Thesis (at least in the form in which it is admitted in Russian constructivism).
\end{remark}

We can easily verify that Definition \ref{221} can be translated into the following more explicit form. 

\mni
{\bf Definition \ref{221}bis}
{\em (complexity of uniformly continuous functions between rationally presented metric spaces, more explicit form).}
Consider two complete metric spaces $X_1$ and $X_2$ given in presentations of class $\ca'$,  $(Y_1,\delta_1,\eta_1)$ and $(Y_2,\delta_2,\eta_2). $ A uniformly continuous function $f\colon X_1 \to X_2 $ is said to be {\em uniformly of class $\ca$} (for the considered presentations) when it is presented by means of two data 

\noindent 
--- the restriction of $f$ to $Y_1$ is presented by a function $\varphi\colon Y_1\times \NN_1 \to Y_2$ which is of class $\ca$ and which verifies 
\[
\delta_2(\varphi(y,n),\varphi(y,n+1),n+1) \leq 1/2^n \; \rm{for \; all} \; 
y \in Y_1,
\]
with $f(y)=\lim_{n\to \infty}\eta_2(\varphi(y,n))$.

\noindent 
--- a sequence $\mu\colon \NN_1 \to \NN_1$ in the class $\ca$ satisfying: for all $x$ and $y$ in $Y_1$
\[
\delta_1(x,y,\mu(n))\le 1/2^{\mu(n)} \; \Rightarrow\; \delta_2(\varphi(x,n+2),\varphi(y,n+2),n+2) \leq 1/2^n.
\]

\medskip
The following results are easy to establish.

\begin{proposition} \label{226}
The composite of two maps uniformly of class $\ca$ is a map uniformly of class~$\ca$. 
\end{proposition}

\begin{proposition} \label{227}
The image of a point (respectively of a family of points) of class $\ca$ by a map uniformly of class $\ca$ is a point (respectively of a family of points) of class $\ca$. 
\end{proposition}

\begin{remark} \label{228}
Consider two rational presentations of the same complete metric space $X$ given by the two families of rational points $(\wi{y})_{y \in Y}$ and $(\wi{z})_{z \in Z}$ respectively. If the first family is $\LINT $ for the first presentation, it can naturally be of greater complexity for the second. \\
In general, {\em to say that the identity of $X$ is uniformly of class $\ca$ when we go from the first to the second presentation is exactly the same as saying that the family $(\wi{y})_{y \in Y}$ is a $\ca$-family of points for the second presentation}. \\
Thus the two presentations are $\ca$-equivalent if and only if the family $(\wi{y})_{y \in Y}$ is a family of class $\ca$ in the second presentation and $(\wi{z})_{z \in Z}$ is a family of class $\ca$ in the first. 
Proposition~\ref{227} allows us to formulate an analogous and more intrinsic result: {\em two rational presentations of the same complete metric space are $\ca$-equivalent if and only if they define the same families of points of class $\ca$}.
\end{remark}

We now turn to the definition of complexity for a family of uniformly continuous maps (with a discrete set of indices).

\begin{definition}[complexity of a family of uniformly continuous maps between rationally presented metric spaces] \label{229} 
Consider a discrete preset $Z$ (this is the set of indices of the family, it is given as a part of a language $A^{\star}$ over a finite alphabet $A$) and two complete metric spaces $X_1$ and $X_2$ given in 
$\ca'$-presentations: $(Y_1,\delta_1,\eta_1)$ and $(Y_2,\delta_2,\eta_2)$.

\noindent 
We denote by $\U(X_1,X_2)$ the set of uniformly continuous maps from $X_1$ into $X_2$. A family of uniformly continuous maps $\wi{f}: Z \to \U(X_1,X_2)$ is said to be {\em uniformly of class $\ca$} (for the considered presentations) when

\noindent 
--- the family has a modulus of uniform continuity $\mu\colon Z \times \NN_1 \to\NN_1$ in class $\ca$ : the function $\mu$ must satisfy 
\[
\forall n \in \NN_1\;\forall x,
y \in X_1 \;\; \; \; \big(d_{X_1}(x,y) \leq 1/2^{\mu(f,n)} \Rightarrow 
d_{X_2}(\wi{f}(x),\wi{f}(y)) \leq 1/2^n \big)
\] 
--- the family $(f,x) \mapsto \wi{f}(x)\colon Z \times X_2 \to
X_2$ is a family of points in $X_2$ of class $\ca$ in the sense of  Definition~\ref{217}, i.e.\ it is presented by a map $\varphi: Z \times Y_1 \times \NN_1 \to Y_2$ which is of class $\ca$ and which verifies:
\[
d_{X_2}\big(\wi{f}(x),\varphi(f,x,n)\big) \leq 1/2^n \; {\rm for\, all} \; (f, x, n) \in Z \times Y_1 \times \NN_.
\]
\end{definition}

\begin{remark} \label{2210}
The notion defined in 2.2.9 is natural because it is the relativisation to the class~$\ca$ of the constructive notion of a family of uniformly continuous maps. 
But this natural notion does not seem to be deducible from Definition \ref{221} by providing $Z \times X_1$ with a suitable structure of rationally presented metric space and by requiring that the map $(f,x) \mapsto \wi{f}(x)\colon Z \times X_2 \to X_2$ be uniformly of class $\ca$. If, for example, we take on $Z \times X_1$ the metric deduced from the discrete metric of $Z$ and the metric of $X_1$ we will obtain the second of the conditions of Definition \ref{229} but the first will be replaced by the requirement that all the maps of the family have the same modulus of uniform continuity of class $\ca$. In other words, the family should be uniformly equicontinuous. This condition is intuitively too strong.  Propositions \ref{2211} and \ref{2212} which follow are a confirmation that Definition \ref{229} is suitable.
\end{remark}

Propositions \ref{226} and \ref{227} generalise to the case of families of maps. The proofs present no difficulty.

\begin{proposition} \label{2211}
Let $X_1$ and $X_2$ be two rationally presented complete metric spaces. Let $\big(\wi{f}\,\big)_{f \in Z}$ be a family in $\U(X_1,X_2)$ uniformly of class $\ca$ and $(x_n)_{n \in \NN_1}$ be a family of class $\ca$ in $X_1$.

\noindent 
Then the family $(\wi{f}(x_n))_{(f,n) \in Z \times \NN_1}$ is a family of class $\ca$ in $X_2$.
\end{proposition}

\begin{proposition} \label{2212}
Let $X_1$, $X_2$ and $X_3$ be three rationally presented complete metric spaces. Let $\big(\wi{f}\,\big)_{f \in Z}$ be a family in $\U(X_1,X_2)$ uniformly of class $\ca$ and $(\wi{g})_{g \in Z'}$ be a family in $\U(X_2,X_3)$ uniformly of class $\ca$.  
Then the family $(\wi{f} \circ \wi{g})_{(f,g) \in Z \times Z'}$ in $\U(X_1,X_3)$ is uniformly of class $\ca$.
\end{proposition}

\subsection{Complexity of locally uniformly continuous maps} \label{subsec23}

The notion of complexity defined in the previous paragraph for uniformly continuous maps is entirely legitimate when the domain space is compact. 
In the case of a locally compact space in the sense of Bishop,\footnote{This means that every bounded subset is contained in a compact space \cite{BB}.} the  maps which are uniformly continuous on every bounded subset (from the classical point of view this is a theorem, from the constructive point of view it is a definition). This leads to the  following natural notion of complexity.

\begin{definition} \label{231}
Consider two complete metric spaces $X_1$ and $X_2$ given in presentations of class $\ca'$: $(Y_1,\delta_1,\eta_1)$ and $(Y_2,\delta_2,\eta_2).$ We assume that we have specified a point $x_0$ of $X_1$ and a point $y_0$ of~$X_2$. A map $f: X_1 \to X_2$ is said to be {\em locally uniformly continuous} if it is uniformly continuous and bounded on any bounded part.

\noindent 
It is said to be {\em locally uniformly of class $\ca$} (for the considered presentations) when

\noindent 
--- it has a bound computable in class $\ca$ on any bounded part, i.e.\ a $\ca$-sequence $\beta\colon \NN_1 \to \NN_1$ satisfying, for all $x$ in $X_1$ and all $m$ in $\NN_1$: 
\[
d_{X_1}(x_0,x) \leq 1/2^n \Rightarrow d_{X_2}(y_0,f(x)) \leq 1/2^{\beta(n)};
\]
--- it has in the class $\ca$ a modulus of uniform continuity on all bounded subsets, i.e.\ a $\ca$-map $\mu\colon \NN_1 \times \NN_1 \to \NN_1$ verifying, for $x, z$ in $X_1$ and $n, m$ in $\NN_1$: 
\[
\left( d_{X_1}(x_0,x) \leq 1/2^m, \; d_{X_1}(x_0,z) \leq 1/2^m, \; d_{X_1}(x,z) \leq 1/2^{\mu(m,n)}\right) \; \Rightarrow\; d_{X_2}(f(x),f(z)) \leq 1/2^n;
\] 
--- the restriction of $f$ to $Y_1$ is in class $\ca$ in the sense of Definition \ref{217}, i.e.\ it is presented by a $\ca$-map $\varphi:\NN_1 \times Y_1 \to Y_2 $ which verifies:
\[
d_{X_2}(f(y),\varphi(y,n)) \leq 1/2^n \;{\rm{for all}} \; y \in Y_1.
\]
\end{definition}
 
Note that the concept defined above does not depend on the choice of points
$x_0$ and $y_0$.
\begin{examples} \label{232}~

\noindent 
--- When $X_1 = \RR$ and $X_2 = \RR$ the above definition is equivalent to the natural notion of a real function computable in the class $\ca$ as found in \cite{Ho90} and \cite{Ko91}.

\noindent 
--- The function $x \mapsto x^2$ is locally uniformly continuous of class $\QL$ but it is not uniformly continuous on $\RR$.
\end{examples}

\begin{remarks}\label{233} ~

\noindent 
1) When $f$ is a uniformly continuous function, Definition \ref{231} and Definition \ref{221} are very similar. However, Definition \ref{221} is a priori more restrictive. Indeed, if $\mu(n)$ is a modulus of uniform continuity in Definition \ref{221}, then we can take in Definition \ref{231} $\mu'(m,n) = \mu(n)$. 
But conversely, suppose we have a modulus of uniform continuity on any bounded subset verifying $\mu'(m,n) = \inf(m, 2^n)$ then, the function is uniformly continuous but the best modulus of uniform continuity that we can deduce~is 
\[
\mu(n) = {\bf Sup} \left\{ \mu'(m,n): m \in \NN \right\} = 2^n
\]
and it has an exponential growth rate whereas $\mu'(m,n)$ is linear. Thus the function $f$ can be of linear complexity as a locally uniformly continuous function, and of exponential complexity as~a uniformly continuous function. 
The two definitions are equivalent if the space $X_1$ has a finite diameter.

\noindent 
2) In constructive analysis, a compact space is a precompact and complete space. It is not possible to prove constructively that any closed part of a compact space is compact. A (uniformly) locally compact space is a complete metric space in which every bounded subset is contained in~a compact~$K_n$, where the sequence $(K_n)$ is an nondecreasing sequence given once and for all (for example, the compact $K_n$ contains the ball $B(x_0, 2^n)$). A function defined on such a space is then said to be continuous if it is uniformly continuous on all bounded subsets. In particular, it is bounded on any bounded subset. In this case, Definition \ref{231} gives an ``algorithmic complexity'' version of the constructive definition of continuity.

\noindent 
3) Proposition \ref{226} about the composition of functions uniformly of class $\ca$ remains valid for functions locally uniformly of class $\ca$. Proposition \ref{227} is also valid.
\end{remarks}

\subsection{A general approach to the complexity of continuous functions} 
\label{subsec24}
 
The question of the complexity of continuous functions has clearly not been exhausted. 
As we have already pointed out, while it is true that a continuous function $f(x)$ is classically well known from a presentation $(y, n) \mapsto \varphi(y,n)$ which allows it to be computed with arbitrary precision on a dense part $Y$ of $X$, the complexity of $\varphi$ can only be considered a pale reflection of the complexity of $f$ (cf.\ remarks \ref{224} and \ref{233}(1)).
A crucial question, little studied until now, is to what extent we can certify that such a $\varphi$ datum corresponds to a continuous function $f$. If the answer is in the affirmative, we need to explain the procedure for calculating approximate values of~$f(x)$ when~$x$ is an arbitrary point of $X$ (given, for example, by a Cauchy sequence of rational points of the presentation in the case of a rationally presented metric space).

\noindent 
We propose a somewhat informal approach to this question.
Let $\phi\colon X_1 \to X_2$ be a map between metric spaces and $(F_{\alpha})_{\alpha \in M}$ a family of subsets of $X_1$. \emph{A modulus of uniform continuity for $\phi$ near each subset $F_{\alpha}$} is by definition a map $\mu \colon M \times \NN_1 \to \NN_1$ verifying:	
\[ 
\forall \alpha \in M \; \forall x \in F_{\alpha} \; \forall x'\in X_1\; \; \; \left(d_{X_1}(x,x') 
\leq 1/2^{\mu(\alpha,n)} \Rightarrow d_{X_2}(\phi(x),\phi(x'))
 \leq 1/2^n \right).
\]
 
\begin{definition}[general but somewhat informal definition of what a continuous function is and what its complexity is]\label{241}

Consider two complete metric spaces $X_1$ and $X_2$ given in rational presentations of class $\ca'$: $(Y_1,\delta_1,\eta_1)$ and $(Y_2,\delta_2,\eta_2)$. We assume that we have specified a $y_0\in X_2$.
Suppose we have defined a family $F_{X_1} = (F_{\alpha})_{\alpha \in M}$ of parts of $X_1$ with the following property:

\noindent 
--- (2.4.1.1) every compact of $X_1$ is contained in some $F_{\alpha}$.

\noindent 
We will say that a map $\phi: X_1 \to X_2$ is {\em $F_{X_1}$-uniformly continuous} if it is bounded on each part $F_{\alpha}$ and if it has a modulus of uniform continuity near each subset $F_{\alpha}$.

\noindent 
Suppose now in addition that the family $(F_{\alpha})_{\alpha \in M}$ has the following property:

\noindent 
--- (2.4.1.2) the set of indices $M$ has a certain computability structure.

\noindent 
We will say that a map $\phi: X_1 \to X_2$ is \emph{$F_{X_1}$-uniformly continuous of class $\ca$} if it verifies the following properties:

\noindent 
--- a bound on each part $F_{\alpha}$ can be calculated as a map $ \beta\colon M \to \NN_1$ in class $\ca$ verifying  
\[
\forall \alpha \in M \; \forall x \in F_{\alpha} \; \;
 d_{X_2}(y_0, \phi (x)) \leq 2^{\beta(\alpha)};
\] 
--- a modulus of uniform continuity near to $F_{\alpha}$ can be calculated in the class $\ca$: a $\ca$-map $\mu\colon M \times \NN_1 \to \NN_1$  satisfying
\[
\forall \alpha \in M \; \forall x \in F_\alpha\; \forall x' \in X_1\;\;  \left(d_{X_1}(x,x') \leq 1/2^{\mu(\alpha,n)} \Rightarrow d_{X_2}(\phi(x),\phi(x')) \leq 1/2^n \right).
\]
--- the restriction of $\phi$ to $Y_1$ is $\ca$-computable.
\end{definition}

This is clearly an extension of Definitions \ref{221} and \ref{231}.
The informal nature of the definition is obviously due to the ``computability structure'' of $M$.
A priori we would like to take for $(F_{\alpha})_{\alpha \in M}$ a family of parts sufficiently simple to verify condition (2.4.1.2) and sufficiently large to verify condition (2.4.1.1).
But these two conditions pull in opposite directions.

Note that Definition \ref{241} is inspired by the notion of a continuous map defined by D.~Bridges in  \cite[Constructive functional analysis]{Br}.

\subsection {Open and closed subspaces} \label{subsec25}

We turn to the definition of a computability structure for an open subspace $U$ of a metric space~$X$ given with a computability structure. In general, the metric induced by $X$ on $U$ cannot satisfy us because the resulting space is generally not complete. Nevertheless, we have a construction that works in an important special case.

Let $X$ be a complete metric space and $f: X \to \RR$ a locally uniformly continuous map. The open $U_f = \{x \in X: f(x) > 0 \}$ is a complete metric space for the distance $d_f$ defined by 
\[
d_f(x,y) = d_X(x,y) + \abs{1/f(x) - 1/f(y)}.
\]

\begin{propdef} \label{251}
Let $X$ be a complete metric space given with a $\ca$-presentation $(Y, \delta, \eta)$, and $f: X \to \RR$ a locally uniformly  $\ca$-map represented by a modulus of uniform continuity and a discrete $\ca$-map,
$ \varphi\colon Y \times \NN_1 \to \DD$. Then the complete metric space $(U_f, d_f)$ can be provided with a $\ca$-presentation where the set of rational points is coded by the set $Y_U$ of pairs $(y,n)$ of $Y \times \NN_1$ satisfying: 
\[
n \geq 10 \; \; {\rm and} \; \; \varphi(y,n) \geq 1/2^{n/8}.
\]
\end{propdef}
\begin{proof} A detailed proof is given in \cite{Mo}. 
\end{proof}

We now turn to the definition of a computability structure for a closed subspace $F$ of a metric space $X$ with a computability structure. In constructive analysis, a closed subset is only really useful when it is \emph{located}, i.e.\ when the function $D_F$ ``distance to the closed subset $F$'' is computable. When translating this notion into terms of recursive computability or complexity, we must take care that in the constructive definition, the fact that the function $D_F$ is the function ``distance to the closed subset $F$'' must also be made explicit.

\begin{definition} \label{252}
Consider a complete metric space $X$ given by a $\ca$-presentation $(Y, \delta, \eta)$. 
A part~$F$ of $X$ is called a \emph{$\ca$-located closed subset of $X$} if:
\begin{itemize}
\item [i)] the function $D_F\colon x \mapsto d_X(x,F)$ from $X$ to $\R^{\geq 0}$ is $\ca$-computable,
\item [ii)] there exists a $\ca$-function  $P_F\colon Y \times \NN_1\to X$ which certifies the function $D_F$ in the following sense: 
for all $(y,n) \in Y \times \NN_1$,
\[
D_F(P_F(y,n)) = 0,\; \; \; \; d_X(y,P_F(y,n)) \leq D_F(y) + 1/2^n.
\] 
\end{itemize}
\end{definition}

\begin{remarks} \label{253}~

\noindent 
1) The function $P_F(y,n)$ calculates an element of $F$ whose distance from $y$ is close enough to $D_F(y)$. However, the function $P_F(y,n)$ does not in general define, by extension by continuity to~$X$ and by passage to the limit when $n$ tends to $\infty$, a projector on the closed subset $F$. 

\noindent 
2) We can show that the points $P_F(y,n)$ (coded by the pairs $(y,n) \in Y \times \NN_1$) form a dense countable part of $F$ which is the set of rational points of a $\ca$-presentation of $F$ (cf.\ \cite{Mo}).
\end{remarks}

\subsection{Rationally presented Banach spaces}\label{subsec26}

We give a minimal definition. It goes without saying that for each particular Banach space, notions of complexity naturally attached to the space under consideration may be taken into account in addition to obtain a truly reasonable notion. 

\begin{definition}[rational presentation of a Banach space] \label{261} 
A rational presentation of a separable Banach space $X$ over the field $\KK$ ($\RR$ or $\cit$) will be said to be of class $\ca$ if, on the one hand, it is of class $\ca$ as a presentation of the metric space and, on the other hand, the following vector space operations are in class $\ca$

\noindent 
--- the product by a scalar,

\noindent 
--- the sum of a list of vectors chosen from the rational points.

\end{definition}

In the context of the field $\CC$ we will denote by $\DD_\KK$ the set of complex numbers whose real and imaginary parts are in $\DD$. In the real context $\DD_\KK$ will only be another denomination of $\DD$.

\noindent 
The following proposition is not difficult to establish. 

\begin{proposition} \label{262}
To give a rational presentation of class $\ca$ of a Banach space $X$ in the sense of the definition above is equivalent to giving  

\noindent 
--- an encoding of class $\ca$ for a countable part $G$ of $X$ which generates $X$ as a Banach space,\footnote{We can suppose that all the elements of $G$ are vectors of norm between 1/2 and 1.}

\noindent 
--- a map $\nu\colon \lst(\DD_{\KK} \times G) \times \NN_1 \to \DD$ which is in the class $\ca$ (for the considered encoding of $G$) and which calculates the norm of a linear combination of elements of $G$ in the following sense:

\noindent 	
for all $\big([(x_1,g_1),(x_2,g_2),\ldots,(x_n,g_n)],m\big)$ in $\lst(\DD_{\KK} \times G) \times \NN_1$, we have the inequality
\[
\Abs{\;\nu\big([(x_1,g_1),(x_2,g_2),\ldots,(x_n,g_n)],m\big) - \Norme{ x_1.g_1 + x_2.g_2 +\cdots.+ x_n.g_n }_X \; }\; \leq 1/2^m.
\]
\end{proposition}

\begin{remarks}\label{263}~

\noindent 
1) As for Definition \ref{211} (cf.\ remark \ref{212} (3)), it may be that the certificate of inclusion of $G$ in~$X$ implies a notion of complexity, which it will be inevitable to take into account in a more precise definition, on a case-by-case basis.

\noindent 
2) The spaces $L^p(\RR)$ of functional analysis, with $1 \leq p < \infty$ can be rationally presented in different ways, according to different possible choices for the set of rational points and an encoding of this set. All reasonable choices turn out to give $\Prim$-equivalent presentations. 

\noindent 
3) It would be interesting to know whether the framework proposed by M. Pour El and I. Richards \cite{PR} concerning computability in Banach spaces can have concrete consequences which would go beyond what can be treated by rational presentations (which offer a natural framework not only for recursion problems but also for complexity problems). The ``counter example'' concerning $L^\infty$ given in \cite{PR} suggests the opposite. 
\end{remarks}

\section[Continuous real functions on a compact interval \ldots]{Space of continuous real functions on a compact interval, first properties}\label{sec3}

In this section, we introduce the problem of rational presentations for the space $\czu$: the space of uniformly continuous real functions on the interval [0,1], given the usual norm:
\[
\norme{f}_{\infty} = {\bf Sup} \{ \abs{f(x)} ; 0 \leq x \leq 1 \}.
\] 
Consider a rational presentation of the space $\czu$ given by a family $\big(\wi{f}\,\big)_{f \in Y}$ of uniformly continuous functions indexed by a part $Y$ of a language $A^{\star}$. We are interested in the following complexity problems:

\noindent 
--- the complexity of the set of (codes of) rational points of the presentation, i.e.\ more precisely the complexity of $Y$ as a part of the language $A^{\star}$ (i.e.\ the complexity of the membership test);

\noindent 
--- the complexity of vector space operations (the product by a scalar on the one hand, the sum of a list of vectors on the other);

\noindent 
--- the complexity of calculating the norm (or distance function);

\noindent 
--- the complexity of the set $\big(\wi{f}\,\big)_{f \in Y}$ of rational points of the presentation, as a family of uniformly continuous functions on $[0,1]$;

\noindent 
--- the complexity of the evaluation function $\Ev\colon \czu \times [0,1] \to \RR$: $ (g,x) \mapsto g(x)$.

\noindent 
It goes without saying that the interval $[0,1]$ can be replaced by another interval $[a,b]$ with $a$ and~$b$ in $\DD$ or of low complexity in $\RR$. 

\subsection{Definition of a rational presentation of the space of continuous functions}\label{subsec31}
 
The complexity of the set $\big(\wi{f}\,\big)_{f \in Y}$ of rational points of the presentation, as a family of uniformly continuous functions on $[0,1]$ is nothing other than the complexity of the map $f \mapsto \wi{f}$ from the set of codes of rational points $Y$ to the space $\czu$. 
We must therefore, in accordance with what we said in remark \ref{212}(3), include in the definition of what is a rational presentation of class $\ca$ of $\czu$, the fact that $\big(\wi{f}\,\big)_{f \in Y}$ is a family uniformly of class $\ca$ in the sense of Definition \ref{229}.
The problem of the complexity of the evaluation function is also an important one, since it would be ``immoral'' for the evaluation function not to be a function of class $\ca$ when we have a rational presentation of class $\ca$. 
However, the evaluation function is not uniformly continuous, nor even locally uniformly continuous.

\noindent 
To deal in general with the question of functions which are continuous but not locally uniformly continuous on the space $\czu$ we use the informal Definition \ref{241} with the following family of parts of $\czu$: 

\begin{notation} \label{311}
If $\alpha$ is a nondecreasing function from $\NN_1$ to $\NN_1$ and $r \in \NN_1$, we denote by $F_{\alpha,r}$ the part of $\czu$ formed by all the functions which, on the one hand, accept $\alpha$ as a modulus of uniform continuity, and on the other hand have their norm bounded by $2^r$.
\end{notation}

The ``calculability structure'' on the set of indices 
\[
M := \{(\alpha,r); \; \alpha \; \hbox{is a non-decreasing function} \; \NN_1\to\NN_1 \;{\rm and}\;r\in\NN_1\}
\]
is not a well-defined thing in the literature, but we will only need to resort to perfectly elementary operations such as ``evaluating $\alpha$ on an integer $n$''.\footnote{Note that the classical Ascoli theorem states that every compact part of $\czu$ is contained in a part $F_{\alpha,r}$, and that the parts $F_{\alpha,r}$ are compact. In constructive mathematics the direct part is still valid, but the second part of the statement needs to be refined, cf.\ \cite{BB} Chapter 4 Theorem 4.8, pages 96 to 98.}
The modulus of uniform continuity for the evaluation function is then ``very simple'' (uniformly linear for any reasonable definition of this notion).

\noindent 
Indeed, near the part $F_{\alpha , r} \times [0,1]$ of $\czu \times [0,1]$ a modulus of uniform continuity for the function $\Ev$ is given by 
\[
\mu(n,\alpha,r) = \max (\alpha(n+1),n+1) \; {\rm for} \; n \in \NN_1 \; {\rm and} \; (\alpha,r) \in M,
\] 
as is very easy to check. And the bound on $F_{\alpha ,r}$ is obviously given by $\beta(\alpha,r) = r$.

\noindent 
The whole question of the complexity of the evaluation function in a given presentation is therefore concentrated on the question of the complexity of the evaluation function restricted to the set of rational points 
\[
(f,x) \mapsto \wi{f}(x)\qquad  Y \times \DD_{[0,1]} \to \RR.
\]
Now this complexity is subordinate to that of $\big(\wi{f}\,\big)_{f \in Y}$ as a family of uniformly continuous functions: this is what we specify in the following proposition (the proof of which is immediate).

\begin{proposition} \label{312}
Let us consider on the space $\czu$ the family of parts $(F_{\alpha,r})_{(\alpha,r) \in M}$ in order to control questions of continuity on $\czu$ (cf.\ notation \ref{311} and Definition \ref{241}).

\noindent 
Then if $\big(\wi{f}\,\big)_{f \in Y}$ is a family uniformly of class $\ca$ and if we consider the rational presentation of the metric space $\czu$ attached to this  family when we see it as the set of rational points of the presentation, the evaluation function 
\[
\Ev\colon \czu \times [0,1] \to \RR\colon (g,x) \mapsto g(x)
\]
is itself of class $\ca$.
\end{proposition}

As we also require that the Banach space structure itself be of class $\ca$, this finally gives us the following definition.

\begin{definition}\label{313}
A rational presentation of class $\ca$ of the space $\czu$ is given by a family of functions $\big(\wi{f}\,\big)_{f \in Y}$ which is a family uniformly of class $\ca$, dense in $\czu$ and such that the following calculations are also in class $\ca$:

\noindent 
--- the product by a scalar,

\noindent 
--- the sum of a list of functions chosen from the rational points,

\noindent 
--- the calculation of the norm.
\end{definition}

In what follows, we shall be interested in a precise study of the complexities involved in Definition \ref{313}. Our conclusion is that there is no paradise in polynomial time for continuous functions, at least if $\p \not= \np$. 

\subsection{Two significant examples of rational presentations of the space \texorpdfstring{$\czu$}{C[0,1]}}\label{subsec32}

We now give two significant examples of presentations of $\czu$ (other examples will be given later).

\subsubsection{Presentation by binary semilinear circuits} 
\label{subsubsec321} 

This presentation and the set of codes for rational points will be noted $\csl$ and $\ysl$ respectively. We will call {\em a semilinear map with coefficients in $\DD$} a piecewise linear map which is equal to a combination by $\max$ and $\min$ of functions $x \mapsto ax+b$ with $a$ and $b$ in $\DD$.

\begin{definition}
\label{321}
A binary semilinear circuit is a circuit whose input ports are ``real'' variables~$x_i$ (here, only one will suffice because the circuit calculates a function of a single variable) and the two constants 0 and 1. There is only one output port.

\noindent 
Gates that are not input gates are of one of the following types:

\noindent 
--- single-entry gates of the following types: $x \mapsto 2x$, $x \mapsto x/2$, $x \mapsto -x$;

\noindent 
--- two-entry gates of the following types: 
$(x,y) \mapsto x + y$, $(x,y) \mapsto \max(x,y)$, $(x,y) \mapsto \min(x,y)$.

\noindent 
A binary semilinear circuit with a single input variable defines a semilinear map with coefficients in $\DD$. Such a circuit can be encoded by an evaluation program. The set $\ysl$ is the set of (codes of these) semilinear circuits, it is the set of rational points of the presentation $\csl$.
\end{definition}

We shall see later that this presentation is in some way the most natural, but that it fails to be a presentation of class $\p$ because of the calculation of the norm.

\noindent 
We have an easy bound of a Lipschitz modulus for the map defined by the circuit: 
\[
\abs{\wi{f}(x) - \wi{f}(y)}\;  \leq 2^p \abs{ x- y} \; 
{\rm where} \; p \; \hbox{is the depth of the circuit}
\]
This gives the modulus of uniform continuity $\mu(k) = k+n$. In particular, this implies that we need not restrict ourselves to a controlled precision in $\DD_{[0,1]}$ in the following proposition.
 
\begin{proposition}[complexity of the family of functions $\big(\wi{f}\,\big)_{f \in \ysl}$] \label{322} 

\noindent 
The family of functions $\big(\wi{f}\,\big)_{f \in \ysl }$ is uniformly of class $\p$. To be precise, this family has $\mu(f,k) = k+\depth(f)$ as a modulus of uniform continuity and we can provide a function $\varphi\colon \ysl \times \DD_{[0,1]} \times \NN_1 \to \DD_{[0,1]}$ of class $\DRT(\Lin, \Oo(N^2))$ (where $N$ is the size of the input $(f,x,k)$) with 
\[
\forall (f, x, k) \in \ysl \times \DD_{[0,1]} \times \NN_1 \;\abs{\wi{f}(x) - \varphi(f,x,k)}\; \leq 1/2^k.
\]
More precisely still, as it is not necessary to read all the bits of $x$, the size of $x$ does not come into play, and the function $\varphi$ is in the classes $\DRT(\Lin, \Oo(\sz(f)(\depth(f)+k)))$ and $\DSPA(\Oo(\depth(f) (\depth(f)+k)))$.
\end{proposition}

\begin{proof} 
To calculate $\varphi (f,x,k)$ we evaluate the circuit $f$ on the input $x$ for which only the first $k + 2 \depth(f)$ bits are considered, truncating the intermediate result calculated at the $\pi$ gate to  precision $k+2\ \depth(f)-\depth(\pi)$. Finally, only the precision $k$ is kept for the final result.

\noindent 
Such a naively applied method requires us to store all the results obtained at a fixed depth $p$ while we calculate results at depth $p+1$.
We then do $\sz(f)$ elementary calculations 
$(\bullet + \bullet, \bullet - \bullet,\bullet \times 2, \bullet/2, \max(\bullet, \bullet),\min(\bullet, \bullet))$ 
on objects of size $\leq k + 2 \depth(f)$. Each elementary calculation takes  time $\Oo(k+\depth(f))$, and so the global calculation is done in time $\Oo(\sz(f)(\depth(f)+k))$. And it also takes $\Oo(\sz(f)(\depth(f)+k))$ as  calculation space.

\noindent 
There is another method of evaluating a circuit, which is a little less time-consuming but much more economical space-consuming, based on Borodin's idea \cite{Bo}. With such a method we reduce the calculation space which becomes $\Oo(\depth(f)(\depth(f)+k))$.
\end{proof}

\begin{remark}\label{323}
We have not taken into account in our calculation the problem posed by the management of $t = \sz(f)$ objects (here dyadic numbers) of sizes bounded by $s = \depth(f)+k$. In the RAM model this management would a priori be in time $\Oo(t \Lg(t) s)$ which does not significantly increase the $\Oo(t s)$ we have found, and which remains in $\Oo(N^2)$ if we remember that encoding the semilinear circuit by an evaluation program gives it a size of $\Oo(t \Lg(t))$. In the TM model, on the other hand, this management requires a priori time $\Oo(t^2s)$ because the tape where the objects are stored must be traversed $t$ times over a total length $\leq ts$. We have therefore underestimated time to some extent by concentrating on what we consider to be the central problem: estimating the total cost of the arithmetic operations themselves. In the following, we will systematically omit the calculation of the {\em management time} of intermediate values (which is very sensitive to the chosen calculation model) each time we evaluate circuits.
\end{remark}

\subsubsection{Presentation \texorpdfstring{$\crf$}{Crff}
(via controlled rational functions given in formula presentation)} 
\label{subsubsec322}

The previous presentation of the space $\czu$ is not of class $\p$ (unless $\p$ = $\np$ as we shall see in Section 4.3) because the norm is not computable in polynomial time. To obtain a presentation of class $\p$ it is necessary to restrict quite considerably the set of rational points in the presentation, so that the norm becomes computable in polynomial time. A significant example is where the rational points are well-controlled rational functions given in dense presentation.
There are several variants to choose from and we have chosen to give the denominator and numerator in a presentation known as ``by formula''. 
A formula is a tree whose leaves are labelled by the variable~$X$ or by an element of $\DD$ and whose nodes are labelled by an arithmetic operator. 
In the formulae we are considering, the only operators used are $ \bullet + \bullet $, $\bullet - \bullet $ and $\bullet \times \bullet$, so that the tree is a binary tree (each node of the tree is a sub-formula and represents a polynomial of $\DD$[X]).

\begin{lemma} 
\label{324} Let us denote by $\DD[X]_f$ the set of polynomials with coefficients in $\DD$, given in presentation by formula. For a one-variable polynomial with coefficients in $\DD$ the transition from the dense representation to the representation by formula is $\LINT$ and the transition from the representation by formula to the dense representation is polynomial. More precisely, if we proceed in a naive way, we are in $\DTI (\Oo(N^2{\cal M}(N))$).
\end{lemma}

\begin{proof} First of all, the dense representation can be considered as a special case of representation by formula, according to Horner's scheme.

\noindent 
 Then, moving from the formula representation to the dense representation means evaluating the formula in $\DD[X]$. Let's introduce the following control parameters. A polynomial $P \in \DD[X]$ has a degree denoted by $d_P$ and the size of its coefficients is controlled by the integer $\sigma(P) := \log(\sum_i \abs{a_i})$ where the $a_i$ are the coefficients of $P$. A formula $F \in \DD[X]_f$ contains a number of arithmetic operators denoted by $t_F$ and the size of its coefficients is controlled by the integer $\lambda(F) = \sum_i (\Lg(b_i))$ where the $b_i$ are the dyadics appearing in the formula.

\noindent 
The size $\flo F$ of the formula $F$ is obviously an upper bound of $t_F$ and $\lambda (F)$.

\noindent 
For  $a$ and $b\in\DD$ we always have $\Lg(a \pm b) \leq \Lg(a) + \Lg(b)$ and $ \Lg(ab) \leq \Lg(a)+ \Lg(b)$.

\noindent 
We can then easily check that $\sigma (P \pm Q) \leq \sigma(P) + \sigma(Q)$ and $\sigma (PQ) \leq \sigma (P) + \sigma (Q)$. The (naive) calculation time for $PQ$ is $\Oo(d_Pd_Q {\cal M}(\sigma(P) + \sigma(Q)))$.

\noindent 
We then show by induction on the size of the formula $F$ that the corresponding polynomial $P \in \DD[X]$ satisfies $d_P \leq t_F$ and $\sigma (P) \leq \lambda (F)$. We also show by induction that the time to calculate $P$ from $F$ is bounded by $t_F^2{ \cal M}(\lambda(F))$. 
 \end{proof}

\begin{definition}\label{325}
The set $\yrf \subset \DD[X]_f \times \DD[X]_f$ is the set of rational functions (with one variable) with coefficients in $\DD$, whose denominator is bounded below by $1$ on the interval $[0,1]$. The space $\czu$ provided with the set $\yrf$ as a family of codes of rational points is denoted by $\crf$.
\end{definition} 

\begin{proposition}[complexity of the family of maps $\big(\wi{f}\,\big)_{f \in \yrf }$] \label{326}  
The family of maps $\big(\wi{f}\,\big)_{f \in \yrf }$ is uniformly of class $\p$, more precisely of class $\DRT(\Lin, \Oo({\M}(N)N))$, where $\M(N)$ is the complexity of the multiplication of two integers of size~$N$.
\end{proposition}

\begin{proof} We need to compute a modulus of uniform continuity for the family $\big(\wi{f}\,\big)_{f \in \yrf }$. We must also specify a function $\varphi$  of class $\DRT(\Lin,\Oo(\M(N) N))$.
\[ 
\varphi\colon \yrf \times \DD_{[0,1]} \times \NN_1 \to \DD \; ,\;
(f,x,n) \mapsto \varphi(f,x,n)
\]
satisfying $\abs{\wi{f}(x) - \varphi(f,x,n)}\; \leq 1/2^n$.

\noindent In fact, the modulus of uniform continuity can be deduced from the calculation of $\varphi $.

\noindent 
Since the denominator of the fraction is bounded below by $1$, we only need to give a modulus of uniform continuity and a calculation procedure in time $\Oo(\M(N) N)$ to evaluate a formula $F \in \DD[X]_f$ with a precision of $1/2^n$ on the interval $[0,1]$ ($N = n+ \flo F$). We assume without loss of generality that the size $m = \flo F$ of the formula $F = F_1 * F_2$ is equal to $ m_1+m_2+2$ if $m_1 = \flo{F_1}$ and $m_2 = \flo{F_2}$ ( $*$ designates one of the operators $+,\;-$ or $\times$). 
We then establish (by induction on the depth of the formula) the following two facts:

\noindent 
--- when the formula $F \in \DD[X]_f$ is evaluated exactly at $x \in [0,1]$, the result is always bounded in absolute value by $2^m$, 

\noindent 
--- when we evaluate the formula $F$ in an $x \in [0,1]$ in an approximate way, taking $x$ and all the intermediate results with precision  $1/2^{n+m}$, the final result is guaranteed with precision $1/2^n$. 

\noindent 
We then conclude without difficulty.
\end{proof}

\begin{proposition} \label{327}~

\noindent 
a) The family of real numbers $(\Norme{ \wi{f}})_{f \in \yrf }$ is of complexity $\p$.

\noindent 
b) The membership test ``$\,f \in\yrf\; ?\,$'' is of complexity $\p$.
\end{proposition}

\begin{proof} 
a) Let $f = P/Q \in \yrf $. To calculate an approximate value to within $1/2^n$ of the norm of~$\wi{f}$, proceed as follows:

\noindent 
--- calculate $(P'Q - Q'P)$ as an element of $\DD[X]$ (\lemref{324}),

\noindent 
--- calculate  the precision $m$ 
 required on $x$ to be able to evaluate $\wi{f} (x)$ with precision $1/2^n$ (Proposition \ref{326}),

\noindent 
--- calculate the roots $(\alpha_i)_{1 \leq i \leq n})$ of $(P'Q - Q'P)$ on $[0,1]$ to within $2^{-m}$, 

\noindent 
--- calculate $\max \{\wi{f}(0);\wi{f}(1);\wi{f}(\alpha_i)
 \; 1 \leq i \leq s \}$ with precision $1/2^n$ (Proposition \ref{326}).

\noindent 
b) The code $f$ contains the codes of $P$ and $Q$. The point is to see that we can test in polynomial time that the denominator $Q$ is bounded below by $1$ on the interval $[0,1]$. 
This is a classical result concerning calculations with real algebraic numbers: it is a matter of comparing to $1$ the inf of $Q(\alpha_i)$ with $\alpha_i=0,1$ or a zero of $Q'$ on the interval. 
\end{proof}

\begin{remarks}\label{328}~

\noindent 
1) It is well known that calculating the real roots lying in a given rational interval of a polynomial in  $\ZZ[X]$ is a calculation of class $\p$. 
Perhaps the most efficient method is not the one we have indicated, but a slight variant. The search for the complex roots of a polynomial (for a given precision) is now extremely fast (cf.\ \cite {Pa}). 
Rather than looking specifically for the real zeros on the interval $[0,1]$ we could therefore search with precision $2^{-m}$ for the real or complex zeros $(\beta_j)_{1 \leq j \leq t}$ sufficiently close to the interval $[0,1]$ (i.e.\ their imaginary part is in absolute value $\leq 2^{-m}$ and their real part is on $[0,1]$ to within $2^{-m}$) and evaluate $\wi{f}$ at $\Re(\beta_j)$.

\noindent 
2) As follows from Proposition \ref{335} below, any semilinear function with coefficients in $\DD$ is a point of complexity $\p$ in $\crf$. 
The fact that the presentation $\csl$ is not of class $\p$ (if $\p \neq \np$, cf.\ Section \ref{subsec43}) implies on the other hand that the family of functions $\big(\wi{f}\,\big)_{f \in \ysl }$ is not a family of class~$\p$ in $\crf$.

\noindent 
3) It is easy to show that vector space operations are also in polynomial time. 
\end{remarks}

The previous results are summarised as follows.

\begin{theorem} \label{329}
The presentation $\crf$ of $\czu$ is of class $\p$.
\end{theorem}

\subsection{Newman's approximation theorem and its algorithmic complexity}\label{subsec33}

Newman's theorem is a fundamental theorem in approximation theory. The following statement is a special case.\footnote{We have taken the bound $n^2/2$ on the degrees so as to obtain a bound in $e^{-n}$.}

\begin{theorem}[Newman's theorem, \cite{Ne}, see for example \cite{PP} p. 73--75] \label{331}~

\noindent Let $n$ be an integer $\geq 6$, define
\[
H_n(x) = \prod_{ 1 \leq k < n^2}(x + e^{-k/n}).
\]
and consider the two polynomials $P_n(x)$ and $Q_n(x)$, of degrees  $\leq n^2/2$, given by
\[
P_n(x^2) = x(H_n(x) - H_n(-x)) \; \; {\rm and } \; \; Q_n(x^2) = H_n(x) + H_n(-x)
\]
Then for all $x \in [-1,1]$ we have 
\[
\abS{ \; \abs{x}\; - \;(P_n/Q_n)(x^2) \; } \; \leq 3e^{-n}\; \leq 2^{-(n+1)}
\]
and 
\[
\abS{Q_n(x^2)}\; \geq \;2H_n(0) = 2/e^{(n^3 - n)/2} \;\geq 1/2^{3(n^3 - 
n)/4} 
\]
\end{theorem}

It follows from Newman's theorem that semilinear functions can be individually approximated by rational functions that are easy to write and well controlled. This is made clear by \lemref{333} below and its corollaries, the propositions and theorems that follow. First, we recall the following result (cf.\ \cite{Brent}).

\begin{lemma}[Brent's theorem] \label{332} 
 Let $a \in [-1,1]\cap \DD_m$. The calculation of $\exp(a)$ with precision~$2^{-m}$ can be done in time $\Oo(\M(m)\log(m))$.  
In other words, the exponential function on the interval $[-1,1]$ is of complexity $\DTI(\Oo(\M(m)\log(m)))$ 
\end{lemma}

We can easily deduce the following result, by noting $\DD[X]_f$ the presentation of $\DD[X]$ by formulae.

\begin{lemma} \label{333}
There exists a sequence 
\[
\NN_1 \rightarrow \DD[X]_f \times \DD[X]_f \; ; \; n \mapsto (u_n, v_n),
\] 
 of class $\DRT(\Oo(n^5), \Oo({\M}(n^3) n^2 \log^3(n)))$, such that for all $x \in [0,1]$ 
\[
\abs{\;\abs{x} - p_n(x^2)/q_n(x^2) \;} \;\;\leq \;2^{-n}.
\]
The degrees of the polynomials $p_n$ and $q_n$ are bounded by $n^2/2$, their sizes in presentation by formula are bounded by $\Oo(n^5)$, and $q_n(x^2)$ is bounded below by $1/e^{(n^3-n)}$.
\end{lemma} 

\begin{proof} We define $p_n$ and $q_n$ as $P_n$ and $Q_n$ by replacing in the definition the real $e^{-k/n}$ by a sufficient dyadic approximation $c_{n,k}$  calculated by means of \lemref{332}. 

\noindent 
If $\abs{e^{-k/n}-c_{n,k}}\;\leq \varepsilon$ we check that for all $x\in [0,1]$ 
\[
\abs{P_n(x) - p_n(x)}\;  \leq  (n^{2}-1) 2^{n^2} \varepsilon  \quad  
{\rm  and}   \quad  \abs{Q_n(x)- q_n(x)}\; \leq  (n^2-1) 2^{n^2} \varepsilon  = \varepsilon_1. 
\]
For the difference between rational functions we use
\[
\abs{A/B - a/b};  \leq \;  
\abs{A/B} \;  \abs{b - B}  / b + \; \abs{A - a}  / b 
\leq   3 \varepsilon_1/ b. 
\]
As $ B \geq 1/2^{3(n^3-n)/4}$ we have $1/b \leq 2.2^{3(n^3-n)/4}$ if 
$\abs{b - B}\;\leq  B/2$, in particular if 
\[
(n^{2}-1) 2^{n^2} \varepsilon \leq (1/2) \cdot 1/2^{3(n^3-n)/4}.
\]
We are then led to take an $\varepsilon $ such that 
\[
3\varepsilon_1/b\leq 6(n^{2}-1)2^{n^2}2^{3(n^3-n)/4}\varepsilon 
\leq 1/2^{n+1}.
\]
It is therefore sufficient to take $\varepsilon \leq 2 ^{-n^3}$ (for $n$ large enough). 
large).

\noindent 
We will therefore have to describe $p_n$ and $q_n$ using formulae of 
``algebraic'' size $\Oo(n^2)$ involving base terms $(x+c_{n,k})$ 
where $c_{n,k}$ is a dyadic of size $\Oo(n^3)$. 
The (boolean) size of the formula is therefore $\Oo(n^5)$. 
Most of the calculation time is taken up by calculating $c_{n,k}$ using \lemref{332}. 
\end{proof}

Note that the increase in computation time is only slightly less than the 
size of the result.

We can immediately deduce the following results.

\begin{theorem} \label{334}
The function $x \mapsto \abs{ x - 1/2 }$ is a point of complexity $\p$ (more precisely $\DRT(\Oo(n^5), \Oo({\M}(n^3) n^2 \log^3(n)))$ 
in the space $\crf$. 
In other words, there is a sequence 
$\NN_1 \to \yrf \; , \; n \mapsto (u_n,v_n)$, of class
 $\DRT(\Oo(n^5), \Oo({\M}(n^3) n^2 \log^3(n)))$, such that 
\[
\Norme{ \; \abs{ x-1/2 } - u_n(x)/v_n(x) \;} \; \leq 2^{-n}
\]
The degrees of the polynomials $u_n$ and $v_n$ are bounded by $n^2$.
\end{theorem}

\begin{proposition} \label{335}
The function $x \mapsto \abs{x}$ on the interval $[-2^m , 2^m]$ can be approximated to within $1/2^n$ by a rational function $p_{n,m}/q_{n,m}$ whose denominator is less than $1$ (over the same 
interval), and the calculation 
\[
(n,m) \mapsto (p_{n,m},q_{n,m}) \; \; \NN_1 \times \NN_1 \to
 \DD[X]_f \times \DD[X]_f 
\] 
is of complexity $\DRT(\Oo(N^5), \Oo({\M}(N^3) N^2 \log^3(N)))$ 
where $N = n+m$. The degrees of the polynomials 
$p_{n,m}$ and $q_{n,m}$ are bounded by $N^2$.
Similarly the function $(x,y) \mapsto \max(x,y)$ 
(resp. $(x,y) \mapsto \min(x,y)$) on the square 
$[-2^m,2^m] \times [-2^m,2^m]$ can be approximated to within $1/2^n$  
by a rational function of the same type and complexity 
as the previous ones.
\end{proposition}

\begin{proof} For the absolute value function:

\noindent 
If $x \in [-2^m , 2^m]$ we write 
$
\abs{x}\;  = 2^m \abs{ x/2^m } 
$
with $x/2^m \in [-1 , 1]$. So, using the notations from the proof of  Lemma \ref{333}, it suffices to take $p_{n,m}(x) = 2^m p_{n+m}(x/2^m)$ and $q_{n,m}(x) = 2^m q_{n+m}(x/2^m)$.

\noindent 
For the functions $\max$ and $\min$, simply use the formulae:
\[
\max(x,y) = \frac{x+y\; + \abs{x-y}} {2}\quad  {\rm and} \quad  
\min (x,y) = \frac{x+y\; - \abs{x-y}}{2} 
\]
\end{proof}

In the following, we shall need to ``approximate'' the discontinuous function
\[ 
C_a(x) = \left\{
\begin{array}{cl}
1, & {\rm if } \; x \geq a 
\\
0, & {\rm otherwise} 
\end{array}
\right.
\]
Such an ``approximation'' is given by the continuous semilinear function $C_{p,a}$:
\[
 C_{p,a} = \min(1, \max(0, 2^p(x - a))) \; {\rm where} \; p \in \NN_1\; {\rm and}\; a \in \DD_{[0,1]}
\] 
\begin{figure}[htbp] 
\begin{center}
\includegraphics*[width=14cm]{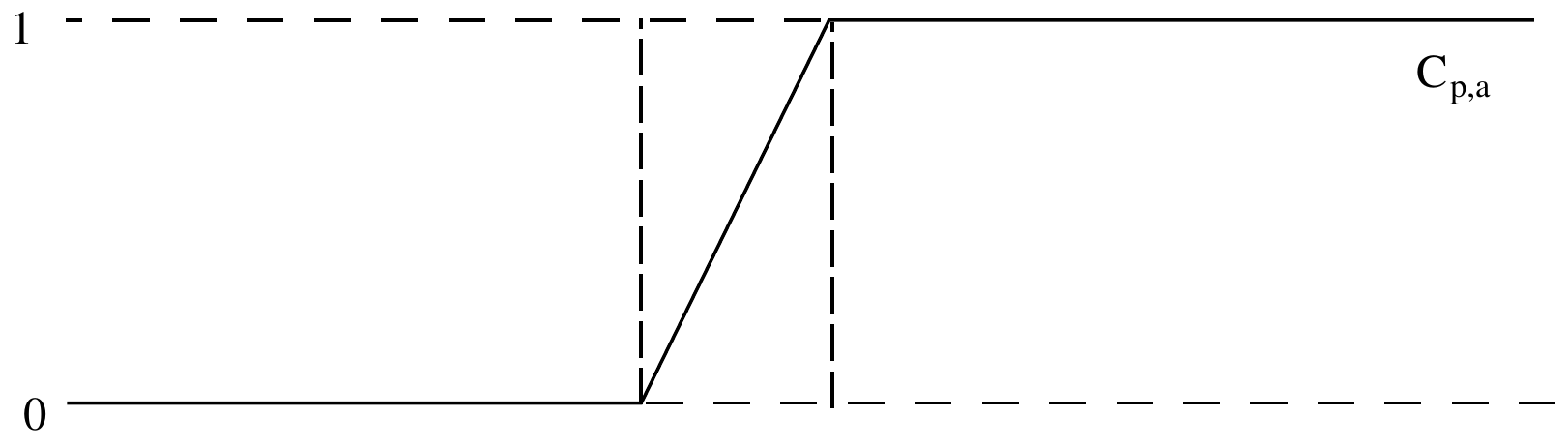}
\end{center}
\caption[representative curve of the function $C_{p,a}$]{\label{fi336} 
the function $C_{p,a}$} 
\end{figure} 

\noindent 
The complexity of the family of functions $(p,a) \mapsto C_{p,a}$ is given in the following proposition. 

\begin{proposition} \label{337}
The family of functions 
$
 \NN_1 \times \DD_{[0,1]} \;: \; (p,a) \mapsto C_{p,a}
$ 
defined as:
\[
C_{p,a} = \min(1,\max(0, 2^p(x - a)))
\] 
is a family of class $\p$. 
More precisely, it has complexity 
$\DRT(\Oo(N^5), \Oo({\M}(N^3) N^2 \log^3(N)))$ where $N = \max(\Lg(a), n+p)$.
\end{proposition}

The following proposition concerns the square root function.

\begin{proposition} \label{338}
The function $ x \mapsto \sqrt {\abs { x - 1/2 }} $ on $[0,1]$ is a 
$\p$-point of $\crf$.
\end{proposition}

\begin{figure}[htbp] 
\begin{center}
\includegraphics*[width=10cm]{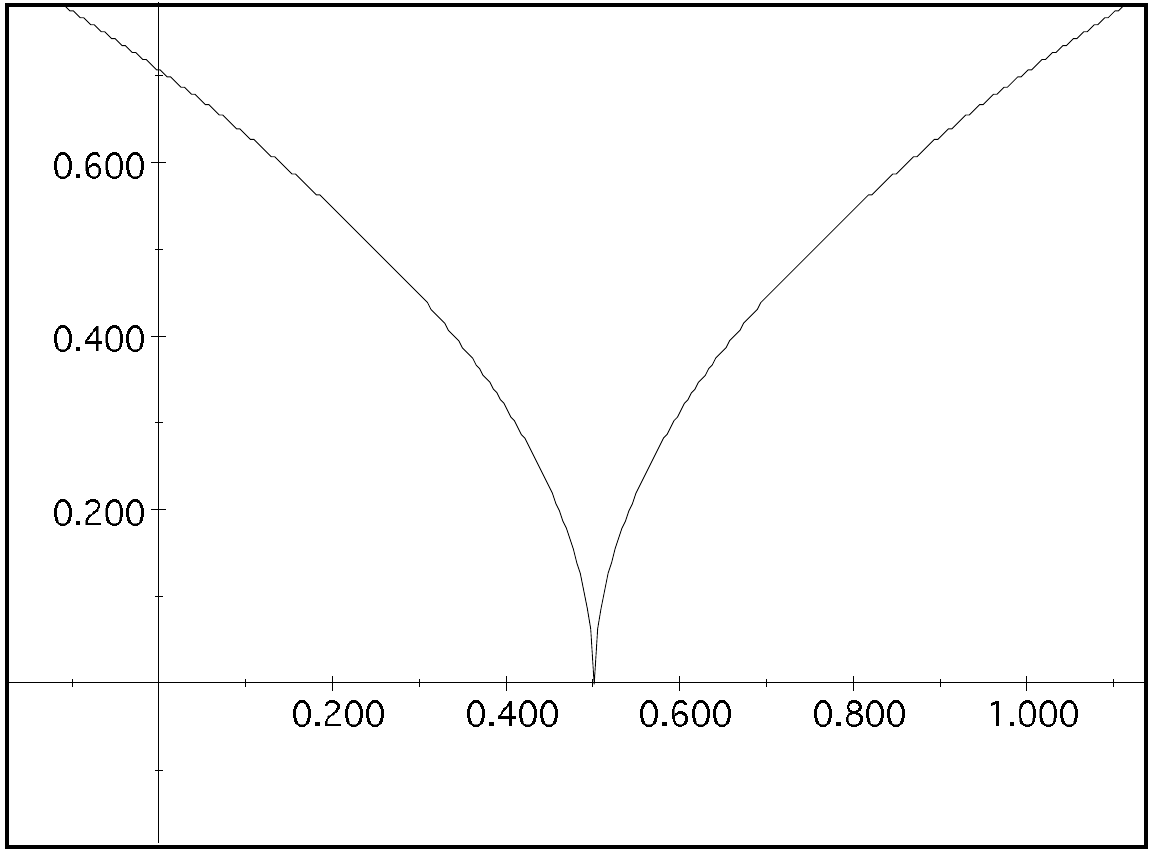}
\end{center}
\caption[Curve of the function $\sqrt {\abs{ x - 1/2 }} $]{\label{fi339} 
the function $ x \mapsto \sqrt {\abs { x - 1/2}} $} 
\end{figure} 

\begin{proof} Sketch. 
In the polynomials $P_n(x^2)$ and $Q_n(x^2)$ from Newman's theorem, if we 
replace $x^2$ by a good approximation of $\abs{x}$ on $[0,1]$, then 
by Newman's theorem, we obtain a rational function 
 $R_n(x)/S_n(x)$ such that 
\[
\forall x \in [0,1] \abs{\sqrt{ \abs{x}} - (R_n/S_n)(x) }\;  \leq 2^{-
n}.
\]
The degrees of $R_n(x)$ and $S_n(x)$ will be  $\Oo(n^4)$. 

\noindent 
Note that the fraction $(P_n/Q_n)(x)$ is not odd, so it provides a 
good approximation to the function~$\sqrt{ \abs{x}}$ only on the interval 
the interval $[0,1]$. 
\end{proof}

\section[A natural presentation of the space \texorpdfstring{$\czu$}{C[0,1]}]{A natural presentation of the space \texorpdfstring{$\czu$}{C[0,1]} 
and some equivalent 
equivalent presentations}\label{sec4}

A natural notion of complexity for points and sequences of points 
of the space $\czu$ is given by Ker-I.\ Ko and H.~Friedman in \cite{KF82}. 
In this section we study a rational presentation of the space $\czu$ which gives (more or less) the same notions of complexity. We also study 
other presentations using circuits, which turn out to be equivalent from the point of view of complexity. 
The proof of these equivalences is based on the proof of an analogous but less general result given in \cite{Ho87,Ho90}. 

\subsection {Definitions of some presentations of the space 
 \texorpdfstring{$\czu$}{C[0,1]}}\label{subsec41}
 
\subsubsection{KF presentation (KF as in Ko-Friedman)}
Recall that $\DD_n = \{ k/2^n ; k \in \ZZ \}$ and $\DD_{n,[0,1]} = 
\DD_n \cap [0,1]$. 

\noindent 
In the definition given by Ko and Friedman, for $f \in \czu$, an Oracle Turing Machine (OTM) ``computes'' the function function $f$ in the following sense. 
For the question ``$\,m\;?\,$'', where $m\in\NN_1$, the oracle delivers an approximation $\xi \in\DD_{m,[0,1]}$ of $x$ to within $2^{-m}$. 
For the input $n \in \NN_1$, the machine calculates, with the help of the previous oracles, a
$\zeta \in \DD_n$ approximating $f(x)$ to within $1/2^n$.

\noindent 
Reading the OTM program alone is obviously no way of knowing whether the OTM calculates~a function of $\czu$. 
Consequently, since we wish to have objects clearly identified as ``rational points'' of the presentation we want to define, we introduce control parameters. 
But for such parameters to be truly effective, it is necessary to restrict the execution of the machine to an output precision given a priori. 
As a result, the required input accuracy is also limited a priori. 
Thus, the OTM is approximated by a sequence of ordinary machines (without oracle), each executing only a finite number of computations. 
We are therefore led to define a rational presentation of $\czu$ which will be denoted by $\ckf$ (the set of codes for rational points will be denoted by $\ykf$) as follows.  
 
\begin{definition} \label{411}
We consider a language, chosen once and for all, to describe TM programs with a single input, in $\DD_{[0,1]}$ (denoted by $x$), and with a single output, in $\DD$ (denoted by $y$). Let ${\bf Prog}$ be the part of this language formed by well-written programs 
(according to a syntax specified once and for all).
Let $f = (Pr, n, m, T) \in {\bf Prog} \times \NN_1 \times \NN_1 
\times \NN_1$ where we have:

\noindent 
--- $n$ is the precision required for $y$,

\noindent 
--- $m$ is the precision with which $x$ is given.

\noindent 
The quadruplet $(Pr,n,m,T)$ is said to be {\em correct} when the following conditions are satisfied.

\noindent 
--- the program $Pr$ calculates a function from $\DD_{m,[0,1]}$ to 
$\DD_n$, i.e.\ for an entry in $\DD_{m,[0,1]}$ it obtains an output 
an element of $\DD_n$;

\noindent 
--- $T$ bounds the maximum execution time for all $x$ in 
$\DD_{m,[0,1]}$; 

\noindent 
--- for two consecutive elements ($1/2^m$ apart) of 
$\DD_{m,[0,1]}$ as input, the program gives two elements of 
$\DD_n$ no more than $1/2^n$ apart.

\noindent \rdb 
The set of correct quadruplets $(Pr,n,m,T)$ is denoted by $\ykf$\label{ykf}.
 When the data $f$ is correct, it defines the following rational point  
$\wi{f}$.
It is the piecewise linear function which joins the points of the graph given on the grid $\DD_{m,[0,1]} \times \DD_n$ by executing the program $Pr$ for all possible inputs in $\DD_{m,[0,1]}$.
So $\ckf=(\ykf,\eta,\delta)$ where $\eta(f)=\wi f$ and $\delta$ must be defined as required in Definition \ref{211}. The complexity of $\eta$ and $\delta$ is studied in Proposition \ref{413}.
\end{definition} 

\begin{remark} \label{412}
 Note that the first two conditions could be fulfilled 
by syntactic constraints that are easy to implement. 
 On the other hand, as we shall see in Proposition \ref{441}, 
the third is uncontrollable in polynomial time (unless $\p$ = $\np$) (this correctness condition defines a $\cnp$-complete problem). Note also that 
if we had not imposed the correctness condition for the points of $\ykf$ 
we would have had no a priori control of the modulus of uniform continuity for the piecewise linear function defined by the data, and the family $\big(\wi{f}\,\big)_{f \in \ykf}$ would not have been uniformly of class~$\p$.
\end{remark}

\begin{proposition}[complexity of the family of functions attached to $\ykf$] \label{413}~\\
The family of continuous functions $\big(\wi{f}\,\big)_{f \in \ykf}$ 
is uniformly of class $\DSRT(\Lin, \Lin, \Oo(N^2))$. 
\end{proposition}

\begin{proof} 
First of all, we note that the function $\wi{f}$ corresponding to 
$f = (Pr, n, m, T)$ is piecewise linear and that, since the data is 
correct, the slope of each piece is bounded in absolute value by 
$2^{m-n}$, which gives the modulus of uniform continuity $\mu(f,k) = k+m-n$ for the family 
$\big(\wi{f}\,\big)_{f \in \ykf }$. 

\noindent 
We need to find a map 
$\varphi\colon \ykf \times \DD_{[0,1]} \times \NN_1 \to \DD$ of complexity 
$\DSRT(\Lin, \Lin, N^2)$ such that
\[
\forall (f,x,k) \in \ykf  \times \DD_{[0,1]} \times \NN_1 \; \; 
\abs{ \varphi(f,x,k) - \wi{f}(x) }\;  \leq 2^{-k}.
\]
Let $z=(f,x,k)=((Pr,n,m,T),x,k) \in \ykf \times \DD_{[0,1]} \times \NN_1$.
We will assume without loss of generality that $k \geq n$ and that $x$ 
is given with at least $m$ bits.

\noindent 
By simply reading $x$ we can identify two consecutive elements $a$ and 
$b$ of $\DD_{m,[0,1]}$ such that $a \leq x \leq b$ and we find the 
dyadic $r \in \DD_{[0,1]}$ such that $x=a+r/2^n$. 

\noindent 
Let $\varepsilon \in \{-1,0,1\}$ be the integer satisfying 
$\wi{f}(a)=\wi{f}(b)+ \varepsilon/2^n$. 

\noindent 
Then $\wi{f}(x) = \wi{f}(a) + \varepsilon r/2^n$. In addition 
$\wi{f}(a)= \Exec (Pr,a)$ is the result of executing the program 
program $Pr$ for input $a$.
The complete calculation therefore essentially consists of:

\noindent 
--- reading $x$ (and deducing $a$, $b$ and $r$),

\noindent 
--- calculating $\Exec (Pr,a)$ and $\Exec (Pr,b)$.

\noindent 
The complexity is therefore bounded by the complexity of the Universal Turing Machine that 
we use to run the program $Pr$. According to \lemref{131} 
this is done in time $\Oo(T(T+ \flo{Pr}))$ and space $\Oo(T+ \flo{Pr})$. And the size of the output is bounded by $T$. 
\end{proof}

The fact that the presentation $\ckf$ which results from Definition \ref{411} is the Ko-Friedman presentation is justified by the following proposition.

\begin{proposition} \label{414}
A function $f\colon [0,1] \to \RR$ is computable in polynomial time in the sense of Ko-Friedman if and only if it is a $\p$-point of $\ckf$. More 
precisely,
\begin{itemize}
\item [a)] if the function $f$ has time complexity $T(n)$ in the sense of Ko-Friedman then it is a $\DTI (T)$-point of $\ckf$;
\item [b)] if the function $f$ is a $\DTI (T)$-point of $\ckf$ then it has complexity $T^2(n)$ in the Ko-Friedman sense. 
\end{itemize}
\end{proposition}
We will first prove a characterisation of the complexity of a function in the sense of Ko-Friedman, for a class $\ca$ of complexity in time or space. We have already stated this without proof in  Example  \ref{223}. We give a more precise statement here.

\begin{proposition} \label{415}
Let $\ca$ be an elementarily stable class of type $\DTI(\bullet)$ 
or $\DSPA(\bullet)$ or $\DSRT(\bullet, \bullet, \bullet)$ or $\DRT(\bullet, 
\bullet)$ or $\DSR(\bullet, \bullet)$ and consider a continuous function 
$f\colon [0,1] \to \RR$. The following properties are equivalent:
\begin{itemize}
\item  [1)] the function $f$ is computable in the Ko-Friedman sense in the class $\ca$.
\item  [2)] the function $f$ is uniformly of class $\ca$.
\end{itemize}
{\bf NB}. 
As far as the complexity of an OTM is concerned, we mean here that the questions to the oracle must be counted among the machine's outputs. 
In other words, the size of the $m$ integers which are the questions to the oracle must have the bound required for the output in classes of the type $\DSRT(\bullet, \bullet, \bullet)$ or $\DRT(\bullet, \bullet, \bullet)$ or $\DSR(\bullet, \bullet, \bullet)$.
\end{proposition}

\begin{proof}[Proof of Proposition \ref{415}] 
Recall that a function is uniformly of class $\ca$ (cf.\ definition 
\ref{211}) when:

\noindent 
a) the function $f$ has a modulus of uniform continuity in the class $\ca$.

\noindent 
b) the family $(f(a))_{a \in \DD}$ is a $\ca$-family of 
real numbers.

\noindent 
Let $f \in \czu$ and $M$ be an OTM which, for any input $n \in \NN_1$ and any $x \in [0,1]$ (given as an oracle) calculates $f(x)$ to within $2^{-n}$ in the $\ca$ class. If we replace the oracle, which 
gives an approximation of $x$ to within $2^{-m}$ in the class by reading an arbitrary dyadic number $a$, we obtain a usual TM that calculates $(f(a))_{a \in \DD}$ as a family of real numbers, in the 
family of real numbers, in the class $\ca$. 

\noindent 
Let's look at the question of the modulus of uniform continuity. On the input $n$ (precision required for the output) and for any oracle for any $x \in [0,1]$, the OTM will calculate $y=f(x)$ to within $2^{-n}$ by querying the oracle for certain precision $m$. Since $\ca$ is a complexity class of the type expected, the largest of the integers $m$ used on the input $n$ can be  augmented by $\mu(n)$ where $\mu\colon \NN_1 \to \NN_1$ is a function in the class $\ca$.\footnote{For example, for a class of complexity class where we specify that the outputs are of linear size, the size of $m$ depends linearly on $n$ (independently of the oracle) since $m$, as a question put to the oracle, is one of the outputs.} This is the modulus of uniform continuity we are looking for.

\noindent 
Suppose now that the function $f$ is uniformly of class $\ca$.
 Let $M$ be the TM (without oracle) that calculates $(f(x))_{x \in \DD}$ as a family of real numbers in class class. 
Let $M'$ be a TM that computes a modulus of uniform continuity $\colon \NN_1 \to \NN_1$ in the class $\ca$. The Ko-Friedman OTM is then as follows. 
On the input $n$ it calculates $m = \mu(n+1)$ (using $M'$), it then queries the oracle with precision $m$, the oracle gives an element $a \in \DD_m$. The OTM then uses $M$ to calculate $f(a)$ to within $1/2^{n+1}$, which, deprived of its last bit, constitutes the output of the OTM. 
\end{proof}

Proposition \ref{415} shows, by contrast, that the non-perfect match obtained in Proposition \ref{414} is only due to the difficulty of putting exactly and every time the notion of complexity of a uniformly continuous map (as a function) into the mould of the notion of complexity of a point in a metric space. This perfect match occurs for classes such as $\p$, $\Prim $ or $\Rec $ but not for $\DRT(\Lin, \Oo(n^k))$. 
\\ 
Finally, note the great similarity between the proofs of the propositions \ref{415} and \ref{414}, with a slightly greater complication for the latter. 

\begin{proof}[Proof of Proposition \ref{414}] Let $f \in \czu$ and $M$ an OTM which, for any input $n \in \NN_1$ and any $x \in [0,1]$ (given as an oracle) calculates $f(x)$ to within $2^{-n}$, in time bounded by $T(n)$. 

\noindent 
Consider the sequence $f_n = (Pr_n, n, T(n), T(n))_{n \in \NN_1}$ where we have:

\noindent 
$Pr_n$ is obtained from the $Pr$ program of the $M$ machine by replacing the answer $(x_m)$ given by the oracle to the question ``$\,m\;?\,$'' by the instruction ``read $x$ with precision $m$''. 

\noindent 
The correction of $f_n$ is clear. 
Let $\wi{f_n}$ be the piecewise linear function corresponding to $f_n$. 
We have for $\xi \in [0,1]$ and $d$ an approximation of $\xi$ to within $1/2^{T(n)}$
\[ 
{\Exec}(Pr_n, d) = {\Exec} (Pr, n, \hbox{Oracle  for } \xi) = M^{\xi}(n)
\]
and therefore
\[
\abs{\wi{f_n}(\xi) - f(\xi)}\; \leq 
\abs{f(d) - {\Exec}(Pr_n, \xi)} + 
\abs{\wi{f_n}(d) - {\Exec} (Pr_n, \xi) }
\; \leq 2^{-n} + 2^{-n} \leq 2^{-(n-1)}
\] 
Finally, the sequence $n \mapsto f_{n+1}$ is of complexity $\DTI (\Oo(n))$.

\noindent 
Conversely, suppose $\norme{ f_n - f } \leq 2^{-n}$ with 
$n \mapsto f_n$ of class $\DTI (T(n))$. 

\noindent 
We have $f_n = (Pr_n, q(n), m(n), t(n))$. Consider the OTM $M$ which, for any integer $n \in \NN_1$ and any $xi \in [0,1]$, performs the following tasks 
as follows: 

\noindent 
--- take the element $f_{n+1} = (Pr_{n+1}, q(n+1), m(n+1), t(n+1))$, from the sequence, corresponding to $n+1$. 
sequence, corresponding to $n+1$. 

\noindent 
--- ask the oracle the question $Q(n) = m(n+1) + n+1 - q(n)$ (we obtain 
an approximation $d$ of~$\xi$ to within $2^{-Q(n)}$). 

\noindent 
--- calculate $\wi{f_{n+1}}(d)$ by linear interpolation (cf.\ Proposition \ref{413}): this is the output of $M$.

\noindent 
We then have for all $\xi \in [0,1]$
\[
\abs{\wi{f_n}(\xi) - M^{\xi}(n)} \; \leq 
\abs{\wi{f}_n(\xi) - \wi{f}_{n+1}(\xi)} + 
\abs{\wi{f}_{n+1}(\xi) - \wi{f}_{n+1}(d)}\; \leq 2^{-(n+1)} + 2^{-(n+1)} = 2^{-n}  
\]
so 
\[
\norme{ f - M^{\xi}(n) } \; \leq\; \norme{ f - \wi{f_n} } + \Norme{ \wi{f_n} - M^{\xi}(n) } \; \leq\; 2^{-n} + 2^{-n}
\]
Moreover the machine $M$ has complexity $\Oo(T^2(n))$ (cf.\ Proposition \ref{413}). 
\end{proof}

In the same style as Proposition \ref{415}, we obtain two natural characterisations, to within $\p$-equivalence, of the rational presentation $\ckf$ of $\czu$.
 
\begin{proposition} \label{416}
Let $\sC^{\star}$ be a rational presentation of the space $\czu$. 
The two following assertions are equivalent. 
\begin{enumerate}
\item The presentation $\sC^{\star}$ is $\p$-equivalent to $\ckf$.
\item A family $\big(\wi{f}\,\big)_{f \in Z}$ in $\czu$ is uniformly of class $\p$ if and only if it is a $\p$-family of points in~$\sC^{\star}$.
\end{enumerate}
\end{proposition}

\begin{proof} 
For any $\ca$-elementarily stable class, two rational presentations of an arbitrary 
metric space are $\ca$-equivalent if and only if they 
define the same $\ca$-families of points (cf.\ remark \ref{228}). It is then sufficient to
check that the rational presentation $\ckf$ satisfies condition (2). The 
proof is identical to that of Proposition \ref{414}. 
 \end{proof}

\begin{theorem} \label{417}
The presentation $\ckf$ is universal for evaluation in the following sense. \\
If $\sC_Y=(Y,\delta,\eta)$ is a rational presentation of $\czu$ for which the family $\big(\wi{f}\,\big)_{f \in Y}$ is uniformly of class $\p$,  then the function $\Id_\czu$ from $\sC_Y$ to $\ckf$ is uniformly of class $\p$.

\noindent The result generalises to any elementarily stable class of complexity $\ca$ which verifies the following property:
if $F \in \ca$ then the computation time of $F\colon \NN_1 \to \NN_1,\; N 
\mapsto F(N)$, is bounded by a function of class~$\ca$.
\end{theorem}

\begin{proof} 
This follows from \ref{416}. However, let us give the proof. We must show that the family $\big(\wi{f}\,\big)_{f \in Y}$ is a family of points of class $\p$ for the presentation $\ckf$. Since the family $\big(\wi{f}\,\big)_{f \in Y}$ is uniformly of class $\p$, we have two functions $\psi$ and $\mu$ of class 
$\p$:
\[
 \psi\colon Y \times \DD_{[0,1]} \times \NN_1 \to \DD \qquad  
(f,d,n) \mapsto \psi(f,d,n) = \wi{f}(d) \; \hbox{to within} \; 1/2^n
\]
and 			
\[
\mu: Y \times \NN_1 \to \DD \qquad (f,n) \mapsto \mu(f,d,n) = m 
\]
where $\mu$ is a modulus of uniform continuity for the family. 
Let $S\colon \NN_1 \to \NN_1$ be a map of class~$\p$ which gives a bound for the computation time of $\psi$. 

\noindent 
Let us then consider the function $\varphi\colon Y \times \NN_1 \to \ykf $ defined by:
\[
\varphi(f,n) = (Pr(f,n), n, \mu(f,n), S(\flo{f} + \mu(f,n) + n))
\]
where $\flo{f}$ is the size of $f$ and $Pr$ is the program calculating 
$\psi$, barely modified: 
it takes as input only the $d \in \DD_{m,[0,1]}$ and retains for output only the significant digits in $\DD_n$.

\noindent 
The $\varphi$ function is also in  class $\p$. 

\noindent 
Furthermore, it is clear that the rational point coded by $(Pr,n,m,S(\flo{f} +m+n))$ approximates~$\wi  f$ to within~$1/2^n$. 	 
\end{proof}

Note that this universal property of the presentation $\ckf$ is obviously shared by all presentations $\p$-equivalent to $\ckf$.

\subsubsection{Presentation by boolean circuits} \label{subsubsec412}

Boolean circuit presentation is essentially the same as Ko-Friedman presentation. It is a little more natural within our framework of rational 
presentations. 

Let us denote by $\CB$ the set of (codes for) boolean circuits. A boolean circuit $C$ is simply given by a boolean evaluation program (boolean straight-line program) which represents the execution of the circuit $C$.
A rational point of the presentation by boolean circuits is a map $\wi{f}$ encoded by a quadruplet $f = (C, n, m, k)$ where $C$ is (the code of) a boolean circuit, $n$ is the required precision for $f(x)$, $m$ is the required precision input, and $2^{k+1}$ is an upper bound for the norm of $\wi{f}$. 

\begin{definition} \label{418}
A rational point of the presentation by boolean circuits $\cbo$ is encoded by a quadruplet $f = (C, n, m, k) \in \CB \times \NN_1 \times \NN_1 \times 
\NN_1$ where $C$ is (the code of) a boolean circuit with~$m$ input gates and $k+n+1$ output gates: the $m$ input gates are used to code an element of $\DD_{m,[0,1]}$ and the $k+n+1$ output gates are used to code elements of $\DD$ of the form $\pm 2^k\left( \sum_{i=1}^{n+k}b_i2^{-i}\right)$ ($b_i\in\{0,1\}$, $b_0$ codes $\pm$).

\noindent 
The data $f = (C, n, m, k)$ is said to be {\em correct} when two entries 
encoding elements  $1/2^m$ apart of $\DD_{m,[0,1]}$ give outputs 
encoding elements at most $1/2^n$ apart of $\DD$. The set of correct data is $\ybc$.

\noindent 
When a given $f$ is correct, it defines a piecewise linear function function $\wi{f}$ which is well controlled (it is the ``rational point'' defined by the data).
\end{definition}

It should be noted that the control parameters $n$, $m$ and $k$ are mainly indicated for the user's convenience. In fact $n+k$ and $m$ are 
directly readable on the $C$ circuit. 

From this point of view, the situation is slightly improved compared to the presentation for which the control parameters are absolutely essential. 
The size of $C$ alone controls the execution time of the circuit 
(i.e.\ the function which calculates, from the code of a boolean circuit and the list of its inputs, the list of its outputs is a function of very low complexity).

\begin{proposition}[complexity of the family of functions attached to $\ybc$] \label{419} 
The family of continuous functions $\big(\wi{f}\,\big)_{f \in \ybc }$ is uniformly of 
class $\DSRT(\Lin, \Lin, \QLin)$.
\end{proposition}

\begin{proof}
We copy the proof of Proposition \ref{413} concerning the family $\big(\wi{f}\,\big)_{f \in \ykf}$. The only difference is the replacement of the function $\Exec(Pr,a)$ by the evaluation of a boolean circuit. 
The function which calculates, from the code of a boolean circuit and the list of its inputs, the list of its outputs is a function of time complexity $\Oo(t\log(t))$ (where $t = \sz(C)$), 
so of class $\DSRT(\Lin,\Lin,\QLin)$. 	 
\end{proof}

\noindent {\bf NB}. If we had taken ``management time'' into account (cf.\ remark 
\ref{323}) we would have found a complexity $\DSRT(\Lin, \Lin, \Oo(N^2))$, i.e.\ the 
same result as for the family $\big(\wi{f}\,\big)_{f \in \ykf}$ (the complexity of $\Exec(Pr,a)$ precisely takes into account ``management time'').

 \subsubsection{Presentation by arithmetic circuits (with 
magnitude)}\label{subsubsec 413}
Remember that an {\em arithmetic circuit} is, by definition a circuit whose input gates are ``real'' variables $x_i$ and constants in $\QQ$. Obviously, we could tolerate only the two constants $0$ and $1$ without significant change. 

\noindent 
The other gates are:

\noindent 
--- single-input gates of the following types: $x \mapsto x^{-1}, \; x \mapsto -x$, 

\noindent 
--- two-input gates of the following two types: $x+y, \; x \times y$. 

\noindent 
An arithmetic circuit calculates a rational function, (or possibly gives error if you ask it to invert the identically zero function).
We shall consider arithmetic circuits with a single input variable and a single output gate, and we shall denote by $\CA$ the set of (codes of) arithmetic circuits with a single input. 
The code of a circuit is the text of an evaluation program (in a language and with a syntax) corresponding to the arithmetic circuit in question.

\begin{definition} \label{4110}
An integer $M$ is called {\em a coefficient of magnitude} for an 
arithmetic circuit~$\alpha$ with a single input when $2^M$ bounds in absolute value 
all the rational functions in the circuit (i.e.\ those calculated at all  the gates of the circuit) at any point in the interval~$[0,1]$. 

\noindent 
A pair $f = (\alpha, M) \in \CA \times \NN_1$ is a  \emph{correct}  datum when $M$ is  a magnitude coefficient for the circuit~$\alpha$. 
This defines the elements of $\yaf$. An element $f$ of $\yaf$ 
defines the rational function $\wi{f}$ when it is  
seen as an element of $\czu$. 
The family $\big(\wi{f}\,\big)_{f \in \yaf }$ is the family of rational points of a presentation of $\czu$ which will be noted $\caf$.
\end{definition} 

Note that, because of the presence of multiplications and the transition  to the inverse, the size of $M$ can be exponential compared to that of $\alpha$, even when no rational function of $\alpha$ has a pole on $[0,1]$. This kind of mishap did not occur with binary semilinear circuits.
It even seems unlikely that the correction of a pair $(\alpha, M) \in \CA \times \NN_1$ could be tested in polynomial time. 

\begin{proposition}[complexity of the family of functions attached to $\yaf$] \label{4111} 
The family of continuous functions $\big(\wi{f}\,\big)_{f \in \yaf }$ is uniformly of class  $\p$, and more precisely of class $\DSRT(\Oo(N^3),\Lin,\Oo(N{\M}(N^2)))$.
\end{proposition}

\begin{proof} Let $ f = (\alpha,M) $ be an element of $\yaf$. Let $ p $ be the 
depth of the product $\alpha$. We show by induction on $\pi$ that, for 
a gate of depth $\pi$ the corresponding 
has a derivative bounded by $2^{2M \pi}$ on the interval 
the interval $[0,1]$. This gives the modulus of uniform continuity $\mu(f,k) = k + 2Mp$ (which is $\Oo(N^2)$) for the family $\big(\wi{f}\,\big)_{f \in \yaf }$. 

\noindent 
We now need to show a function 
$\psi\colon \yaf \times \DD_{[0,1]} \times \NN_1 \to \DD$ 
of complexity $ \DRT(\Lin,\Oo(N\M(N^2))) $ such that
\[
\forall (f, x, k) \in \yaf  \times \DD_{[0,1]} \times \NN_1 \; \;
\abs{ \psi(f,x,k) - \wi{f}(x)} \;  \leq 2^{-k}.
\]
Let $(f, x, k) \in \yaf \times \DD_{[0,1]} \times \NN_1$ where 
$f = (\alpha,M)$. Let $ t = \sz(\alpha) $ be the number of gates in the circuit $\alpha$. The size of the inputs is $N = t+M+k$. 

\noindent 
To calculate $\psi((\alpha,M), x, k)$ proceed as follows. 
Read $ m = k+2Mp+p $ bits of $x$ which gives the two consecutive points
 $a$ and $b$ of $\DD_{m,[0,1]}$ such that $a \leq x \leq b$. 
We execute the $\alpha$ circuit at input~$a$, truncating the calculation performed on each gate to the first $m$ significant bits. The value obtained at the output, truncated to the first $k$ significant bits, is the element $\psi((\alpha,M), x, k)$. 
As multiplication (or division) takes place in time ${\M}(m)$, the time of the arithmetic operations themselves is an $\Oo(t\,\M(k+2Mp+p))= \Oo(N\M(N^2))$ and the used space is in $\Oo(N^3)$. 
\end{proof}

\noindent {\bf NB}. If we had taken ``management time'' into account (cf.\ Remark \ref{323}) we would have said: 

\noindent 	
the time for the arithmetic operations themselves is $\Oo(N \M(N^2)),$ 
and the management of objects is $\Oo(N^4).$ So if we consider the 
fast (respectively naive) multiplication, the evaluation function 
is in $ \DRT(\Lin,\Oo(N^4))$ (respectively $\DRT(\Lin,\Oo(N^5))$) where $N$ is the size of the data.

\subsubsection{Presentation by polynomial arithmetic circuits (with 
magnitude)}\label{subsubsec414}

The following presentation of $\czu$ will be denoted by $\capo$. The set of codes for rational points will be denoted by~$\yap$. 

\noindent 
This is the same as the arithmetic circuits, except that 
the ``inverse pass'' gates are removed. We therefore do not give 
the definition in detail.

\noindent 
We find the same difficulties with magnitude, but a little less serious. 
 It does not seem possible to obtain for polynomial circuits significantly better complexity bounds than those obtained in Proposition 
\ref{4111} for circuits with divisions.

\subsection{Comparisons of previous presentations}\label{subsec42}

In this section we show an important result, the equivalence between 
the Ko-Friedman presentation and the four circuit presentations in the previous section, from the point of view of polynomial time complexity. 
 In particular we obtain an  algorithmically fully controlled formulation 
for the Weierstrass approximation theorem. 

\noindent 
To show the equivalence between these five presentations from the point of view of complexity, we will follow the plan below.

\begin{figure}[htbp]
\begin{center}
$$
\xymatrix @R=15pt{
  \ckf\ar[r] &  \cbo\ar[r] &  \csl\ar[r] &  \caf\ar[r] & \capo\ar@/^25pt/[llll] 
}
$$ 
\end{center}
\medskip\caption[Proof diagram]{\label{fi421}proof diagram  
}  
\end{figure}
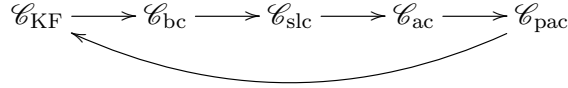

To give an equivalence between two different presentations, we need to 
construct maps that transform a rational point of the first presentation into a rational point of the other presentation, which approximates conveniently the first rational point. And vice versa.

\begin{proposition} \label{421}
The identity of $\czu$ from $\ckf$ to $\cbo$ is uniformly of class $\p$, in fact of class  $\DTI (N^{14})$.
\end{proposition}

The proof of this proposition is based on \lemref{131} which describes 
the complexity of a Universal Turing Machine and on the following lemma.
\begin{lemma} \label{422}
{\rm (cf.\ \cite{St} and \cite{Mo}) } 

\noindent 
Let $ M $ be a fixed TM, and $ T $ and $ m $ be elements of 
$\NN_1$. Then we have a function computable in time $\Oo((T+m)^7)$ which computes 
$\gamma_{T,m}$: a boolean circuit simulating the first $ T $ configurations of~$M$ for any entry in $\{0,1\}^m$.
\end{lemma} 

\begin{remark} \label{423}
J. Stern has given in \cite {St}, for any integer $T \in \NN_1$, a fairly simple construction of a circuit $\gamma_T$ which calculates the configuration of $M$ obtained after $T$ calculation steps from an initial configuration of size $\leq T$ (we can assume that $T \geq m$ or take $\max(T,m)$). The size of the circuit $\gamma_T$ is $\Oo(T^2)$. The author also mentions that this construction is done in polynomial time. A more precise bound ($\Oo(T^7)$) is given in \cite{Mo}.
\end{remark}

\begin{proof}[Proof of Proposition \ref{421}] Consider the Turing Machine 
 MU from \lemref{131}.\\ Let $f=(Pr,n,m,T)$ be an element of 
$\ykf$. Let $ p $ be the size of the program $Pr$. 
The size of $f$ is $N = p+n+m+T.$ The machine $ MU $ takes as inputs the program $ Pr $, an input $ x \in \DD_{m,[0,1]}$ and the number of steps $T$. 
It executes in $\Oo(T (\max(T,m)+p)) = \Oo(N^2)$ steps (cf.\ \lemref{131}) the task of calculating the output for the program $Pr$ on the same input $x$ after $T$ calculation steps. 
By applying \lemref{422} we obtain time $\Oo(N^{14})$ for 
calculating an element $g$ of $\ybc$ from the input $f$ 
(a boolean circuit and its control parameters) 
for which we have $\wi{g} = \wi{f}$. 
\end{proof}

\begin{proposition} \label{424}
The identity of $\czu$ from $\cbo$ to $\csl$ is uniformly of class 
$\LINT$. More precisely, we have a discrete function which, starting from an element $ f = (\gamma,n,m,k) $ of $\ybc$ and an integer 
$ q \in \NN_1$, calculates in time $ \Oo(N) $ (where $ N $ is the size of the the input $ ((\gamma,n,m,k),q)$, a binary semilinear circuit $ g $ such that 
\[
\forall x \in [0,1]  \;\;
\abs{ \wi{f}(x) - \wi{g}(x)} \;  \leq 2^{-q}    \eqno(\F4.2.4).
\]  
\end{proposition}

The proof of this proposition uses the following lemma.

\begin{lemma} \label{425} 
There exists a function computable in time $\Oo(N)$ which transforms any 
element $f = ((\gamma,n,m,k),h)$ of $\ybc \times \NN_1$ into an element 
$f' = (\gamma',n+h,m+h,k)$ of $\ybc$ corresponding to the same function.
\end{lemma}

\begin{proof} 
Suppose that the circuit $\gamma$ calculates, for the input $x = u/2^m$ where $0 \leq u \leq 2^m - 1$, the value $y = \ell/2^n$, and for $x' = (u+1)/2^m$, the value $y' = \ell'/2^n$ with $\ell , \ell' \in \ZZ$ and $\abs{ \ell - \ell'} \;  \leq 1$.
 Then the circuit $\gamma'$ calculates, for the input $x+r2^{-h}2^{-m}$, 
$0 \leq r \leq 2^h$, the value $y+(y'-y)r2^{-h}$. 

\noindent 
Note also that $ \sz(\gamma') = \Oo(\sz(\gamma )+h) $ and $\depth(\gamma') = \Oo(\depth(\gamma)+h)$.	 
\end{proof}

\begin{proof}[Proof of Proposition \ref{424}] 
According to \lemref{425}, the condition (\F4.2.4) of the 
can be replaced by the condition: 
\[\forall x \in [0,1]  \;\;
 \abs {\wi{f}(x) - \wi{g}(x)} \;  \leq 2^{-(n-1)}  
   \eqno(\F4.2.5)
\] 
Indeed: if $q < n$ then (\F4.2.4) follows from (\F4.2.5), otherwise, we use linear interpolation given in the proof of \lemref{425}.

\noindent 
We now seek to simulate the boolean circuit $f = (\gamma,n,m,k)$ by a binary semilinear circuit $g$ to within $ 1/2^n$. 

\noindent 
First of all, let's note that it's easy to simulate all the gates exactly, ``except for the input'', by a simple $\lambda$ semilinear circuit which consists of:

\noindent 
1) replacing the boolean constants $ 0$ and $ 1$ of $ \gamma$ by the rational rational constants $0$ and $1$.

\noindent 
2) replacing each vertex of $\gamma$ calculating $\neg u$ by a vertex of 
$\lambda$ calculating $1-u$. 

\noindent 
3) replacing each vertex of $ \gamma $ calculating $u \land v$ by a vertex 
of $ \lambda $ calculating $\min(u,v)$.

\noindent 
4) replacing each vertex of $ \gamma $ calculating $u \lor v$ by a vertex
 of $ \lambda $ calculating $\max(u,v)$. 

\noindent 
5) calculating, from the outputs $c_0, c_1, \ldots,c_{n+k}$ of the circuit 
$\gamma$, the rational point:
\[
\pm 2^k \sum_{i=1,\ldots,n+k}c_i2^{-i} \; (c_0 \;\rm{ codes \;the \;sign})
\]
It is clear that the circuit $\lambda$ is constructed in linear time and 
has a size $\sz(\lambda) = \Oo(\sz(\gamma))$ and a depth $\depth(\lambda) = \Oo(\depth(\gamma))$.

\noindent 
Now we want to ``simulate the input'' of the circuit $\gamma$, 
i.e.\ the bits encoding $x$, by a binary semilinear circuit.

\noindent 
Usually, to determine the binary expansion of a real number $x$, we use the following pseudo-algorithm which uses the discontinuous function function $ C $ defined by:
\[
C(x) = \left\{
\begin{array}{cl} 
1 & {\rm if } \; x \geq 1/2 
\\
0 &
{\rm otherwise} 
\end{array}
\right.
\]

\sni {\bf Algorithm 1} (calculation of the first $m$ bits of a real number $x\in [0,1]$)
\begin{itemize}

\item [] Inputs: $ x \in [0,1], \; m \in \N$ 

\noindent 
 Output: the list $(b_1, b_2,\ldots, b_m) \in \{ 0,1 \}^m$ 
	\begin{itemize}
	
\item [] 

For $j:=1$ to $m$ Do
		\begin{itemize}
		
\item [] 

$ b_j \leftarrow C(x)$ 

\noindent 
		$	 x \leftarrow 2x - b_j$ 
		\end{itemize}
	 Done 

\noindent 
 	End.
\end{itemize}
\end{itemize}

\sni The map $C$ is discontinuous, so Algorithm 1 {\em is not really 
an algorithm}; and it cannot be simulated by a binary semilinear 
binary circuit. We then consider the continuous function
\[
C_p(x) := C_{p,1/2} = \min(1, \max(0,2^p(x-1/2))) \eqno\hbox{(cf.\ Figure \ref{fi336})}
\]
which ``approximates'' the 
function $C$. And we consider the following algorithm:

\mni {\bf Algorithm 2} (approximate calculation of the first $ m $ bits which 
encode a real number $x$)
\begin{itemize}

\item [] 

Inputs: $ x \in [0,1], \; m \in \NN_1$ 

\noindent 
 Output: a list $(b_1, b_2,\ldots, b_m) \in [ 0,1 ]^m$ 
	\begin{itemize}
	
\item [] 

$p \leftarrow m+2$ 

\noindent 
			 For $j:=1$ to $m$ Do
		\begin{itemize}
		
\item [] 

$ b_j \leftarrow C_p(x)$ 

\noindent 
			 $x \leftarrow 2x - b_j$ 

\noindent 
				$p \leftarrow p-1$
	 \end{itemize}
	 Done 

\noindent 
 	End.
\end{itemize}
\end{itemize}

\sni This is achieved by a circuit of size and depth $\Oo(m)$ in the following form.

\mni {\bf Algorithm 2bis} (semilinear circuit form of algorithm 2)
\begin{itemize}

\item [] 

Inputs: $ x \in [0,1], \; m \in \NN_1$ 

\noindent 
 Output: a list $(b_1, b_2,\ldots, b_m) \in [ 0,1 ]^m$ 
	\begin{itemize}
	 
\item [] 

$p \leftarrow m+2 , \; q \leftarrow 2^p$ 

\noindent 
			 For $j:=1$ to $m$ Do
				\begin{itemize}
			 
\item [] 

$y \leftarrow q(x-1/2)$ 

\noindent 
			 		 $ b_j \leftarrow \min (1, \max(0,y))$ 

\noindent 
					$x \leftarrow 2x - b_j$ 

\noindent 
 				 $q \leftarrow q/2$
				 \end{itemize}
	 Done 

\noindent 
 	End
	 \end{itemize}
\end{itemize}

\sni But a crucial problem arises when the input real $x$ is in 
an interval of type $\;  (k/2^m, k/2^m + 1/2^p) \; $ where 
$0 \leq k \leq 2^m-1$. In this case, at least one bit calculated in algorithm 2 
is in $(0,1)$. Consequently, the final result of the 
boolean circuit simulation may be inconsistent. 
To get round this difficulty, we use a technique introduced by 
Hoover \cite{Ho87,Ho90}. We make the following remarks.

\noindent 
--- For all $x \in [0,1]$ at most one of the values 
$x_{\sigma} = x + \frac \sigma   {2^{m+2}}$ (where $ \sigma \in \{-1, 0, 1 \} $) is in an interval of the previous type.

\noindent 
--- Let $z_{\sigma} = \sum_{j=1,\ldots,m}b_j2^{-j}$ where the $b_j$ are 
provided by Algorithm 2bis on the input $x_{\sigma}$. According to the 
previous remark, at least two values 
$\wi {\lambda}(z_{\sigma})$ $(\sigma \in \{-1,0,1\})$, 
correspond exactly to the output of the arithmetic circuit $\gamma$ 
when the first $m$ bits of $x_{\sigma}$ are input. 
In other words, for at least two values of $ \sigma $, we have $ z_{ \sigma } \in 
\DD_m $ and 
$\wi \lambda(z_{\sigma}) = \wi{f}(z_{\sigma})$ with in addition
 $\abs{ z_{\sigma} - x }\;  \leq 1/2^m + 1/2^{m+2} \leq 1/2^{m-1}$ 
	(hence $\abS{ \wi{f}(z_{\sigma}) - \wi{f}(x) }\;  \leq 1/2^{n-1}$).

\noindent 
--- Therefore, any incorrect value calculated by the semilinear 
semilinear circuit can only be the first or the third of the values if 
we order them in ascending order.
Thus, if $y_{-1}$, $y_0$ and $y_1$, are the three values respectively 
calculated by the circuit at the points $x_{\sigma}$ ($\sigma=-1,0,1$), 
then the second of the 3 values correctly approaches $\wi{f}(x)$. 
Given three real numbers $y_{-1}$, $y_0$ and $y_1$, the 2nd one in ascending order, $\theta(y_{-1},y_0,y_1)$, is calculated by one of the two semilinear circuits represented by the second member in the equations:
\[
\theta (y_{-1},y_0,y_1) = y_{-1}+y_0+y_1 - \min (y_{-1},y_0,y_1) - \max (y_{-1},y_0,y_1)
\]
\[
\theta (y_{-1},y_0,y_1) = \min (\max (y_0,y_1), \max (y_1,y_{-1}), \max (y_0,y_{-1}))
\]
To summarise: from the input $x$ we calculate the $x_{\sigma} = x +\frac \sigma {2^{m+2}}$ (where $\sigma \in \{-1,0,1 \}$), and we apply successively:

\noindent 
--- the circuit $\varepsilon$ (of size $\Oo(m) = \Oo(N)$) which decodes 
correctly the first $m$ digits of $x_{\sigma}$ for at least two of the three $x_{\sigma}$; 

\noindent 
--- the circuit $\lambda$ which simulates the boolean circuit proper and 
recodes the output digits in the form of a dyadic, this circuit is of size 
$\Oo(N)$;

\noindent 
--- then the circuit $\theta$ which chooses the 2nd calculated value in 
ascending order.

\noindent 
We therefore obtain a circuit which calculates the function:
\[
\wi{g}(x) = \theta (\lambda(\varepsilon (x-1/2^p)),\lambda(\varepsilon (x)), \lambda(\varepsilon (x+1/2^p))) \; {\rm with} \; p = m+2
\]
of size and depth $\Oo(N)$ and such that
\[
\abs{ \wi{f}(x) - \wi{g}(x) }\;  \leq 2^{-n} \; \forall x \in 
[0,1].
\]
\begin{figure}[htbp] 
\begin{center}
\includegraphics*[width=3cm]{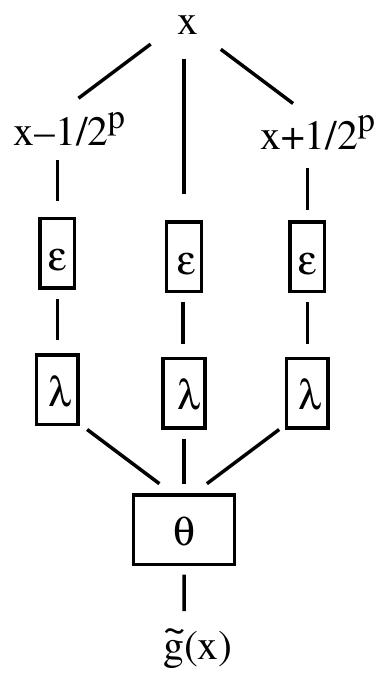}
\end{center}
\caption[global view of the calculation]{\label{fi422} 
global view of the calculation} 
\end{figure} 
\noindent In addition, the time taken to calculate (the code for) $g$ from the 
input $f$ is also in $\Oo(N)$. 
\end{proof}

We now move on to the simulation of a binary semilinear circuit by an arithmetic circuit (with divisions). First we give a circuit version of  Proposition \ref{335}.

\begin{proposition} \label{426}
The functions $(x,y) \mapsto \max(x,y)$ and $(x,y) \mapsto \min(x,y)$ on the square $[-2^m, 2^m] \times [-2^m, 2^m]$ can be approximated to within $1/2^n$ by arithmetic circuits of size $\Oo(N^3)$, of magnitude $\Oo(N^3)$ and which can be calculated in the class $\DTI (\Oo(N^3))$ where  $(N = n+m)$.
\end{proposition} 

\begin{proof} 
The proof of \lemref{333} is repeated. The representation of polynomials 
approximating $ P_n $ and $ Q_n $ by means of circuits is more economical in space. First of all, we need to construct a circuit that calculates an 
approximation of $ e^{-1/n} $ to within $1/2^{n^3}$. 
Consider the Taylor expansion 
\[
F_m(z) = \sum_{0\leq k\leq m} \frac{(-1)^k}{k!} z^k
\] 
which approaches $ e^{z}$ to within $1/2^{m+2}$ if $m\geq 5$ and $0\leq z\leq 1$. 
Here, we can simply give a circuit which calculates $d_n = F_{n^3}(1/n)$, 
whose size and calculation time are a priori in $ \Oo(n^3) $ (whereas we 
had to explicitly give a dyadic approximation $c_{n,1}$ of $d_n$). 
We then construct a circuit that calculates a good approximation of $H_n(x) = \prod_{1 \leq k < n^2} (x+e^{-k/n})$ in the form $h_n(x) = (x+d_n)(x+d_n^2)\cdots(x+d_n^{n^2-1})$.
This requires a calculation time (and circuit size) in $\Oo(n^2)$. 
This means that the arithmetic circuit which calculates an approximation to within 
$1/2^n$ of $\abs{x}$ on $[0,1]$ 
is calculated in time $\Oo(n^3)$. 
 
\noindent It is also easy to see that the coefficient of magnitude is bounded by the size of $1/h_n(0)$, i.e.\ also an $\Oo(n^3)$. 	 
\end{proof}

Note that the magnitude coefficient can hardly be improved. 
On the other hand, it does not seem impossible that $ d_n $ could be 
 calculated by a circuit of smaller size than the naive size in $\Oo(n^3)$.

\begin{proposition} \label{427}
The identity of $\czu$ from $\csl$ to $\caf$ is uniformly of class $\p$, in fact of class $\DTI (N^4)$. More precisely, we have a discrete function which, starting from an element~$g$ of $\ysl$ of size 
$\sz(g) = t $ and depth $ \depth(g) = p $, and an integer 
$n \in \NN_1$, calculates in time $ \Oo( t (n+p)^3 ) $ an element 
$f = (\alpha,M) \in \yaf$

\noindent 
\spa of size $ \sz(\alpha) = \sz(g) \Oo((n+p)^3) $, 

\noindent 
\spa of depth $ \depth(\alpha) = \Oo(p(n+p)^3) $, 

\noindent 
\spa with magnitude coefficient $ M = \Mag(\alpha) = \Oo((n+p)^3) $, 

\noindent 
and such that $\abS{\wi{f}(x)-\wi{g}(x)}\; \leq 2^{-n} \; \; \forall x \in [0,1]$. 
\end{proposition} 

\begin{proof}
Let $ p = \depth(g). $ The gates of $ g $ for an input $x \in [0,1]$ take values in the interval 
$[-2^p,2^p]$. 
We want to simulate to within $2^{-n}$ the circuit $g$ by an arithmetic circuit. 
 Simply simulate the $\max$ and $\min$ gates of the 
semilinear circuit to within $2^{-(n+p)}$ on the interval $[-2^p,2^p]$. 
According to Proposition \ref{426}, each simulation requires an 
arithmetic circuit (with division), all the characteristics of which are 
bounded by an $\Oo(n+p)^3$.	
\end{proof}

We move on to the simulation of an arithmetic circuit with divisions by a polynomial arithmetic circuit.

\begin{proposition} \label{428}
The identity of $\czu$ from $\caf$ to $\capo$ is uniformly of class $\p$, in fact of class $\DTI (N^2)$. 
More precisely, we have a discrete function which, starting from an element 
 $f=(\alpha,M)$ of $\yaf$ and an integer 
$n\in\NN_1$, calculates in time $\Oo(N^2)$ an arithmetic polynomial circuit $(\Gamma, M') \in \yap $ ($N$ is the size of the input $((\alpha,M),n)$ 

\noindent 
--- of size $ \sz(G) = \Oo( \sz(\alpha) (n + M) ) $, 

\noindent 
--- of depth $ \depth(G) = \Oo( \depth(\alpha) (n + M) ) $, 

\noindent 
--- of magnitude $ \Mag(G) = M'= \Oo( n + M ) $, 

\noindent 
--- and such that $\abS{ \wi{f}(x) - \wi{g}(x) }\;  \leq 2^{-n}$ 
for all $x \in [0,1]$.
\end{proposition} 

The problem only arises at the level of the ``pass to inverse'' gates. 
We therefore try to simulate them using polynomial circuits while 
keeping the magnitude well bounded.

\begin{lemma} \label{429}
The map $ x \mapsto 1/x $ on the interval $[2^{-m}, 2^m]$ can be 
realised to within $ 1/2^n $ by a polynomial circuit of 
magnitude $ \Oo(m)$, of size $ \Oo(m+n) $ and which is constructed in 
linear time.
\end{lemma}

\begin{proof}
Like Hoover, we use Newton's method to calculate 
the inverse of a real $z$ to within $2^{-n}$ in a reasonable number of additions and multiplications w.r.t.\ to $n$.

\smallskip \noindent 
{\bf Newton's method} (for calculating the inverse of $z$)

\noindent 
 For $2^{-m} < z < 2^m$ we define 
\[
C(x) = \left\{
\begin{array}{l} 
y_0 = 2^{-m} 
\\
y_{i+1} = y_i (2- zy_i) 
\end{array}
\right.
\]
It is easy to check that, for $i \geq 3m + \log(m+n)$, we have
\[
\abs{ z^{-1} - y_i }\;  \leq 2^{-n}.
\]
So for the input $(n,m) \in \NN_1 \times \NN_1$, the application of Newton's method up to iteration $i = 3m + \log(m+n)$ is represented by a polynomial circuit of size $\Oo(m+n)$ and magnitude $\Oo(m+n)$. It is also easy to check that the circuit is constructed in linear time.
\end{proof}

\begin{proposition} \label{4210}
The identity of $\czu$ from $\capo$ to $\ckf$ is uniformly of class $\p$, in fact of class $\DRT(\Lin, \Oo(N^4))$.
\end{proposition} 

\begin{proof}
The fact of being of class $\p$ follows from \thref{417} and Proposition \ref{4111}. A precise reading of the proofs of \thref{417} and Proposition \ref{4111} gives the result $\DRT(\Lin,\Oo(N^4))$ (taking into account the Nota bene after \ref{4111}). 
\end{proof}

Summarising the previous results we obtain the following theorem.

\begin{theorem} \label{4211}
The five presentations $\ckf$, $\cbo$, $\csl$, $\caf$ and $\capo$ of $\czu$ are $\p$-equivalent.
\end{theorem}

\begin{notation} \label{4212}
As long as we are at a sufficiently high level of complexity to make the comparison theorem valid (in particular the class $\p$ suffices) there is no reason to make any difference between the five presentations $\ckf$, $\cbo$, $\csl$, $\caf$ and $\capo$ of $\czu$. 
Consequently, from now on the notation $\czu$ will mean that we consider the space $\czu$ with the calculability structure~$\csl$.
\end{notation}

\subsection{Complexity of the norm problem}\label{subsec43}
We know that determining the maximum, on an interval $[0,1]$ of a function computable in polynomial time $f\colon \NN \to \{ 0,1 \}$ is more or less the same thing as the most classic $\np$-complete problem: SAT.

It is therefore not surprising to find as an $\np$-complete problem a problem related to the calculation of the norm for a continuous function. First we need to formulate the norm problem attached to a given rational presentation of the space $\czu$ in a sufficiently precise and invariant way. 
 
\begin{definition}\label{431}~\\
We call ``the norm problem'', relative to a presentation $\ca_1 = (Y_1, \delta_1, \eta_1)$ of $\czu$) the problem:

\noindent 
--- Solve \emph{approximately} the question ``$\, a \leq \norme{f}_{\infty}\;?\,$'' in the presentation $\ca_1$ of $\czu$.

\noindent 
The precise formulation of this problem is as follows:

\noindent 
--- Inputs: $(f,a,n) \in Y_1 \times \DD \times \NN_1$ 

\noindent 
--- Output: Correctly provide one of the two answers:

\noindent 
~~~~~~~~~~~~--- I give you an $x \in \DD$ such that 
$\abs{f(x)} \geq a - 1/2^n$, i.e.\  $\norme{f}_{\infty}\geq a - 1/2^n$,

\noindent 
~~~~~~~~~~~~--- there is no $x \in \DD$ satisfying 
$\abs{f(x)} \geq a$,  i.e.\  $\norme{f}_{\infty}\leq a$.
\end{definition}

This definition is justified by the following lemma.

\begin{lemma} \label{432}
For two polynomially equivalent rational presentations $\ca_1$ and $\ca_2$ of $\czu$, the corresponding norm problems are also polynomially equivalent.
\end{lemma}

\begin{proof}
The transformation from the problem corresponding to one presentation $\ca_1$ to the problem corresponding to another presentation $\ca_2$ is carried out by an algorithm with the same complexity as the algorithm used to present the identity function between $\ca_1$ and $\ca_2$. For the data $(f,a,n) \in Y_1 \times \DD \times \NN_1$, we look for $g \in Y_2$ such that $\norme{ f-g }_{\infty} \leq 2^{-(n+2)}$, then solve the problem with the inputs $(g,a-1/2^{n+1}, n+2)$. 

\noindent 
If we find $x \in \DD$ such that $\abs{g(x)}  \geq a - 2^{-(n+1)}- 2^{-(n+2)}$, then $\abs{f(x)}  \geq a - 2^{-n}$. 

\noindent 
If we declare a forfeit, it means that there is no $x \in \DD$ such that $\abs{g(x)}\;  \geq a - 2^{-(n+1)}$. 
A fortiori, there is no $x \in \DD$ such that $\abs{f(x)} \geq a$. 
\end{proof}

\begin{theorem} \label{433}~\\
The norm problem is $\np$-complete for the presentations $\ckf$, $\cbo$, $\csl$, $\caf$ and $\capo$.
\end{theorem}
\begin{proof}
According to \lemref{432} it suffices to do the proof for the presentation $\cbo$. 
The $\np$ character of the problem is immediate. 
To see the $\np$-hardness, we consider the norm problem limited to the inputs $((\gamma,1,m,1), 3/4, 2)$ where $\gamma$ is an arbitrary boolean circuit with $m$ inputs and one output (the quadruplet is then obviously correct), and the answer yes corresponds to the satisfiability of the  circuit $\gamma$.
\end{proof}

We also have the following result, which is essentially negative,\footnote{Because $\p \neq \np$ !.} and therefore less interesting.

\begin{proposition} \label{434}
For the considered presentations, the norm function $f \mapsto \norme{f}_{\infty}$ from $\czu$ to $\R^{+}$ is uniformly of class $\p$ if and only if $\p = \np$.
\end{proposition}
\begin{proof}
It suffices to reason with the presentation $\cbo$. If the norm function is $\p$-computable,\footnote{We sometimes use the terminology $\p$-computable as an abbreviation for computable in polynomial time.} the norm problem is solved in polynomial time, and therefore $\p$ = $\np$.

\noindent 
If $\p = \np$, the norm problem is solved in polynomial time, which means that the norm can be calculated by dichotomy by initialising with the $2^k$ upper bound, until the accuracy of $1/2^q$ is obtained. This requires $k+q$ dichotomy steps. 
The whole calculation is in polynomial time on the input $(g,q) \in \ybc \times \NN_1$. 
\end{proof}

\begin{corollary} \label{435}
The identity of $\czu$ from $\ckf$ to $\crf$ is not computable in polynomial time, at least if $\p \neq \np$.
\end{corollary}
\begin{proof}
The norm function is computable in polynomial time for the presentation $\crf$ according to Proposition \ref{327}. We conclude with the previous proposition. 
\end{proof}

\begin{proposition} \label{436}
For the five previous presentations of $\czu$, if the evaluation is in $\DSPA(S(N))$ with $S(N) \geq N$, then the norm function is also in $\DSPA(S(N))$.
\end{proposition}
\begin{proof}
If $k \mapsto \mu(k)$ is a modulus of uniform continuity for a rational point $\wi{f}$ of a given presentation of $\czu$, then to compute the norm with precision $n$ it suffices to evaluate $\wi{f}$ over the elements of $\DD_{m,[0,1]}$ (where $ m = \mu(n)$) and take the maximum value.
Since useless intermediate results are immediately deleted, and since $S(N) \geq N \geq m$ the computation space of the norm function is the same as that of the evaluation function. 
\end{proof}

\subsection{Complexity of the considered rational presentations}\label{subsec44}
In this section, we briefly present the complexity of the membership test and the vector space operations for the five considered presentations, and we summarise all the results obtained. 

\begin{proposition} \label{441}
The membership test (for the set of codes of rational points) is:

\noindent 
--- $\LINS$ and $\cnp$-complete for  presentations $\ckf$ and $\cbo$; 

\noindent 
--- $\LINT$ for the presentation $\csl$.
\end{proposition}
 
\begin{proof}
For the presentation $\csl$, this is obvious. The proofs are essentially the same for the two presentations $\ckf$ and $\cbo$. 
We only give one for each.

\noindent 
Let us see that the membership test is $\LINS$ for $\ykf$. 
For an entry $(Pr,n,m,T)$ we do the following calculation: 

\noindent 
 for $i = 1,\ldots,2^m$ check that
\[
Pr(i/2^m) \in \DD_n \quad  {\rm and} \quad  \Abs{ Pr((i-1)/2^m) - Pr(i/2^m) }\;  
\leq 1/2^n.
\]
This calculation is $\LINS$.\\
For the $\cnp$-completeness of the membership test, we give the proof for boolean circuits. We can restrict ourselves to the inputs $(\gamma,2,m,0)$ where $\gamma$ is a circuit which only calculates an output corresponding to the first bit, the other two bits being zero. Consistency is required at two consecutive points on the grid. Only constant functions are therefore tolerated.

\noindent 
The problem opposite to the membership test is to find out whether a boolean circuit is non-constant, which involves solving the satisfiability problem. 
\end{proof}

\begin{proposition} \label{442}
 The vector space operations (on the set of of codes for rational points) are in $\LINT$ for the five presentations $\ckf$, $\cbo$, $\csl$, $\caf$ and $\capo$ of $\czu$.
\end{proposition} 
\begin{proof}
The calculations are obvious. For example if $(f_1,\ldots,f_s) \in \lst(\ykf )$ and $n \in \NN_1$, we can easily calculate $f \in \ykf $ such that: 
\[
\NOrme{ \wi{f} - \sum_{i=1,\ldots,s} \wi{f_i} } \leq 
1/2^n
\] 
because it is sufficient to know each $\wi{f_i}$ with the precision $1/2^{n+\log(s)}$. 
\end{proof}

The only ``drama'' is obviously that the presentations $\ckf$ and $\cbo$ are not $\p$-presentations of $\czu$ (unless $\p = \np$ cf.\ Proposition \ref{434}.).

To finish this section we give a summary table in which we group together almost all the complexity results established for the five rational presentations $\ckf$, $\cbo$, $\csl$, $\caf$ and $\capo$ of~$\czu$.

\bni {
\begin{tabular}{ c l p{3cm} c c } 
& Evaluation & Norm & Membership 
& Vectorial \cr
& & function 
& test 
& space \cr
&&&&operations\cr\cr
 $\ckf$ & $\DSRT(\Lin,\Lin,\Oo(N^2))$ & $\LINS$ and & $\LINS$ and & $\LINT $\cr
 and $\cbo$ & &$\np$-complete 
& $\cnp$-complete & \cr\cr
 $\csl$ & $\DSRT(\Oo(N^2),\Lin,\Oo(N^2))$ & $\DSPA(\Oo(N^2))$ 
& $\LINT $ & $\LINT $\cr
 & &and $\np$-complete\cr\cr
$\caf$ & $\DSRT(\Oo(N^3),\Lin,\Oo(N^4))$ & 
$\DSPA(\Oo(N^3))$ 
& $\PSP$ & $\LINT $\cr
 and $\capo$ & & and $\np$-complete
\end{tabular}
}

\bni For the $\PSP$ membership test, its complexity is probably much less.

\begin{remark}\label{444}
Despite the ease with which the evaluation function can be calculated for the presentations $\ckf$ and $\cbo$, it is still the presentation using binary semilinear circuits that seems to be the simplest. 
Its consideration has also shed light on the comparison \thref{4211}, which is a strengthened, uniform version of the results established by Hoover. 
\\ 
The unavoidable flaw (if $\p \neq \np$) of the presentations defined so far is the non-feasibility of calculating the norm. 
This prevents us from having a feasible control procedure for Cauchy sequences of rational points. 
This reduces the interest in $\p$-points of $\csl$. 
This underlines the fact that it is somewhat artificial to study the $\p$-points of a space that is given in a presentation of non-polynomial complexity.\\
In addition, problems that are at least as difficult a priori as the calculation of the norm, such as the calculation of a primitive or the solution of a differential equation, are also without hope of a reasonable solution within the framework of the presentations we have just studied.
\\ 
It is therefore legitimate to turn to other rational presentations of $\czu$  to see to what extent they are better suited to the aims of numerical analysis.
\end{remark}

\section{Some  \texorpdfstring{$\p$}{P}-presentations 
for the space \texorpdfstring{$\czu$}{C[0,1]}}\label{sec5}

In this section we address the question of how far rational $\p$-presentations  of the space $\czu$ provide a suitable framework for numerical analysis. This is only a first study, which should be seriously developed.

\subsection{Definitions of some $\p$-presentations}\label{subsec51}

\subsubsection{Presentation \texorpdfstring{$\cw$}{Cw}
 (à la Weierstrass)}\label{subsubsec511}

The set $\yw$ of codes for rational points of the presentation $\cw$ is simply the set $\DD[X]$ of (one-variable) polynomials with coefficients in $\DD$ given in dense presentation.

\smallskip So $\cw=(\yw,\eta,\delta)$ where the reader will give precisely $\eta$ and $\delta$ as required in Definition \ref{211}.

\smallskip A $\p$-point $f$ of $\cw$ is therefore given by a $\p$-computable sequence: 
\[
m \mapsto u_m \; : \; \NN_1 \to \DD[X] \quad {\rm with }\quad \forall m \, \norme{ u_m - f }_{\infty} \leq 1/2^m.
\]

And a $\p$-sequence $f_n$ of $\cw$ is given by a $\p$-computable double sequence: 
\[
(n,m) \mapsto u_{n,m} \; : \; \NN_1 \times \NN_1 \to \DD[X] \quad {\rm with }\quad  \forall n,m  \; \norme{ u_{n,m} - f_n }_{\infty} 
\leq 1/2^m
\]

\begin{remark} \label{511}
An equivalent definition for a $\p$-point $f$ of $\cw$ is obtained by requiring that $f$ be written as the sum of a series $\sum_ms_m$, where $(s_m)_{m\in\NN_1}$ is a $\p$-computable sequence in $\DD[X]$ satisfying: 
$\norme{ s_m }_{\infty} \leq 1/2^m$. 
This gives a nice way of presenting the $\p$-points of $\cw$. Indeed, we can check in polynomial time (with respect to $m$) that the sequence is correct for the terms from~$1$ to~$m$. Moreover, from the point of view of lazy computation, we can control the sum of the series as the need for precision increases. This remark is valid for any other rational presentation of class~$\p$ whereas it would not be valid for the presentations studied in Section \ref{sec4}.
\end{remark}

The following result is immediate.

\begin{proposition} \label{512}
The presentation $\cw$ of $\czu$ is of class $\p$.
\end{proposition}

\smallskip Here is a result comparing the rational presentations $\crf$ and $\cw$.

\begin{proposition} \label{513}~

\noindent 
--- The identity of $\czu$ from $\cw$ to $\crf$ is $\LINT$.

\noindent 
--- The identity of $\czu$ from $\crf$ to $\cw$ is not of class $\p$.
\end{proposition}

\begin{proof}
The first statement is trivial. 
The second results from the fact that the function $x \mapsto \abs {x-1/2}  $ is a $\p$-point of $\crf$ (\thref{334}) whereas all the $\p$-points of $\cw$ are infinitely differentiable functions (cf.\ \thref{527} below). 
\end{proof}

The interest of the presentation $\cw$ is underlined in particular by the characterisation theorems (cf.\ Section \ref{subsec52}) which specify  ``well-known''  phenomena in numerical analysis, with Chebyshev polynomials as a method of attacking the problems.

\subsubsection{Presentation \texorpdfstring{$\csp$}{Csp}
(via semipolynomials in dense presentation)}\label{subsubsec512}

This is a presentation which significantly increases the set of $\p$-points (compared with $\cw$). An element of $\ysp$ represents a piecewise polynomial function (also called a semi-polynomial) given in dense presentation. 

More precisely $\ysp \subset \lst(\DD) \times \lst(\DD[X])$, and the two lists in $\DD$ and $\DD[X]$ are subject to the following conditions:

\noindent 
--- the list $(x_i)_{0 \leq i \leq t}$ of dyadic rational numbers is ordered in ascending order: 
\[
0 = x_0 < x_1 < x_2 <\cdots< x_{t-1} < x_t =1;
\]
--- the list $(P_i)_{1 \leq i \leq t}$ in $\DD[X]$ verifies $P_i(x_i) = P_{i+1}(x_i)$ for $ 1 \leq i \leq t-1$.

\noindent 
The code $f = ((x_i)_{0 \leq i \leq t},(P_i)_{1 \leq i \leq t})$ defines the rational point $\wi{f}$, that is the function which coincides with  $\wi{P_i}$ on each interval $[x_{i-1},x_i]$. 

\smallskip So $\csp=(\ysp,\eta,\delta)$ where the reader will give precisely $\eta$ and $\delta$ as required in Definition \ref{211}.

\smallskip 
The presentation $\csp$ of $\czu$ is clearly of class $\p$.
The following proposition is proved in the same way as Proposition \ref{513}.

\begin{proposition} \label{514} ~

\noindent 
--- The identity of $\czu$ from $\cw$ to $\csp$ is $\LINT$.

\noindent 
--- The identity of $\czu$ from $\csp$ to $\cw$ is not of class $\p$.
\end{proposition} 

\subsubsection{Presentation $\csr$ (via controlled semirational function s given in presentation by formula)}\label{subsubsec513}

The set $\ysr$ of codes of rational points is now the set of codes of piecewise rational functions (also called semi-rational function s) with dyadic coefficients. 

More precisely, $\ysr \subset \lst(\DD) \times \lst(\DD[X]_f \times \DD[X]_f)$, and the two lists in $\DD$ and $\DD[X]_f \times \DD[X]_f$ are subject to the following conditions:

\noindent 
--- the list $(x_i)_{0 \leq i \leq t}$ of dyadic rational numbers is ordered in ascending order: 
\[
0 = x_0 < x_1 < x_2 <\cdots< x_{t-1} < x_t =1;
\]
--- each pair $(P_i,Q_i) \; (1 \leq i \leq t$ from the 2nd list represents a rational function $R_i = P_i/Q_i$ with its denominator $\geq  1$ on the interval $[x_{i-1} , x_i]$;

\noindent 
--- the list $(R_i)_{1 \leq i \leq t}$ verifies $R_i(x_i) = R_{i+1}(x_i)$
 for $1 \leq i < t$.
\\
The code  $f = ((x_i)_{0 \leq i \leq t},(P_i,Q_i)_{1 \leq i \leq t})$ defines the rational point  \smash{$\wi{f}$}, that is the function which coincides with $\wi{R_i}$ on each interval $[x_{i-1} , x_i]$.

\smallskip So $\csr=(\ysr,\eta,\delta)$ with $\eta(f)=\wi f$ and the reader will give $\delta$ as required in Definition~\ref{211}.

\smallskip The presentation $\csr$ of $\czu$ is clearly of class $\p$.

\subsection {Results concerning the Weierstrass presentation}\label{subsec52}

This section is mainly a  devlopment of Section C-c in the paper \cite{Lo89} which is the third part of the PhD thesis of the second author.

We use the classical references for the theory of approximations 
\cite[Bakhvalov, 1973]{Ba},
\cite[Cheney, 1966]{Ch} et \cite[Rivlin, 1974]{Ri}. Concerning the Gevrey class we refer to \cite[Hörmander, 1983]{Ho}. Our notations are as in \cite{Ch}.   

\smallskip  
Before characterising the $\p$-points of $\cw$, we need to recall some classical results from the theory of uniform approximation by polynomials.

\medskip\noindent  {\bf Warning!} Given the usual way in which the theory of approximation is formulated, we will use the interval $[-1,1]$ to give the results and proofs concerning $\cw$.

\subsubsection{Some definitions and results in the theory of uniform approximation by polynomials}\label{subsubsec521}

See for example \cite{Ba}, \cite{Ri} and \cite{Ch}. 

\begin{notation} \label{521}~ 
\begin{itemize}\itemsep2pt

\item [] 
$\cab$ is the space of continuous real functions on the segment $[a,b]$.

\item [] 
$\C$ is the space $\cuu$, the uniform norm on this interval is denoted by $\norme{f}_{\infty}$ and the corresponding distance $d_{\infty}$. 

\item []
 $\C^{(k)}$ is the space of functions $k$ times continuously differentiable on $[-1,1]$. 

\item [] 
$\Ci$ is the space of indefinitely differentiable functions on $[-1,1]$. 

\item [] 
$\po_n$ is the space of polynomials of degree $\leq n$. 

\item [] 
$\Tch_n$ is the Chebyshev polynomial of degree $n$: 
\[
\Tch_n\big(\varphi (z)\big)= \varphi(z^n) \; {\rm with } \; \varphi(z) = \frac {1} {2} (z + 1/z)
\]
they can also be defined by $\Tch_n\big(\cos (x)\big) = \cos (nx)$ or by 
\[
F(u,x) = \frac {1-u\,x} {1-u^2-2u\,x} = \sum_{n=0}^{\infty} \Tch_n(x)u^n
\]
\item []
We note \fbox{$E_n(f) = d_{\infty} (f, \po_n)$} for $f \in \C=\cuu$. 

\item [] 
We consider on $\C$ the scalar product 
\[
\left < g,h \right > := \int_{-1}^1 \frac {g(x)\,h(x)} {\sqrt {1-x^2}}\, dx 
= \int_0^{\pi} g\big(\cos (x)\big)\,h \big(\cos (x)\big)dx. 
\]
Let $\norme{f}_2$ be the norm in the sense of this scalar product. 
The polynomials $(\Tch_i)_{0 \leq i \leq n}$ form an orthogonal $\po_n$-basis for this scalar product, with 
\[
\left < \Tch_0,\Tch_0 \right > = \pi \quad  {\rm and}  \quad  \left < \Tch_i,\Tch_i \right > = 
\pi /2 \; {\rm for } \; i>0.
\]

\item [] 
We note
\[
A_k = A_k(f):= \frac {2} {\pi} \int_{-1}^{1}  \frac {f(x)\,\Tch_k(x)} {\sqrt {1-x^2}}\, {dx}= \frac {2} {\pi} \int_{0}^{\pi} \cos (kx) f(\cos (x))dx
\]
The $A_k$ are called the \emph{Chebyshev  coefficients} of $f$.

\item [] 
The function 
\[
s_n(f) := A_0/2 + \sum_{i=1}^{n}A_i\Tch_i  \hbox{ also denoted by } 
 {\sum_{i=0}^n}\,'\, A_i\Tch_i
\]  
is the orthogonal projection of $f$ onto $\po_n$ in the sense of the above scalar product. 

\item [] 
The corresponding series is called {\em the Chebyshev series} of $f$.\footnote{It converges in the sense of $L^2$ for the scalar product. The Chebyshev series are to continuous functions on $[-1,1]$ what the Fourier series are to periodic continuous functions, which is easy to understand when you consider the change of variable $z\mapsto 1/2( z + 1/z )$ which transforms the unit circle of the complex plane into the segment $[-1,1]$ and the function $z\mapsto z^n$ into the polynomial $\Tch_n$.} 

\item [] 
We note \fbox{$S_n(f) := \norme{ f - s_n(f) }_{\infty}$}.
We immediately have $\abs {A_{n+1}(f)}\; \leq S_n(f) + S_{n+1}(f)$. 

\item [] 
The zeros of $\Tch_n$ are the points 
\[
\xi_i^{[n]} = \cos \left(\frac{2i-1} {n}\cdot \frac {\pi} {2}\right) \; \; \; i =1,\ldots,n
\]
and we have 
\[
\Tch_n(x) = 2^{n-1} \prod_{i=1}^{n} (x - \xi_i^{[n]}) \; \; ({\rm for}\; n \geq 1).
\] 

\item [] 

The extrema of $\Tch_n$ on $[-1,1]$ are equal to $\pm 1$ and obtained at the points
\[
\eta_i^{[n]} = \cos (\frac{i} {n}\cdot \pi) \; \; \; i =0,1,\ldots,n.
\] 

\item [] 

Approximate values of $s_n(f)$ can be calculated by means of interpolation formulae. We let
\[
\alpha_k^{[m]} = \frac {2} {m} \;
{\sum_{i=0}^m}\,'f(\xi_i^{[m]})\Tch_k(\xi_i^{[m]}), \; \; \; \; u_n^{[m]} = \sum_{k=1}^{n} \alpha_k^{[m]}\Tch_k(x)
\]
and we have: $u_n^{[n+1]}$ is the polynomial which interpolates $f$ at the zeros of $\Tch_{n}$
\end{itemize}
\end{notation}

The theory of uniform approximation by polynomials establishes close links between ``being sufficiently well approximated by polynomials'' and ``being sufficiently regular''.

\subsubsection{Some classical results}\label{subsubsec522}

You find most of these results in \cite{Ch}.

\smallskip  In this subsection the functions are in $\C=\cuu$.

\smallskip \noindent {\bf Evaluation of a polynomial} $P=\sum_{k=0}^{n}a_k\Tch_k$

\noindent 
 The recursive formulae $\Tch_{m+1}(x) = 2x\Tch_m(x) - \Tch_{m-1}(x)$ lead to a Horner-style algorithm:
\[
B_{n+1} = B_{n+2} = 0, \; B_k = 2xB_{k+1} - B_{k+2} + a_k, \; P(x) = \frac 
{B_0 - B_2} {2}
\]

\smallskip\noindent {\bf Markov's inequalities} 

\noindent 
If $g \in \po_n$ then (A.A. Markov, \cite[page 91]{Ch})
\begin{equation} \label{F 5.2.1}
\Norme{ g' }_\infty \leq n^2 \norme{ g }_{\infty}
\end{equation}
and for $k \geq 2$ (V.A. Markov, \cite[Theorem 2.24]{Ri})
\begin{equation} \label{F 5.2.2}
\Norme{ g^{(k)} }_\infty \leq \Tch_n^{(k)}(1) \norme{ g }_{\infty} = \frac {n^2(n^2-1)\cdots(n^2-(k-1)^2)} {1.3.5\cdots(2k-1)} \,\norme{ g }_\infty
\end{equation}

\smallskip\noindent {\bf Comparison of $E_n(f)$ and $S_n(f)$}
\begin{equation} \label{F 5.2.3}
E_n(f) \leq S_n(f) \leq \left(4+ \frac {4} {\pi^{2}} \log (n)\right)\,E_n(f)
\end{equation}

\smallskip\noindent {\bf Comparison of $E_n(f)$ and $A_{n+1}(f)$}

\noindent 
For $n \geq 1$ we have	
\[
\int_{-1}^{1} \frac {\abs {\Tch_n(x)}} {\sqrt {1-x^2}}\,dx = 2
\]
from which we deduce 
\begin{equation} \label{F 5.2.4}
(\pi /4) \abs {A_{n+1}(f)}  \leq E_n(f)
\end{equation}

\smallskip\noindent {\bf Jackson's Theorems}

\noindent 
Let $f \in \C$. For any integer $n \geq 1$ we have 
\begin{equation} \label{F 5.2.5}
E_n(f) \leq \pi \lambda /(2n+2) \; \; \; \hbox{if} \; \; \; \abs{f(x)-f(y)}\;  \leq \lambda \abs{x-y}
\end{equation}
\begin{equation} \label{F 5.2.6}
E_n(f) \leq (\pi /2)^k \Norme{ f^{(k)} }_{\infty} \big/ \big((n+1)(n)(n-
1)\cdots(n-k+2)\big) \; \; \hbox{if} \; f \in \C^{(k)} \; \hbox{and} \; n\geq k
\end{equation}

\smallskip\noindent {\bf Convergence of the Chebyshev series of a function}

\noindent 
The Chebyshev series of a function $f \in \C^{(k)}$ converges uniformly to $f$ if $k \geq 1$, and is absolutely convergent (for the norm $\norme{f}_{\infty}$) if $k \geq 2$. 
\begin{equation} \label{F 5.2.7}
S_n(f) = \norme{ s_n(f)-f }_{\infty}\; \leq \sum_{j=n+1}^{\infty} \abs {A_j}
\end{equation}
and (cf.\ \cite{Ri} Theorem 3.12 p. 182)
\begin{equation} \label{F 5.2.8}
\NOrme{ s_n(f)- u_n^{[n+1]} }_\infty\; \leq \sum_{j=n+2}^\infty \abs {A_j} 
\end{equation}

\smallskip\noindent {\bf Uniform approximation of functions in $\Ci$ by polynomials} 

\noindent 
The following properties are equivalent. 
\begin{itemize}\itemsep2pt

\item [(i)] 
$\forall k \; \; \exists M > 0 \; \; \forall n> 0, \; \; E_n(f) \leq M/n^k$. 

\item [(ii)]
$\forall k \; \; \exists M > 0 \; \; \forall n> 0, \; \; S_n(f) \leq M/n^k$.

\item [(iii)]
$\forall k \; \exists M > 0 \; \forall n> 0, \; \abs {A_n(f)}\; \leq M/n^k$.

\item [(iv)] 
$\forall k \; \exists M > 0 \; \forall n> 0, \; \NOrme{ u_n^{[n+1]} - f }_{\infty} \leq M/n^k$.

\item [(v)] 
The function $f$ is of class $\ca^{\infty}$ (i.e., $f \in \Ci$).
\end{itemize}

%
\begin{proof}
(i) and (ii) are equivalent according to (\ref{F 5.2.3}). 

\noindent 
(iv) $\Rightarrow$ (i) trivially. 

\noindent 
(ii) $\Rightarrow$ (iii) because $\abS{A_n(f)}\;  \leq S_n(f) + S_{n-
1}(f)$. 

\noindent 
(iii) $\Rightarrow$ (iv) according to (\ref{F 5.2.7}) and (\ref{F 5.2.8}). 

\noindent 
(iii) $\Rightarrow$ (v). The series $\sum'A_i\Tch_i^{(h)}$ is absolutely convergent according to (\ref{F 5.2.2}) and the inequalities~(iii); therefore we can derive $h$ times term by term the Chebyshev series. 

\noindent 
(v) $\Rightarrow$ (i) from (\ref{F 5.2.6}). 
\end{proof}

\smallskip\noindent {\bf Analyticity and uniform approximation by polynomials} 

\noindent 
The following properties are equivalent 
\begin{itemize}\itemsep2pt

\item [(i)] 

$ \exists M > 0, \;  r < 1 \; \; \forall n> 0  \;\;\, E_n(f) \leq Mr^n$.

\item [(ii)] 

$\exists M > 0, \;  r < 1 \; \; \forall n> 0  \;\;\, S_n(f) \leq Mr^n$. 

\item [(iii)] 

$\exists M >0, \;  r < 1 \;\; \forall n> 0 \; \abS{A_n(f)} \leq Mr^n$.

\item [(iv)] 

$ \exists M > 0, \;  r < 1 \; \; \forall n> 0  \; \NOrme{ u_n^{[n+1]} - f }_{\infty}   \leq Mr^n$.

\item [(v)] 

$\exists r < 1$ such that $f$ is analytic in the complex plane inside the ellipse $\sE_\rho$ of foci $1$, $-1$ and whose half-sum of principal diameters is equal to $\rho = 1/r$.

\item [(vi)] 

$ \exists M > 0, \;  R > 1 \; \; \forall n \;  \NOrme{f^{(n)} }_{\infty}  \leq MR^nn!$\,.

\item [(vii)] 

$f$ is analytic on the interval $[-1,1]$.
\end{itemize}
Furthermore the lower bound on the possible values of $r$ is the same in the first 5 cases.\footnote{The equivalences (i) \ldots (iv) are shown as for the previous proposition. For the equivalence with (v) see for example \cite{Ri}. Condition (vi) very nearly represents analyticity in the open $U_R$ formed by points whose distance from the interval is less than $1/R$.}

\begin{figure}[htbp] 
\begin{center}
\includegraphics*[width=10cm]{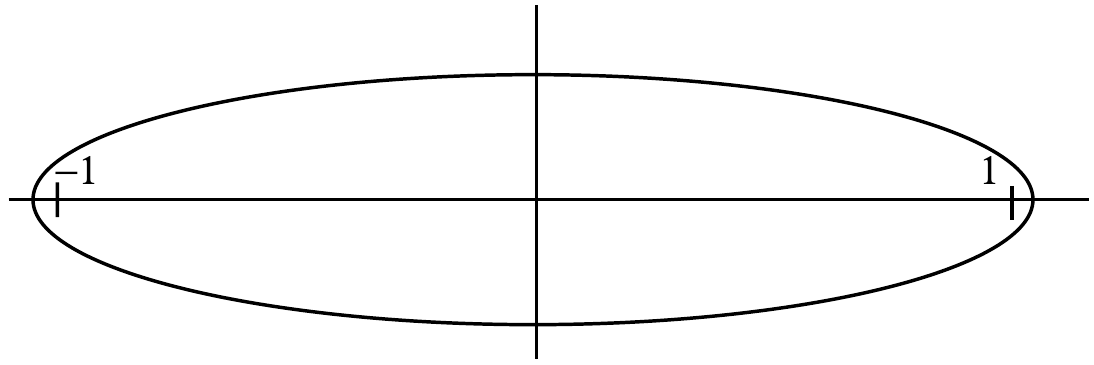}
\end{center}
\caption[The ellipse $\sE_\rho$]{\label{fi521} 
the ellipse $\sE_\rho$} 
\end{figure} 

\begin{remarks}\label{rem-anal}~

\noindent 
1) The space of analytic functions on a compact interval therefore has a good constructive description, in terms of Chebyshev series for example. It appears as a nested countable union of complete metric spaces (those obtained using definition (iii) and setting the integer $M$ and the rational $r$, for example). 
The space of $\Ci$-functions is much more difficult to describe constructively, mainly because there is no pleasant way of generating rapidly decaying sequences of rationals, due to the $\forall k\; \exists M$ in the definition of rapid decay.\footnote{This alternation of quantifiers takes an explicit form when we give explicitly $M$ as a function of $k$. But, by virtue of Cantor's diagonal argument, there is no effective way of generating effective functions from $\NN$ to $\NN$.} 

\noindent 
2) Condition (i) can also be read as follows: the function $f$ can be approximated to within $1/2^n$ (for the uniform norm) by a polynomial of degree $\leq c.n$, where $c$ is a fixed constant, i.e.\ again: there exists an integer $h$ such that
$E_{hn}(f) \leq 1/2^n$. 

\noindent 
 The same applies to conditions (ii), (iii) and (iv). This implies that the function $f$ can be approximated to within $1/2^n$ by a polynomial with dyadic coefficients whose size (in dense presentation on the basis of $X^n$ or on the basis of $\Tch_n$) is in $\Oo(n^2)$. The size of the sum of the absolute values of the coefficients is $\Oo(n)$. 
Bakhvalov (cf.\ \cite {Ba} IV-8 Th.\ p.\ 233) gives a sufficient condition of the same kind for a function $f$ to be analytic in a lens with extremities $-1$ and $1$ of the complex plane (and no longer in a neighbourhood of the segment): it is sufficient that the sum of the absolute values of the coefficients of a polynomial giving $f$ to within $1/2^n$ is bounded by $M2^{qn}$ (where $M$ and $q$ are fixed constants). In other words
the size of the sum of the absolute values of the coefficients of a polynomial approximating $f$ to within $1/2^n$ is $\Oo(n)$.
\begin{figure}[htbp] 
\begin{center}
\includegraphics*[width=12cm]{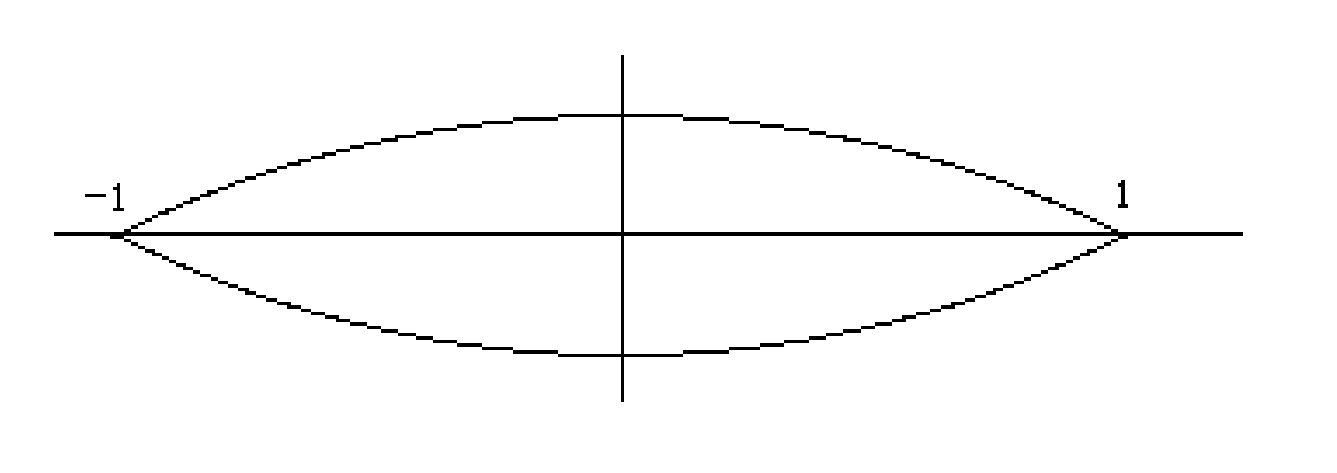}
\end{center}
\caption[Bakhvalov's lens]{\label{fi522} 
Bakhvalov's lens} 
\end{figure} 
\end{remarks}

\smallskip\noindent {\bf Gevrey class and uniform approximation by polynomials} 

\noindent 
If $f$ is a $\p$-point of $\cw$ given by a $\p$-sequence $\; m \mapsto P_m\; $
 (with $\Norme{ f - P_m}_{\infty} \leq 2^m$), then the degree of $P_m$ is bounded by a polynomial in $m$, so there exists an integer $ k$ and a constant $ B$ such that the degree of $ P_m$ is bounded by $(Bm)^k$. 
Let $ n $ be arbitrary, and consider the largest integer~$ m $ such that $(Bm)^k \leq n$, i.e.\ $m:= \Flo{\sqrt[k]{n}/B}$. 
We therefore have $m+1 \geq \sqrt[k]{n}/B$. By letting $r := 1/2^{1/B}$
 and $\gamma := 1/k$, we obtain: 
\[
E_n(f) \leq 1/2^m \leq 2.r^{n^{\gamma}}, \; {\rm with} \; r \in  (0,1) , \; 
\gamma > 0.
\]
In particular, the sequence $E_n(f)$ decreases rapidly and $f\in\Ci$.
This leads us to study the functions $f$ for which this kind of bound is obtained.

\begin{definition}[Gevrey class\footnote{Cf. for example Hörmander \cite{Ho}: The Analysis of Linear Partial Differential Operators I p 281 (Springer 1983). 
A function is Gevrey of order $ 1 $ if and only if it is analytic.}] \label{522}

\noindent 
A function $f\in \Ci$ is said to be in the Gevrey class of order $\alpha > 0 $ if its derivatives satisfy a majorisation: 
\[
\NOrme{f^{(n)} }_{\infty}\leq  MR^n n^{\alpha n}
\]
The Gevrey class is obtained when the order $\alpha$ is not specified. 
\end{definition}

\begin{theorem} \label{523}
Let $f\in \C$. The following properties are equivalent. 
\begin{itemize}

\item [(i)] 

$ \exists M > 0, \; \; r < 1, \; \; \gamma > 0 \; \; \forall n> 0, \; \; E_n(f) \leq Mr^{n^{\gamma}}$,

\item [(ii)]
 
$ \exists M > 0, \; \; r < 1, \; \; \gamma > 0 \; \; \forall n> 0, \; \; S_n(f) \leq Mr^{n^{\gamma}}$,

\item [(iii)]
 
$ \exists M > 0, \; \; r < 1, \; \; \gamma > 0 \; \; \forall n> 0, \; \; \abS{A_n(f)}\; \leq Mr^{n^{\gamma}}$, 

\item [(iv)]
 
$ \exists M > 0, \; r < 1 \; \forall n> 0, \; \NOrme{ u_n^{[n+1]} - f }_{\infty} \leq Mr^{n^{\gamma}}$,

\item [(j)] 

$ \exists c, \beta > 0 \; \forall n> 0\; \forall m \geq cn^{\beta}, \; E_m(f) \leq 1/2^n $,

\item [(jj)]

$ \exists c, \beta > 0 \; \forall n> 0,\; \forall m \geq cn^{\beta}, \; S_m(f) \leq 1/2^n $,

\item [(jjj)] 

$ \exists c, \beta > 0 \; \forall n> 0,\; \forall m \geq cn^{\beta}, \; 
\abS{A_m(f)}\; \leq 1/2^n $,

\item [(jw)]
 
$ \exists c, \beta > 0 \; \forall n> 0\;\forall m \geq cn^{\beta}, \NOrme{ u_m^{[m+1]} - f }_{\infty} \leq 1/2^n $,

\item [(k)] 

$f$ is in the Gevrey class.
\end{itemize}
\end{theorem}

\begin{proof}
(i) $\Leftrightarrow$ (ii) from Equation (\ref{F 5.2.3}). 

\noindent 
(i) $\Rightarrow$ (iii) from Equation (\ref{F 5.2.4}). 

\noindent 
(iv) $\Rightarrow$ (i) is trivial.

\noindent 
The 4 equivalences of type (i) $\Leftrightarrow$ (j) result from the same kind of calculation as the one done before the theorem.

\noindent 
The implication (jjj) $\Rightarrow$ (jw) results from a simple inequality  using the equations (\ref{F 5.2.7}) and (\ref{F 5.2.8}). 

\noindent 
Assume (k), i.e.\ that $f$ is Gevrey of order $\alpha$, and show (i). 
Finding an upper bound is only tricky for $\alpha \geq 1$, which we will now assume. Applying Jackson's theorem, we obtain an inequality $E_n(f) \leq \pi^{k} \Norme{f^{(k)} }_{\infty} /n^k$
 as soon as $n \geq 2k$, which with Gevrey's bound gives $E_n(f) \leq A(Ck^{\alpha}/n)^k$. 
We can assume $C^{1/ \alpha} \geq 2$ and take for $k$ an integer close to $(n/2C)^{1/ \alpha}$ ($ \leq n/2$), hence to a very close approximation:
\[
E_n(f) \leq A(1/2)^{(n/2C)^{1/ \alpha}} = Ar^{n^{\gamma}}, \; {\rm with} \; \gamma = 1/ \alpha.
\]
Now we suppose that $f$ verifies (j) and we show that $f$ is Gevrey. 

\noindent 
Finding an upper bound is only tricky for $\beta \geq 1$, which we will now assume. We write $f^{(k)} = \sum' A_m \Tch_m^{(k)}$. 
Hence $\Norme{f^{(k)} }_{\infty} \leq \sum' \abS{A_m} m^{2k}$ according to V.A.\ Markov's inequality (\ref{F 5.2.2}). 
We now use the bound (jjj). We take $c$ and $\beta$ to be integers for simplicity (this is not a restriction). In the sum above, group the terms for $ m $ between $cn^{\beta}$ and $c(n+1)^{\beta}$. In the packet obtained, each term is bounded by $ (1/2^n) m^2k$, and the number of terms is bounded by $ c(n+1)b,$ hence:
\[
\Norme{f^{(k)} }_\infty \leq \sum_n (c(n+1)^{\beta}/2^n)(c(n+1)^{\beta})^{2k} \leq 2c^{2k+1} \sum_n (n+1)^{\beta (2k+1)}/2^n
\]
\[
\leq 4c^{2k+1} \sum_n n^h/2^n, \; {\rm where} \; h = \beta(2k+1).
\] 
We bound this series by the series obtained by deriving $h$ times the series $\sum_n x^n$ (then making $x=1/2$) and we obtain that $f$ is Gevrey of order $2\beta$. 
\end{proof}

\begin{remarks}\label{524}~

\noindent 
1) The space of Gevrey functions therefore has a pleasant constructive presentation.

\noindent 
2) For $\gamma = 1$ we obtain the analytic functions. For $\gamma > 1$, we obtain the entire functions.

\noindent 
3) For $\gamma \leq 1$, the upper limit of possible $\gamma$ is the same in (i), (ii), (iii) and (iv), the lower limit of possible $\beta$ is the same in (j), (jj), (jjj) and (jw), with $\gamma = 1/ \beta$. 

\noindent 
4) Based on the case of analytic functions ($\alpha = \beta = \gamma = 1$), we can hope, for the implication (j) $\Rightarrow$ (k), to obtain that $f$ is Gevrey of order $\beta$ by means of a more sophisticated majorisation calculation.

\noindent 
5) In (j), (jj), (jw) we can remove the quantifier $\forall m$ if we take $ c$ and $\beta$ to be integers and $m = cn^{\beta}$.
\end{remarks}

\subsubsection{Returning to questions of complexity in the space $\cw$}
\label{subsubsec 523}

We start with an

\smallskip \noindent 
{\bf Important remark}. With respect to $\DD[X]$, the ordinary dense presentation (based on $X^n$) and the dense presentation based on Chebyshev polynomials $\Tch_n$, are equivalent in polynomial time. We will use either of the two bases, depending on the convenience of the moment.

\smallskip  Recall also that the norm $\; P \to \norme{ P }_{\infty}\; $ is a $\p$-computable function from $\DD[X]$ to $\RR$.

The proof of the following proposition is immediate. In fact, any functional defined on $\cw$ which has a polynomial modulus of uniform continuity and whose restriction to $\DD[X]$ is easy to compute is itself easy to compute. This proposition takes on its full value in view of the characterisation \thref{527}.

\begin{proposition}[good behaviour of usual functionals] \label{526}~ 

\noindent 
The functionals:
\[
\cw \to \RR \quad  f \mapsto \norme{f}_{\infty}, \; \norme{f}_2, \; \norme{f}_1
\]
are  uniformly of class $\p$.\\
The functionals:
\[
\cw \times [0,1] \times [0,1] \to \RR \quad  (f,a,b) \mapsto \sup_{x 
\in [a,b]} (f(x)), \; \int_{a}^{b} f(x) dx
\]
are uniformly of class $\p$.
\end{proposition}

\begin{theorem}[characterisation of the $\p$-points of $\cw$] \label{527}~ 

\noindent 
Let $f \in \C$. The following properties are equivalent.

\begin{itemize}
\item [a)]
 The function $f$ is a $\p$-point of $\ckf$ and is in the Gevrey class.

\item  [b)]  The sequence $A_n(f)$ is a $\p$-sequence in $\RR$ and verifies an inequality $\abS{A_n(f)}\; \leq Mr^{n^{\gamma}}$ with $M > 0$, $\gamma > 0$ and $0 < r < 1$. 

\item  [c)] The function $f$ is a $\p$-point of $\cw$.
\end{itemize}
\end{theorem}
\begin{proof}
The implications (c) $\Rightarrow$ (a) and (c) $\Rightarrow$ (b) are easy from \thref{523}. 

\noindent 
(b) $\Rightarrow$ (c). A polynomial (in dense presentation on the basis of $\Tch_n$) approximating $f$ with precision $1/2^{n+1}$ is obtained with the partial sum extracted from the Chebyshev series of $f$ stopping at the index $(Bn)^h$ (where $B$ and $h$ are calculated from $M$ and $\gamma$). It remains to replace each Chebyshev coefficient by a dyadic approximating it with precision 
\[
1/ \Flo{(Bn)^h2^{n+1}}  = 1/2^{n+1+h \log(Bn)}.
\]

\noindent 
(a) $\Rightarrow$ (c). A polynomial approximating $f$ to within $1/2^{n+1}$ is obtained with $u_m^{[m+1]}$ (where $m = (Cn)^k$), $C$ and $k$ are calculated from $M$ and $\gamma$, taking into account equations (\ref{F 5.2.7}) and (\ref{F 5.2.8})). 
The formula defining $u_m^{[m+1]}$ provides its coefficients on the basis of $\Tch_n$ and we can calculate (in polynomial time) an approximation to within $1/2^{n+1+k \log(Cn)}$ of these coefficients by taking advantage of the fact that the double sequence $\xi_i^{[n]}$ is a $\p$-sequence of reals and that the function~$f$ is a $\p$-point of $\ckf$. 
\end{proof}

An immediate consequence of the previous theorem is obtained in the case of analytic functions.

\begin{theorem} \label{528}
Let $f \in \C$. The following properties are equivalent. 
\begin{itemize}

\item [(a)] 

The function $f$ is an analytic function and it is a $\p$-point of $\ckf$.

\item [(b)] 

The sequence $A_n(f)$ is a $\p$-sequence in $\RR$ and verifies an inequality
\[
 \abS{A_n(f)}\; \leq Mr^n \; (M > 0, \; r < 1).
\]

\item [(c)] 

The function $f$ is an analytic function and is a $\p$-point of $\cw$.
\end{itemize}
\end{theorem}

\begin{definition}[$\p$-analytic functions]\label{529} 

\noindent 
When these properties are verified, we say that the function $f$ is $\p$-analytic on the interval $[-1,1]$.
\end{definition}

\begin{theorem}[fairly good behaviour of derivation with respect to complexity] \label{5210} 
~

\noindent 
Let $f$ be a $\p$-point of $\cw$. Then the sequence $k \mapsto f^{(k)}$ is a $\p$-sequence of $\cw$. 
More generally, if $(f_p)$ is a $\p$-sequence of $\cw$ then the double sequence $(f_p^{(k)})$ is a $\p$-sequence of $\cw$.
\end{theorem}

\begin{proof} We give the proof for the first part of the proposition. It would apply without change for the case of a $\p$-sequence of $\cw$. 

\noindent 
The function $f$ is a $\p$-point of $\cw$ given as the limit of a  $\p$-computable sequence $n \mapsto  P_n$. 
The double sequence $P_n(k)$ is $\p$-computable (unary entries). There are two integers $a$ and $b$ such that the degree of $\po_n$ is bounded by $2^a n^b$. Therefore, according to V.A.\ Markov's inequality (\ref{F 5.2.2}) we have the bound
\[
\NOrme{ P_n^{(k)} - P_{n-1}^{(k)} }_\infty \leq (2^an^{2b})^k \Norme{ P_n - P_{n-1} }_{\infty} \leq (2^an^{2b})^k /2^{n-2} = 1/2^{n-(k.(a+2b \log(n))+2)}.
\]
We can then easily determine a constant $n_0$ such that, for $n \geq 2n_0k$, we have 
\[
n \geq 2(k.(a+2b \log(n))+2),
\] 
and therefore 
\[
\NOrme{ P_n^{(k)} - P_{n-1}^{(k)} }_\infty \leq 1/2^{n/2},
\] 
so that by posing $\nu(n) := 2 \sup(n_0k,n)$, we have, for $q \geq \nu(n),$ 
\[
\NOrme{ P_q^{(k)} - P_{q+1}^{(k)} }_\infty \leq 1/2^n,
\]
and therefore, since $\nu(n+1) = \nu(n)$ or $\nu(n)+2,$ 
\[
\NOrme{ P_{\nu(n)}^{(k)} - P_{\nu(n+1)}^{(k)} }_\infty \leq 1/2^{n-1},
\]
 hence finally 
\[
\NOrme{ P_{\nu(n)}^{(k)} - f^{(k)} }_\infty \leq 1/2^{n-2}.
\]
We end by noting that the double sequence $(n,k) \mapsto P_{\nu(n+2)}^{(k)}$ is $\p$-computable. 
\end{proof}

\begin{corollary} \label{5211}
If $f$ is a $\p$-point of $\cw$ and $a, b$ two $\p$-points of $[-1,1]$, then the sequences 
\[
\NOrme{f^{(n)} }_{\infty}, \; \NOrme{f^{(n)} }_2, \; \NOrme{f^{(n)} }_1, \; \NOrme{f^{(n)}(a)} \;{\rm and} \; \sup\nolimits_{x \in [a,b]} (f^{(n)}(x))
\]
are $\p$-sequences in $\RR$.

\noindent
More generally, if $(f_p)$ is a $\p$-sequence of $\cw$ then the double sequences 
\[
\NOrme{f_p^{(n)}}_\infty, \; \NOrme{f_p^{(n)} }_2, \; \NOrme{f_p^{(n)} }_1, \; \NOrme{f_p^{(n)}(a)} \; {\rm and} \; \sup\nolimits_{x \in [a,b]} (f_p^{(n)}(x))
\]
are $\p$-sequences in $\RR$.
\end{corollary}

The proof of \thref{5210} (and therefore of Corollary \ref{5211}) is somewhat uniform and has a more general meaning. We will now define the natural framework in which this theorem applies and give a new, more general and more satisfactory statement.

\begin{definition}\label{5212}
For $c$ and $\beta > 0$ let $\Gv_{c,\beta}$ be the class of Gevrey functions satisfying the inequality (of the kind (jjj) in \ref{523})
\[
 \forall m > cn^{\beta} \; \; \abS{A_m(f)}\; \leq 1/2^n
\] 
This is a closed convex part of $\C$. For any $c$ and $\beta$, 
let $\Y_{\Gv_{c, \beta}}$ be the elements of $\DD[X]$ which are in the class $\Gv_{c, \beta}$. 
This set $\Y_{\Gv_{c, \beta}}$ can be taken as the set of rational points of a rational presentation $\sC_{\Gv,c,\beta}$ of $\Gv_{c,\beta}$.
\end{definition}

Note that the test for membership of the subset $\Y_{\Gv_{c, \beta}}$  of 
$\DD[X]$ is in polynomial time, since the $A_m(f)$'s for a polynomial $f$ are its coefficients on the Chebyshev basis.
In this new framework \thref{528} has a more uniform and efficient formulation.

\begin{theorem} \label{5213}
Each functional $f \mapsto f^{(n)}$ is a function uniformly of class $\p$ from $\sC_{\Gv_{c, \beta}}$ to $\cw$. 
More precisely, the sequence of functions 
\[
(k,f) \mapsto f^{(k)} \; : \; \NN_1 \times \sC_{\Gv_{c, \beta}} \to \cw 
\]
is uniformly of class $\p$ (in the sense of Definition \ref{229}).
\end{theorem}

\begin{proof} 
The double sequence $(k,f) \mapsto f^{(k)}$ is of low complexity as a function from $\NN_1 \times \DD[X]$ to~$\DD[X]$ and therefore also as a function from $\NN_1 \times \Y_{\Gv, \beta}$ to $\DD[X]$.

\noindent 
The whole problem is therefore to show that we have a polynomial modulus of uniform continuity (in the sense of \ref{229}). 
We need to calculate a function $\mu(k,h)$ such that for all $f$ and $g$ in $\Y_{\Gv, \beta}$:
\[
\norme{ f - g }_\infty \leq 1/2^{\mu (k,h)} \Rightarrow \NOrme{f^{(k)} - g^{(k)} }_\infty \leq 1/2^h. 
\] 
This calculation is rather similar to the one used in the proof of Theorem \ref{528}. We write 
\[
\NOrme{f^{(k)} - g^{(k)} }_\infty \leq \NOrme{f^{(k)} - s_n(f)^{(k)} }_{\infty} + \NOrme{ g^{(k)} - s_n(g)^{(k)} }_{\infty} + \NOrme{ s_n(f-g)^{(k)} }_{\infty}.
\]
In the sum of the second member, the first two terms are bounded as follows
\[
\NOrme{f^{(k)} - s_n(f)^{(k)} }_{\infty}\leq  \sum_{q>n} \abS{A_q(f)} \NOrme{ T_q^{(k)} }_\infty \leq \sum_{q>n} \abS{A_q(f)} q^{2k}.
\]
Since we have: $\forall q > cn^{\beta} \; \abS{A_q(f)}\; \leq 1/2^n$, $\sum_{q>n} \abS{A_q(f)} q^{2k}$ is convergent and we can give exlicitly a polynomial $\alpha(k,h)$ such that (see explanation at the end of the proof),
\[
{\rm with} \; n = \alpha(k,h) \; \; \forall f \in \Y_{\Gv,c,\beta} \; : \; \; 
\sum_{q>n} \abS{A_q(f)} q^{2k} \leq 1/2^{h+2}.
\]
Once we have set $n = \alpha(k,h)$ we need to make small the term $\NOrme{ s_n(f-g)^{(k)} }_{\infty}$.
V.A.\ Markov's inequality (\ref{F 5.2.2}) implies that 
\[
\NOrme{ s_n(f-g)^{(k)} }_\infty \leq \Norme{ s_n(f-g) }_{\infty} n^{2k}.
\]
All that remains is to obtain a suitable bound for 
$\norme{s_n(f-g)}_{\infty}$ from $\norme{ f-g }_{\infty}$.
For example, we can use the inequality $S_n(f) \leq (4+ \log (n)) E_n(f)$ (from Formula (\ref{F 5.2.3})) whence
\[
\Norme{f^{(k)} }_\infty \leq \norme{f}_\infty + S_n(f) \leq \norme{f}_\infty + (4+ \log (n)) E_n(f) \leq (5+ \log (n)) \norme{f}_\infty.
\]
Finally, let's explain $\alpha(k,h)$. We have the inequalities
\[
\sum_{q \geq cn_0^{\beta}} \abS{A_q(f)} q^{2k} \leq \sum_{n \geq n_0}\sum_{q \leq c(n+1)^{\beta}}q^{2k}/2^n \leq \sum_{n \geq n_0} c(n+1)^{\beta}(c(n+1)^{\beta})^(2k)/2^n,
\]
so
\[
\sum_{q \geq cn_0^{\beta}} \abS{A_q(f)} q^{2k} \leq \sum_{n \geq n_0} (c(n+1)^{\beta})^{2k+1}/2^n \leq \sum_{n \geq n_0} 1/2^{\varphi (n, \beta ,c, k)}
\]
with, if $2^a \geq c$,
\[
\varphi (n, \beta ,c, k) \geq n- (2k+1)a-(2k+1) \beta \log(n+1).
\]
If we have 
\[
{\rm for} \; n \geq n_0 \; \; \varphi (n, \beta ,c, k) \geq h+n/2+4, \eqno (\star)
\]
we obtain
\[
\sum_{q \geq cn_0^{\beta}} \abS{A_q(f)} q^{2k} \leq \sum_{n \geq n_0} 1/2^{\varphi (n, \beta ,c, k)} \leq (1/2^{h+2}(1/4) \sum_{n \geq n_0} 1/2^{n/2} \leq 1/2^{h+2}.
\]
And we can take $\alpha (k,h) = cn_0^{\beta}$. 

\noindent 
It remains to be seen how we can achieve the $(\star)$ condition. 

\noindent 
For any integer $b$ we have an integer $\nu(b)\leq \max(8,b^2)$ for which 
\[
n > \nu (b) \; \Rightarrow\; n \geq b \log (n+1).
\]
If therefore $n \geq \nu (4(2k+1) \beta)$ we obtain
\[
\varphi (n, \beta ,c, k) \geq n-(2k+1)a-(2k+1)\beta \log(n+1) \geq (3n/4)-(2k+1)a
\]
and the condition $\varphi (n, \beta ,c, k)\geq h+n/2+4$ is fulfilled if $n/4 \geq (2k+1)a +h +4$.
Hence $\alpha(h,k) = c \max(\nu (4(2k+1)\beta ),4((2k+1)a + h +4))^{\beta}$. 
\end{proof}
 
Applying Proposition \ref{526}, we obtain:

\begin{corollary} \label{5214}~\\
i) The three sequences of functionals 
\[
(n,f) \mapsto \Norme{f^{(n)} }_{\infty} , \;  \Norme{f^{(n)} }_2, \; \Norme{f^{(n)} }_1 \qquad  \NN_1 \times \sC_{\Gv,c, \beta} \to \RR
\]
are uniformly of class $\p$ (in the sense of Definition \ref{229}).

\noindent 
ii) The sequence of functionals
\[(n,f,x) \mapsto f^{(n)}(x) \qquad 
\NN_1 \times \sC_{\Gv,c, \beta} \times [-1,1] \to \RR  
\] 
is uniformly of class $\p$.

\noindent 
iii) The sequence of functionals
\[  (n,f,a, b) \mapsto \sup_{x \in [a,b]} (f^{(n)}(x)) \qquad 
\NN_1 \times \sC_{\Gv,c, \beta} \times [-1,1] \times [-1,1] \to \RR 
\] 
is uniformly of class $\p$. 
\end{corollary}

\begin{remark} 
Theorems \ref{527}, \ref{528}, \ref{5210}, \ref{5213}, Proposition \ref{526} and Corollaries \ref{5211} and \ref{5214} significantly improve the results of \cite{KF82}, \cite{KF88} and \cite{Mu87} on analytic functions computable in polynomial time (in the Ko-Friedman sense).
\end{remark}

\subsection {Comparisons of different presentations of class $\p$}\label{subsec53}

In this section, we obtain the following chain of functions uniformly of class $\p$ for the identity of $\czu$. 
\[
\cw \to \csp \to \crf \equiv \csr \to \ckf 
\] 
and none of the $\to$ arrows in the line above is a $\p$-equivalence except perhaps $\csp \to \crf$ and very possibly $\crf \to \ckf$ (it would imply $\p = \np$).
First of all, it is clear that the identity of $\czu$ is of class $\LINT$ for the following cases: 
\[
\cw \to \csp ; \quad  \csp \to \csr; \quad  \crf \to \csr .
\]
Furthermore, the identity of $\czu$ is of class $\p$ in the following case 
\[
\crf \to \caf 
\] 
(in fact, only the calculation of the magnitude is not completely trivial, and it is surely in $\DTI (\Oo(N^2))$). 

We still have to show that the identity of $\czu$ from $\csp$ to $\crf$ and the one from $\csr$ to $\crf$ are of class~$\p$.

\begin{theorem} \label{531}
The identity of $\czu$ from $\csp$ to $\crf$ is uniformly of class $\p$.
\end{theorem}

\begin{proof} 
Let $n \in \NN_1$ and $f \in \ysp$. We need to calculate an element $g \in \yrf $ such that 
\[
\Norme{ \wi{f} - \wi{g} }_{\infty} \leq 1/2^n.
\] 
We have $f = ((x_0,x_1,\ldots,x_t),(P_1,P_2,\ldots,P_t))$ with $x_0=0$, $x_t=1$ and $P_i(x_i) = P_{i+1}(x_i)$ for $i = 1,\ldots,t-1$.
 We calculate $m \in \NN_1$ such that $2^m$ bounds  
$\norme{P_i}_{\infty}$ and  $\norme{P_i'}_{\infty}$ for each $i$. 

\noindent 
We let $p=m+n+1$ and $z_i = x_i - 1/2^{p+1}$ for $i=0,\dots,t-1$. For $i = 0$  we let $h_0 := -C_{p,z_1}$. For $i = t-1$ we let $h_{t-1}:=C_{p,z_{t-1}}$. 
For $i = 1,\ldots,t-2$ we let $h_i := C_{p,z_i} - C_{p,z_{i+1}}$  (see Figure \ref{fi531}). 

\noindent 
\begin{figure}[htbp] 
\begin{center}
\includegraphics*[width=12cm]{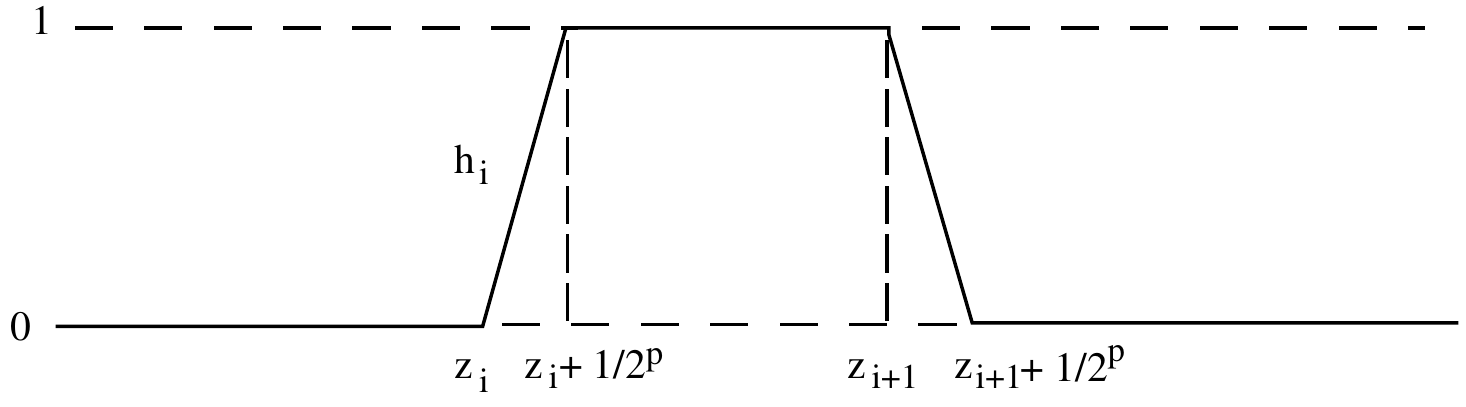}
\end{center}
\caption[The $h_i$ function]{\label{fi531} 
the function $h_i$} 
\end{figure} 

\noindent 
The function $\wi{f}$ is roughly equal to 
$h = \sum_i h_iP_{i+1}$: on the intervals $[z_i+1/2^p, z_{i+1} ]$ 
we get $h =\wi{f}$, while on an interval 
\[
[z_i, z_i +1/2^p] = [x_i - 1/2^{p+1}, x_i +1/2^{p+1}]
\]
we obtain $h = h_{i-1}P_i + h_i P_{i+1}$ which is a weighted average of $P_i$ and $P_{i+1}$, instead of $P_i$ or $P_{i+1}$. 
At a point $x$ on this interval, we have $\abs{x-x_i}\; \leq 1/2^{p+1}$, 
we apply the mean value theorem and using $P_i(x_i) =P_{i+1}(x_i)$ we obtain
\[
\abs{P_i(x) - P_{i+1}(x)}\; \leq \abs{P_i(x) - P_i(x_i)} + \abs{P_{i+1}(x) - P_{i+1}(x_i)}\; \leq 2^{m+1}/2^{p+1} \leq 1/2^{n+1}
\]
and therefore
\[
\Norme{ \wi{f} - h }_{\infty} \leq 1/2^{n+1}.
\]
It remains to replace each $h_i$ by an element $g_i$ of $\yrf$ satisfying $\Norme{ \wi{g_i} - h_i }_{\infty} \leq 1/(t2^{p})$
 (so that
$\Norme{\wi{g_i}P_{i+1}-h_iP_{i+1}}_{\infty}\leq 1/(t2^{n+1})$
and therefore
\[
\NOrme{ h - \sum\nolimits_i \wi{g_i}P_{i+1} }_\infty \leq 1/2^{n+1}.
\]
In view of Proposition \ref{337} concerning the approximation of functions $C_{p,a}$ by rational functions, the calculation of $g_i$ is done in polynomial time from the data $(f,n)$. It remains to express $\sum_i\wi{g_i}P_{i+1}$ in form $\wi{g}$ with $ g \in \yrf$, which is not difficult, to obtain $\Norme{ \wi{f} - \wi{g} }_{\infty} \leq 1/2^n.$ 
\end{proof}

\begin{theorem} \label{532}
The representations $\crf$ and $\csr$ of $\czu$ are $\p$-equivalent.
\end{theorem}

\begin{proof}
From an arbitrary element $f=((x_0,x_1,\ldots,x_t),((P_1,Q_1),\ldots,(P_t,Q_t)))$ of $\ysr$ we can calculate in polynomial time dyadic numbers $d_1, d_2,\ldots,d_t \geq 1$  satisfying $d_iQ_i(x_i) = d_{i+1}Q_{i+1}(x_i)$ (and therefore also $d_iP_i(x_i) = d_{i+1}P_{i+1}(x_i)$) for $i = 1, \ldots, t-1$.
Then $\wi{f} = \wi{g} / \wi{h}$ where $g,h \in \Y_{Sp}$ are given by:
\[
g = ((x_0,x_1,\ldots,x_t),(d_1P_1,\ldots,d_tP_t)) \quad  {\rm and} \quad  h = ((x_0,x_1,\ldots,x_t),(d_1Q_1,\ldots,d_tQ_t)).
\]
We conclude by using  \thref{531} which allows us to approximate $\wi g$ and $\wi h$ by rational functions. 
\end{proof}

\section*{Conclusion}\label{sec Conclusion}
\addcontentsline{toc}{section}{Conclusion}
Hoover has linked in an interesting way the natural notion of complexity of continuous real functions given by Ko and Friedman to another notion, based on arithmetic circuits. In this way he gave a certain ``polynomial time'' version of the Weierstrass approximation theorem. 

In this paper we have generalised Hoover's approach, introducing a uniform  point of view to the notion of rational presentation of a metric space.
This provides a satisfactory general framework for the study of many algorithmic complexity problems in analysis. 
We have also generalised to this approach the results of Ko, Friedman and Müller concerning analytic functions computable in polynomial time. 
The rational presentation $\csl$ $\p$-equivalent to $\ckf$ is the most natural from the point of view of theoretical computer science. 
However, it is not a $\p$-presentation and is ill-suited to numerical analysis as soon as problems more complicated than evaluation arise (the calculation of the norm, or of a primitive, for example).

Among the representations we have studied, the presentation $\cw$ seems to be the easiest to use for many numerical analysis problems. 

As for the presentation by rational functions, it deserves further study. We would like to obtain an analogue of the theorem (given for the presentation $\cw$) concerning the characterisation of $\p$-points.
It would also be interesting to obtain for $\crf$ the $\p$-calculability of certain usual operations in numerical analysis, such as the calculation of a primitive or more generally the calculation of the solution of an ordinary differential equation. 

\bni
{\bf Acknowledgements.}
We would like to thank Maurice Margenstern and the referee for their comments.


\normalsize
\endgroup
\stopcontents[english]

\clearpage
\newpage
\thispagestyle{empty}


~

\clearpage
\newpage

\renewcommand\thepage{F\arabic{page}}\renewcommand\theHsection{F\arabic{section}}

\selectlanguage{french}
\def\frenchproofname{\textsl{Démonstration}}

\setcounter{page}{1} 
\setcounter{section}{0}
\setcounter{subsection}{0}
\setcounter{equation}{0}

\selectlanguage{french}
\def\frenchproofname{\textsl{Démonstration}}

\FrenchFootnotes


\theoremstyle{plain}
\newtheorem{ftheorem}{Théorème}[subsection]
\newtheorem{fproposition}[ftheorem]{Proposition}
\newtheorem{fpropdef}[ftheorem]{Proposition et définition}
\newtheorem{flemma}[ftheorem]{Lemme}
\newtheorem{fcorollary}[ftheorem]{Corollaire}

\theoremstyle{definition}
\newtheorem{fdefinition}[ftheorem]{Définition}
\newtheorem{fnotation}[ftheorem]{Notation}
\newtheorem{fexample}[ftheorem]{Exemple}
\newtheorem{fexamples}[ftheorem]{Exemples}

\theoremstyle{remark}
\newtheorem{fremark}[ftheorem]{Remarque}
\newtheorem{fremarks}[ftheorem]{Remarques}


\renewcommand\csl{{\sC_{\rm csl}}}
\renewcommand\cbo{{\sC_{\rm cb}}}
\renewcommand\ckf{{\sC_{\rm KF}}}
\renewcommand\caf{{\sC_{\rm ca}}}
\renewcommand\capo{{\sC_{\rm cap}}}
\renewcommand\cw{{\sC_{\rm W}}}
\renewcommand\crf{{\sC_{\rm frf}}}
\renewcommand\csr{{\sC_{\rm sfrf}}}
\renewcommand\csp{{\sC_{\rm sp}}}

\renewcommand\CB{{\rm CB}}
\renewcommand\CA{{\rm CA}}

\renewcommand\ysl{{\Y_{\rm csl}}}
\newcommand\ybo{{\Y_{\rm cb}}}
\renewcommand\ykf{{\Y_{\rm KF}}}
\renewcommand\yaf{{\Y_{\rm ca}}}
\renewcommand\yap{{\Y_{\rm cap}}}
\renewcommand\yw{{\Y_{\rm W}}}
\renewcommand\yrf{{\Y_{\rm frf}}}
\renewcommand\ysr{{\Y_{\rm sfrf}}}
\renewcommand\ysp{{\Y_{\rm sp}}}

\renewcommand \loca {localement\xspace}
\renewcommand \uni {uniformément\xspace}
\newcommand \uniz {uniformément}
\renewcommand \mcu {module de continuité uniforme\xspace}
\newcommand \mcuz {module de continuité uniforme}
\renewcommand \unico {\uni continue\xspace}
\newcommand \unicoz {\uni continue}
\newcommand \unicos {\uni continues\xspace}
\newcommand \unicosz {\uni continues}

\renewcommand \com {complexité\xspace}
\newcommand \comz {complexité}
\newcommand \pres {présentation\xspace}
\newcommand \equivas {équivalentes\xspace}
\renewcommand \equiva {équivalente\xspace}
\renewcommand \etpo {en temps \poll}
\newcommand \elem {élément\xspace}

\newcommand \pol{polynôme\xspace}
\newcommand \pols{polynômes\xspace}

\newcommand \polt{polynomialement\xspace}
\newcommand \poll{polynomial\xspace}
\newcommand \polle{polynomiale\xspace}
\newcommand \polles{polynomiales\xspace}

\renewcommand\rp{\pres rationnelle\xspace}
\renewcommand\rps{présentations rationnelles\xspace}
\renewcommand\rapr{rationnellement présenté\xspace}
\newcommand\rapre{rationnellement présentée\xspace}
\newcommand\raprs{rationnellement présentés}

\newcommand\ta{{\rm t}}
\newcommand\prof{{\rm prof}}

\renewcommand\ni{\noindent }

\thickmuskip = 7mu plus 2mu

\pagestyle{headings}
\patchcmd{\sectionmark}{\MakeUppercase}{}{}{}

\title{Espaces métriques rationnellement présentés et complexité, le cas 
de l'espace des fonctions réelles \unicos sur un intervalle compact}
\author{S. Labhalla \\ Dépt. de Mathématiques \\ Univ. de Marrakech, Maroc 
\\{\tt labhalla@ucam.ac.ma} 
\and H. Lombardi \\ Laboratoire de Mathématiques de Besançon \\
Université Marie et Louis Pasteur, France \\  {\tt Henri.Lombardi@univ-fcomte.fr}
\and  E. Moutai\\ Dépt. de Mathématiques \\ Univ. de Marrakech, Maroc}

\maketitle

\begin{quotation} 
\small Cet article est paru en 2001 dans \emph{Theoretical Computer Science.} {\bf 250}, \num 1-2, 265--332  (reçu en avril 1997; révisé en mars 1999). 
Nous avons corrigé quelques erreurs de détail.

\end{quotation}

\rdb
\label{beginfrench}

\begin{abstract}
 Nous définissons la notion de {\em \rp d'un
espace métrique complet} comme moyen d'étude des espaces métriques et des
fonctions continues du point de vue de la complexité algorithmique. Nous
étudions dans ce cadre différentes manières de présenter l'espace $\czu$ 
des fonctions réelles \unicos sur l'intervalle $[0,1]$, muni de la norme 
usuelle:
 $\norme{f}_{\infty} = {\bf Sup} \{ \abs{f(x)}  \;
 0 \leq x \leq 1\}.$ Ceci nous permet de
faire une comparaison de nature globale entre les notions de complexité
attachées à ces présentations. En particulier, nous obtenons une
généralisation des résultats de Hoover concernant le {\em théorème 
d'approximation de Weierstrass en temps \poll}. Nous obtenons également une
généralisation des résultats de Ker-I. Ko, H. Friedman et N. Müller 
concernant les fonctions analytiques calculables en temps \poll.
\end{abstract}

\newpage 
\startcontents[french]

\setcounter{tocdepth}{4}
\markboth{Table des matières}{Table des matières}

\printcontents[french]{}{1}{}
\normalsize
\newpage

\section*{Introduction}\label{fsec0}
\addcontentsline{toc}{section}{Introduction}
\markboth{Introduction}{Introduction}
Notons   $\czu$ l'espace des fonctions réelles \unicos sur l'intervalle 
$[0,1]$.\\
Dans \cite{fKF82}, Ker-I. Ko et Friedman ont introduit et étudié la notion de 
\com des fonctions réelles, définie via une machine de Turing à oracle. \\
Dans l'article  \cite{fHo90}, Hoover a étudié les présentations de l'espace 
$\czu$ via les circuits booléens et les circuits arithmétiques.\\
Dans ces deux articles la \com dans l'espace  $\czu$  est étudiée ``point 
par point''.  \\
Autrement dit, sont définies des phrases comme : 
\begin{itemize}
\item $f$ est un $\p$-point au sens de Ker-I. Ko et Friedman.
\item $f$ est un $\p$-point au sens des circuits booléens.
\item $f$ est un $\p$-point au sens des circuits arithmétiques.
 \end{itemize}
Ko et Friedman ont étudié les propriétés des  $\p$-points  (au sens de 
Ker-I. Ko et Friedman).   
Hoover~a démontré que les  $\p$-points au sens des circuits arithmétiques sont 
les mêmes que les $\p$-points au sens de Ker-I. Ko et Friedman.

\smallskip  Dans cet article, nous étudions différentes présentations de 
l'ensemble  $\czu$  en nous basant sur la notion de {\em présentation 
rationnelle d'un espace métrique complet}. Ceci nous permet de faire une 
comparaison de nature globale entre les notions de \com attachées à 
différentes présentations.

\smallskip  Après quelques préliminaires dans la section~\ref{fsec1}, 
la section ~\ref{fsec2} contient  essentiellement une exposition de ce 
qui nous semble être une problématique naturelle concernant les 
questions de \com relatives aux espaces métriques complets séparables. 

\smallskip 
Usuellement un espace métrique contient des objets de nature infinie (le 
paradigme étant un nombre réel défini à la Cauchy), 
ce qui exclut une présentation informatique directe (c.-à-d. codée sur un 
alphabet fini) de ces objets. 
Pour contourner cette difficulté, on procède comme il est usuel pour 
l'espace  $\RR$. On considère une partie dense  Y  de l'espace métrique 
considéré  X, qui soit suffisamment simple pour que 
\begin{itemize}
\item ses éléments puissent être codés comme (certains) mots sur un 
alphabet fini fixé.
\item la fonction distance restreinte à  $Y$  soit calculable, c.-à-d. 
donnée par une fonction calculable:
\[
\delta\colon  Y\times Y\times \NN_1\rightarrow \QQ \; \qquad \hbox{avec}\qquad \; 
\abs {d_X(x,y) - \delta(x,y,n)} \;  \leq 1/2^n
\]
\end{itemize}
L'espace  $X$  apparaît alors comme le séparé-complété de  $Y$.\\
On dira que le codage proposé pour  $Y$  et la description proposée pour la fonction distance constituent une {\em présentation rationnelle} de l'espace 
métrique  $X$.
 
Il est à noter qu'on n'envisage pas de traiter des espaces métriques non 
complets, pour lesquels les difficultés de codage semblent insurmontables, et pour lesquels on ne peut pratiquement rien démontrer de sérieux en analyse constructive.

Pareillement les espaces métriques traités sont ``à base dénombrable'' 
(on dit aussi ``séparables'').

\smallskip Les espaces métriques étudiés en analyse constructive (cf. 
\cite{fBB}) sont très souvent définis via une présentation rationnelle de 
ce type, ou au moins faciles à définir selon ce schéma. Le problème qui 
se pose est en général de définir constructivement une partie 
dénombrable dense de l'espace considéré. C'est évidemment impossible 
pour des espaces de Banach classiques non séparables du style  $L^\infty$, 
mais justement, ces espaces ne sont pas traitables sous leur forme classique par 
les méthodes 
constructives. Le problème est plus délicat pour des espaces qui sont 
classiquement séparables mais pour lesquels il n'y a pas de procédé 
naturel constructif qui donne une partie dénombrable dense. Par exemple c'est 
le cas pour un sous-espace fermé arbitraire d'un espace complet séparable, 
et c'est encore le cas pour certains espaces de fonctions continues.

\smallskip La présentation en unaire et la présentation en binaire de  
$\ZZ$  ne sont pas \equivas du point de vue la 
\com en temps \poll. La présentation en unaire est une présentation 
naturelle de  $\ZZ$  comme groupe tandis que 
la présentation en binaire est une \pres naturelle de  $\ZZ$  comme anneau. 
On peut se poser des questions analogues concernant les espaces métriques 
classiques usuels.

\smallskip La première question qui se pose est la comparaison des 
différentes présentations d'un espace métrique usuel.  Notez que la 
présentation usuelle de  $\RR$  est considérée comme la seule naturelle 
et que c'est à partir de cette présentation de  $\RR$  qu'est définie la 
\com d'une présentation  $(Y_1, \delta_1)$  de  $X$  ou la \com de 
l'application $\Id_X$  de  $(Y_1, \delta_1)$  vers  $(Y_2, \delta_2)$  lorsqu'on 
compare deux présentations distinctes de  $X $.

\smallskip La question qui se pose ensuite est celle de la \com des fonctions continues calculables entre espaces métriques.  S'il~y~a une notion naturelle de \com des fonctions dans le cas des espaces compacts, la question est nettement plus délicate dans le cas général, car elle renvoie à la \com des fonctionnelles de type 2 arbitraires.
 
\smallskip Une autre question est celle de la complexité d'objets liés de 
manière naturelle à l'espace métrique qu'on étudie. Par exemple dans le 
cas des réels, et en ne considérant que la structure algébrique de  $\RR$  
on est intéressé par le fait que la \com de l'addition, celle de la 
multiplication, ou celle de la recherche des racines complexes d'un \pol 
à coefficients réels soient toutes ``de bas niveau''. 
Ce genre de résultats légitime a posteriori le choix qui est fait usuellement pour présenter $\RR$, et les résultats de nature inverse disqualifient d'autres présentations (cf. \cite{fLL}), moins efficaces que la \pres via les suites de Cauchy de rationnels écrits en binaire{\footnote{Néanmoins, le test de signe est indécidable pour les réels présentés à la Cauchy. 
On se contente d'avoir en temps \poll le test constructif:  
$x + y \geq  1/2^n  \Rightarrow   ( x \geq  1/2^{n+2} \; \hbox{ou}\;  
y \geq  1/2^{n+2} )$.}}. 
De même, un espace métrique usuel est en général muni d'une structure 
plus riche que la seule structure métrique, et il s'agit alors d'étudier,
 pour chaque présentation, la \com de ces ``éléments naturels de 
structure''.

\medskip 
Dans la section \ref{fsec3}, nous introduisons l'espace  $\czu$  des fonctions réelles \uni continues sur l'intervalle $[0,1]$ du point de vue de ses présentations rationnelles.
 
La fonction d'évaluation  $(f, x) \mapsto f(x)$  n'est pas \loca \unico. Vue son importance, nous discutons ce que signifie la \com de cette fonction lorsqu'on a choisi une \rp  de l'espace  $\czu$. 

Nous donnons ensuite deux exemples significatifs de telles présentations 
rationnelles:
\begin{itemize}
\item Une \pres par circuits semilinéaires binaires notée  $\csl$, où les 
points rationnels sont exactement les fonctions semilinéaires binaires. Une telle fonction peut être définie par un circuit semilinéaire binaire (cf. définition~\ref{f321}) et codée par un programme d'évaluation correspondant au circuit.
\item une présentation, notée $\crf$, via des fractions rationnelles 
convenablement contrôlées et données en \pres par formule (cf. définition~\ref{f325}).
\end{itemize}
Enfin nous établissons des résultats de \com liés au théorème 
d'approximation de Newman. En bref, le théorème de Newman a une \com 
\polle et en conséquence les fonctions linéaires par morceaux sont, 
chacune individuellement, des points de \com $\p$  dans l'espace~$\crf$.

\medskip Dans la section~\ref{fsec4} nous définissons et étudions des \rps  
``naturelles'' de  $\czu$,  \equivas du point de vue de la \com en temps 
\poll:
\begin{itemize}
\item La \pres dite ``à la Ko-Friedman'', et notée $\ckf$, pour laquelle un point rationnel est donné par un quadruplet $(Pr, n, m, T)$ 
où  $Pr$  est un programme de machine de Turing. Les entiers $n$, $m$,  $T$  sont des paramètres de contrôle.  
Nous faisons le lien avec la notion de \com introduite par Ko et Friedman. Nous démontrons une propriété universelle 
qui caractérise cette \rp  du point de vue de la \com en temps \poll.
\item La \pres par circuits booléens qu'on note par $\cbo$ et pour laquelle un point rationnel est donné par un quadruplet  $(C,n,m,k)$  où  $C$  code un circuit booléen et  $n$,  $m$,  $k$  sont des paramètres de contrôle.
\item La \pres par circuits arithmétiques fractionnaires (avec magnitude) 
notée  $\caf$. 
Un point rationnel de cette \pres est donné par un couple $(C,M)$ où $C$  est le code d'un circuit arithmétique, et $M$ est un paramètre de contrôle.
\item La \pres par circuits arithmétiques polynomiaux (avec magnitude) notée  $\capo$ analogue à la précédente, mais ici le circuit est \poll 
(c.-à-d. ne contient pas les portes ``passage à l'inverse''). 
\end{itemize}
Nous démontrons dans la section \ref{fsubsec42} que les présentations  $\ckf$,  
$\cbo$,  $\csl$,  $\caf$  et  $\capo$  sont \equivas en temps \poll. Ce 
résultat généralise et précise les résultats de Hoover. Non seulement 
les $\p$-points de  $\ckf$  sont ``les mêmes'' que les $\p$-points de  $\capo$, 
mais bien mieux, les bijections  
$$ \ckf \rightarrow \capo \;\; \;  \hbox{et} \; \; \; \capo \rightarrow \ckf $$
qui représentent l'identité de  $\czu$  sont, globalement, calculables en 
temps \poll. 
Ainsi nous obtenons une formulation complètement contrôlée du point de vue 
algorithmique pour le théorème d'approximation de Weierstrass. \\

\smallskip Dans la section~\ref{fsubsec43} , nous démontrons que ces 
présentations ne sont pas de classe  $\p$  en établissant la non 
faisabilité du calcul de la norme (si  $\p\neq \np$). Plus précisément, 
nous définissons convenablement ``le problème de la norme'' et nous 
démontrons qu'il s'agit d'un problème $\np$-complet pour les présentations 
considérées. 

Pour ce qui concerne le test d'appartenance à l'ensemble des (codes des) 
points rationnels,  c'est un problème co-$\np$-complet pour les 
présentations $\ckf$  et  $\cbo$, tandis qu'il est en temps linéaire pour la 
\pres  $\csl$. Cette dernière est donc 
légèrement plus satisfaisante.

\medskip Nous définissons dans la section \ref{fsec5}  d'autres présentations 
de l'espace  $\czu$  à savoir:
\begin{itemize}
\item La \pres notée  $\cw$  (comme Weierstrass) pour laquelle l'ensemble des 
points rationnels est l'ensemble des \pols (à une variable) à 
coefficients rationnels donnés en \pres dense.
\item La \pres notée  $\csp$  pour laquelle l'ensemble des points rationnels 
est l'ensemble des fonctions \polles par morceaux, chaque \pol 
étant donné comme pour  $\cw$.
\item La \pres notée  $\csr$  est obtenue à partir de  $\crf$  de la même 
manière que   $\csp$    est obtenue à partir de  $\cw$.
\end{itemize}
Nous essayons de voir jusqu'à quel point le caractère de classe $\p$ de ces 
présentations fournit un cadre de travail adéquat pour l'analyse 
numérique.\\
Nous caractérisons les $\p$-points de  $\cw$ en établissant l'équivalence 
entre les propriétés suivantes: (cf. théorème \ref{f527})
\begin{itemize}
 \item[a)] $f$  est un $\p$-point de  $\cw$
\item[b)] la suite   $A_n(f)$ (qui donne le développement de  $f$  en série 
de Chebyshev) est une $\p$-suite dans  $\RR$   et vérifie une majoration : 
\[
\abS{A_n(f)}\;  \leq Mr^{n^{\gamma}} \;\hbox{avec}\;M > 0,~\gamma > 0,~0< r < 1.
\]

\item[c)] $f$ est un $\p$-point de  $\ckf$  et est dans la classe de Gevrey 
\end{itemize}
Nous en déduisons une équivalence analogue entre (cf. \thref{f528}):
\begin{itemize}
\item[a)] $f$ est une fonction analytique et est un $\p$-point de  $\cw$.
\item[b)] $f$ est une fonction analytique et est un $\p$-point de  $\ckf$.
\end{itemize}

\smallskip Les calculs usuels sur les $\p$-points de  $\cw$ (calcul de la norme, du maximum et de l'intégrale d'un $\p$-point de  $\cw$) sont en temps \poll (cf. proposition \ref{f526}). 

De plus, un assez bon comportement de la dérivation vis à vis de la \com est obtenu en démontrant que pour tout $\p$-point (ou toute $\p$-suite) de  $\cw$  la suite (ou la suite double) de ses dérivées est une $\p$-suite de  $\cw$ (cf. théorème \ref{f5210} et corollaire \ref{f5211}). 

Dans le théorème \ref{f5213} et son corollaire \ref{f5214}, nous 
donnons une version plus uniforme des résultats précédents. Combinés avec la proposition \ref{f526} on obtient une amélioration sensible des résultats de  \cite{fMu87}  et  \cite{fKF88}. En outre nos preuves sont plus conceptuelles. 

\smallskip Par ailleurs, ces résultats, combinés avec le théorème de Newman (dans sa version précisée à la section \ref{fsubsec33}), démontrent que le passage de la \pres par fraction rationnelles  $\crf$  à la \pres par \pols (denses)  $\cw$ n'est pas en temps \poll. Le théorème de Newman démontre également que les présentations  $\crf$  et  $\csr$  sont \polt 
équivalentes.

\bigskip En résumant les résultats des sections \ref{fsec4} et \ref{fsec5}, 
nous obtenons  que l'identité de $\czu$ est \uni de classe  $\p$  dans les cas suivants: 
\[
\cw \rightarrow \csp \rightarrow \crf \equiv \csr 
\rightarrow \ckf \equiv \cbo \equiv \csl \equiv \caf 
\equiv \capo.
\] 
et aucune des flèches dans la ligne ci-dessus n'est une $\p$-équivalence 
sauf peut être  $\csp \rightarrow \crf $  et très éventuellement  $\crf 
\rightarrow \ckf$  (mais cela impliquerait $\p = \np$).

\bigskip Nous terminons cette introduction par quelques remarques sur le 
constructivisme.

Le travail présenté ici est écrit dans le style des mathématiques 
constructives à la Bishop telles qu'elles ont été développées 
notamment dans les livres de Bishop \cite{fBi}, Bishop \& Bridges \cite{fBB},  
Mines, Richman \& Ruitenburg \cite{fmrr}. Il s'agit d'un corps de mathématiques constructives en quelque sorte minimal\footnote{Un travail précurseur dans un style purement formel est celui de Goodstein \cite{fGo1,fGo}.}. Tous les théorèmes 
démontrés de cette manière sont en effet également valables pour les 
mathématiciens classiques, le bénéfice étant qu'ici les théorèmes et 
les preuves ont toujours un contenu algorithmique. Le lecteur ou la lectrice classique peut donc lire ce travail comme un prolongement des travaux en analyse récursive par Kleene et Turing puis plus tard notamment par Ker-I. Ko \& H. Friedman (\cite{fKF82}, \cite{fKo91}), N. Th. Müller \cite{fMu86,fMu87}, Pour El \& Richards \cite{fPR}, qui se situent eux dans un cadre de mathématiques classiques.

Par ailleurs les théorèmes et leurs preuves dans le style de Bishop sont 
également acceptables dans le cadre des autres variantes du constructivisme, comme l'intuitionnisme de Brouwer ou l'école constructive russe de A.A. Markov, G.S. Ceitin, N.A. Shanin, B.A. Kushner et leurs élèves. On trouvera une discussion éclairante sur ces différentes variantes de constructivisme dans le livre de Bridges \& Richman 
\cite{fBR}. 

Le livre très complet de Beeson \cite{fBe} donne également des 
discussions approfondies de ces points de vue et de leurs variantes, notamment en s'appuyant sur une étude remarquable des travaux des logiciens qui ont tenté de formaliser les mathématiques constructives.  

Les mathématiques constructives russes peuvent être découvertes à 
travers les livres de Kushner \cite{fKu} et O. Aberth \cite{fAb}. Un article historique très pertinent sur la question est écrit par M.~Margenstern \cite{fMa}. 
Le chapitre IV de \cite{fBe} est également très instructif. Beeson déclare page 58: ``nous espérons démontrer que [cet] univers [mathématique] est un endroit extrêmement divertissant, plein de surprises (comme n'importe quel pays étranger), mais en aucun cas trop chaotique ni invivable''. Les mathématiques constructives russes restreignent leurs objets d'étude aux ``êtres récursifs'' et se rapprochent en cela de certains travaux d'analyse récursive classique développés 
notamment par Grzegorczyk, Kreisel, Lacombe, Schoenfield et Specker, avec 
lesquels elles partagent de nombreux résultats. Un des principes des mathématiques constructives russes est par exemple que seuls existent les réels récursifs donnés par des algorithmes (qui calculent des suites de rationnels à vitesse de convergence controlée). Et une fonction réelle définie sur les réels est un algorithme  $F$  qui prend en entrée un algorithme  $x$  et donne en sortie, sous la condition que  $x$  est un algorithme produisant un réel récursif, un algorithme  $y$  produisant un réel récursif. Ceci conduit par exemple au théorème de Ceitin, faux classiquement, selon lequel toute fonction réelle définie sur les réels est continue en tout point. 
L'analyse récursive classique énonce quant à elle: si un algorithme 
termine chaque fois que l'entrée est le code d'un réel récursif et donne 
alors en sortie le code d'un réel récursif, et s'il définit une fonction 
(c.-à-d. deux codes du même réel récursif en entrée conduisent à 
deux codes d'un même réel récursif en sortie) alors il définit une 
fonction continue en tout point réel récursif (théorème de Kreisel-
Lacombe-Schoenfield). Beeson a analysé la preuve de Kreisel-Lacombe-
Schoenfield pour en préciser les aspects non constructifs. La preuve de 
Ceitin, qui est plus constructive que celle de 
Kreisel-Lacombe-Schoenfield, est analysée dans \cite{fBR}.
 
Dans le style de Bishop, comme on considère que la notion d'effectivité est une notion primitive qui ne se réduit pas nécessairement à la 
récursivité, et que la récursivité, au contraire, ne peut être définie sans cette notion primitive d'effectivité, les nombres réels et les fonctions réelles sont des objets qui conservent une plus grande part de 
``liberté'' et sont donc plus proches des réels et des fonctions réelles 
telles que les conçoivent intuitivement les mathématiciens classiques. 
Et le théorème ``russe'' précédent est donc indémontrable dans le style 
Bishop. De manière symétrique, les théorèmes de mathématiques 
classiques qui contredisent directement des résultats du 
constructivisme russe (comme par exemple la possibilité de définir sans 
aucune ambigüité des fonctions réelles discontinues) ne peuvent être 
démontrés dans les mathématiques du style Bishop. 

En conclusion, les mathématiques constructives russes présentent un 
intérêt historique indéniable, développent une philosophie 
mathématique très cohérente et peuvent sans doute trouver un renouveau à 
partir de préoccupations d'informatique théorique. Il serait donc également intéressant de faire une étude de ces mathématiques du point de vue de la complexité algorithmique.    

\section{Préliminaires}\label{fsec1}
\subsection{Notations}\label{fsubsec11}
\begin{description}
\item[$\NN_1$] ensemble des entiers naturels en unaire.
\item[$\NN$] ensemble des entiers naturels en binaire. Du point de vue de la
 \com, $\NN_1$ est isomorphe à la partie de  $\NN$  formée des puissances 
de  $2$.
\item[$\ZZ$] ensemble des entiers relatifs en binaire.
\item[$\QQ$] ensemble des rationnels présentés sous forme d'une fraction 
avec numérateur et dénominateur en binaire.
\item[$\QQ^{\NN_1}$] ensemble des suites de rationnels, où l'indice est en 
unaire.
\item[$\DD$] ensemble des  nombres de la forme  $k/2^n$  avec $(k,n) \in  \ZZ \times  \NN_1$.
\item[$\DD_{[0,1]}$] $\DD \cap  [0,1]$.
\item[$\DD_n$] ensemble  des nombres de la forme  $k/2^n $ avec $k \in  \ZZ$.
\item[$\DD_{n,[0,1]}$] $\DD_n \cap  [0,1]$.
\item[${\DD[X]}$] ensemble des \pols à coefficients dans  $\DD$   
donnés en \pres dense.
\item[${\DD[X]_f}$] ensemble des \pols à coefficients dans  $\DD$   
donnés en \pres par formule.
\item[$\RR$] ensemble des nombres réels présentés par les suites de 
Cauchy dans  $\QQ$.
\item [$\CB$] ensemble des codes de circuits booléens.
\item [$\CA$] ensemble des codes de circuits arithmétiques.
\item[{\rm MT}] machine de Turing.
\item[{\rm MTO}] machine de Turing à oracle.
\item[$\mu$] \mcu (cf.  définition \ref{f221})
\item[$\Lg(a)$] la longueur du codage binaire du nombre dyadique $\vert a \vert$ (pour  $a \in  \DD$).
\item[$\ta(C)$] taille du circuit  $C$  (le nombre de ses portes).
\item[$\prof(C)$] profondeur du circuit  $C$.
\item[$\Mag(C)$] magnitude du circuit arithmétique  $C$ (cf. définition 
\ref{f4110}).
\item[$\wi{f}$] fonction continue codée par  $f$.
\item[$\log(n)$] pour un entier naturel $n$ c'est la longueur de son codage binaire.
\item[$\M(n)$] \com du calcul de la multiplication de deux entiers en représentation binaire (en multiplication rapide $\M(n) = \Oo(n \log(n) \log\log(n)$).
\item[$\flo{d}$]	la longueur (du codage discret) de l'objet  $d$  (le codage de  $d$  est un mot sur un alphabet fini fixé).
\item[$\Flo{x}$] la partie entière du nombre réel $x$.  
\end{description}

\paragraph{Présentations rationnelles de l'espace  $\czu$}~

\begin{itemize}\itemsep2pt
\item[$\ckf$] \pres par Machine de Turing, à la Ko-Friedman\\
$\ykf$ est l'ensemble (des codes de) points rationnels (cf. définition \ref{f411})
\item[$\caf$] \pres par circuits arithmétiques\\
$\yaf$ est l'ensemble (des codes de) points rationnels (cf. définition \ref{f4110})
\item[$\capo$] \pres par circuits arithmétiques sans divisions (ou polynomiaux)\\
$\yap$ est l'ensemble (des codes de) points rationnels (cf. section \ref{fsubsubsec414})
\item[$\cbo$] \pres par circuits booléens\\
$\ybo$ est l'ensemble (des codes de) points rationnels  (cf. définition \ref{f418})
\item[$\csl$] \pres par circuits semilinéaires binaires\\
$\ysl$ est l'ensemble (des codes de) points rationnels (cf. définition 
\ref{f321})
\item[$\crf$] \pres par fractions rationnelles dans un codage par formules\\
$\yrf$ est l'ensemble (des codes de) points rationnels (cf. 
définition \ref{f325})
\item[$\cw$]  \pres ``à la Weierstrass'' par \pols dans un codage dense\\
$\yw$ est l'ensemble (des codes de) points rationnels (cf. 
section \ref{fsubsec51})
\item[$\csp$] \pres par fonctions \polles par morceaux dans un codage dense\\
$\ysp$ est l'ensemble (des codes de) points rationnels (cf. section \ref{fsubsubsec512})
\item[$\csr$] \pres par fractions rationnelles par morceaux dans un codage par formules\\
$\ysr$ est l'ensemble (des codes de) points rationnels (cf. section \ref{fsubsubsec513})
\end{itemize}

\subsection{Classes de fonctions discrètes intéressantes}\label{fsubsec12} 
Nous considérerons des classes de fonctions discrètes~$\ca$   (une fonction 
discrète est une fonction de $A^{\star}$
vers  $B^{\star}$  où  $A$  et  $B$  sont deux alphabets finis{\footnote{Si 
on veut donner une allure plus mathématique et moins informatique à la 
chose, on pourra remplacer  $A^{\star}$  et  $B^{\star}$  par des ensembles  
$\NN^k$.}}) jouissant des propriétés de stabilité élémentaires 
suivantes
\begin{itemize}
\item $\ca$ contient les fonctions arithmétiques usuelles et le test de 
comparaison dans  $\ZZ$   (pour  $\ZZ$  codé en binaire).
\item $\ca$ contient les fonctions calculables en temps linéaire  (c.-à-d. 
$\LINT  \subset  \ca$  )
\item $\ca$ est stable par composition.
\item $\ca$ est stable par listes:  si  $f\colon A^{\star} \rightarrow
B^{\star}$ est dans~$\ca$ alors $\lst(f) \colon  \lst(A^{\star})
\rightarrow \lst(B^{\star})$ est également dans~$\ca$ 
($\lst(f)[x_1, x_2,\ldots, x_n] = [f(x_1), f(x_2),\ldots, f(x_n)]$).
\end{itemize}
Une classe vérifiant les propriétés de stabilité précédentes sera 
dite {\em élémentairement stable}.
En pratique, on sera particulièrement intéressé par les classes 
élémentairement stables suivantes:
\begin{itemize}
\item $\Rec :$ la classe des fonctions récursives.
\item $\Fnc :$ la classe des fonctions constructivement définies  (en 
mathématiques constructives, ce concept est un concept primitif qui ne 
coïncide pas avec le précédent, et il est nécessaire de l'avoir au 
préalable pour pouvoir définir le précédent, la récursivité étant 
interprétée comme une constructivité purement mécanique dans son 
déroulement comme processus de calcul)
\item $\Prim :$ la classe des fonctions primitives récursives.
\item $\p :$ la classe des fonctions calculables en temps \poll.
\item $\cE:$ la classe des fonctions élémentairement récursives, c.-à-d. 
encore calculables en temps majoré par une composée d'exponentielles.
\item $\PSP :$ la classe des fonctions calculables en espace \poll (avec 
une sortie 
polynomialement majorée en taille).
\item $\LINS :$ la classe des fonctions calculables en espace linéaire (avec une sortie linéairement majorée en taille).
\item $\DSRT(\Lin, \Lin, \Poly):$ la classe des fonctions calculables en espace 
linéaire, \etpo  et avec une sortie linéairement majorée en taille.
\item $\DSRT(\Lin, \Lin, \Oo(n^k)):$ la classe des fonctions calculables en espace 
linéaire, en temps $\Oo(n^k)$ (avec  $k > 1$)  et avec une sortie 
linéairement majorée en taille.
\item $\DSRT(\Lin, \Lin, \Exp):$ la classe des fonctions calculables en espace linéaire, en temps $\exp(\Oo(n))$  et avec une sortie linéairement majorée en taille.
\item $\DRT(\Lin,\Oo(n^k)):$ la classe des fonctions calculables en temps 
$\Oo(n^k)$, (avec  $k > 1$)  et avec une sortie linéairement majorée en 
taille.
\item $\DSR(\Poly, \Lin):$ la classe des fonctions calculables en espace 
polynomial  avec une sortie linéairement majorée en taille.
\item $\DSR(\Oo(n^k), \Lin):$ la classe des fonctions calculables en espace 
$\Oo(n^k)$, (avec  $k > 1$)   avec une sortie linéairement majorée en taille.
\item $\QL:$ la classe des fonctions calculables en temps 
quasilinéaire (cf. Schnorr \cite{fSc}). \\ ($\QL: = \cup_b \DTI 
(\Oo(n.\log^b(n))) = \DTI (\QLin)$).
\end{itemize}

\smallskip On remarquera que, hormis les classes $\Rec$  et $\Fnc$, toutes les classes que nous avons considérées sont des classes de \com (au sens de Blum). 
Il n'y a cependant aucune nécessité à cela, comme le montrent justement les exemples de $\Rec$  et $\Fnc$. 
Lorsque nous considérons uniquement la classe $\Fnc$, nous développons un chapitre de mathématiques constructives abstraites.
Notez que $\DTI (\Oo(n^k))$ pour $k > 1$ n'est pas stable par composition. Mais, le plus souvent, les calculs en temps $\Oo(n^k)$ que nous aurons à considérer sont dans la classe $\DRT(\Lin, \Oo(n^k))$  ou même $\DSRT(\Lin, \Lin, \Oo(n^k))$  et ces classes ont les bonnes propriétés de stabilité.

\subsection{Complexité d'une Machine de Turing Universelle}\label{fsubsec13} 
Nous aurons besoin dans la suite d'utiliser une Machine de Turing Universelle et d'estimer sa \com algorithmique. 
Le résultat suivant, pour lequel nous n'avons pas trouvé de référence, semble faire partie du folklore, il nous a été signalé par M. Margenstern.

\begin{flemma} \label{f131} 
Il existe une machine de Turing universelle $MU$ qui fait le travail suivant.\\
Elle prend en entrée:\\
---  le code  (dont la taille est  $p$)  d'une machine de Turing $M_0$ 
(fonctionnant  sur un alphabet fixé, avec une bande d'entrée, une bande de 
sortie, plusieurs bandes de travail)  supposée être de \com en temps  $T $ 
et en espace  $S$  (avec  $S(n) \geq  n$)\\ 
--- une entrée  $x$   (de taille $n$)  pour $M_0$.\\
Elle donne en sortie le résultat du calcul exécuté par  $M_0$  pour 
l'entrée  $x$.\\
Elle exécute cette tâche en un temps  $\Oo(T(n)(S(n)+p))$  et en utilisant un 
espace  $\Oo(S(n)+p).$
\end{flemma} 

\proof 
La machine $MU$ utilise une bande de travail pour y écrire, à chaque étape 
élémentaire de la machine $M_0$,  qu'elle simule, la liste des contenus de 
chacune des variables de~$M_0.$ Pour simuler une étape de $M_0$ la machine  
$MU$  a besoin de  $\Oo(S(n)+p)$  étapes élémentaires, en utilisant un 
espace du même ordre de grandeur.\eop

\subsection{Circuits et programmes d'évaluation}\label{fsubsec14} 
Les familles de circuits constituent des modèles de calcul intéressants, 
notamment du point de vue du parallélisme. Les familles de circuits booléens constituent dans une certaine mesure une alternative au modèle standard des Machines de Turing. Des circuits arithmétiques de faible taille sont capables de calculer des \pols de très grand degré. C'est par exemple le cas pour un circuit arithmétique qui simule une itération de Newton pour une fonction donnée par une fraction rationnelle.
 
Dans tous les cas se pose le problème de savoir quel codage on adopte pour un circuit. Nous choisirons de toujours coder un circuit par un des programmes d'évaluation (ou straight-line program) qui exécutent la même tâche que lui\footnote{À un même circuit peuvent correspondre différents programmes d'évaluation selon l'ordre dans lequel on écrit les instructions à exécuter.}.

Par ailleurs, concernant les circuits arithmétiques, qui représentent des 
\pols ou des fractions rationnelles, nous les envisagerons non pas du point de vue du calcul exact (ce qui serait trop coûteux), mais du point de vue du calcul approché. Se pose alors la question d'évaluer leur temps 
d'exécution lorsqu'on veut garantir une précision donnée sur le 
résultat, pour des entrées dans  $\DD$  données elles-mêmes avec une 
certaine précision. Aucune majoration raisonnable du temps d'exécution ne 
peut être obtenue par des arguments d'ordre général si la profondeur du 
circuit n'est pas très faible, car les degrés obtenus sont trop grands et 
les nombres calculés risquent de voir leur taille exploser. Comme on n'arrive pas toujours à se limiter à des circuits de très faible profondeur, il s'avère indispensable de donner un paramètre de contrôle, appelé magnitude, qui assure que, malgré un éventuel très grand degré, la 
taille de tous les nombres calculés par le circuit (qui seront évalués 
avec une précision elle même limitée) n'est pas trop grande lorsque 
l'entrée représente un réel variant dans un intervalle compact.

\section{Espaces métriques complets \raprs   }\label{fsec2}

\subsection[Présentation rationnelle d'un espace métrique, \com des 
points \ldots ]{Présentation rationnelle d'un espace métrique, \com des points 
et des familles de points}\label{fsubsec21}
Sauf mention expresse du contraire, les classes de fonctions discrètes que 
nous considérons seront des classes élémentairement stables.
\begin{fdefinition}[\rp  de classe~$\ca$  pour un espace métrique complet] \rm \label{f211}
~\\
Un espace métrique complet  $(X, d_X)$  est {\em donné dans une \pres  
rationnelle de classe~$\ca$} de la manière suivante. On donne un triplet  
$(Y, \delta, \eta)$  où
\begin{itemize}
\item $Y$  est une~$\ca$-partie d'un langage  $A^{\star}$
\item $\eta$ est une application de $Y$ dans $X$
\item$\delta\colon  Y\times Y\times \NN_1\rightarrow
\DD$   est une fonction dans la classe~$\ca$   et vérifiant  (pour  
$n\in\NN_1$  et  $x ,  y , z   \in  Y$):\\
--- $\abs{d_X(x,y)-\delta(\eta(x),\eta(y))}\; \leq 1/2^n$\\
--- $\delta(x,y,n) \in \DD_n $\\ 
--- $\delta(x,y,n) = \delta(y,x,n)$\\
--- $\delta(x,x,n) = 0$\\ 
--- $\abs{\delta(x,y,n+1)-\delta(x,y,n)}\;  \leq 1/2^{n+1}$\\ 
--- $\delta(x,z,n) \leq \delta(x,y,n) + \delta(y,z,n) + 2/2^n$ 
\end{itemize}
Si on définit l'écart  $d_Y(x,y)$   comme la limite de $\delta(x,y,n)$ 
lorsque $n$  tend vers l'infini, l'application  $\eta$  de $Y$ dans $X$ 
identifie $(X,d_X)$
au séparé complété de $(Y,d_Y)$.
L'ensemble  $\eta(Y)$  est appelé l'ensemble des {\em points rationnels}  de  
$X$  pour la \pres considérée.  Si  $y \in  Y$  on notera souvent   
$\wi{y}$   pour le point rationnel $\eta(y)$,  et  $y$  est appelé le code de  
$\wi{y}$.
\end{fdefinition}
Nous abrégerons parfois  ``$(Y,\delta,\eta)$ est une \pres  rationnelle de 
classe~$\ca$   pour  $(X,d_X)$''  en  ``$(Y,\delta)$ est une~$\ca$-\pres de  
$(X,d_X)$''.

\begin{fremarks} \label{f212}
1) On peut remplacer  $\DD$  par  $\QQ$  dans la définition ci-dessus sans 
modifier la notion qui est définie. 
Le fait de choisir  $\QQ$  est plus joli mathématiquement, tandis que le fait de choisir   $\DD$  est plus naturel d'un point de vue informatique. 
En outre la contrainte  $\delta(x,y,n) \in \DD_n$  répond à la requête naturelle de ne pas utiliser plus de place que nécessaire pour représenter une approximation à $1/2^n$  près d'un nombre réel.\\
2) Lorsqu'on a donné une \rp  pour un espace métrique abstrait  $(X,d_X)$, 
on dira qu'{\em on l'a muni d'une structure de calculabilité}. Cela revient essentiellement à définir un langage  
$Y \subset  A^{\star}$  puis une application $\eta$ de  $Y$  vers  $X$  dont l'image soit une partie dense de  $X$. La \pres est complètement définie uniquement lorsqu'on a aussi donné une application  
$\delta \colon  Y \times  Y \times  \NN_1  \rightarrow   \DD$  
vérifiant les requêtes de la définition \ref{f211}. 
{\em Dans la suite, on se permettra néanmoins de dire qu'un codage d'une 
partie dense  $Y$  de  $X$  définit une \pres de classe~$\ca$   de  $X$  
lorsqu'on démontre que les requêtes de la définition \ref{f211} peuvent être satisfaites.} \\
3)	Comprise au sens constructif, la phrase incluse dans la définition 	
	``on donne une application  $\eta$  de  $Y$  vers $X$''	 
réclame qu'on ait	 
	``un certificat que  $\eta(y)$  soit bien un élément de  $X$  pour tout $y$  de $Y$''. 	
Il se peut que ce certificat d'appartenance implique lui-même une notion 
naturelle de \com. Lorsque ce sera le cas, par exemple pour l'espace  
$\czu$, il sera inévitable de prendre en compte cette \com. 
{\em Ainsi, la définition \ref{f211} doit en l'état actuel être considérée comme incomplète et à préciser au cas par cas.} C'est sans doute dommage si on se place du point de vue l'élégance formelle des définitions 
générales. Mais c'est une situation de fait qui semble très difficile à 
contourner. 
\end{fremarks}

\begin{fexamples} \label{f213}~\\
--- L'espace  $\RR$  est défini usuellement dans la \pres où l'ensemble des 
points rationnels est  $\QQ$.  De manière équivalente, et c'est ce que nous 
ferons dans la suite, on peut considérer comme ensemble des points rationnels 
de  $\RR$  l'ensemble  $\DD$  des nombres dyadiques{\footnote{Cela correspond 
à la notation $\RR_{\rm conv}$  dans \cite{fLL}  et   $\RR_{\rm con}$  dans 
\cite{fKo83}.}}.\\ 
--- On définira sans difficulté la \pres produit de deux présentations 
pour deux espaces métriques.\\ 
--- Tout {\em ensemble discret}  $Z$  ($Z$  est donné comme une partie d'un 
langage  $A^{\star}$  avec une relation d'équivalence qui définit 
l'égalité dans  $Z$  et qui est testable dans la classe~$\ca$  )  donne 
lieu à un {\em espace métrique discret}, c.-à-d. dans lequel la distance 
de deux points distincts est égale à  $1$. \\
--- Les espaces métriques complets des mathématiques constructives dans le 
style Bishop \cite{fBB} admettent en général une \rp  de classe  $\Fnc$. 
Dans chaque cas concret la \pres s'avère être une \pres de classe $\Prim$ ou 
même  $\p$. 
\end{fexamples}

\begin{fdefinition}[complexité d'un point dans un 
espace métrique rationnellement présenté] \label{f214}~\\
On considère un espace métrique complet  $X$  donné dans  une \pres  
$(Y,\delta,\eta)$  de classe $\ca'$. Un point  $x$  de  $X$  est dit {\em de 
classe}~$\ca$   lorsqu'on connaît une suite  ($n \mapsto y_n$)  de classe 
$\ca$  (en tant que fonction   $\NN_1   \rightarrow   Y$)   avec   
$d_X(x,\eta(y_n)) \leq 1/2^n$.
 On dit encore que  $x$  est {\em un~$\ca$-point dans  $X$}.
\end{fdefinition}

\begin{fexample}  \label{f215}
Lorsque  $X = \RR$, la définition ci-dessus correspond à la notion usuelle 
de ``nombre réel de classe~$\ca$''    (au sens de Cauchy).
\end{fexample}

\begin{fremark}\label{f216}
Il semblerait naturel de demander que la classe $\ca'$ contienne la classe  
$\ca$, mais ce n'est pas complètement indispensable. Cette remarque vaut pour  à peu près toutes les définitions qui suivent.
\end{fremark}

\mni
{\bf Définition \ref{f214}bis  ~}
{\em (Complexité d'un point dans un espace métrique 
rationnellement présenté, version plus explicite)} 
On considère un espace métrique complet  $X$  donné dans  une \pres de 
classe  $\ca'$: $(Y, \delta,\eta)$. \\
Un point  $x$  de classe~$\ca$   dans  $X$  est donné par une suite  ($n 
\mapsto y_n$)  de classe~$\ca$   
(en tant que fonction   $\NN_1   \rightarrow   Y$)   
vérifiant les conditions
$\delta(y_n,y_{n+1},n+1) \leq 1/2^n$ 
et $x=\lim_{n\rightarrow \infty}\eta(y_n)$.

\medskip 
Nous laissons à la lectrice ou au lecteur le soin de vérifier que les deux 
définitions \ref{f214} et \ref{f214}bis sont équivalentes.
Nous passons maintenant à la définition de la \com pour une famille de 
points (avec un ensemble d'indices discret)

\begin{fdefinition}[complexité d'une famille de points dans un espace métrique rationnellement présenté] \label{f217} 
On considère un préensemble discret  $Z$  (c'est l'ensemble des indices de 
la famille, il est donné comme une partie d'un langage  $A^{\star}$  sur un alphabet fini  $A$)  et un espace métrique complet~$X$  donné dans  une \pres de classe $\ca':~(Y, \delta,\eta).$  Une fonction (ou famille)   
$f\colon  Z \rightarrow X$  est dite de classe~$\ca$  lorsqu'on connaît une fonction $\varphi : Z \times \NN_1 \rightarrow Y$ qui est de classe~$\ca$   et qui vérifie 
$d_X(f(z),\eta(\varphi(z,n))) \leq 1/2^n$  pour tout $z \in Z$. 
On dit alors que  $\varphi$  est {\em une \pres de classe~$\ca$  de la famille $f(z)_{z\in Z}$ de points de  $X$.}
\end{fdefinition}

\begin{fexamples} \label{f218}~\\
--- Lorsque  $Z = \NN_1$  et  $X = \RR$  la définition ci-dessus correspond 
à la notion usuelle de ``suite de réels de classe~$\ca$'' (on dit encore: 
{\em une~$\ca$-suite dans  $\RR$}). \\
--- La définition d'un espace \rapr dans la classe~$\ca$  peut être relue de 
la manière suivante. On donne: \\
\spa -- une~$\ca$-partie  $Y$  d'un langage  $A^{\star}$. \\
\spa -- une fonction $\varphi \colon  Y \rightarrow X$ telle que $\varphi(Y)$ soit 
dense dans  $X$  et telle que la famille de nombres réels  $\; \; Y \times Y
\rightarrow \RR$, $\; \; (y_1, y_2) \mapsto  d_X(\varphi(y_1), \varphi(y_2))$ 
soit de classe~$\ca$. 
\end{fexamples}
 
\begin{fproposition} \label{f219}
Soit  $(x_n)$  une suite de classe~$\ca$   (l'ensemble d'indices est  $\NN_1$) 
dans un espace métrique \rapr  $(X,d)$.  Si la suite est explicitement de 
Cauchy avec la majoration   $d(x_n, x_{n+1}) \leq 1/2^n$  alors la limite de la 
suite est un point de classe~$\ca$   dans $X$.
\end{fproposition}
\proof Soient  $(Y, \delta, \eta)$  une re\rp  de  $(X,d)$  et  $x$  la limite de la suite  $(x_n)$.
La suite  $(x_n)$  est de classe~$\ca$   dans  $X$, donc il existe une fonction  $\psi \colon  \NN_1 \times \NN_1 \rightarrow Y$  de classe~$\ca$  telle que: 
$$d(x_n,\psi(n,m)) \leq 1/2^m \; \;\hbox{pour tout}\; n, m \in \NN_1.
$$
On pose   $z_n = \psi(n+1,n+1)$, c'est une suite de classe~$\ca$    et
$$d(x,z_n) \leq d(x,x_{n+1}) + d(x_{n+1},\psi(n+1,n+1)) \leq 1/2^{n+1} +
1/2^{n+1} = 1/2^n.
$$
Donc  $x$  est un~$\ca$-point dans  $X$.
\eop

\begin{fremark}\label{f2110}
Si on ralentit suffisamment la vitesse de convergence de la suite de Cauchy, on 
peut obtenir comme point limite d'une suite de classe  $\DTI (n^2)$  un point 
récursif arbitraire de~$X$  (cf. \cite{fKF82} pour l'espace  $[0,1]$).
\end{fremark} 
\mni
{\bf Définition \ref{f217}bis  }
{\em (complexité d'une famille de points dans un espace métrique 
rationnellement présenté, version plus explicite)} 
On considère un préensemble discret  $Z$  (une partie d'un 
langage~$A^{\star}$)  et un espace métrique complet  $X$  donné dans  une \pres de classe $\ca': (Y,\delta,\eta).$  
Une famille $ f \colon  Z \rightarrow X$ est dite  de classe~$\ca$    si on a une fonction  
$ \varphi \colon  Z \times \NN_1 \rightarrow Y$ de classe~$\ca$   qui vérifie $\delta(\varphi(z,n),\varphi(z,n+1),n+1) \leq 1/2^n  $
et $f(z)=\lim_{n\rightarrow \infty}\eta(\varphi(z,n))$ pour tout $z\in Z$.

\subsection{Complexité des fonctions \uni continues} \label{fsubsec22}

Nous commençons par donner une définition ``raisonnable'' qui sera 
justifiée par les exemples et propositions qui suivent.

\begin{fdefinition}[\com des fonctions \unicos entre espaces métriques rationnellement présentés] \label{f221}
On considère deux espaces métriques complets  $X_1$  et  $X_2$  donnés 
dans  des présentations de classe  $\ca'$:   $(Y_1,\delta_1,\eta_1)$  et  
$(Y_2,\delta_2,\eta_2).$  Soient  $f \colon  X_1 \rightarrow X_2 $  une fonction \unico 
et  $\mu \colon  \NN_1 \rightarrow \NN_1$ une suite d'entiers.\\ 
On dira que $\mu$ est {\em un \mcu pour la fonction   $f$}  si on a :  
$$d_{X_1}(x,y) \leq 1/2^{\mu(n)} \Rightarrow d_{X_2}(f(x),f(y)) \leq
1/2^n \; \hbox{pour tous} \; x, y \in X_1 \;\hbox{et}\; n \in \NN_1$$
On dira que la fonction   $f$   est {\em \uni de classe}~$\ca$    (pour les 
présentations considérées) lorsque\\
--- elle possède un \mcu  $\mu \colon  \NN_1  \rightarrow   \NN_1$   dans la 
classe~$\ca$\\  
--- la restriction de  $f$  à  $Y_1$  est dans la classe~$\ca$   au sens de 
la définition \ref{f217}, c.-à-d. qu'elle est présentée par une fonction   
$\varphi\colon \NN_1 \times Y_1 \rightarrow Y_2 $  qui est de classe~$\ca$   et qui 
vérifie:
$$d_{X_2}(f(y),\eta_2(\varphi(y,n))) \leq 1/2^n \; \hbox{pour tout}\,  y \in 
Y_1.$$	
Lorsque  $X_1 = X_2 =X$  la fonction  $\Id_X$  admet  $\Id_{\NN_1}$  pour 
\mcu. Si les deux fonctions  $\Id_X$  (de  $X_1$  vers  $X_2$  et de  $X_2$  
vers  $X_1$)  sont dans la classe~$\ca$   on dira que {\em les deux 
présentations sont 
(\uniz)~$\ca$-équivalentes.}
\end{fdefinition}

\begin{fremarks}\label{f222}~\\
1)	On ne demande pas que, pour  $n$  fixé, la fonction  $y \mapsto
\varphi(y,n)$  soit \unico sur  $Y_1$  ni même qu'elle soit continue en chaque 
point de  $Y_1$.\\
2)	Notez que la définition ici donne un \mcu correspondant (à très peu 
près) à la définition donnée en  théorie 
classique de l'approximation. On appelle en général \mcu une suite croissante d'entiers  $\NN\to\ZZ:\;n \mapsto \nu(n)$  qui tend vers $+\infty$ et qui vérifie:
\[
d_{X_1}(x,y) \leq 1/2^n \Rightarrow d_{X_2}(f(x),f(y)) \leq 1/2^{\nu(n)}
\]	
La fonction $\nu$ est en quelques sorte ``réciproque'' de la fonction $\mu$.
\end{fremarks}

\begin{fexamples}\label{f223}~\\
--- Lorsque  $X_1 = [0,1]$ et $X_2 = \RR$  et lorsque~$\ca$    est une classe 
de \com en temps ou en espace {\em élémentairement stable}, la définition 
ci-dessus est \equiva à la notion usuelle de fonction réelle calculable dans 
la classe~$\ca$  (cf. \cite{fKF82}, \cite{fKo91}), comme nous le prouverons en 
détail à la proposition~\ref{f415}.\\
--- Lorsque  $X_1$ est un espace métrique discret la définition~\ref{f221} 
redonne la définition~\ref{f217}. 
\end{fexamples}
La définition \ref{f221} rend accessible la notion de \com pour les fonctions 
\unicos. Cela traite donc toutes les fonctions continues dans le cas où 
l'espace de départ est compact. Rappelons à ce sujet qu'en analyse 
constructive on définit, dans le cas d'un espace compact, la continuité 
comme signifiant la continuité uniforme (cf. \cite{fBB}) et non pas la 
continuité en tout point. 

\begin{fremark}\label{f224} 
Le contrôle de la continuité donné par le \mcu est essentiel dans la 
définition \ref{f221}. On peut par exemple définir une fonction  $\varphi \colon [0,1] \rightarrow \RR$  qui est continue au sens classique, dont la restriction 
à  $\DD_{[0,1]} = \DD\cap [0,1]$  est $\LINT$  mais qui ne possède pas de 
\mcu récursif. Pour cela on considère une fonction  $\theta \colon  \NN 
\rightarrow \NN$  de classe $\LINT $ injective et d'image non récursive. Pour 
chaque $m \in \NN_1$ on considère  $n = \theta(m)\; , a_n = 1/(3.2^n)$  et on 
définit  $\psi_m \colon  [0,1] \rightarrow\RR$  partout nulle sauf sur un 
intervalle centré en   $a_n$  sur lequel le graphe de  $\psi_m$  fait une 
pointe de hauteur  $1/2^n$  avec une pente égale à  $1/2^m$.  Enfin, on 
définit  $\varphi$  comme la somme de la série  $\sum_m \psi_m$.  Bien que 
la suite  $\varphi(a_n)$  ne soit pas une suite récursive de nombres réels 
(ce qui implique d'ailleurs que $\varphi$ ne puisse avoir de \mcu récursif), 
la restriction de $\varphi$ aux nombres dyadiques est très facile à calculer  
(les nombres réels litigieux  $a_n$  ne sont pas des dyadiques). Cet exemple 
met bien en évidence que {\em la définition classique de la continuité (la 
continuité en tout point) ne permet pas d'avoir accès au calcul des valeurs 
de la fonction à partir de sa restriction à une partie dense de l'espace de 
départ.}
\end{fremark}

\begin{fremark} \label{f225}
Dans \cite{fKF82}  il est ``démontré'' que si une fonctionnelle définie via 
une machine de Turing à oracle (MTO) calcule une fonction de  $[0,1]$  
vers  $\RR$,  alors la fonction possède un module de continuité uniforme 
récursif. Si on évite tout recours aux principes non constructifs (caché 
dans le théorème de Heine-Borel), la preuve de \cite{fKF82} peut être 
facilement transformée en une preuve constructive du théorème suivant:  
si une fonction  $f\colon [0,1] \rightarrow \RR$  est \unico et calculable par une MTO, alors son \mcu est récursif.\\ 
A contrario, si on fait l'hypothèse (nullement invraisemblable) selon laquelle tout oracle d'une machine de Turing serait fourni par un procédé mécanique (mais inconnu), il est possible de définir des fonctions ``pathologiques'' de  $[0,1]$  dans  $\RR$  au moyen de machines de Turing à oracles: ce sont des fonctions continues en tout point réel récursif mais non \unicos. 
Ceci est basé sur l'``arbre singulier de Kleene'': un arbre binaire récursif infini qui ne possède aucune branche infinie récursive. Cf. \cite{fBe}, théorème de la section 7 du chapitre 4, page 70 où est donnée une fonction  $t(x)$  continue en tout point réel récursif sur $[0,1]$, mais qui n'est pas bornée (donc pas \unicoz) sur cet intervalle.
\\  
Il semble étrange que dans une preuve concernant les questions de 
calculabilité, on puisse utiliser sans même la mentionner l'hypothèse 
que les oracles d'une MTO se comportent de manière antagonique avec la 
Thèse de Church (du moins sous la forme où elle est admise dans le
 constructivisme russe).
\end{fremark}

On vérifie sans peine que la définition \ref{f221} peut être traduite sous 
la forme plus explicite suivante. 

\mni
{\bf Définition \ref{f221}bis  ~}
{\em (\com des fonctions \unicos entre espaces métriques rationnellement 
présentés, forme plus explicite)}
On considère deux espaces métriques complets  $X_1$  et  $X_2$  donnés 
dans  des présentations de classe  $\ca'$:   $(Y_1,\delta_1,\eta_1)$  et  
$(Y_2,\delta_2,\eta_2).$  Une fonction \unico  $f \colon  X_1 \rightarrow X_2 $ est 
dite {\em  uniformément de classe~$\ca$}   (pour les présentations 
considérées) lorsqu'elle est présentée au moyen de deux données:\\
--- la restriction de  $f$  à $Y_1$   est présentée par une fonction   
$\varphi \colon  Y_1\times \NN_1 \rightarrow Y_2$ qui est de classe~$\ca$   et qui vérifie:
$$\delta_2(\varphi(y,n),\varphi(y,n+1),n+1) \leq 1/2^n \; \hbox{pour tout} \; 
y \in Y_1.$$
avec $f(y)=\lim_{n\rightarrow \infty}\eta_2(\varphi(y,n))$.\\
--- une suite  $\mu\colon  \NN_1  \rightarrow   \NN_1$   dans la classe~$\ca$    
vérifiant:  pour tous  $x$  et  $y$  dans  $Y_1$
$$\delta_1(x,y,\mu(n))\le 1/2^{\mu(n)} \; \Rightarrow\;  
\delta_2(\varphi(x,n+2),\varphi(y,n+2),n+2) \leq 1/2^n.$$

\medskip 
Les résultats qui suivent sont faciles à établir.
\begin{fproposition} \label{f226}
La composée de deux fonctions \uni de classe~$\ca$  est une fonction \uni de 
classe~$\ca$. 
\end{fproposition}
\begin{fproposition} \label{f227}
L'image d'un point  (respectivement d'une famille de points)  de classe~$\ca$  
par une fonction \uni de classe~$\ca$   est un point  (respectivement d'une 
famille de points)  de classe~$\ca$. 
\end{fproposition}

\begin{fremark}\label{f228}
Considérons deux \rps  d'un même espace métrique complet~$X$ données 
par les deux familles de points rationnels respectifs  $(\wi{y})_{y \in 
Y}$  et  $(\wi{z})_{z \in Z}$.  Si la première famille est $\LINT $ pour 
la première présentation elle peut naturellement être de plus grande 
\com pour la deuxième. \\
De manière générale, {\em dire que l'identité de  $X$  est \uni de classe  
$\ca$   lorsqu'on passe de la première à la deuxième \pres revient 
exactement à dire que la famille $(\wi{y})_{y \in Y}$  est une~$\ca$-
famille de points pour la deuxième présentation.} \\
Ainsi  deux présentations de $X$ sont~$\ca$-\equivas si et seulement si la famille  $(\wi{y})_{y \in Y}$  est une famille de classe~$\ca$  dans la seconde \pres et $(\wi{z})_{z \in Z}$   est une famille de classe~$\ca$   dans la première. 
La proposition~\ref{f227} nous permet de formuler un résultat analogue et plus 
intrinsèque: {\em deux \rps  d'un même espace métrique complet sont 
$\ca$-\equivas si et seulement si elles définissent les mêmes familles de 
points de classe~$\ca$}.
\end{fremark}
Nous passons maintenant à la définition de la \com pour une famille de 
fonctions \unicos (avec un ensemble d'indices discret).

\begin{fdefinition}[\com d'une famille de fonctions \unicos entre espaces métriques 
rationnellement présentés]\label{f229}

On considère un préensemble discret  $Z$  (c'est l'ensemble des indices de 
la famille, il est donné comme une partie d'un langage  $A^{\star}$  sur un alphabet fini  $A$)  et deux espaces métriques complets  $X_1$  et  $X_2$  donnés dans  des présentations de classe  $\ca'$:   $(Y_1,\delta_1,\eta_1)$  et  $(Y_2,\delta_2,\eta_2)$.\\  
Nous notons  $\U(X_1,X_2)$  l'ensemble des fonctions \unicos de  $X_1$  dans  $X_2$.  
Une famille de fonctions \unicos $\wi{f}\colon Z\rightarrow 
\U(X_1,X_2)$ est dite {\em \uni de classe~$\ca$} (pour les présentations considérées) lorsque\\
--- la famille possède un \mcu  $\mu \colon  Z \times \NN_1 \rightarrow\NN_1$   
dans la classe~$\ca$ :   la fonction  $\mu$  doit vérifier     
\[
\forall  n \in \NN_1\;\forall  x,
y \in X_1 \;\; \; \;  \big(d_{X_1}(x,y) \leq 1/2^{\mu(f,n)} \Rightarrow 
d_{X_2}(\wi{f}(x),\wi{f}(y)) \leq 1/2^n \big)
\] 
--- la famille   $(f,x) \mapsto \wi{f}(x) \colon  Z \times X_2 \rightarrow
X_2$  est une famille de points dans  $X_2$  de classe~$\ca$   au sens de la 
définition~\ref{f217}, c.-à-d. qu'elle est présentée par une fonction    
$\varphi\colon  Z \times Y_1 \times \NN_1 \rightarrow Y_2 $    qui est de classe  
$\ca$   et qui vérifie:
$$d_{X_2}\big(\wi{f}(x),\varphi(f,x,n)\big) \leq 1/2^n \; \hbox{pour tout} \; 
(f, x, n) \in Z \times Y_1 \times \NN_.$$
\end{fdefinition}

\begin{fremark} \label{f2210}
La notion définie en 2.2.9 est naturelle car elle est la relativisation à la classe~$\ca$ de la notion constructive de famille de fonctions \unicos. 
Mais cette notion naturelle ne semble pas pouvoir se déduire de la 
définition \ref{f221} en munissant   $Z \times  X_1$  d'une structure 
convenable d'espace métrique \rapr et en demandant que la fonction   $(f,x) \mapsto \wi{f}(x) \,;\, Z \times X_2 \rightarrow X_2$   soit \uni de classe  
$\ca$. Si par exemple, on prend sur   $Z \times  X_1$  la métrique déduite 
de la métrique discrète de  $Z$  et de la métrique de  $X_1$  on obtiendra 
la deuxième des conditions de la définition  \ref{f229}  mais la première 
sera remplacée par la demande que toutes les fonctions de la famille aient un 
même \mcu de classe~$\ca$. C'est-à-dire que la famille devrait être \uni 
équicontinue. Cette condition est intuivement trop forte. Les propositions 
\ref{f2211} et \ref{f2212} qui suivent sont une confirmation que la définition 
\ref{f229} est convenable.
\end{fremark}
Les propositions \ref{f226}  et \ref{f227} se généralisent au cas des familles 
de fonctions. Les preuves ne présentent aucune difficulté.
\begin{fproposition} \label{f2211}
Soient  deux espaces métriques complets  $X_1$  et  $X_2$  
\raprs.  Soit   $\big(\wi{f}\,\big)_{f \in Z}$  une famille dans   $\U(X_1,X_2)$  
\uni de classe~$\ca$    et  $(x_n)_{n \in \NN_1}$  une famille de classe~$\ca$   
dans  $X_1$.\\  
Alors la famille  $(\wi{f}(x_n))_{(f,n) \in Z \times \NN_1}$  est une 
famille de classe~$\ca$   dans  $X_2$.
\end{fproposition}

\begin{fproposition} \label{f2212}
Soient  trois espaces métriques complets  $X_1$, $X_2$  et  $X_3$  \raprs.  
Soit   $\big(\wi{f}\,\big)_{f \in Z}$ une famille dans  $\U(X_1,X_2)$  \uni de 
classe~$\ca$    et    $(\wi{g})_{g \in Z'}$  une famille dans  
$\U(X_2,X_3)$  \uni de classe~$\ca$. \\ 
Alors la famille  
$(\wi{f} \circ \wi{g})_{(f,g) \in Z \times Z'}$  
dans  $\U(X_1,X_3)$  est \uni de classe~$\ca$.
\end{fproposition}

\subsection{Complexité des fonctions ``\loca \unicos''} \label{fsubsec23}
La notion de \com définie au paragraphe précédent pour les fonctions 
\unicos est entièrement légitime lorsque l'espace de définition est 
compact. Dans le cas d'un espace \loca compact au sens de Bishop{\footnote{I.e., toute partie bornée est contenu dans un compact, \cite{fBB}.}} les fonctions continues sont les fonctions \unicos sur tout borné (du point de vue classique c'est un théorème, du point de vue constructif, c'est une définition). Ceci conduit à la notion de \com naturelle suivante.

\begin{fdefinition} \label{f231}
On considère deux espaces métriques complets  $X_1$  et  $X_2$  donnés 
dans  des présentations de classe  $\ca'$:   $(Y_1,\delta_1,\eta_1)$  et  
$(Y_2,\delta_2,\eta_2).$  On suppose avoir spécifié un point  $x_0$  de  
$X_1$  et un point  $y_0$  de  $X_2$. Une fonction  $f\colon  X_1  \rightarrow   X_2$   
est dite {\em \loca \unicoz}   si elle est \unico et bornée sur toute partie 
bornée.\\
Elle est dite {\em \loca \uni de classe~$\ca$}   (pour les présentations 
considérées) lorsque\\
--- elle possède dans la classe~$\ca$   une borne sur tout borné, c.-à-d. 
une~$\ca$-suite  $\beta \colon  \NN_1 \rightarrow \NN_1$ vérifiant, pour tout  $x$  dans  
$X_1$  et tout  $m$   dans  $\NN_1$:    
\[
d_{X_1}(x_0,x) \leq 1/2^n \Rightarrow d_{X_2}(y_0,f(x)) \leq 1/2^{\beta(n)};
\]
--- elle possède dans la classe~$\ca$   un \mcu sur tout borné, c.-à-d. 
une~$\ca$-fonction  $\mu \colon  \NN_1 \times \NN_1 \rightarrow \NN_1$   vérifiant,   pour   
$x, z$  dans  $X_1$  et   $n, m$  dans  $\NN_1$:    
\[
\left( d_{X_1}(x_0,x) \leq 1/2^m, \;d_{X_1}(x_0,z) \leq 1/2^m, \; d_{X_1}(x,z) \leq 1/2^{\mu(m,n)}\right) \; \Rightarrow\; d_{X_2}(f(x),f(z)) \leq 1/2^n;
\] 
--- la restriction de $f$ à $Y_1$ est dans la classe~$\ca$ au sens de 
la définition \ref{f217}, c.-à-d. qu'elle est présentée par une~$\ca$-fonction    
$\varphi\colon \NN_1 \times Y_1 \rightarrow Y_2 $     qui vérifie:
\[
d_{X_2}(f(y),\varphi(y,n)) \leq 1/2^n \;\hbox{pour tout} \; y \in Y_1.
\]
\end{fdefinition}
 
Notez que la notion définie ci-dessus ne dépend pas du choix des points
$x_0$ et $y_0$.

\begin{fexamples}\label{f232}~\\
--- Lorsque  $X_1 = \RR$  et  $X_2 = \RR$  la définition ci-dessus est 
\equiva à la notion naturelle de fonction réelle calculable dans la classe  
$\ca$   telle qu'on la trouve dans \cite{fHo90} et \cite{fKo91}.\\
--- La fonction  $x \mapsto x^2$  est \loca \unico de classe $\QL$ mais elle 
n'est pas \unico sur $\RR$.
\end{fexamples}

\begin{fremarks}\label{f233} ~\\
1)	Lorsque  $f$  est une fonction \unico, la définition \ref{f231} et la 
définition \ref{f221} se ressemblent beaucoup. Cependant la définition 
\ref{f221} est a priori plus contraignante. En effet, si $\mu(n)$  est un module 
de continuité dans la définition \ref{f221}, alors on peut prendre dans la 
définition \ref{f231}  $\mu'(m,n) = \mu(n)$. Mais réciproquement, supposons 
qu'on ait un \mcu sur tout borné vérifiant  $\mu'(m,n) = \inf(m, 2^n)$   
alors, la fonction est \unico mais le meilleur \mcu qu'on puisse en déduire 
est  
$$\mu(n) = {\bf Sup} \left\{ \mu'(m,n)\,;\, m \in \NN \right\} = 2^n$$
et il a un taux de croissance exponentiel alors que  $\mu'(m,n)$  est 
linéaire. Ainsi la fonction  $f$  peut être de \com linéaire en tant que 
fonction \loca \unico, et exponentielle en tant que fonction \unico.  
Les deux définitions sont \equivas si l'espace  $X_1$  a un diamètre fini.
\\
2)	En analyse constructive un espace compact est un espace précompact et 
complet. Il n'est pas possible de démontrer constructivement que toute partie 
fermée d'un espace compact est compacte. Un espace 
(\uniz) \loca compact est un espace métrique complet dans lequel tout borné 
est contenu dans un compact  $K_n$,  où la suite  $(K_n)$  est une suite 
croissante donnée une fois pour toutes (par exemple le compact  $K_n$  
contient la boule  $B(x_0, 2^n)$). Une fonction définie sur un tel espace est 
alors dite continue si elle est \unico sur toutes les parties bornées.  En 
particulier elle est bornée sur toute partie bornée. La définition 
\ref{f231} permet donc de donner dans ce cas une version ``\com'' de la 
définition constructive de la continuité.
\\
3)	La proposition \ref{f226} sur la composition des fonctions \uni de classe 
$\ca$  reste valable pour les fonctions \loca \uni de classe~$\ca$. La 
proposition \ref{f227} également.
\end{fremarks}

\subsection{Une approche générale de la \com des fonctions continues  } 
\label{fsubsec24} 
La question de la \com des fonctions continues n'est manifestement pas 
épuisée. 
Comme nous l'avons déjà signalé, s'il est bien vrai qu'une fonction 
continue  $f(x)$  est classiquement bien connue à partir d'une \pres  $(y, 
n) \mapsto \varphi(y,n)$  qui permet de la calculer avec une précision 
arbitraire sur une partie dense  $Y$  de  $X$, la \com de  $\varphi$  ne peut 
être tenue que pour un pâle reflet de la \com de  $f$  (cf. les remarques 
\ref{f224} et \ref{f233}(1)).
Une question cruciale, et peu étudiée jusqu'à présent, est de savoir 
jusqu'à quel point on peut certifier qu'une telle donnée  $\varphi$  
correspond bien à une fonction continue~$f$. Dans le cas d'une réponse 
positive, il faut expliquer par quelle procédure on peut calculer des valeurs approchées de~$f(x)$  lorsque~$x$  est un point arbitraire de  $X$  (donné par exemple par une suite de Cauchy de points rationnels de la \pres dans le cas d'un espace métrique rationnellement présenté).\\
Nous proposons une approche un peu informelle de cette question.
Soit  $\phi \colon  X_1 \rightarrow X_2$  une fonction entre espaces métriques et  
$(F_{\alpha})_{\alpha \in M}$  une famille de parties de $X_1$.  {\em Un \mcu 
pour  $\phi$  près de chaque partie  $F_{\alpha}$ } est par définition une 
fonction   $\mu \colon M \times \NN_1 \rightarrow \NN_1$  vérifiant:	
\[ 
\forall \alpha \in M \; \forall x \in F_{\alpha} \; \forall x'\in X_1\; \; \; \left(d_{X_1}(x,x') 
\leq 1/2^{\mu(\alpha,n)} \Rightarrow d_{X_2}(\phi(x),\phi(x'))
 \leq 1/2^n \right).
\]

\begin{fdefinition}[définition générale mais un peu informelle pour ce qu'est une fonction continue et ce qu'est sa \com] \label{f241}
On considère deux espaces métriques complets  $X_1$  et  $X_2$  donnés 
dans  des \rps  de classe  $\ca'$:   $(Y_1,\delta_1,\eta_1)$  et  
$(Y_2,\delta_2,\eta_2).$  On suppose avoir spécifié un point  $y_0$ de 
$X_2$.\\ 
Supposons que nous ayons défini une famille  $F_{X_1} = (F_{\alpha})_{\alpha \in M}$  de parties de  $X_1$  avec la propriété suivante:\\
--- (2.4.1.1)   tout compact de  $X_1$  est contenu dans un des  $F_{\alpha}$. 
\\
On dira qu'une fonction  $\phi\colon  X_1 \rightarrow X_2$  est   {\em $F_{X_1}$-
\unicoz}   si elle est bornée sur chaque partie  $F_{\alpha}$  et si elle 
possède un \mcu près de chaque partie $F_{\alpha}$.\\  
Supposons maintenant en plus que la famille $(F_{\alpha})_{\alpha \in M}$  ait 
la propriété suivante:\\
--- (2.4.1.2)   l'ensemble d'indices  $M$  a une certaine ``structure de 
calculabilité''\\
On dira qu'une fonction  $\phi\colon  X_1 \rightarrow X_2$  est  {\em $F_{X_1}$-\unico 
de classe~$\ca$  }  si elle vérifie les propriétés suivantes:\\
--- une borne sur chaque partie  $F_{\alpha}$  peut être calculée en 
fonction de $\alpha$   dans la classe~$\ca$, c.-à-d. qu'on a une fonction de 
classe~$\ca$,     $ \beta \colon  M \rightarrow \NN_1$   vérifiant:  
\[
\forall \alpha \in M \; \forall x \in F_{\alpha} \; \;
 d_{X_2}(y_0, \phi (x)) \leq 2^{\beta(\alpha)},
\] 
--- un \mcu près de  $F_{\alpha}$  peut être calculé dans la classe  
$\ca$:  une fonction   $\mu \colon  M \times \NN_1 \rightarrow \NN_1$   dans la 
classe~$\ca$   vérifiant:
\[
\forall \alpha \in M \; \forall x \in F_{\alpha} \; \forall x' \in X_1\; \; \; \left(d_{X_1}(x,x') 
\leq 1/2^{\mu(\alpha,n)} \Rightarrow d_{X_2}(\phi(x),\phi(x'))
 \leq 1/2^n \right),
\]
--- la restriction de  $\phi$  à  $Y_1$ est calculable dans la classe~$\ca$.
\end{fdefinition}
Il s'agit manifestement d'une extension des définitions \ref{f221} et \ref{f231}.
Le caractère informel de la définition tient évidemment à ``la 
structure de calculabilité'' de  $M$.\\
A priori on voudrait prendre pour  $(F_{\alpha})_{\alpha \in M}$ une famille de 
parties suffisamment simple pour vérifier la condition (2.4.1.2) et 
suffisamment grande pour vérifier la condition (2.4.1.1).\\
Mais ces deux conditions tirent dans deux sens opposés.\\ 
Notez que la définition \ref{f241} est inspirée de la notion de fonction 
continue définie par D. Bridges dans \cite[Constructive functional analysis]{fBr}.

\subsection {Ouverts et fermés} \label{fsubsec25}
Nous passons à la définition d'une structure de calculabilité pour un 
sous-espace ouvert $U$ d'un espace métrique $X$ muni d'une structure de 
calculabilité. La métrique induite par $X$ sur $U$ ne saurait en général 
nous satisfaire car l'espace obtenu n'est généralement pas complet. On a 
néanmoins une construction qui fonctionne dans un cas particulier important.\\ 
Soient $X$ un espace métrique complet et   $f\colon  X \rightarrow \RR$  une 
fonction 
\loca \unico.  L'ouvert  $U_f = \{x \in X\,;\, f(x) > 0 \}$  est un espace 
métrique complet pour la distance  $d_f$  définie par:		
$$d_f(x,y) = d_X(x,y) + \abs{1/f(x) - 1/f(y)}$$

\begin{fpropdef} \label{f251}
Soient $X$ un espace métrique complet donné avec une~$\ca$-\pres $(Y, 
\delta, \eta)$, et   $f\colon  X \rightarrow \RR$  une fonction \loca \uni de classe  
$\ca$    représentée par un \mcu et une fonction discrète de classe  
$\ca$,
$ \varphi \colon  Y \times \NN_1 \rightarrow \DD$. Alors l'espace métrique complet  
$(U_f, d_f)$ peut être muni d'une~$\ca$-\pres où l'ensemble des points 
rationnels est (codé par) l'ensemble  $Y_U$ des couples  $(y,n)$  de  $Y 
\times  \NN_1$  vérifiant:		$$n \geq 10 \; \; \hbox{et} \; \; 
\varphi(y,n) \geq 1/2^{n/8}$$
\end{fpropdef}

\proof Une preuve détaillée est donnée dans \cite{fMo}.\eop

\medskip Nous attaquons maintenant la définition d'une structure de calculabilité 
pour un sous-espace fermé $F$ d'un espace métrique $X$ muni d'une structure 
de calculabilité. En analyse constructive, un sous-ensemble fermé n'est 
vraiment utile que lorsqu'il est situé, c.-à-d. lorsque la fonction  $D_F$  
distance au fermé est calculable. 
Au moment de traduire cette notion en termes de calculabilité récursive ou de \com, nous devons prendre garde que dans la définition constructive, le fait que la fonction  $D_F$  est la fonction distance au fermé doit être aussi rendu explicite.

\begin{fdefinition} \label{f252}
On considère un espace métrique complet $X$ donné par une~$\ca$-\pres  
$(Y, \delta, \eta)$.  Une partie  $F$  de $X$ est appelée {\em un fermé 
$\ca$-situé de $X$}  si:
\begin{itemize}
\item [i)]	la fonction  $D_F\colon   x \mapsto d_X(x,F)$   de $X$ vers  $\RR^{\geq 0}$   est 
calculable dans la classe~$\ca$,
\item [ii)]	il existe une fonction calculable dans la classe~$\ca$ :    $P_F \colon  Y 
\times \NN_1\rightarrow X$ qui certifie la fonction  $D_F$  au sens suivant: pour tout   $(y,n) \in  Y \times  \NN_1$,
\[
D_F(P_F(y,n)) = 0,\; \; \; \;  d_X(y,P_F(y,n)) \leq D_F(y) + 1/2^n 
\] 
%
\end{itemize}
\end{fdefinition}

\begin{fremarks}\label{f253}~\\
1)	La fonction  $P_F(y,n)$  calcule un élément de  $F$  dont la distance 
à  $y$  est suffisamment proche de  $D_F(y)$. Cependant, la fonction  
$P_F(y,n)$  ne définit pas en général, par prolongement par continuité~à~$X$ et par passage à la limite lorsque  $n$  tend vers  $\infty$, un 
projecteur sur le fermé $F$. \\
2)	On peut démontrer que les points  $P_F(y,n)$  (codés par les couples  $(y,n) 
\in  Y \times  \NN_1$)  forment une partie dénombrable dense de  $F$  qui est 
l'ensemble des points rationnels d'une~$\ca$-\pres de  $F$  (cf. \cite{fMo}).
\end{fremarks}

\subsection{Espaces de Banach rationnellement présentés}\label{fsubsec26}
Nous donnons une définition minimale. Il va de soi que pour chaque espace de Banach particulier, des notions de \com naturellement attachées à l'espace considéré peuvent éventuellement être prises en compte en plus pour obtenir une notion vraiment raisonnable. 

\begin{fdefinition}[\rp  d'un espace de Banach]\label{f261} 
Une \rp  d'un espace de Banach séparable $X$ sur le corps  $\KK$  ($\RR$  ou $\CC$) sera dite de classe~$\ca$ si, d'une part elle est de classe~$\ca$  en tant que \pres de l'espace métrique et, d'autre part les opérations 
d'espace vectoriel suivantes sont dans la classe~$\ca$\\  
--- le produit par un scalaire,\\ 
--- la somme d'une liste de vecteurs choisis parmi les points rationnels.
\end{fdefinition}

Dans le contexte du corps des complexes  $\CC$  nous désignerons par  
$\DD_\KK$  l'ensemble des complexes dont les parties réelle et imaginaire 
sont dans  $\DD$. Dans le contexte réel  $\DD_\KK$  sera seulement une 
autre dénomination de  $\DD$.\\
La proposition suivante n'est pas difficile à établir. 

\begin{fproposition} \label{f262}
Donner une \rp  de classe~$\ca$   d'un espace de Banach $X$ au sens de la 
définition ci-dessus revient à donner: \\
--- un codage de classe~$\ca$   pour  une partie dénombrable $G$ de $X$ qui 
engendre $X$ en tant qu'espace de Banach{\footnote{On peut supposer que tous les éléments de $G$ sont des vecteurs de norme comprise entre 1/2 et 1.}},\\    
--- une application   $\nu \colon  \lst(\DD_{\KK} \times G) \times \NN_1 
\rightarrow
 \DD$  qui est dans la classe~$\ca$   (pour le codage de $G$ considéré)  et 
qui calcule la norme d'une combinaison linéaire d'éléments de  $G$  au sens suivant:\\   	
pour tout $\big([(x_1,g_1),(x_2,g_2),\ldots,(x_n,g_n)],m\big)$ dans  $\lst(\DD_{\KK} \times G) \times \NN_1$, on a la majoration
\[
\Abs{\;\nu\big([(x_1,g_1),(x_2,g_2),\ldots,(x_n,g_n)],m\big) - \Norme{ x_1.g_1 + x_2.g_2 +\cdots.+ x_n.g_n }_X \; }\; \leq 1/2^m.
\]
\end{fproposition}

\begin{fremarks} \label{f263}~\\
1)	Comme pour la définition \ref{f211}  (cf. remarque \ref{f212} (3)), il se 
peut que le certificat d'inclusion de $G$ dans $X$ implique une notion de 
\com, qu'il sera inévitable de prendre en compte dans une définition plus 
précise, au cas par cas.\\
2)	Les espaces  $L^p(\RR)$  de l'analyse fonctionnelle, avec  $1 \leq p < 
\infty$  peuvent être rationnellement présentés de différentes 
manières, selon différents choix possibles pour l'ensemble des points 
rationnels et un codage de cet ensemble. Tous les choix raisonnables 
s'avèrent donner des présentations  $\Prim$-équivalentes. \\
3)	Il serait intéressant de savoir si le cadre de travail proposé par M. 
Pour El et I. Richards \cite{fPR} concernant la calculabilité dans les espaces 
de Banach peut avoir des conséquences concrètes qui iraient au delà de ce 
qui peut être traité par les \rps  (lesquelles offrent un cadre naturel non seulement pour les problèmes de récursivité mais aussi pour les 
problèmes de \comz). Le ``contre-exemple'' concernant  $L^{\infty}$  donné 
dans \cite{fPR} incite à penser le contraire. 
\end{fremarks}

\section[Fonctions réelles continues sur un intervalle compact 
\ldots]{Espace des fonctions réelles continues sur un intervalle compact, 
premières propriétés}\label{fsec3}
Dans cette section, nous introduisons le problème des \rps  pour l'espace  
$\czu$:  l'espace des fonctions réelles \unicos sur l'intervalle  [0,1], muni 
de la norme usuelle:
\[
\norme{f}_{\infty} = {\bf Sup}\{ \abs{f(x)} ; 0 \leq x \leq 1 \}.
\] 
Considérons une \rp  de l'espace  $\czu$  donnée par une famille  
$\big(\wi{f}\,\big)_{f \in Y}$  de fonctions \unicos indexée par une partie  $Y$  
d'un langage $A^{\star}$. Nous sommes intéressés par les problèmes de 
\com suivants:\\
--- la \com de l'ensemble des codes des points rationnels de la 
présentation, c.-à-d. plus précisément la \com de  $Y$  en tant que 
partie du langage  $A^{\star}$   (c.-à-d. la \com du test d'appartenance);\\
--- la \com des opérations d'espace vectoriel (le produit par un scalaire 
d'une part, la somme d'une liste de vecteurs d'autre part);\\
--- la \com du calcul de la norme (ou de la fonction distance);\\ 
--- la \com de l'ensemble  $\big(\wi{f}\,\big)_{f \in Y}$  des points rationnels de 
la présentation, en tant que famille de fonctions \unicos sur  $[0,1]$;\\
--- la \com de la fonction d'évaluation  $\Ev \colon  \czu \times [0,1] 
\rightarrow \RR$: $ (g,x) \mapsto g(x)$.\\
Il va de soi que l'on peut remplacer l'intervalle  $[0,1]$  par un autre 
intervalle  $[a,b]$  avec  $a$  et  $b$  dans~$\DD$  ou de faible \com dans  
$\RR$. 

\subsection{La définition d'une \rp  de l'espace des fonctions 
continues}\label{fsubsec31} 
La \com de l'ensemble  $\big(\wi{f}\,\big)_{f \in Y}$  des points rationnels de la 
présentation, en tant que famille de fonctions \unicos sur  $[0,1]$    n'est 
rien d'autre que la \com de l'application  $f \mapsto \wi{f}$    de 
l'ensemble des codes de points rationnels $Y$ vers l'espace  $\czu$.  Nous 
devons donc, conformément à ce que nous disions dans la remarque 
\ref{f212}(3), inclure dans la définition de ce qu'est une \rp  de classe  
$\ca$   de  $\czu$, le fait que  $\big(\wi{f}\,\big)_{f \in Y}$  est une famille 
\uni de classe~$\ca$   au sens de la définition \ref{f229}.
Le problème de la \com de la fonction d'évaluation est également un 
problème important car il serait ``immoral''  que la fonction d'évaluation ne 
soit pas une fonction de classe~$\ca$   lorsqu'on a une \rp  de classe~$\ca$. 
Cependant, la fonction d'évaluation n'est pas \unico, ni même 
\loca \unico.\\
Pour traiter en général la question des fonctions continues mais 
non \loca \unicos sur l'espace  $\czu$  nous faisons appel à la définition 
informelle \ref{f241} avec la famille suivante de parties de  $\czu$:  

\begin{fnotation}  \label{f311}
Si  $\alpha$  est une fonction croissante de  $\NN_1$ vers  $\NN_1$  et  $r 
\in \NN_1$,  on note  $F_{\alpha,r}$  la partie de  $\czu$  formée par toutes 
les fonctions qui, d'une part acceptent  $\alpha$   comme \mcu, et d'autre part 
ont leur norme majorée par  $2^r$.
\end{fnotation}

La ``structure de calculabilité''  sur l'ensemble d'indices  
$$M := \{ (\alpha,r); \alpha \; \hbox{est une  fonction croissante de} \;
 \NN_1 \;\hbox{vers} \; \NN_1 \; \hbox{et} \; r \in \NN_1 \}$$
n'est pas une chose bien définie dans la littérature, mais nous n'aurons 
besoin de faire appel qu'à des opérations parfaitement élémentaires 
comme ``évaluer  $\alpha$  en un entier  $n$''{\footnote{Notez que le 
théorème d'Ascoli classique affirme que toute partie compacte de $\czu$ 
est contenue dans une partie $F_{\alpha,r}$ , et que les parties  $F_{\alpha,r}$   
sont compactes. En mathématiques constructives la partie directe est encore 
valable, mais la deuxième partie de l'énoncé doit être raffinée, cf 
\cite{fBB} chap 4 théorème 4.8, pages 96 à 98.}}.
Le module de continuité de la fonction d'évaluation est alors très 
simple (\uni linéaire pour toute définition raisonnable de cette notion).\\
En effet, près de la partie  $F_{\alpha, r} \times [0,1]$    de  $\czu \times 
[0,1]$   un module de continuité de la fonction $\Ev$ est donné par:
$$\mu(n,\alpha,r) = \max (\alpha(n+1),n+1) \; \hbox{pour} \; n \in \NN_1 \; 
\hbox{et} \; (\alpha,r) \in M
$$                             
comme il est très facile de le vérifier. Et la borne sur  $F_{\alpha,r}$  
est évidemment donnée par  $\beta(\alpha,r) = r$.\\
Toute la question de la \com de la fonction d'évaluation dans une \pres 
donnée est donc concentrée sur la question de la \com de la fonction 
d'évaluation restreinte à l'ensemble des points rationnels  
$$(f,x) \mapsto \wi{f}(x) \qquad  Y \times \DD_{[0,1]} \rightarrow \RR.$$
Or cette \com est subordonnée à celle de  
$\big(\wi{f}\,\big)_{f \in Y}$  en tant que famille de fonctions \unicosz: c'est 
ce que nous précisons dans la proposition suivante (dont la démonstration est immédiate).

\begin{fproposition} \label{f312}
Considérons sur l'espace  $\czu$  la famille de parties  
$(F_{\alpha,r})_{(\alpha,r) \in M}$  pour contrôler les questions de 
continuité sur  $\czu$  (cf. notation \ref{f311} et définition \ref{f241}).\\  
Alors si   $\big(\wi{f}\,\big)_{f \in Y}$  est une famille \uni de classe~$\ca$    
et si on considère la \rp  de l'espace métrique  $\czu$  attachée à 
cette famille considérée comme ensemble des points rationnels de la 
présentation, la fonction d'évaluation      
$$\Ev \colon  \czu \times [0,1] \rightarrow \RR \colon  (g,x) \mapsto g(x)$$
est elle même de classe~$\ca$.
\end{fproposition}

Comme en outre nous demandons que la structure d'espace de Banach soit elle-même de classe~$\ca$, cela nous donne finalement la définition suivante.

\begin{fdefinition} \label{f313}
Une \rp  de classe~$\ca$   de l'espace  $\czu$  est donnée par une famille de 
fonctions  $\big(\wi{f}\,\big)_{f \in Y}$  qui est une famille \uni de classe  
$\ca$, dense dans $\czu$  et telle que soient également dans la classe~$\ca$   
les calculs suivants:\\
--- le produit par un scalaire,\\  
--- la somme d'une liste de fonctions choisies parmi les points rationnels,\\
--- le calcul de la norme.
\end{fdefinition}

Dans toute la suite, nous nous intéresserons à une étude précise des 
complexités impliquées dans la définition \ref{f313}. Notre conclusion 
est qu'il n'existe pas de paradis \etpo des fonctions continues, du moins si  
$\p \not= \np$.

\subsection{Deux exemples significatifs de \rps  de l'espace 
\texorpdfstring{$\czu$}{C[0,1]}}\label{fsubsec32}

Nous donnons maintenant deux exemples significatifs de présentations de  
$\czu$ (d'autres exemples seront donnés plus loin)   

\subsubsection{Présentation par circuits semilinéaires binaires} 
\label{fsubsubsec321}  
Cette \pres et l'ensemble des codes des points rationnels seront notés 
respectivement  $\csl$  et  $\ysl$.  Nous appellerons {\em  fonction 
semilinéaire à coefficients dans  $\DD$ }  une fonction linéaire par 
morceaux qui est égale à une combinaison par  $\max$  et  $\min$  de 
fonctions  $x \mapsto ax+b$  avec  $a$  et  $b$  dans  $\DD$.

\begin{fdefinition} \label{f321}
Un {\em  circuit semilinéaire binaire} est un circuit qui a pour portes 
d'entrée des variables réelles  $x_i $  (ici, une seule suffira parce que 
le circuit calcule une fonction d'une seule variable) et les deux constantes  0  
et  1. Il y a une seule porte de sortie.\\
Les portes qui ne sont pas des portes d'entrée sont de l'un des types 
suivants:\\
--- des portes à une entrée, des types suivants:  $x \mapsto 2x$ ,$x 
\mapsto x/2$, $x \mapsto -x$\\ 
--- des portes à deux entrées, des types suivants:  
$(x,y) \mapsto x + y$, $(x,y) \mapsto \max(x,y)$, $(x,y) \mapsto \min(x,y)$.\\ 
Un circuit semilinéaire binaire avec une seule variable d'entrée définit 
une fonction semilinéaire à coefficients dans  $\DD$. Un tel circuit peut 
être codé par un programme d'évaluation. L'ensemble  $\ysl$   est 
l'ensemble (des codes) de ces circuits semilinéaires: ses éléments codent les 
points rationnels de la \pres  $\csl$.
\end{fdefinition}

Nous verrons plus loin que cette \pres est en quelque sorte la plus naturelle, 
mais qu'elle manque d'être une \pres de classe  $\p$   à cause du calcul de 
la norme.\\
On a une majoration facile d'un module de Lipschitz de la fonction définie par le circuit:  
\[
\abs{\wi{f}(x) - \wi{f}(y)}\;  \leq 2^p \abs{ x- y} \; \hbox{où} \; p \;  \hbox{est la profondeur du circuit}
\]
Ceci donne pour \mcu  $\mu(k) = k+n$.  Ceci implique en particulier qu'on n'a 
pas besoin de contrôler la précision dans  $\DD_{[0,1]}$   dans la 
proposition suivante.  

\begin{fproposition}[complexité de la famille de fonctions  $\big(\wi{f}\,\big)_{f \in \ysl }$] \label{f322}
~ \\
La famille de fonctions  $\big(\wi{f}\,\big)_{f \in \ysl }$  est  \uni de classe  
$\p$. Précisément, cette famille admet  
$\mu(f,k) = k+\prof(f)$
  pour  \mcu et on peut expliciter une fonction  $\varphi \colon  \ysl  \times \DD_{[0,1]} 
\times \NN_1 \rightarrow \DD_{[0,1]}$   de classe $\DRT(\Lin, \Oo(N^2))$  (où $N$ est 
la taille de l'entrée  $(f,x,k)$)  avec
\[
 \forall (f, x, k) \in \ysl  \times 
\DD_{[0,1]} \times \NN_1 \;
\abs{\wi{f}(x) - \varphi(f,x,k)}\;  \leq 1/2^k.
\]
Plus précisément encore, comme il n'est pas nécessaire de lire $x$ en 
entier,  la taille de $x$ n'intervient pas, et la fonction  $\varphi$   est dans 
les classes  $\DRT(\Lin, \Oo(\ta(f)(\prof(f)+k)))$  et  $\DSPA(\Oo(\prof(f) 
(\prof(f)+k)))$.
\end{fproposition}

\proof 
Pour calculer  $\varphi (f,x,k)$  on évalue le circuit $f$ sur l'entrée $x$ 
dont on ne considère que les  $k + 2 \prof(f)$  premiers bits, en tronquant le 
résultat intermédiaire calculé à la porte  $\pi$ à la précision  
$k+2\prof(f)-\prof(\pi)$.  Enfin, pour le résultat final on ne garde que la 
précision  $k$.\\
Une telle méthode appliquée naïvement nécessite de garder stockés 
tous les résultats obtenus à une profondeur fixée  $p$  pendant qu'on 
calcule des résultats à la profondeur  $p+1$.
Nous faisons alors  $\ta(f)$  calculs élémentaires  $(\bullet + \bullet, 
\bullet - \bullet,\bullet \times 2, \bullet/2, \max(\bullet, \bullet),
\min(\bullet, \bullet))$  
sur des objets de taille $\leq  k + 2 \prof(f)$.  Chaque calcul élémentaire 
prend un temps $\Oo(k+\prof(f))$, et donc le calcul global se fait en temps  
$\Oo(\ta(f)(\prof(f)+k))$.  Et cela prend aussi un $\Oo(\ta(f)(\prof(f)+k))$ comme 
espace de calcul.\\  
Il existe une autre méthode d'évaluation d'un circuit, un peu moins 
économe en temps mais nettement plus économe en espace, suivant l'idée de 
Borodin \cite{fBo}. Avec une telle méthode on économise l'espace de calcul 
qui devient un  $\Oo(\prof(f)(\prof(f)+k))$.
\eop

\begin{fremark} \label{f323}
Nous n'avons pas pris en compte dans notre calcul le problème posé par la 
gestion de  $t = \ta(f)$  objets (ici des nombres dyadiques) de tailles 
majorées par  $s = \prof(f)+k$. Dans le modèle RAM cette gestion serait a 
priori en temps   $\Oo(t \Lg(t) s)$  ce qui n'augmente pas 
sensiblement le  $\Oo(t s)$  que nous avons trouvé, et ce qui reste en  
$\Oo(N^2)$  si on se rappelle que le codage du circuit semilinéaire par un 
programme d'évaluation lui donne une taille de $\Oo(t \Lg(t))$. Dans le 
modèle des machines de Turing par contre, cette gestion réclame a priori un 
temps  $\Oo(t^2s)$  car il faut parcourir~$t$ fois la bande où sont stockés 
les objets sur une longueur totale  $\leq  ts$. Nous avons donc commis une 
certaine sous-estimation en nous concentrant sur le problème que nous 
considérons comme central: estimer le coût total des 
opérations arithmétiques proprement dites. Nous omettrons dans la suite 
systématiquement le calcul du {\em temps de gestion} des valeurs 
intermédiaires (très sensible au modèle de calcul choisi) chaque fois 
qu'il s'agira d'évaluer des circuits.
\end{fremark}

\subsubsection{Présentation \texorpdfstring{$\crf$}{Cfrac}
(via des fractions rationnelles contrôlées et données en \pres par formule)} \label{fsubsubsec322}
La \pres précédente de l'espace  $\czu$  n'est pas de 
classe  $\p$  (sauf si $\p$ = $\np$ comme nous le verrons à la section 4.3)  parce que la norme n'est pas calculable en temps \poll.  
Pour obtenir une \pres de classe $\p$ il est nécessaire de restreindre assez considérablement l'ensemble des points rationnels de la présentation, de manière à ce que la norme devienne une fonction calculable en temps \poll. 
Un exemple significatif est celui où les points rationnels sont des fractions rationnelles bien contrôlées et données dans une \pres du 
type dense.
On a le choix entre plusieurs variantes et nous avons choisi de donner le 
dénominateur et le numérateur dans une \pres dite ``par formules''.  Une 
formule est un arbre dont les feuilles sont étiquetées par la variable $X$ 
ou par un élément de  $\DD$  et dont chaque noeud est étiqueté par un 
opérateur arithmétique.  
Dans les formules que nous considérons, les seuls opérateurs utilisés  sont $\bullet + \bullet, \bullet - \bullet$ et $\bullet \times  \bullet,$ de sorte que l'arbre est un arbre binaire (chaque noeud de l'arbre est une sous-formule et représente un \pol de $\DD$[X].)

\begin{flemma} 
\label{f324} Désignons par $\DD[X]_f$  l'ensemble des \pols à 
coefficients dans $\DD$, donnés en \pres par formule.  Pour un \pol à 
une variable et à coefficients dans  $\DD$   le passage de la re\pres dense 
à la re\pres par formule est  $\LINT $  et le passage de la re\pres par 
formule à la re\pres dense est \poll. Plus précisément si on 
procède de manière naïve on est en  $\DTI (\Oo(N^2{\cal M}(N)))$).
\end{flemma}

\proof  Tout d'abord, la re\pres dense peut être considérée comme un cas 
particulier de re\pres par formule, selon le schéma de Horner.\\ 
Ensuite, passer de la re\pres par formule à la re\pres dense revient à 
évaluer la formule dans  $\DD[X]$. Introduisons les paramètres de controle 
suivants. Un \pol $P \in \DD[X]$   a un degré noté  $d_P$  et la taille 
de ses coefficients est controlée par l'entier $\sigma(P) := \log(\sum_i \abs{a_i})$ où les  $a_i$  sont les coefficients de  $P$.  Une formule  $F 
\in \DD[X]_f$  contient un nombre d'opérateurs arithmétiques noté  $t_F$ 
et la taille de ses coefficients est controlée par l'entier $\lambda(F) = 
\sum_i (\Lg(b_i))$   où les  $b_i$ sont les dyadiques apparaissant dans la 
formule.\\  
La taille  $\flo{F}$ de la formule $F$ est évidemment un majorant de   
$t_F$ et  $\lambda (F)$.\\
Pour deux dyadiques  $a$  et  $b$  on a toujours  $\Lg(a \pm b) \leq \Lg(a) + 
\Lg(b)$  et $ \Lg(ab) \leq  \Lg(a)+\Lg(b)$.\\
On vérifie alors facilement que $\sigma (P \pm Q) \leq \sigma(P) + \sigma(Q)$   
et  $\sigma (PQ) \leq \sigma (P) + \sigma (Q)$.  Le temps de calcul (naïf) 
de  $PQ$  est un  $\Oo(d_Pd_Q {\cal M}(\sigma(P) + \sigma(Q)))$.\\
On démontre ensuite par récurrence sur la taille de la formule $F$ que le 
\pol correspondant  $P \in \DD[X]$  vérifie $d_P \leq t_F$ et  $\sigma 
(P) \leq \lambda (F)$. On démontre également par récurrence que le temps pour 
calculer  $P$  à partir de $F$ est majoré par  $t_F^2{ \cal M}(\lambda(F))$. 
\eop

\begin{fdefinition} \label{f325}
L'ensemble  $\yrf  \subset \DD[X]_f \times \DD[X]_f$  est l'ensemble des 
fractions rationnelles (à une variable) à coefficients dans  $\DD$,  dont 
le dénominateur est minoré par  $1$  sur l'intervalle $[0,1]$.  L'espace  
$\czu$  muni de l'ensemble  $\yrf$  comme famille des codes des points 
rationnels est noté~$\crf$.
\end{fdefinition} 

\begin {fproposition}[complexité de la famille de fonctions  $\big(\wi{f}\,\big)_{f \in \yrf}$] \label{f326}
~ \\
La famille de fonctions  $\big(\wi{f}\,\big)_{f \in \yrf }$   est  \uni de classe  
$\p$, plus précisément de classe  
$\DRT(\Lin, \Oo({\M}(N)N))$, où $\M(N)$ est la \com de la multiplication de deux entiers de taille~$N$.
\end{fproposition}

\proof Nous devons calculer un \mcu pour la famille $\big(\wi{f}\,\big)_{f \in \yrf}$.  Nous devons aussi expliciter une fonction $\varphi$ de classe  $\DRT(\Lin,\Oo(\M(N) N))$
\[
\varphi \colon  \yrf  \times \DD_{[0,1]} \times \NN_1 \rightarrow \DD \; ,\;
(f,x,n) \mapsto \varphi(f,x,n)
\]
vérifiant   $\abs{\wi{f}(x) - \varphi(f,x,n)}\;  \leq 1/2^n$.\\
En fait, le \mcu va se déduire du calcul de  $\varphi $.\\ 
Puisque le dénominateur de la fraction est minoré par  $1$, il nous suffit 
de donner un \mcu et une procédure de calcul en temps  $\Oo(\M(N) N)$ pour 
évaluer une formule  $F \in \DD[X]_f$  avec une précision  $1/2^n$  sur 
l'intervalle $[0,1]$   ($N = n+ \flo{F}$).  Nous supposons sans perte de 
généralité que la taille  $m = \flo{F}$  de la formule  $F = F_1 * 
F_2$  est égale à $ m_1+m_2+2$  si  $m_1 = \flo{F_1}$  et  $m_2 = 
\flo{F_2}$   ( $*$  désigne un des opérateurs $+, \; - \;$ ou $\;\times$).  
On établit alors (par récurrence sur la profondeur de la formule) les deux 
faits suivants:\\ 
--- lorsqu'on évalue de manière exacte la formule  $F \in \DD[X]_f$  en un  
$x \in [0,1]$    le résultat est toujours majoré en valeur absolue par  
$2^m$. \\
--- lorsqu'on évalue la formule $F$ en un  $x \in [0,1]$    de manière 
approchée, en prenant $x$ et  tous les résultats intermédiaires avec une 
précision (absolue)  $1/2^{n+m}$  le résultat final est garanti avec la 
précision  $1/2^n$. \\
On conclut ensuite sans difficulté.
 \eop
 
\begin{fproposition} \label{f327}~\\
a) La famille de nombres réels  
$(\Norme{ \wi{f}})_{f \in \yrf }$  est de \com  $\p$.\\
b) Le test d'appartenance  ``$\,f \in\yrf\; ?\,$''  est de \com  $\p$.
\end{fproposition}
\proof 
a) Soit  $f = P/Q \in \yrf $.  Pour calculer une valeur approchée à  $1/2^n$  
de la norme de  $\wi{f}$, on procède comme suit:\\
--- calculer  $(P'Q - Q'P)$  en tant qu'élément de  $\DD[X]$  (lemme 
\ref{f324})\\
--- calculer  $m$: la précision requise sur $x$ pour pouvoir évaluer  
$\wi{f} (x)$   avec la précision  $1/2^n$  (proposition \ref{f326})\\
--- calculer les racines   $(\alpha_i)_{1 \leq i \leq n})$  de  $(P'Q - Q'P)$   
sur  $[0,1]$    avec la précision  $2^{-m}$ \\
--- calculer  $\max  \{\wi{f}(0);\wi{f}(1);\wi{f}(\alpha_i)
 \; 1 \leq i \leq s \}$ avec la précision  $1/2^n$ (proposition \ref{f326})\\
b) Le code $f$ contient les codes de $P$ et $Q$. Il s'agit de voir qu'on peut tester en temps \poll que le dénominateur $Q$ est minoré par $1$ sur 
l'intervalle $[0,1]$. 
Ceci est un résultat classique concernant les calculs avec les nombres réels algébriques: il s'agit de comparer à $1$ le inf des $Q(\alpha_i)$ avec $\alpha_i=0,1$ ou un zéro de $Q'$ sur l'intervalle. 
\eop

\begin{fremarks} \label{f328}~\\
1)	Il est bien connu que le calcul des racines réelles d'un \pol de  
$\ZZ[X]$  situées dans un intervalle rationnel donné est un calcul de 
classe  $\p$. Peut-être la méthode la plus performante n'est pas celle que 
nous avons indiquée, mais une légère variante. 
En effet la recherche des 
racines complexes d'un \pol (pour une précision donnée) est aujourd'hui 
extrêmement rapide (cf. \cite {fPa}). 
Plutôt que de chercher spécifiquement les zéros réels sur l'intervalle $[0,1]$ on pourrait donc chercher avec la précision  $2^{-m}$  les zéros réels ou complexes  $(\beta_j)_{1 \leq j \leq t}$ suffisamment proches de l'intervalle   $[0,1]$    (i.e. leur partie imaginaire est en valeur absolue  $\leq2^{-m}$ et leur partie réelle est sur  $[0,1]$  à   $2^{-m}$  près) et évaluer $\wi{f}$ en les  $Re(\beta_j)$.\\    
2) Comme cela résulte de la proposition \ref{f335} ci-dessous,  toute 
fonction semilinéaire à coefficients dans $\DD$  est un point de \com  
$\p$ dans $\crf$. 
Le fait que la \pres $\csl$ ne soit pas de classe  $\p$ (si $\p\neq\np$, cf. section \ref{fsubsec43}) implique par contre que la famille de fonctions $\big(\wi{f}\,\big)_{f \in \ysl }$ n'est pas une famille de classe $\p$   dans $\crf$.\\
3) On démontre facilement que les opérations d'espace vectoriel sont aussi en temps \poll. 
\end{fremarks}

Les résultats précédents se résument comme suit.

\begin{ftheorem} \label{f329}
La \pres  $\crf$  de  $\czu$  est de classe  $\p$.
\end{ftheorem}

\subsection{Le théorème d'approximation de Newman et sa \com 
algorithmique}\label{fsubsec33}
Le théorème de Newman est un théorème fondamental en théorie de 
l'approximation. L'énoncé ci-après est un cas particulier{\footnote{Nous 
avons pris la borne  $n^2/2$ sur les degrés de manière à obtenir une 
majoration en  $e^{-n}.$}}.

\begin{ftheorem}[Théorème de Newman, \cite{fNe}, voir par exemple \cite{fPP} p. 73--75] \label{f331}~  

\noindent   
Soit  $n$  entier  $\geq  6$,   définissons
\[
H_n(x) = \prod_{ 1 \leq k < n^2}(x + e^{-k/n}).
\]
et considérons les deux \pols  $P_n(x)$  et 
  $Q_n(x)$, de degrés majorés par  $n^2/2$,  donnés par
\[
P_n(x^2) = x(H_n(x) - H_n(-x)) \; \; \hbox{et} \; \; Q_n(x^2) = H_n(x) + H_n(-x)
\]
Alors on a pour tout  $x \in  [-1,1]$  
\[
\abS{ \; \abs{x}\; - \;(P_n/Q_n)(x^2) \; } \; \leq 3e^{-n}\; \leq 2^{-(n+1)}
\]
et 
\[\abS{Q_n(x^2)}\; \geq \;2H_n(0) = 2/e^{(n^3 - n)/2} \;\geq 1/2^{3(n^3 - 
n)/4} 
\]
\end{ftheorem}

Du théorème de Newman découle que les fonctions semilinéaires se 
laissent individuellement bien approcher par des fractions rationnelles 
faciles à écrire et bien contrôlées. C'est ce que précisent le 
lemme \ref{f333}  ci-après et ses corollaires, les propositions et 
théorèmes qui suivent. Nous rappelons d'abord le résultat suivant (cf. 
\cite{fBrent}).
\begin{flemma}[théorème de Brent] \label{f332}  
 Soit   $a \in  [-1,1]\cap \DD_m$. Le calcul de  
$\exp(a)$  avec la précision~$2^{-m}$ peut être fait en temps $\Oo(\M(m)\log(m))$.   
En d'autres termes, la fonction exponentielle sur l'intervalle $[-1,1]$ est de \com $\DTI(\Oo(\M(m)\log(m)))$   
\end{flemma} 

On en déduit facilement, en notant $\DD[X]_f$ la présentation de $\DD[X]$ 
par formules.

\begin{flemma} \label{f333}
Il existe une suite 
\[
\NN_1 \rightarrow \DD[X]_f \times \DD[X]_f \; ; \; n \mapsto (u_n, v_n),
\] 
de classe $\DRT(\Oo(n^5), \Oo({\M}(n^3) n^2 \log^3(n)))$, telle que pour tout   $x \in [0,1]$      
\[
\abs{\;\abs{x} - p_n(x^2)/q_n(x^2) \;} \;\;\leq \;2^{-n}.
\]
Les degrés des \pols $p_n$ et $q_n$ sont majorés par $n^2/2$, leurs tailles en présentation par formule sont majorées par un $\Oo(n^5)$,  
et $q_n(x^2)$ est minoré par  $1/e^{(n^3-n)}$.
\end{flemma} 

\proof On définit $p_n$ et  $q_n$  comme   $P_n$  et  $Q_n$  en rempla\c{c}ant 
dans la définition le réel  $e^{-k/n}$ par une approximation  dyadique  
$c_{n,k}$  suffisante calculée au moyen du lemme \ref{f332}. \\ 
Si  $\abs{e^{-k/n}-c_{n,k}}\;\leq \varepsilon$ on vérifie que 
pour tout   $x\in [0,1]$  
\[
\abs{P_n(x) - p_n(x)}\;  \leq  (n^{2}-1) 2^{n^2} \varepsilon  \quad  
\hbox{et}  \quad  \abs{Q_n(x)- q_n(x)}\; \leq  (n^2-1) 2^{n^2} \varepsilon  = \varepsilon_1. 
\]
Pour l'écart entre les fractions rationnelles on utilise
\[
\abs{A/B - a/b};  \leq \;  
\abs{A/B} \;  \abs{b - B}  / b + \; \abs{A - a}  / b 
\leq   3 \varepsilon_1/ b. 
\]
Comme $ B \geq  1/2^{3(n^3-n)/4} $ on a  $1/b\leq2.2^{3(n^3-n)/4}$ si  
$\abs{b - B}\;\leq  B/2$, en particulier si  
\[
(n^{2}-1) 2^{n^2} \varepsilon \leq (1/2) \cdot 1/2^{3(n^3-n)/4}.
\]
On est alors conduit à prendre un  $\varepsilon $  tel que   
\[
3\varepsilon_1/b\leq 6(n^{2}-1)2^{n^2}2^{3(n^3-n)/4}\varepsilon 
\leq 1/2^{n+1}.
\]
Il suffit donc de prendre  $\varepsilon  \leq  2 ^{-n^3}$ (pour  $n$  assez 
grand).\\
On va donc être amené à décrire   $p_n$  et  $q_n$  par des formules de 
taille  algébrique   $\Oo(n^2)$   portant sur des termes de base  $(x+c_{n,k})$  
où  $c_{n,k}$  est un dyadique de taille  $\Oo(n^{3}).$ 
La taille (booléenne) de la formule est donc un  $\Oo(n^{5}).$ 
L'essentiel du temps de calcul est absorbé par le calcul des  $c_{n,k}$  en 
utilisant le lemme \ref{f332}.  
\eop

Notez que la majoration du temps de calcul est à peine moins bonne que la 
taille du résultat.

On en déduit immédiatement les résultats suivants.
\begin{ftheorem} \label{f334}
La fonction   $x \mapsto \abs{ x - 1/2 }$ est un point de \com $\p$   (plus 
précisément  $\DRT(\Oo(n^5), \Oo({\M}(n^3) n^2 \log^3(n)))$   
dans l'espace  $\crf$. 
En d'autres termes, il existe une suite  
$\NN_1 \rightarrow \yrf  \; , \; 
n \mapsto (u_n,v_n)$, de classe
 $\DRT(\Oo(n^5), \Oo({\M}(n^3) n^2 \log^3(n)))$, telle que         
\[
\Norme{ \; \abs{ x-1/2 } - u_n(x)/v_n(x)\; } \; \leq 2^{-n}
\]
Les degrés des \pols  $u_n$  et  $v_n$ sont majorés par   $n^2$.
\end{ftheorem}
\begin{fproposition} \label{f335}
La fonction  $x \mapsto \abs{x}$ sur l'intervalle  $[-2^m , 2^m]$  peut 
être approchée à  $1/2^n$   près  par une fraction rationnelle  
$p_{n,m}/q_{n,m}$  dont le dénominateur est minoré par  $1$  (sur le même 
intervalle),  et le calcul  
$$(n,m) \mapsto (p_{n,m},q_{n,m}) \; \; \NN_1 \times \NN_1 \rightarrow
 \DD[X]_f \times \DD[X]_f $$
est de \com  $\DRT(\Oo(N^5), \Oo({\M}(N^3) N^2 \log^3(N)))$  
où  $N = n+m$.  Les degrés des \pols  
$p_{n,m}$  et  $q_{n,m}$ sont majorés par   $N^2$.
De même la fonction $(x,y) \mapsto \max(x,y)$ 
(resp. $(x,y) \mapsto \min(x,y)$)  sur le carré  
$[-2^m,2^m] \times  [-2^m,2^m]$  peut être approchée à  $1/2^n$   
près  par une fraction rationnelle du même type et de même \com 
que les précédentes.
\end{fproposition}

\proof 	Pour la fonction valeur absolue:\\
Si   $x \in [-2^m , 2^m]$  on écrit   
\[
\abs{x}\;  = 2^m \abs{ x/2^m } 
\]
avec  $x/2^m \in [-1 , 1]$. Donc, avec les notations de la preuve du lemme 
\ref{f333}, il suffit de prendre    $p_{n,m}(x) = 2^m p_{n+m}(x/2^m)$  et   
$q_{n,m}(x) = 2^m q_{n+m}(x/2^m)$.\\
Pour les fonctions $\max$ et $\min$, il suffit d'utiliser les formules:	
\[
\max(x,y) = \frac{x+y\; + \abs{x-y}} {2}\quad  \hbox{et} \quad  
\min (x,y) = \frac{x+y\; - \abs{x-y}}{2} 
\]
\eop

\medskip Dans la suite, on aura besoin d'``approcher'' la fonction discontinue
\[  
C_a(x) = \left\{
\begin{array}{cl}
1, & \hbox{si} \; x \geq a \\
0, & \hbox{sinon} 
\end{array}
\right.
\] 
Une telle ``approximation''  est donnée par la fonction semilinéaire continue  
$ C_{p,a}$:
\[
C_{p,a} = \min(1, \max(0, 2^p(x - a))) \; \hbox{où} \; p \in \NN_1 
\; \hbox{et} \; a \in \DD_{[0,1]}
\]

\begin{figure}[htbp]  
\begin{center}
\includegraphics*[width=14cm]{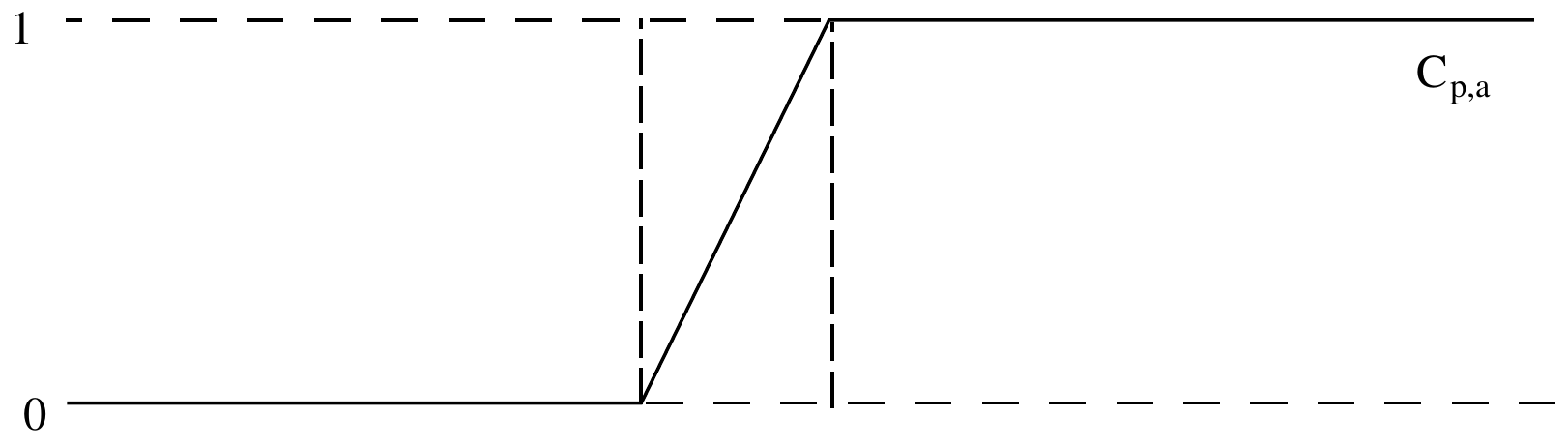}
\end{center}
\caption[Courbe représentative de la fonction $C_{p,a}$]{\label{ffi336}  
la fonction $C_{p,a}$}  
\end{figure}  

\noindent    
La \com de la famille de fonctions  $(p,a) \mapsto C_{p,a}$  est donnée dans la 
proposition suivante.

\begin{fproposition} \label{f336}
La famille de fonctions   
\[
\NN_1 \times \DD_{[0,1]} \; \; (p,a) \mapsto C_{p,a}
\]
définie par:
$
 C_{p,a} = \min(1,\max(0, 2^p(x - a)))
$  
est une famille de classe $\p$. 
Plus précisément, elle est de \com  
$\DRT(\Oo(N^5), \Oo({\M}(N^3) N^2 \log^3(N)))$   où  $N = \max(\Lg(a), n+p)$.
\end{fproposition}

La proposition suivante concerne la fonction racine carrée.

\begin{fproposition} \label{f337}
La fonction  $ x \mapsto \sqrt {\abs { x - 1/2 }} $  sur $[0,1]$   est un  
$\p$-point de  $\crf$.
\end{fproposition}
\begin{figure}[htbp]  
\begin{center}
\includegraphics*[width=10cm]{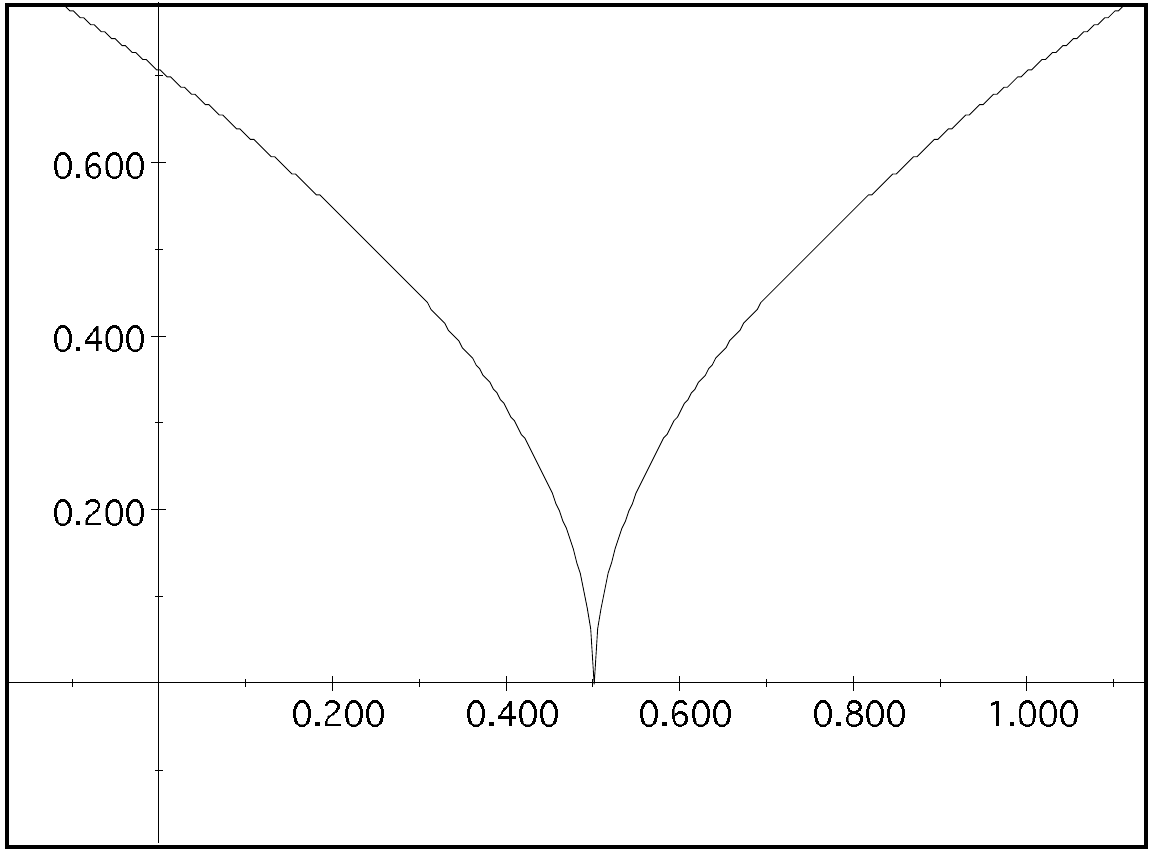}
\end{center}
\caption[Courbe de la fonction $\sqrt {\abs{ x - 1/2 }} $]{\label{ffi339}  
courbe de la fonction $ x \mapsto \sqrt {\abs { x - 1/2}} $}  
\end{figure}  

\proof (esquisse) 
Dans les \pols  $P_n(x^2)$ et $Q_n(x^2)$  du théorème de Newman, si on 
remplace  $x^2$  par une bonne approximation de $\abs{x}$ sur  $[0,1]$  
obtenue elle même par le théorème de Newman, on obtiendra une fraction 
rationnelle  $R_n(x)/S_n(x)$  telle que 
\[
\forall x \in [0,1] \abs{\sqrt{ \abs{x}} - (R_n/S_n)(x) }\;  \leq 2^{-
n}.
\]
Les degrés de  $R_n(x)$  et  $S_n(x)$  seront en  $\Oo(n^4)$. \\  
Notez que la fraction  $(P_n/Q_n)(x)$  n'est pas impaire, elle fournit donc une 
bonne approximation de la fonction  $\sqrt{ \abs{x}}$  uniquement sur 
l'intervalle $[0,1]$.
\eop

\section[Une \pres naturelle de l'espace \texorpdfstring{$\czu$}{C[0,1]}]{Une \pres naturelle de l'espace \texorpdfstring{$\czu$}{C[0,1]} 
et quelques présentations 
équivalentes}\label{fsec4}

Une notion naturelle de \com pour les ``points'' et pour les ``suites de points'' 
de l'espace   $\czu$  est donnée par  K. I Ko et H. Friedman dans \cite{fKF82}. 
Nous étudions dans cette section une \rp  de l'espace  $\czu$  qui donne (à très peu près) les mêmes notions de \com. Nous étudions également 
d'autres présentations utilisant des circuits, qui s'avèrent \equivas du 
point de vue de la \com  $\p$. La preuve de ces équivalences est basée sur 
la preuve d'un résultat analogue mais moins général donnée par 
\cite{fHo87,fHo90}. 

\subsection {Définitions de quelques présentations de l'espace 
 \texorpdfstring{$\czu$}{C[0,1]}}\label{fsubsec41}
 
\subsubsection{Presentation KF (KF comme Ko-Friedman) }\label{fsubsubsec411}
Rappelons qu'on note  $\DD_n = \{ k/2^n ; k \in \ZZ \}$  et  $\DD_{n,[0,1]} = 
\DD_n \cap [0,1]$.   \\ 
Dans la définition donnée par Ko et Friedman, on considère, pour  
$f \in \czu$, une machine de Turing à Oracle  (MTO)  qui ``calcule'' la 
fonction $f$ au sens suivant. 
 L'oracle délivre, pour la question  ``$\,m\;?\,$'' écrite en unaire  une approximation 
  $\xi \in\DD_{m,[0,1]}$ de $x$ avec la précision  $2^{-m}$. 
Pour l'entrée  $n \in  \NN_1$, la machine calcule, en utilisant l'oracle, une 
approximation $ \zeta  \in  \DD_n$  de  $f(x)$  à $1/2^n$  près.\\
La seule lecture du programme de la  MTO  ne permet évidemment pas de savoir si la  MTO  calcule bien une fonction de  $\czu$.  En conséquence, puisque nous souhaitons avoir des objets clairement identifiés comme ``points rationnels'' de la \pres que nous voulons définir, nous prenons le parti d'introduire des paramètres de contrôle. 
Mais pour que de tels paramètres soient vraiment efficaces, il est nécessaire de limiter l'exécution de la machine à une précision de sortie donnée a priori.  
En conséquence la précision nécessaire en entrée est également 
limitée a priori. Ainsi, la  MTO  est approximée par une suite de machines 
de Turing ordinaires (sans oracle) n'exécutant chacune qu'un nombre fini de calculs.
Nous sommes donc conduits à définir une \rp  de  $\czu$  qui sera notée  
$\ckf$ (l'ensemble des codes des points rationnels  sera noté  $\ykf$)  de la 
manière suivante. 
 
\begin{fdefinition}\label{f411}
On considère un langage, choisi une fois pour toutes, pour décrire les 
programmes de machines de Turing avec une seule entrée, dans $\DD_{[0,1]}$  
(notée  $x$), et avec une seule sortie, dans $\DD$  (notée  $y$). Soit  
${\bf Prog}$  la partie de ce langage formée par les programmes bien écrits (conformément à une syntaxe précisée une fois pour toutes).
Soit  $f = (Pr, n, m, T) \in {\bf Prog} \times \NN_1 \times \NN_1 
\times \NN_1$  où  on a:\\
--- $n$  est la précision réclamée pour  $y$,\\
--- $m$  est la précision avec laquelle est donnée  $x$.\\
Le quadruplet  $(Pr,n,m,T)$  est dit {\em correct} lorsque sont vérifiées 
les conditions suivantes:\\
--- le programme  $Pr$  calcule une fonction de $\DD_{m,[0,1]}$  vers  
$\DD_n$,  c.-à-d. pour une entrée dans  $\DD_{m,[0,1]}$   on obtient en 
sortie un élément de  $\DD_n$;\\
--- $T$  majore le temps d'exécution maximum pour tous les $x$ de  
$\DD_{m,[0,1]}$; \\
--- pour deux éléments consécutifs (distants de  $1/2^m$) de  
$\DD_{m,[0,1]}$  en entrée, le programme donne en sortie deux éléments de  
$\DD_n$  distants d'au plus  $1/2^n$.\\
L'ensemble des quadruplets  $(Pr,n,m,T)$  corrects est noté  $\ykf$\label{fykf}.
  Lorsque la donnée $f$ est correcte, elle définit le point rationnel  
$\wi{f}$  suivant:  c'est la fonction linéaire par morceaux qui
 joint les points du graphe donnés sur la grille  
$\DD_{m,[0,1]} \times  \DD_n$  par l'exécution du programme $Pr$  pour 
toutes les entrées 
possibles dans  $\DD_{m,[0,1]}$.\\
Ainsi $\ckf=(\ykf,\eta,\delta)$ où  \smash{$\eta(f)=\wi f$} et $\delta$ doit être précisé conformément à la définition~\ref{f211}. La complexité de $\eta$ et $\delta$ est traitée dans la proposition \ref{f413}. 
\end{fdefinition} 

\begin{fremark}\label{f412}
 On notera que les deux premières conditions pourraient être réalisées 
de manière automatique par des contraintes de type syntaxique faciles à 
mettre en oeuvre. Par contre, comme on le verra dans la proposition \ref{f441}, 
la troisième est incontrôlable \etpo (sauf si  $\p$ = $\np$ ) (cette 
condition de correction définit un problème $\cnp$-complet). 
Notez aussi que si on n'avait pas imposé la condition de correction pour les points de  $\ykf$  
on n'aurait eu aucun contrôle a priori du \mcu pour la fonction linéaire par 
morceaux définie par une donnée, et la famille  $\big(\wi{f}\,\big)_{f \in \ykf 
}$  n'aurait pas été \uni de classe  $\p$.
\end{fremark}

\begin{fproposition}[complexité de la famille de fonctions attachées à  $\ykf$] \label{f413}~\\
La famille de fonctions continues  $\big(\wi{f}\,\big)_{f \in \ykf}$  
est \uni de classe  $\DSRT(\Lin, \Lin, \Oo(N^2))$. 
\end{fproposition}

\proof 
Tout d'abord, on remarque que la fonction  $\wi{f}$  correspondant à  
$f = (Pr, n, m, T)$ est linéaire par morceaux et que, puisque la donnée est 
correcte, la pente de chaque morceau est majorée en valeur absolue par  
$2^{m-n}$  ce qui donne le \mcu  $\mu(f,k) = k+m-n$  pour la famille  
$\big(\wi{f}\,\big)_{f \in \ykf }$. \\
Nous devons exhiber une fonction  
$\varphi \colon  \ykf  \times \DD_{[0,1]} \times \NN_1 \rightarrow \DD$ de \com  
$\DSRT(\Lin, \Lin, N^2)$  telle que:
\[
\forall (f,x,k) \in \ykf  \times \DD_{[0,1]} \times \NN_1 \; \; 
\abs{ \varphi(f,x,k) - \wi{f}(x) }\;  \leq 2^{-k}.
\]
Soit  $z=(f,x,k)=((Pr,n,m,T),x,k) \in \ykf  \times \DD_{[0,1]} \times \NN_1$.
  Nous supposerons sans perte de généralité que  $k \geq  n$  et que $x$ 
est donné avec au moins $ m $ bits.\\
Par simple lecture de $x$ on repère deux éléments consécutifs  $a$  et  
$b$  de $\DD_{m,[0,1]}$   tels que   $a \leq  x \leq  b$  et on trouve le 
dyadique  $r \in \DD_{[0,1]}$  tel que  $x=a+r/2^n$. \\
Notons  $\varepsilon \in \{-1,0,1\}$   l'entier vérifiant   
$\wi{f}(a)=\wi{f}(b)+ \varepsilon/2^n$. \\  
Alors   $\wi{f}(x) = \wi{f}(a) + \varepsilon r/2^n$.  En outre  
$\wi{f}(a)= \Exec (Pr,a)$  est le résultat de l'exécution du 
programme  $Pr$   pour l'entrée  $a$.
Le calcul complet consiste donc essentiellement à \\ 
--- lire $x$  (et en déduire  $a$,  $b$  et  $r$), puis\\
--- calculer  $\Exec (Pr,a)$  et   $\Exec (Pr,b)$.\\
La \com est donc majorée par la \com de la Machine de Turing Universelle que nous utilisons pour exécuter le programme  $Pr$. D'après le lemme \ref{f131} cela se fait en temps $\Oo(T(T+ \flo{Pr}))$ et en espace $\Oo(T+ \flo{Pr})$. Et la taille de la sortie est majorée par $T$.
\eop

\medskip Le fait que la \pres  $\ckf$  qui résulte de la définition \ref{f411} soit 
appelée ``\pres à la Ko-Friedman'' est justifié par la proposition 
suivante.

\begin{fproposition} \label{f414}
Une fonction   $f \colon  [0,1] \rightarrow \RR$ est calculable \etpo  au sens de Ko-Friedman si et seulement si elle est un  $\p$-point de $\ckf$. Plus précisément 
\begin{itemize}
\item [a)] Si la fonction $f$ est de \com en temps  $T(n)$  au sens de Ko-Friedman alors  
elle est un  $\DTI (T)$-point de $\ckf$,
\item [b)]	Si la fonction $f$ est un  $\DTI (T)$-point de $\ckf$ alors elle est de \com en temps $T^2(n)$ au sens de Ko-Friedman. 
\end{itemize}
\end{fproposition}
Nous prouverons d'abord une caractérisation de la \com d'une fonction au sens de Ko-Friedman, pour une classe~$\ca$   de \com en temps ou en espace élémentairement stable. Nous l'avons déjà affirmée sans preuve dans le premier exemple \ref{f223}. Nous en donnons un énoncé plus précis ici.

\begin{fproposition} \label{f415}
Soit~$\ca$    une classe élémentairement stable de type $\DTI (\bullet)$  
ou  $\DSPA(\bullet)$  ou $\DSRT(\bullet, \bullet, \bullet)$  ou  $\DRT(\bullet, 
\bullet)$  ou  $\DSR(\bullet, \bullet)$  et soit une fonction continue  
$f \colon  [0,1] \rightarrow \RR$. Les propriétés suivantes sont équivalentes. 
\begin{enumerate}
\item la fonction $f$ est calculable au sens de Ko-Friedman dans la classe~$\ca$.
\item la fonction $f$ est \uni de classe~$\ca$.
\end{enumerate}
{\bf NB}: 
Pour ce qui concerne la \com d'une MTO, nous entendons ici que les questions 
posées à l'oracle doivent être comptées parmi les sorties de la machine. 
Autremant dit la taille des entiers~$m$ qui sont les questions à l'oracle 
doivent avoir la majoration requise pour la sortie dans les classes du type  $\DSRT(\bullet,\bullet,\bullet)$ ou $\DRT(\bullet,\bullet,\bullet)$  ou  $\DSR(\bullet,\bullet,\bullet)$.
\end{fproposition}

\proof (de la proposition \ref{f415})
Rappelons qu'une fonction est \uni de classe~$\ca$   (cf. définition 
\ref{f211}) lorsque:\\
a)	la fonction $f$ possède un \mcu dans la classe~$\ca$.\\
b)	la famille   $(f(a))_{a \in \DD}$   est une~$\ca$-famille de nombres 
réels.\\
Soit tout d'abord  $f \in \czu$ et   $M$  une  MTO  qui, pour toute entrée  
$n \in \NN_1$ et tout  $x \in [0,1]$   (donné comme oracle) calcule  $f(x)$  
à  $2^{-n}$   près dans la classe~$\ca$.  Si on remplace l'oracle, qui 
donne à la demande une approximation de $x$ à   $2^{-m}$  près dans  
$\DD_{m}$,  par la lecture d'un nombre dyadique  $a$  arbitraire, on obtient 
une machine de Turing usuelle qui calcule  $(f(a))_{a \in \DD}$   en tant que 
famille de nombres réels, ceci dans la classe~$\ca$. \\
Voyons la question du \mcu. Sur l'entrée  $n$  (précision requise pour la 
sortie) et pour n'importe quel oracle pour n'importe quel  $x \in [0,1]$, la MTO 
va calculer  $y=f(x)$  à~$2^{-n}$  près en interrogeant l'oracle pour 
certaines précisions  $m$. Puisque~$\ca$   est une classe de \com du type 
prévu, le plus grand des entiers  $m$  utilisés sur l'entrée  $n$  peut 
être majoré par $\mu(n)$ où $\mu\colon  \NN_1 \rightarrow  \NN_1$  est une 
fonction dans la classe~$\ca${\footnote{Par exemple pour une classe de 
complexité où on spécifie que les sorties sont de taille linéaire,  la 
taille de $m$ dépend linéairement de $n$ (indépendamment de l'oracle) 
puisque $m$, en tant que question posée à l'oracle, est une des sorties.}}. 
C'est le \mcu que nous cherchons.\\
Supposons maintenant que la fonction $f$ soit  \uni de classe~$\ca$.
 Soit  $M$  la machine de Turing (sans oracle) qui calcule  
$(f(x))_{x \in \DD}$   en tant que famille de nombres réels dans la classe~$\ca$.            
Soit  $M'$ une machine de Turing qui calcule un \mcu  $\mu \colon  \NN_1 \rightarrow 
\NN_1$  dans la classe~$\ca$. La MTO à la Ko-Friedman est alors la suivante. 
Sur l'entrée  $n$  elle calcule  $m = \mu(n+1)$  (en utilisant  $M'$), elle 
interroge l'oracle avec la précision  m, l'oracle donne un 
\elem  $a \in \DD_m$. La MTO utilise alors  $M$  pour calculer  $f(a)$  avec 
une approximation de $1/2^{n+1}$, qui, privée de son dernier bit, constitue la 
sortie de la MTO. \eop  

\medskip La proposition \ref{f415} démontre, par contraste, que la non parfaite adéquation 
obtenue dans la proposition \ref{f414} est seulement due à la difficulté de 
faire entrer exactement et à tout coup la notion de complexité d'une 
fonction \unico (en tant que fonction) dans le moule de la notion de 
``complexité d'un point dans un espace métrique''. Cette adéquation 
parfaite a lieu pour des classes comme  $\p$, $\Prim $ ou $\Rec $ mais pas pour  
$\DRT(\Lin, \Oo(n^k))$. \\
Enfin on pourra noter la grande similitude entre les preuves des propositions 
\ref{f415} et \ref{f414}, avec une complication un peu plus grande pour cette 
dernière.   
\sni \proof (de la proposition \ref{f414}) Soit tout d'abord  $f \in \czu$  et   
$M$  une  MTO  qui, pour toute entrée  $n \in \NN_1$  et tout $x \in [0,1]$ 
(donné comme oracle) calcule  $f(x)$ à  $2^{-n}$ près, en temps majoré 
par  $T(n)$. \\
Considérons la suite  $f_n = (Pr_n, n, T(n), T(n))_{n \in \NN_1}$ où on 
a:\\
$Pr_n$  est obtenu à partir du programme  $Pr$  de la machine  $M$  en 
rempla\c{c}ant la réponse  $(x_m)$  de l'oracle  à la question  ``$\,m\;?\,$''  
par l'instruction   ``lire  $x$  avec la précision  $m$'' \\ 
La correction de la donnée $f_n$ est claire. Soit  $\wi{f_n}$  la 
fonction linéaire par morceaux correspondant à la donnée  $f_n$.  On a  
pour  $\xi \in [0,1]$  et  $d$  une approximation de  $\xi$ à  $1/2^{T(n)}$  
près
$$ {\Exec}(Pr_n, d) = {\Exec} (Pr, n, \hbox{Oracle pour } \xi) 
= M^{\xi}(n)$$
et donc
$$\abS{ \wi{f_n}(\xi) - f(\xi) }\;  \leq 
\abs{f(d) - {\Exec}(Pr_n, \xi) } + \abS{ \wi{f_n}(d) - {\Exec} (Pr_n, \xi) }\;  \leq 2^{-n} + 2^{-n} \leq 2^{-(n-1)}$$ 
Enfin la suite  $n \mapsto f_{n+1}$  est de complexité  $\DTI (\Oo(n))$.\\
Réciproquement, supposons que  $\norme{ f_n - f}  \leq 2^{-n}$  avec    
$n \mapsto f_n$   de classe  $\DTI (T(n))$. \\
On a  $f_n = (Pr_n, q(n), m(n), t(n))$. Considérons la MTO  $M$ qui, pour tout 
entier  $n \in \NN_1$ et tout  $\xi \in [0,1]$, effectue les tâches 
suivantes: \\
--- prendre l'élément  $f_{n+1} = (Pr_{n+1}, q(n+1), m(n+1), t(n+1))$, de la 
suite, correspondant à  $n+1$. \\
--- poser à l'oracle la question  $Q(n) = m(n+1) + n+1 - q(n)$  (on obtient 
une approximation  $d$  de  $\xi$ à $2^{-Q(n)}$ près). \\
--- calculer  $\wi{f_{n+}}(d)$ par interpolation linéaire (cf. 
proposition \ref{f413}): ceci est la sortie de la machine~$M$.\\   
On a alors pour tout $ \xi \in [0,1]$
\[
\abS{\wi{f_n}(\xi) - M^{\xi}(n) }\;  \leq 
\abS{\wi{f}_n(\xi) - \wi{f}_{n+1}(\xi)} + 
\abS{\wi{f}_{n+1}(\xi) - \wi{f}_{n+1}(d) }\;  
\leq 2^{-(n+1)} + 2^{-(n+1)} = 2^{-n}  
\] 
donc 
\[
\Norme{ f - M^{\xi}(n)} \; \leq\;  \Norme{ f - \wi{f_n}} + 
\Norme{ \wi{f_n} - M^{\xi}(n) } 
\; \leq\;  2^{-n} + 2^{-n}
\]
De plus la machine  $M$ est de complexité $\Oo(T^2(n))$ (cf. proposition 
\ref{f413}).\eop

\medskip \noindent Dans le même style que la proposition \ref{f415}, on obtient deux 
caractérisations naturelles, à $\p$-équivalence près, de la \rp   $\ckf$  
de  $\czu$.  

\begin{fproposition} \label{f416}
Soit $\sC^{\star}$  une \rp  de l'espace  $\czu$.  
Alors on a l'équivalence des deux assertions suivantes.   
\begin{enumerate}
\item La \pres  $\sC^{\star}$ est $\p$-\equiva à  $\ckf$.
\item Une famille  $\big(\wi{f}\,\big)_{f \in Z}$  dans  $\czu$  est \uni de classe  
$\p$  si et seulement si elle est une $\p$-famille de points de  $\sC^{\star}$.
\end{enumerate}
\end{fproposition}

\proof 
Pour n'importe quelle classe élémentairement stable~$\ca$, deux \rps  d'un espace métrique arbitraire sont~$\ca$-\equivas si et seulement si elles définissent les mêmes~$\ca$-familles de points (cf. remarque \ref{f228}). Il suffit donc de vérifier que la \rp   $\ckf$  vérifie la condition (2). La preuve est identique à celle de la proposition \ref{f414}. 
\eop

\begin{ftheorem} \label{f417}
La \pres  $\ckf$  est universelle pour l'évaluation au sens suivant.\\
 Si  $\sC_Y=(Y,\delta,\eta)$  est une \pres rationnelle de  $\czu$  pour laquelle la famille  
$\big(\wi{f}\,\big)_{f \in Y}$  est \uni de classe $\p$,  
alors la fonction  $\Id_\czu$  de $\sC_Y$  vers  $\ckf$  est \uni de 
classe~$\p$.\\ 
Le résultat se généralise à toute classe de \com~$\ca$   
élémentairement stable vérifiant la propriété suivante:
si  $F \in \ca$  alors le temps de calcul de $F\colon  \NN_1 \rightarrow \NN_1,\; N 
\mapsto F(N)$,  est majoré par une fonction de classe~$\ca$.
\end{ftheorem}

\proof 
Cela résulte de \ref{f416}. Donnons néanmoins la preuve.  Nous devons démontrer 
que la famille   $\big(\wi{f}\,\big)_{f \in Y}$   est une $\p$-famille de points   pour la \pres  $\ckf$.  Puisque  la famille  $\big(\wi{f}\,\big)_{f 
\in Y}$  est \uni de classe $\p$, on a deux fonctions $\psi$ et  $\mu$  de classe  
$\p$:
$$ \psi \colon  Y \times \DD_{[0,1]} \times \NN_1 \rightarrow \DD \qquad  
(f,d,n) \mapsto \psi(f,d,n) = \wi{f}(d) \;  \hbox {avec  la  
précision} \; 1/2^n$$
et 			
$$ \mu\colon  Y \times \NN_1 \rightarrow \DD \qquad  (f,n) \mapsto \mu(f,d,n) = m $$
où  $\mu$  est un \mcu pour la famille.  
Soit $S \colon   \NN_1 \rightarrow \NN_1$   une fonction de classe~$\p$  qui donne une majoration du temps de calcul de $\psi$. \\  
Considérons alors la fonction  $\varphi \colon  Y \times \NN_1 \rightarrow \ykf $   
définie par:
$$\varphi(f,n) = (Pr(f,n), n, \mu(f,n), S(\flo{f} + \mu(f,n) + n))$$
où  $\flo{f}$  est la taille de $f$ et  $Pr$  est le programme calculant 
$\psi$ à peine modifié: 
il prend en entrée seulement les  $d \in \DD_{m,[0,1]}$ et ne garde pour la sortie que les chiffres significatifs dans $\DD_n$.\\
La fonction $\varphi$  est aussi dans la classe $\p$. \\
De plus il est clair que le point rationnel codé par $(Pr, n, m, S(\flo{f} + m+ 
n))$  ainsi construit approxime~$\wi f$~à~$1/2^n$   prés.  	\eop

\medskip Notez que cette propriété universelle de la \pres  $\ckf$  est évidemment 
partagée par toute autre \pres  $\p$-\equiva à  $\ckf$.

\subsubsection{Présentation  par circuits booléens} \label{fsubsubsec412}
La \pres par circuits booléens est essentiellement la même chose que la 
\pres à la Ko-Friedman. Elle est un peu plus naturelle dans le cadre que nous 
nous sommes fixés (\rp  de l'espace métrique  $\czu$). 

Notons   $\CB$  l'ensemble des (codes des) circuits booléens. Codé 
informatiquement, un circuit booléen  $C$  est simplement donné par un 
programme d'évaluation booléen (boolean straight-line program) qui représente l'exécution du circuit  $C$.
Un point rationnel de la \pres par circuits booléens est une fonction 
linéaire par morceaux  $\wi{f}$    codée par un quadruplet  
$f=(C,n,m,k)$ où $C$ est (le code d')un circuit booléen,  $n$  est la 
précision réclamée pour  $f(x)$, $m$ la précision nécessaire en 
entrée, et $2^{k+1}$  majore la norme de  $\wi{f}$. 

\begin{fdefinition} \label{f418}
Un point rationnel de la \pres par circuits booléens  $\cbo$  est codé par 
un quadruplet  $f = (C, n, m, k) \in \CB \times \NN_1 \times \NN_1 \times 
\NN_1$  où  $C$   est (le code d')un circuit booléen ayant  $m$  portes 
d'entrée et  $k+n+1$   portes de sortie:  les $m$   portes d'entrée 
permettent de coder un élément de $\DD_{m,[0,1]}$  et  les $k+n+1$  portes 
de sortie permettent de coder les éléments de  $\DD$  de la forme 
$\pm 2^k\left( \sum_{i=1,\ldots,n+k}b_i2^{-i}\right)$  (où les $b_i$ sont 
égaux à $0$ ou $1$, $b_0$ code le signe $\pm$).\\
La donnée  $f = (C, n, m, k)$  est dite {\em correcte} lorsque deux entrées 
codant des éléments de  $\DD_{m,[0,1]}$  distants de  $1/2^m$  donnent en 
sortie des codages de nombres distants de au plus  $1/2^n$.  L'ensemble des 
données correctes est  $\ybo$.\\
Lorsqu'une donnée $f$ est correcte, elle définit une fonction linéaire par 
morceau  $\wi{f}$   qui est bien contrôlée (c'est le ``point 
rationnel'' défini par la donnée).
\end{fdefinition}

On notera que les paramètres de contrôle  $n, m$ et $k$  sont surtout 
indiqués pour le confort de l'utilisateur. En fait  $n+k$   et  $m$  sont 
directement lisibles sur le circuit  $C$. \\
De ce point de vue, la situation est un peu améliorée par rapport à la 
\pres  $\ckf$  pour laquelle les paramètres de contrôle sont absolument 
indispensables.
La taille de  $C$  contrôle à elle seule le temps d'exécution du circuit  
(c.-à-d. la fonction qui calcule, à partir du code d'un circuit booléen et 
de la liste de ses entrées, la liste de ses sorties est une fonction de très 
faible \comz).
\begin{fproposition} \label{f419} 
{\em  (Complexité de la famille de fonctions attachées à   $\ybo$)}.  
La famille de fonctions continues  $\big(\wi{f}\,\big)_{f \in \ybo }$ est \uni de 
classe  $\DSRT(\Lin, \Lin, \QLin)$.
\end{fproposition}

\proof
On recopie presque à l'identique la preuve de la proposition \ref{f413} 
concernant la famille  $\big(\wi{f}\,\big)_{f \in \ykf}$. La seule différence est 
le remplacement de la fonction  $\Exec(Pr,a)$  par l'évaluation d'un circuit 
booléen.   La fonction qui calcule, à partir du code d'un circuit booléen 
et de la liste de ses entrées, la liste de ses sorties est une fonction de 
\com en temps $\Oo(t\log(t))$ (avec $t = \ta(C)$), donc de classe  $\DSRT(\Lin, 
\Lin, \QLin)$. 	\eop

\medskip \noindent  {\bf NB}. Si on avait pris en compte ``le temps de gestion'' (cf. remarque \ref{f323}) on aurait trouvé une \com  $\DSRT(\Lin, \Lin, \Oo(N^2))$, c.-à-d. le même résultat que pour la famille  $\big(\wi{f}\,\big)_{f \in \ykf}$  (la \com de  $\Exec(Pr,a)$  prend justement en compte ``le temps de gestion'').

\subsubsection{Présentation par circuits arithmétiques fractionnaires (avec magnitude)}\label{fsubsubsec 413}
Rappelons qu'un circuit {\em arithmétique (fractionnaire)} est par 
définition un circuit qui a pour portes d'entrées des variables ``réelles''  $x_i$  et des constantes dans $\QQ$  (on pourrait évidemment ne tolérer que 
les deux constantes  0  et  1  sans changement significatif). \\
Les autres portes sont:\\
--- des portes à une entrée, des types suivants:  
$x \mapsto x^{-1}, \; x \mapsto -x$ \\
--- des portes à deux entrées, des deux types suivants:  
$x+y, \; x \times y$. \\ 
Un circuit arithmétique calcule une fraction rationnelle, (ou éventuellement 
``error'' si on demande d'inverser la fonction identiquement nulle).
Nous considérerons des circuits arithmétiques avec une seule variable en 
entrée et une seule porte de sortie et nous noterons  $\CA$  l'ensemble des 
(codes de) circuits arithmétiques à une seule entrée.  Le code d'un 
circuit est le texte d'un programme d'évaluation (dans un langage et avec une 
syntaxe fixés une fois pour toutes) correspondant au circuit arithmétique 
considéré.

\begin{fdefinition} \label{f4110}
Un entier  $M$  est appelé {\em un coefficient de magnitude} pour un circuit 
arithmétique $\alpha$ à une seule entrée lorsque $2^M$  majore en valeur 
absolue toutes les fractions rationnelles du circuit (c.-à-d. celles 
calculées à toutes les portes du circuit) en tout point de l'intervalle 
$[0,1]$. \\
Un couple  $f = (\alpha, M) \in \CA \times \NN_1$ est une donnée {\em 
correcte} lorsque $M$ est un coefficient de magnitude pour le circuit  $\alpha$.  
Cela définit les éléments de  $\yaf$.  Un élément $f$ de  $\yaf$  
définit la fraction rationnelle notée  $\wi{f}$    lorsqu'elle est 
vue comme un élément de  $\czu$.  
La famille  $\big(\wi{f}\,\big)_{f \in \yaf }$ est la famille des points rationnels 
d'une \pres de  $\czu$  qui sera notée  $\caf$
\end{fdefinition}
 
Notez que, à cause de la présence des multiplications et du passage à 
l'inverse, la taille de  $M$ peut être exponentielle par rapport à celle de  
$\alpha$, même lorsqu'aucune fraction rationnelle de $\alpha$ n'a de pole sur  
$[0,1]$.  Ce genre de mésaventure n'avait pas lieu avec les circuits 
semilinéaires binaires.
Il semble même improbable que la correction d'un couple 
$(\alpha, M) \in \CA \times \NN_1$  puisse être testée \etpo (par rapport 
à la taille de la donnée $(\alpha, M)$. 

\begin{fproposition}[complexité de la famille de fonctions attachées à  $\yaf$] \label{f4111}~\\    
La famille de fonctions continues  $\big(\wi{f}\,\big)_{f \in \yaf }$  est \uni de 
classe  $\p$, et plus précisément de classe  $\DSRT(\Oo(N^3), \Lin, 
\Oo(N{\M}(N^2)))$.
\end{fproposition}

\proof 	Soit $ f = (\alpha,M) $ un élément de  $\yaf$. Soit $ p $ la 
profondeur du cirduit $\alpha$. On démontre par récurrence sur $\pi$ que, pour 
une porte de profondeur $\pi$  la fonction 
rationnelle correspondante a une dérivée majorée par $2^{2M \pi}$ sur 
l'intervalle $[0,1]$. Ceci fournit le \mcu  $\mu(f,k) = k + 2Mp$  (qui est en 
$\Oo(N^2)$) pour la famille  $\big(\wi{f}\,\big)_{f \in \yaf }$. \\
Nous devons maintenant exhiber une fonction  
$\psi \colon  \yaf  \times \DD_{[0,1]} \times \NN_1 \rightarrow \DD$  
de \com $ \DRT(\Lin,\Oo(N\M(N^2))) $ telle que
\[
\forall (f, x, k) \in \yaf  \times \DD_{[0,1]} \times \NN_1 \; \;
\abs{ \psi(f,x,k) - \wi{f}(x)} \;  \leq 2^{-k}.
\]
Soit  $(f, x, k) \in \yaf  \times \DD_{[0,1]} \times \NN_1$  où  
$f = (\alpha,M)$.  Soit $ t = \ta(\alpha) $ le nombre des portes du circuit 
$\alpha$. La taille des entrées est $N = t+M+k$. \\ 
Pour calculer $\psi((\alpha,M), x, k)$  on procède comme suit.  
On lit $ m = k+2Mp+p $ bits de $x$ ce qui donne les deux points consécutifs
 $ a $ et $ b $ de $\DD_{m,[0,1]}$   tels que $a \leq x \leq b$.  On exécute 
le circuit $ \alpha $ au point $ a $ en tronquant le calcul effectué sur 
chaque porte aux $ m $ premiers bits significatifs. La valeur obtenue à la 
sortie, tronquée aux $ k $ premiers bits significatifs, est l'élément 
$\psi((\alpha,M), x, k)$. Comme une multiplication (ou une division) se fait en 
temps ${\M}(m)$, le temps des opérations arithmétiques proprement dites est 
un $\Oo(t \M(k+2Mp+p)) = \Oo(N\M(N^2))$ et l'espace occupé est en $\Oo(N^3)$.
	\eop 

\mni {\bf NB}. Si on avait pris en compte ``le temps de gestion'' (cf. remarque 
\ref{f323}) on aurait dit: \\ 	
le temps des opérations arithmétiques proprement dites est $\Oo(N \M(N^2)),$ 
et la gestion des objets est en $\Oo(N^4).$ Ainsi si on considère la 
multiplication rapide (respectivement naïve), la fonction d'évaluation 
est dans $ \DRT(\Lin,\Oo(N^4))$ (respectivement $\DRT(\Lin,\Oo(N^5))$) où $N$ est 
la taille des données.
\subsubsection{Présentation par circuits arithmétiques polynomiaux (avec 
magnitude)}\label{fsubsubsec414}

La \pres suivante de  $\czu$  sera notée   $\capo$. L'ensemble des codes des points rationnels  sera noté  $\yap$. 

C'est la même chose que les circuits arithmétiques fractionnaires, sauf 
qu'on supprime les portes ``passage à l'inverse''. Il n'est donc pas 
nécessaire d'écrire la définition en détail.\\
On retrouve les mêmes difficultés concernant la magnitude, en un peu moins 
grave. Il ne semble pas que l'on puisse obtenir pour les circuits polynomiaux de 
majorations sensiblement meilleures que celles obtenues à la proposition 
\ref{f4111} pour les circuits avec divisions.

\subsection{Comparaisons des présentations précédentes}\label{fsubsec42}

Nous démontrons dans cette section un résultat important, l'équivalence entre 
la \pres à la Ko-Friedman et les quatre présentations ``par circuits'' de 
$\czu$ citées dans la section précédente, ceci du point de vue de la \com 
en temps \poll. En particulier nous obtenons 
une formulation complètement contrôlée du point de vue algorithmique 
pour le théorème d'approximation de Weierstrass.  \\
Pour démontrer l'équivalence entre ces cinq présentations du point de vue de 
la \com en temps \poll, nous  suivrons le plan ci-dessous:

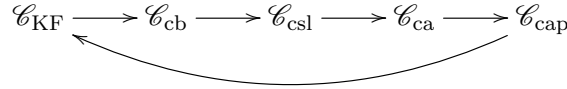
\begin{figure}[htbp]
\begin{center}
$$
\xymatrix @R=15pt{
  \ckf\ar[r] &  \cbo\ar[r] &  \csl\ar[r] &  \caf\ar[r] & \capo\ar@/^25pt/[llll] 
}
$$ 
\end{center}
\medskip\caption[schéma de preuve]{\label{ffi421}schéma de la preuve  
}  
\end{figure}  

Pour donner l'équivalence entre ces différentes présentations, nous devons 
construire des fonctions qui permettent de transformer un point rationnel d'une 
\pres donnée en un point rationnel d'une autre présentation, qui approche 
convenablement le premier.

\begin{fproposition} \label{f421}
L'identité de  $\czu$  de  $\ckf$  vers  $\cbo$  est \uni de classe  $\p$,  en 
fait de classe  $\DTI (N^{14})$.
\end{fproposition}

La preuve de cette proposition est basée sur le Lemme \ref{f131} qui décrit 
la \com d'une Machine de Turing Universelle et sur le lemme suivant.

\begin{flemma} \label{f422}
{\rm (cf. \cite{fSt} et \cite{fMo}) }  \\
Soient $ M $ une machine de Turing fixée, $ T $ et $ m $ des éléments de 
$\NN_1$. Alors on a une fonction calculable en temps $\Oo((T+m)^7)$  qui calcule  
$\gamma_{T,m}$:  un circuit booléen simulant les $ T $ premières 
configurations de $ M $ pour n'importe quelle entrée dans  $\{0,1\}^m$.
\end{flemma} 

\begin{fremark} \label{f423}
J. Stern a donné dans \cite {fSt}, pour tout entier $T \in \NN_1$, une 
construction assez simple d'un circuit  $\gamma_T$  qui calcule la configuration 
de $ M $ obtenue après $ T $ pas de calcul à partir d'une configuration 
initiale de taille $\leq  T$  (on peut supposer que $T \geq  m$ ou prendre 
$\max(T,m)$). La taille du circuit  $\gamma_T$  est un  $\Oo(T^2)$. L'auteur a 
mentionné aussi que cette construction se fait en temps \poll. Une \com 
plus précise (de l'ordre de  $\Oo(T^7)$) est donnée dans \cite{fMo}.
\end{fremark}

\proof (de la proposition \ref{f421})	On considère la Machine de Turing 
Universelle MU du lemme \ref{f131}.  Soit $ f = (Pr,n,m,T) $ un élément de  
$\ykf$.  Notons $ p $ la taille du programme $Pr.$ 
La taille de  $f$ est $N = p+n+m+T.$  La machine $ MU $ prend en entrée le 
programme $ Pr $  ainsi qu'une entrée
 $ x \in  \DD_{m,[0,1]}$  et le nombre d'étapes $ T.$ 
Elle exécute en $\Oo(T (\max(T,m)+p)) = \Oo(N^2)$ étapes (cf. lemme \ref{f131})  
la tâche de calculer la sortie pour le programme $ Pr $ sur la même entrée 
$x$ après $ T $ étapes de calcul.  
En appliquant le lemme \ref{f422}  on obtient un temps  $\Oo(N^{14})$  pour 
calculer à partir de l'entrée $f$ un \elem $ g $ de  $\ybo$  
(un circuit booléen ainsi que ses paramètres de contrôle)  
pour lequel on a 
$ \wi{g} = \wi{f}$.
	\eop

\begin{fproposition} \label{f424}
L'identité de  $\czu$  de  $\cbo$  vers  $\csl$  est \uni de classe    $\LINT 
$.  Plus précisément, on a une fonction discrète qui, à partir d'un 
élément $ f = (\gamma,n,m,k) $ de $\ybo$  et d'un entier 
$ q \in \NN_1$,  calcule en temps $ \Oo(N) $ (où $ N $ est la taille de 
l'entrée $ ((\gamma,n,m,k),q)$), un circuit semilinéaire binaire $ g $ tel 
que  
\[
\forall x \in [0,1]\;\; 
\abs{ \wi{f}(x) - \wi{g}(x)} \;  \leq 2^{-q} \; \;  \eqno (\F4.2.4) .
\]  
\end{fproposition}

La preuve de cette proposition utilise le lemme suivant:
\begin{flemma} \label{f425} 
Il existe une fonction calculable en temps $\Oo(N)$ qui transforme tout 
élément $f = ((\gamma,n,m,k),h)$ de $\ybo \times  \NN_1$   en un élément 
$f' = (\gamma',n+h,m+h,k)$ de $\ybo$  correspondant à la même fonction 
semilinéaire.
\end{flemma}

\proof (du lemme \ref{f425})
Supposons que le circuit $ \gamma $ calcule, pour l'entrée $x = u/2^m$ où
 $0 \leq u \leq 2^m - 1$, la valeur $y = \ell/2^n$, et pour $x' = (u+1)/2^m$, la 
valeur $y' = \ell'/2^n$ avec $\ell , \ell' \in  \ZZ$  et  
$\abs{ \ell - \ell'} \;  \leq 1$.
 Alors le circuit   $\gamma'$  calcule, pour l'entrée  $x+r2^{-h}2^{-m}$, 
$0 \leq r \leq 2^h$, la valeur  $y+(y'-y)r2^{-h}$. \\
Noter de plus que  $  \ta(\gamma') = \Oo(\ta(\gamma )+h)  $  et  $\prof(\gamma') = 
\Oo(\prof(\gamma)+h)$.	\eop

\proof (de la proposition \ref{f424})
D'après le lemme \ref{f425}, la condition (\F4.2.4) de la 
proposition \ref{f423} peut être remplacée par la condition:   
\[
\forall x \in [0,1]\;\; \abs {\wi{f}(x) - \wi{g}(x)} \;  \leq 2^{-(n-1)}  
   \eqno{(\F4.2.5)}
\] 
En effet: si   $q < n$  alors (\F4.2.4) découle de (\F4.2.5), sinon, on utilise 
l'interpolation linéaire donnée dans la preuve du lemme \ref{f425}.\\
Nous cherchons donc maintenant à simuler le circuit booléen   
$f = (\gamma,n,m,k)$   par un circuit semilinéaire binaire  $g$  avec la 
précision  $ 1/2^n$. \\
Tout d'abord remarquons qu'il est facile de simuler de manière exacte toutes 
les portes, ``sauf l'entrée'', par un circuit semilinéaire simple  $\lambda$ 
qui consiste en:\\
1) remplacer les constantes booléennes $ 0 $ et $ 1 $ de $ \gamma $ par les 
constantes rationnelles $ 0$ et $1$.\\
2) remplacer chaque sommet de $ \gamma $ calculant $\neg u$ par un sommet de 
$ \lambda $ calculant $1-u$. \\
3) remplacer chaque sommet de $ \gamma $ calculant $u \land v$ par un sommet 
de $ \lambda $ calculant $\min(u,v)$.\\
4) remplacer chaque sommet de $ \gamma $ calculant $u \lor v$ par un sommet
 de $ \lambda $ calculant $\max(u,v)$. \\
5) calculer, à partir des sorties $c_0, c_1, \ldots,c_{n+k}$ du 
circuit $ \gamma, $ le point rationnel:
\[
\pm 2^k \sum_{i=1,\ldots,n+k}c_i2^{-i} \quad (c_0  \hbox{ code le  signe})
\]
Il est clair que le circuit $ \lambda $ se construit en temps linéaire et 
admet une taille $ \ta(\lambda) = \Oo(\ta(\gamma)) $ et une profondeur 
$ \prof(\lambda) = \Oo(\prof(\gamma))$.\\
Maintenant nous cherchons à ``simuler l'entrée'' du circuit $\gamma,$ 
c.-à-d. les bits codant $x$, par un circuit semilinéaire binaire.\\
Usuellement, pour déterminer le développement binaire d'un nombre réel 
$x$, on utilise le pseudo-algorithme  suivant qui utilise la fonction 
discontinue $ C $ définie par:
$$ C(x) = \left\{
\begin{array}{cl} 
1 & \hbox{ si  } x \geq 1/2 \\ 
0 &
\hbox{sinon} 
\end{array}
\right.
$$

\sni {\bf Algorithme 1} (``calcul'' des $m$ premiers bits d'un nombre réel $x\in [0,1]$)
\begin{itemize}
\item[] Entrées: $ x \in [0,1], \; m \in \NN$ \\
        Sortie:   la liste $(b_1, b_2,\ldots, b_m) \in \{ 0,1 \}^m$ 
	\begin{itemize}
	\item[] Pour $j:=1$ à $m$ Faire
		\begin{itemize}
		\item[] $ b_j \leftarrow C(x)$ \\
			 $ x \leftarrow 2x - b_j$ 
		\end{itemize}
	  Fait \\
 	Fin.
	\end{itemize}
\end{itemize}

\sni La fonction $C$ est discontinue, donc l'algorithme 1 {\em n'est pas vraiment 
un algorithme}; et il ne peut pas être simulé par un circuit semilinéaire 
binaire. On considère alors la 
fonction continue
$$C_p(x) := C_{p,1/2} = \min(1, \max(0,2^p(x-1/2))) \eqno \hbox{(cf. figure \ref{ffi336})}
$$
qui ``approche'' la 
fonction  $C$.  Et l'on considère l'algorithme suivant:

\newpage
\mni {\bf Algorithme 2} (calcul ``approximatif'' des $ m $ premiers  bits qui 
codent un nombre réel $x$)
\begin{itemize}
\item[] Entrées: $ x \in [0,1], \; m \in \NN_1$ \\
  Sortie: une liste $(b_1, b_2,\ldots, b_m) \in [ 0,1 ]^m$ 
	\begin{itemize}
	\item[] $p \leftarrow m+2$ \\
			 Pour $j:=1$ à $m$ Faire
		\begin{itemize}
		\item[] $ b_j \leftarrow C_p(x)$ \\
			    $ x \leftarrow 2x - b_j$  \\
				$p \leftarrow p-1$
	     \end{itemize}
	  Fait \\
 	Fin.
	\end{itemize}
\end{itemize}

\sni Ceci est réalisé par un circuit de taille et profondeur $\Oo(m)$ sous la 
forme suivante:

\mni {\bf Algorithme 2bis} (forme circuit semilinéaire de l'algorithme 2)
\begin{itemize}
\item[] Entrées: $ x \in [0,1], \; m \in \NN_1$ \\
 Sortie: une liste $(b_1, b_2,\ldots, b_m) \in [ 0,1 ]^m$ 
	\begin{itemize}
	 \item[] $p \leftarrow m+2 , \; q \leftarrow 2^p$ \\
			 Pour $j:=1$ à $m$ Faire
				\begin{itemize}
			 \item[] $ y \leftarrow q(x-1/2)$ \\
			 		 $ b_j \leftarrow \min (1, \max(0,y))$ \\
					$x \leftarrow 2x - b_j$ \\
 				    $q \leftarrow q/2$
				 \end{itemize}
	  Fait \\
 	Fin.
	 \end{itemize}
\end{itemize}

\sni Mais un problème crucial se pose quand le réel $x$ en entrée est dans 
un intervalle de type  $\; ]\,k/2^m, k/2^m + 1/2^p\,[\; $   où 
$0 \leq k \leq 2^m-1$. Dans ce cas au moins un bit calculé dans l'algorithme 2 est dans $]0,1[$. Par conséquent le résultat final de la 
simulation du circuit booléen peut être incohérent. 
Pour contourner cette difficulté, on utilise une technique introduite par 
Hoover \cite{fHo87,fHo90}. On fait les remarques suivantes:\\
--- Pour tout  $x \in [0,1]$  au plus une des valeurs  
$x_{\sigma} = x + \frac \sigma   {2^{m+2}}$  (où  $ \sigma \in \{-1, 0, 1 \} $) est dans un  intervalle de type précédent.\\
--- Notons  $z_{\sigma} = \sum_{j=1,\ldots,m}b_j2^{-j}$  où les $b_j$  sont 
fournis par l'algorithme 2bis sur l'entrée  $x_{\sigma}$. D'après la remarque précédente, au moins deux valeurs  $\lambda(\wi {\lambda}(z_{\sigma}) \; (\sigma \in \{-1,0,1\})$, 
correspondent exactement à la sortie du circuit arithmétique  $\gamma$  
lorsqu'on met en entrée les  $m$ premiers bits de~$x_{\sigma}$. 
C.-à-d.,  pour au moins deux valeurs de  $\sigma$, on a  $z_{\sigma} \in 
\DD_m$  et $\wi {\lambda}(z_{\sigma}) = \wi{f}(z_{\sigma})$   avec en outre $\abs{ z_{\sigma} - x }\;  \leq 1/2^m + 1/2^{m+2} \leq 1/2^{m-1}$ 
(d'où  $\abS{ \wi{f}(z_{\sigma}) - \wi{f}(x) }\;  \leq 1/2^{n-1}$).\\
--- Donc, une éventuelle mauvaise valeur calculée par le circuit 
semilinéaire ne peut être que la première ou la troisième des valeurs si 
on les ordonne par ordre croissant.
Ainsi, si   $y_{-1}$, $y_0$ et $y_1$,  sont les trois valeurs respectivement calculées par le circuit aux points $x_{\sigma}$  ($\sigma\in  \{-1, 0,1\}$), alors la deuxième des 3 valeurs approche correctement $\wi{f}(x)$. 
Etant donnés trois nombres réels  $y_{-1}$, $y_0$ et $y_1$, le  2ème par 
ordre croissant, $\theta(y_{-1},y_0,y_1)$,  est calculé par exemple par l'un des deux circuits semilinéaires binaires représentés par le 
deuxième membre dans les équations:
\[
\theta (y_{-1},y_0,y_1) = y_{-1}+y_0+y_1 - \min (y_{-1},y_0,y_1) - 
\max (y_{-1},y_0,y_1)
\]
\[\theta (y_{-1},y_0,y_1) = \min (\max (y_0,y_1), \max (y_1,y_{-1}), 
\max (y_0,y_{-1}))
\]
Résumons: à partir de l'entrée $x$ nous calculons les  
$x_{\sigma} = x + \frac \sigma   {2^{m+2}}$  (où  $\sigma \in \{-1,0,1 \}$), 
et nous appliquons successivement:\\
--- le circuit $\varepsilon$ (de taille $\Oo(m) = \Oo(N)$)  qui décode 
correctement les $m$ premiers digits de $x_{\sigma}$  pour au moins deux des 
trois $x_{\sigma}$; \\
--- le circuit  $\lambda$  qui simule le circuit booléen proprement dit et 
recode les digits de sortie sous forme d'un dyadique, ce circuit est de taille 
$\Oo(N)$;\\
--- puis le circuit  $\theta$  qui choisit la  2ème  valeur calculée par 
ordre croissant.\\
Nous obtenons donc un circuit qui calcule la fonction:
\[
\wi{g}(x) = \theta (\lambda(\varepsilon (x-1/2^p)),
\lambda (\varepsilon (x)), \lambda (\varepsilon (x+1/2^p))) \quad \hbox{avec} 
\; p = m+2
\]
 de taille et de profondeur \Oo(N) et tel que: 
\[
\abs{ \wi{f}(x) - \wi{g}(x) }\;  \leq 2^{-n} \; \forall x \in 
[0,1].
\]
En outre, le temps mis pour calculer  (le code de)  $g$  à partir de 
l'entrée $f$ est également en $\Oo(N)$.

\begin{figure}[htbp]  
\begin{center}
\includegraphics*[width=3cm]{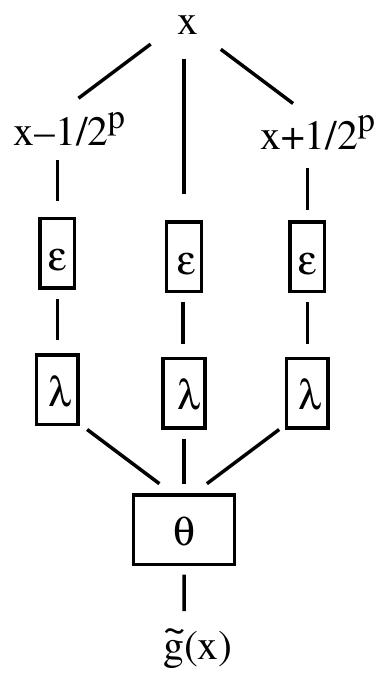}
\end{center}
\caption[vue globale du calcul]{\label{ffi422}  
vue globale du calcul}  
\end{figure}  
\eop

\medskip Nous passons à la simulation d'un circuit semilinéaire binaire par un circuit arithmétique (avec divisions). 
Nous donnons tout d'abord une version ``circuit'' de la proposition \ref{f335}.

\begin{fproposition} \label{f426}
Les fonctions  $(x,y) \mapsto \max(x,y)$ et $(x,y) \mapsto \min(x,y)$   sur le carré  $[-2^m, 2^m] \times [-2^m, 2^m]$   peuvent être approchées à  $1/2^n$ près par des circuits arithmétiques de taille  $\Oo(N^3)$, 
de magnitude  $\Oo(N^3)$ et qui peuvent être calculés dans la classe 
$\DTI (\Oo(N^3))$ où $(N = n+m)$.
\end{fproposition} 

\proof 
On reprend la preuve du lemme \ref{f333}. La re\pres de \pols approchant 
convenablement  $ P_n $ et $ Q_n $  au moyen de circuits est plus économique en espace. Tout d'abord il faut construire un circuit qui calcule une approximation de $ e^{-1/n} $ avec la précision $1/2^{n^3}$. 
On considère le développement de Taylor  
\[
F_m(z) = \sum_{0\leq k\leq m} \frac{(-1)^k}{k!} z^k
\] 
qui approche $ e^{z} $ à $1/2^{m+2}$  près si  $m\geq 5$ et $0\leq z\leq 1$. 
Ici, on peut se contenter de donner un circuit qui calcule $d_n = F_{n^3}(1/n)$ dont la taille et le temps de calcul sont a priori en $ \Oo(n^3) $ (alors qu'on était obligé de donner explicitement une approximation dyadique $c_{n,1}$  de $d_n$). 
Ensuite on construit un circuit qui calcule une bonne approximation 
de  $H_n(x) = \prod_{1 \leq k < n^2} (x+e^{-k/n})$ sous la forme  
$h_n(x) = (x+d_n)(x+d_n^2)\cdots(x+d_n^{n^2-1})$.
Ce qui réclame un temps de calcul (et une taille de circuit) en  $\Oo(n^2)$. 
Ceci fait que le circuit arithmétique qui calcule une approximation  à 
$ 1/2^n $ près  de  $\abs{x}$  sur  $[0,1]$   
est calculé en temps $\Oo(n^3)$. \\ 
On voit aussi facilement que le coefficient de magnitude est majoré par la 
taille de $ 1/h_n(0) $ c.-à-d. également un $ \Oo(n^3).$ 	\eop
 
\smallskip On remarquera que le coefficient de magnitude peut difficilement être amélioré. Par contre, il ne semble pas impossible que $ d_n$ puisse être calculé par un circuit de taille plus petite que la taille naïve en $\Oo(n^3)$.

\begin{fproposition} \label{f427}
L'identité de  $\czu$  de  $\csl$  vers  $\caf$  est \uni de classe  $\p$, en fait de classe $ \DTI (N^4). $ 
Plus précisément, on a une fonction discrète qui, à partir d'un élément $ g $ de $ \ysl $ de taille $ \ta(g) = t $ et de profondeur $ \prof(g) = p, $ et d'un entier $ n \in  \NN_1$,  calcule en temps $ \Oo( t (n+p)^3 ) $ un élément $ f = (\alpha,M)  \in \yaf$\\ 
\spa de taille	 $ \ta(\alpha)  =  \ta(g) \Oo((n+p)^3), $  \\
\spa de profondeur	 $ \prof(\alpha)  = \Oo(p(n+p)^3) $ \\
\spa avec coefficient de magnitude	 $ M = \Mag(\alpha) = \Oo((n+p)^3), $  
\\
et tel que 	
$\abS{\wi{f}(x)-\wi{g}(x)}\; \leq 2^{-n} \; \; \forall x \in [0,1]$. 
\end{fproposition} 

\proof 
Notons   $ p = \prof(g). $   Les portes de $ g $ pour une entrée 
$ x \in  [0,1]$    prennent des valeurs dans l'intervalle $[-2^p,2^p]$.  
On veut simuler à  $2^{-n}$  près le circuit $ g $ par un circuit 
arithmétique.  Il suffit de simuler les portes $\max$ et $\min$ du circuit 
semilinéaire à $2^{-(n+p)}$ près sur l'intervalle $[-2^p,2^p]$.  
Chaque simulation réclame, d'après la proposition \ref{f426}, un circuit 
arithmétique (avec division) dont toutes les caractéristiques sont 
majorées par un  $\Oo(n+p)^3$.	\eop

Nous passons à la simultation d'un circuit arithmétique avec divisions par 
un circuit arithmétique \poll.

\begin{fproposition} \label{f428}
L'identité de  $\czu$  de  $\caf$  vers  $\capo$  est \uni de classe  $\p$, en 
fait de classe  $\DTI (N^2)$.  
Plus précisément, on a une fonction discrète qui, à partir d'un 
élément  $f=(\alpha,M)$ de  $\yaf$ et d'un entier  
$n\in\NN_1$,  calcule en temps  $\Oo(N^2)$  un circuit arithmétique 
\poll   $(\Gamma, M') \in \yap $ ($N$ est la taille de l'entrée 
$((\alpha,M),n)$ \\
--- de taille  $ \ta(G) = \Oo( \ta(\alpha) (n + M) ), $  \\
--- de profondeur  $ \prof(G) = \Oo( \prof(\alpha) (n + M) ) $, \\
--- de magnitude  $ \Mag(G) = M'= \Oo( n + M ) $,  \\
--- et tel que:  
$\abS{ \wi{f}(x) - \wi{g}(x) }\;  \leq 2^{-n}$  
pour tout $x \in [0,1]$.
\end{fproposition}
 
Le problème se pose seulement au niveau des portes ``passage à l'inverse''. 
Nous cherchons donc à les simuler  par des circuits polynomiaux tout en 
gardant la magnitude bien majorée.

\begin{flemma} \label{f429}
La fonction   $ x \mapsto 1/x  $  sur l'intervalle  $[2^{-m}, 2^m]$  peut être 
réalisée avec la précision   $ 1/2^n  $   par un circuit \poll de 
magnitude $\Oo(m)$,  de taille  $ \Oo(m+n) $   et qui est construit en temps 
linéaire.
\end{flemma}

\proof Nous utilisons comme Hoover la méthode de Newton qui permet de calculer l'inverse d'un réel $z$ à $2^{-n}$ prés en un nombre raisonnable 
d'additions et multiplications en fonction de $n$.\\
{\bf Méthode de Newton} (pour le calcul de l'inverse de $z$)\\
Pour   $2^{-m} < z < 2^m$  on pose 
\[
C(x) = \left\{
\begin{array}{l} 
y_0 = 2^{-m} \\ 
y_{i+1} = y_i (2- zy_i) 
\end{array}
\right.
\]
On vérifie facilement que, pour    $i \geq 3m + \log(m+n)$,  on a :
\[
\abs{ z^{-1} - y_i }\;  \leq 2^{-n}.
\]
Ainsi pour l'entrée  $(n,m) \in \NN_1 \times \NN_1$,  
l'application de la méthode de Newton jusqu'à l'itération  
$i = 3m + \log(m+n)$  est représentée par un circuit \poll de taille 
\Oo(m+n) et de magnitude  $\Oo(m+n)$.  De plus, il est facile de vérifier que la 
construction de ce circuit se fait en temps linéaire.
  \eop

\begin{fproposition} \label{f4210}
L'identité de  $\czu$  de  $\capo$  vers  $\ckf$  est \uni de classe  $\p$, en 
fait de classe  $\DRT(\Lin, \Oo(N^4))$.
\end{fproposition} 

\proof 
Le fait d'être de classe  $\p$   résulte du théorème \ref{f417} et de la 
proposition \ref{f4111}.  La lecture précise des démonstrations du théorème 
\ref{f417} et de la proposition \ref{f4111} donne le résultat    
$\DRT(\Lin,\Oo(N^4))$   (en tenant compte du Nota bene après \ref{f4111}).  
\eop

\smallskip En résumant les résultats qui précèdent nous obtenons le théorème 
suivant.
\begin{ftheorem} \label{f4211}
Les cinq présentations    $\ckf$,  $\cbo$,  $\csl$,  $\caf$   et   $\capo$  de  
$\czu$  sont   $\p$-équivalentes.
\end{ftheorem}

\begin{fnotation} \label{f4212}
Pour autant qu'on se situe à un niveau de \com suffisamment élevé pour 
rendre le théorème de comparaison valable (en particulier la classe  $\p$  
suffit)  il n'y a pas de raison de faire de différence entre les cinq 
présentations  $\ckf$,  $\cbo$,  $\csl$,  $\caf$   et   $\capo$  de  $\czu$. 
En conséquence, désormais la notation  $\czu$  signifiera qu'on considère 
l'espace  $\czu$  avec la structure de calculabilité  $\csl$.
\end{fnotation}

\subsection{Complexité du problème de la norme}\label{fsubsec43}
On sait que la détermination du maximum,  sur un intervalle  $[0,1]$   d'une 
fonction calculable \etpo  $f \colon  \NN \rightarrow \{ 0,1 \}$ 
 est à très peu près la même chose que le problème $\np$-complet le 
plus classique:  SAT. 
 
Il n'est donc pas étonnant de trouver comme problème  
$\np$-complet un problème relié au calcul de la norme pour une fonction 
continue. Il nous faut tout d'abord formuler le problème de la norme attaché 
à une \rp  donnée de l'espace  $\czu$  d'une manière suffisamment 
précise et invariante. 
   
\begin{fdefinition}\label{f431}~\\
Nous appelons ``problème de la norme'', relativement à une \pres $\ca_1 = 
(Y_1, \delta_1, \eta_1)$ de  $\czu$)  le problème:\\
--- Résoudre {\em approximativement}  la question  
``$\; a \leq \norme{ f }_{\infty}\;?\;$''  dans la \pres  $\ca_1$  de $\czu$.\\
La formulation précise de ce problème est la suivante:\\
--- Entrées:  $(f,a,n) \in Y_1 \times \DD \times \NN_1$ \\
--- Sortie:    fournir correctement une des deux réponses:\\
\hspace*{1cm}--- voici un $x \in \DD$  tel que 
$\abs{f(x)}  \geq a - 1/2^n$, i.e.\  $\norme{f}_{\infty}\geq a - 1/2^n$, \\
\hspace*{1cm}--- il n'existe pas de  $x \in \DD$  vérifiant  
$\abs{f(x)}  \geq a$,  i.e.\  $\norme{f}_{\infty}\leq a$.       
\end{fdefinition}

Cette définition est justifiée par le lemme suivant.

\begin{flemma} \label{f432}
Pour deux \rps   $\ca_1$ et $\ca_2$  de  $\czu$  \polt équivalentes, 
les problèmes de la norme correspondants sont aussi \polt 
équivalents.
\end{flemma}

\proof 
La transformation du problème correspondant à une \pres $\ca_1$ au 
problème correspondant à une autre \pres  $\ca_2$  se fait par un algorithme 
ayant la même \com que l'algorithme qui permet de présenter la fonction 
identité entre  $\ca_1$  et  $\ca_2$. En effet, pour les données  $(f,a,n) 
\in Y_1 \times \DD \times \NN_1$, on cherche 
$g \in Y_2$ telle que  $\norme{ f-g }_{\infty} \leq 2^{-(n+2)}$,
 puis on résout le problème avec les entrées  $(g,a-1/2^{n+1}, n+2)$. 
 \\
Si on trouve  $x \in \DD$  tel que 
$\abs{g(x)}  \geq a - 2^{-(n+1)}- 2^{-(n+2)}$,  alors 
$\abs{f(x)}  \geq a - 2^{-n}$. 
\\ 
Si on déclare forfait, c'est qu'il n'existe pas de $x \in \DD$ tel que 
$\abs{g(x)}\;  \geq a - 2^{-(n+1)}$. 
A fortiori, il n'existe pas de  
$x \in \DD$  tel que  $\abs{f(x)} \geq a$.   
\eop

\begin{ftheorem} \label{f433}
Le problème de la norme est $\np$-complet pour les présentations  $\ckf$,  
$\cbo$,  $\csl$,  $\caf$   et   $\capo$.
\end{ftheorem}

\proof 
D'après le lemme \ref{f432} il suffit de faire la preuve pour la \pres $\cbo$. 
Le caractère  $\np$  du problème est immédiat.  
Pour voir la  $\np$-dureté, nous considérons le problème de la norme 
limité aux entrées  $((\gamma,1,m,1), 3/4, 2)$  où  $\gamma$  
est un circuit booléen arbitraire à  $m$  entrées et une sortie 
(le quadruplet est alors évidement correct), et la réponse oui 
correspond à la satisfiabilité du circuit  $\gamma$.  \eop

\medskip On a également le résultat suivant, essentiellement 
négatif{\footnote{Puisque  $\p \neq \np$ !}}, et donc moins intéressant.

\begin{fproposition} \label{f434}
Pour les présentations considérées, la fonction norme  
$f \mapsto \norme{ f }_{\infty}$ de $\czu$  vers  $\RR^{+}$ est \uni 
de classe $\p$  si et seulement si  $\p = \np$.
\end{fproposition}

\proof 
Il suffit de raisonner avec la \pres  $\cbo$. Si la fonction norme est  $\p$-
calculable{\footnote{Nous utiliserons quelquefois la terminologie  $\p$-
calculable   comme abréviation pour  calculable en temps \poll.}}, le 
problème de la norme se résout en temps \poll, et donc  $\p$ = $\np$.\\
Si  $\p = \np$, le problème de la norme se résout en temps \poll, ce 
qui permet de calculer la norme par dichotomie en initialisant avec la 
majoration $2^k$, jusqu'à obtenir la précision  $1/2^q$. Ceci réclame  
$k+q$ étapes de dichotomie. 
L'ensemble du calcul est \etpo sur l'entrée  
$(g,q) \in \ybo  \times \NN_1$.  \eop

\begin{fcorollary} \label{f435}
L'identité de  $\czu$  de  $\ckf$  vers  $\crf$  n'est pas calculable en temps 
\poll, au moins si  $\p \neq \np$.
\end{fcorollary}

\proof 
La fonction norme est calculable \etpo pour la \pres  $\crf$  d'après la 
proposition  \ref{f327}.  On conclut par la proposition précédente.  \eop 

\begin{fproposition} \label{f436}
Pour les cinq présentations précédentes de $\czu$, si l'évaluation est 
dans  $\DSPA(S(N))$  avec $S(N) \geq N$, alors la fonction norme est aussi dans 
$\DSPA(S(N))$.
\end{fproposition}

\proof 
Si   $k \mapsto \mu(k)$  est un module de continuité d'un point rationnel  
$\wi{f}$    d'une \pres donnée de $\czu$, alors pour calculer la norme 
avec une précision  $n$  il suffit d'évaluer  $\wi{f}$    sur les 
éléments de $\DD_{m,[0,1]}$  (où  $ m = \mu(n)$)  et prendre la valeur 
maximale. 
Puisque les résultats intermédiaires inutiles sont immédiatement 
effacés, et puisque  $S(N) \geq N \geq m$ l'espace de calcul de la fonction 
norme est le même que celui de la fonction d'évaluation.  \eop

\subsection{Complexité des \rps  étudiées}\label{fsubsec44}
Dans cette section, nous présentons rapidement la complexité du test 
d'appartenance et des opérations d'espace vectoriel, pour les cinq 
présentations considérées, et nous récapitulons l'ensemble des 
résultats obtenus. 

\begin{fproposition} \label{f441}
Le test d'appartenance (à l'ensemble des codes des points rationnels) est:\\
--- $\LINS$  et $\cnp$-complet pour les présentations  $\ckf$ et  $\cbo$;  \\
--- $\LINT $ pour la \pres $\csl$.
\end{fproposition}

\proof    Pour la \pres  $\csl$, c'est évident. Les preuves sont 
essentiellement les mêmes pour les deux présentations  $\ckf$  et  $\cbo$. 
Nous n'en donnons qu'une chaque fois.\\
Voyons que le test d'appartenance est  $\LINS$  pour  $\ykf$.  
Pour une entrée  $(Pr,n,m,T)$  on fait le calcul suivant: \\
 pour  $i = 1,\ldots,2^m$   vérifier que
\[
Pr(i/2^m) \in \DD_n \quad  \hbox{et} \quad  \Abs{ Pr((i-1)/2^m) - Pr(i/2^m) }\;  
\leq 1/2^n
\]
Ce calcul est $\LINS$.\\
Pour la $\cnp$-complétude du test d'appartenance, nous donnons la preuve pour 
les circuits booléens. On peut se limiter aux entrées  $(\gamma,2,m,0)$  
où  $\gamma$  est un circuit qui ne calcule qu'une sortie correspondant au 
premier bit, les deux autres bits sont nuls. On demande la cohérence sur deux 
points consécutifs de la grille. Seules les fonctions constantes sont donc 
tolérées.\\
Le problème opposé du test d'appartenance revient à savoir si un circuit 
booléen est non constant, ce qui implique la résolution du problème de 
satisfiabilité.  \eop

\begin{fproposition} \label{f442}
 Les opérations d'espace vectoriel (sur l'ensemble des points rationnels) sont 
dans $\LINT $ pour les cinq présentations  $\ckf$,  $\cbo$,  $\csl$,  $\caf$   
et   $\capo$  de  $\czu$.
\end{fproposition} 

\proof Les calculs sont évidents. Par exemple si 
$(f_1,\ldots,f_s) \in \lst(\ykf )$  et  $n \in \NN_1$, on peut calculer 
facilement 
$f \in \ykf $  tel que: \\
$$\NOrme{ \wi{f} - \sum_{i=1,\ldots,s} \wi{f_i} } \leq 
1/2^n$$ 
car il suffit de connaître chaque $\wi{f_i}$  avec la précision 
$1/2^{n+\log(s)}$.  \eop

Le seul ``drame'' est évidemment que les présentations $\ckf$ et $\cbo$ ne 
sont pas des $\p$-présentations de  $\czu$ (sauf si $\p = \np$  cf. la 
proposition \ref{f434}.) \\
Pour terminer cette section nous donnons un tableau récapitulatif dans lequel 
nous regroupons presque tous les résultats de \com établis pour les cinq 
\rps   $\ckf$, $\cbo$,  $\csl$,  $\caf$   et   $\capo$  de l'espace $\czu$.

\bni
\begin{tabular}{ c l p{3.5cm} c  c }  
& Evaluation & Fonction  & Test 
& Opérations\cr
&  & Norme & d'appartenance 
&  d'espace \cr
&&&&vectoriel\cr\cr
\ni $\ckf$  & $\DSRT(\Lin,\Lin,\Oo(N^2))$ & $\LINS$ et 
&  $\LINS$ et & $\LINT $\cr
\ni et  $\cbo$ &  &$\np$-complet  
&   $\cnp$-complet & \cr\cr
\ni $\csl$ & $\DSRT(\Oo(N^2),\Lin,\Oo(N^2))$ & $\DSPA(\Oo(N^2))$ 
& $\LINT $ & $\LINT $\cr
\ni  &  &et $\np$-complet\cr\cr
\ni $\caf$  & $\DSRT(\Oo(N^3),\Lin,\Oo(N^4))$ & 
$\DSPA(\Oo(N^3))$ 
& $\PSP$ & $\LINT $\cr
\ni et  $\capo$ &  &et $\np$-complet
\end{tabular}

\bni Pour le test d'appartenance $\PSP $ il est probable que sa complexité 
soit bien moindre.

\begin{fremark} \label{f444}
Malgré la facilité de calcul de la fonction d'évaluation pour les 
présentations $\ckf$  et   $\cbo$, c'est encore la \pres par circuits 
semilinéaires binaires qui semble au fond la plus simple. 
Sa considération a permis en outre d'éclairer le théorème de comparaison 
\ref{f4211}, qui est une version renforcée, uniforme, des résultats établis 
par Hoover. \\
Le défaut inévitable (si $\p \neq \np$) des présentations définies 
jusqu'à maintenant est la non faisabilité du calcul de la norme. 
Ceci empêche d'avoir une procédure de contrôle faisable pour les suites 
de Cauchy de points rationnels. 
Ceci diminue d'autant l'intérêt des  $\p$-points de  $\csl$. Cela souligne 
bien le fait qu'il est un peu artificiel d'étudier les  $\p$-points d'un 
espace qui est donné dans une \pres de \com non \polle.\\
En outre des problèmes a priori au moins aussi difficiles que la calcul de la 
norme, comme par exemple le calcul d'une primitive ou la solution d'une 
équation différentielle, sont également sans espoir de solution 
raisonnable dans le cadre des présentations que nous venons d'étudier.\\ 
Il est donc légitime de se tourner vers d'autres \rps  de l'espace  $\czu$  
pour voir dans quelle mesure elles sont mieux adaptées aux objectifs de 
l'analyse numérique.
\end{fremark}

\section{Quelques présentations de classe  \texorpdfstring{$\p$}{P} 
pour l'espace   \texorpdfstring{$\czu$}{C[0,1]}}\label{fsec5}
Dans cette section on aborde la question de savoir jusqu'à quel point des \rps  
dans la classe  $\p$  de l'espace  $\czu$  fournissent un cadre de travail 
adéquat pour l'analyse numérique. Il ne s'agit que d'une première étude, 
qui devrait être sérieusement développée.

\subsection {Définitions de quelques présentations de classe  
$\p$}\label{fsubsec51}

\subsubsection{Présentation  \texorpdfstring{$\cw$}{Cw}
 (à la Weierstrass)}\label{fsubsubsec511}
L'ensemble  $\yw$ des codes des points rationnels   de la \pres  $\cw$  est simplement l'ensemble  $\DD[X]$  des \pols (à une variable) à coefficients dans  $\DD$   donnés en \pres dense.

\smallskip Ainsi $\cw=(\yw,\eta,\delta)$ où la lectrice donnera précisément $\eta$ et $\delta$ conformément à la définition~\ref{f211}.

\smallskip Un  $\p$-point $f$ de  $\cw$  est donc donné par une suite $\p$-calculable:  
\[
m \mapsto u_m \; : \; \NN_1 \rightarrow \DD[X] \quad \hbox{avec} \quad 
\forall m \; \norme{ u_m - f }_{\infty} \leq 1/2^m.
\]

Et une  $\p$-suite  $f_n$ de  $\cw$  est donnée par une suite double  $\p$-calculable: 
\[
(n,m) \mapsto u_{n,m} \; : \; \NN_1 \times \NN_1 \rightarrow \DD[X] \quad \hbox{avec} \quad \; \forall n,m \; \norme{ u_{n,m} - f_n }_{\infty} 
\leq 1/2^m.
\]

\begin{fremark} \label{f511}
Une définition \equiva pour un $\p$-point $f$ de  $\cw$  est obtenue en 
demandant que $f$ s'écrive comme somme d'une série $\sum_m s_m$,
  où  $(s_m)_{m\in\NN_1}$  est une suite  $\p$-calculable dans~$\DD[X]$   vérifiant:  
$\norme{s_m}_{\infty} \leq 1/2^m$.  Ceci donne une manière 
agréable de présenter les  $\p$-points de~$\cw$. 
En effet on peut contrôler \etpo (par rapport à  $m$)  le fait que la suite est correcte pour les termes de $1$ à  $m$.  
En outre, dans l'optique d'un calcul paresseux, on peut contrôler la somme de série $\sum_m s_m$ au fur et à mesure que la précision requise augmente. 
Cette remarque est valable pour toute autre \rp  de classe  $\p$  alors qu'elle ne le serait pas pour les présentations étudiées dans la section \ref{fsec4}.
\end{fremark}

\smallskip Le résultat suivant est immédiat.

\begin{fproposition} \label{f512}
La \pres  $\cw$  de  $\czu$  est de classe  $\p$.
\end{fproposition}

Voici un résultat comparant les \rps   $\crf$  et $\cw$.

\begin{fproposition} \label{f513}~\\
--- L'identité de  $\czu$  de  $\cw$  vers  $\crf$  est  $\LINT $.\\
--- L'identité de  $\czu$  de  $\crf$  vers  $\cw$  n'est pas de classe  $\p$.
\end{fproposition}

\proof La première affirmation est triviale. 
La seconde résulte du fait que la fonction  $x \mapsto \abs{ x-1/2} $  est 
un  $\p$-point de  $\crf$  (théorème \ref{f334}) tandis que tous les 
$\p$-points de  $\cw$  sont des fonctions infiniment dérivables 
(cf. ci-dessous le théorème \ref{f527}).  \eop 

\smallskip L'intérêt de la \pres   $\cw$  est notamment souligné par les 
théorèmes de caractérisation (cf. section \ref{fsubsec52}) 
qui précisent des phénomènes ``bien connus'' en analyse numérique, avec 
les \pols de Chebyshev comme méthode d'attaque des problèmes.

\subsubsection{Présentation   \texorpdfstring{$\csp$}{Csp}
(via des semi-\pols en \pres 
dense)}\label{fsubsubsec512}
Il s'agit d'une \pres qui augmente notablement l'ensemble des $\p$-points (par 
rapport à~$\cw$). Un élément de  $\ysp$   représente une fonction 
\polle par morceaux (ou encore un  semi-\pol) donné ``en \pres 
dense''. 

Plus précisément $\ysp \subset \lst(\DD) \times \lst(\DD[X])$, et les 
deux listes dans $\DD$ et $\DD[X]$ sont assujetties aux conditions 
suivantes:\\
--- la liste  $(x_i)_{0 \leq i \leq t}$ de nombres rationnels dyadiques est 
ordonnée par ordre croissant: 
$$0 = x_0 < x_1 < x_2 <\cdots< x_{t-1} < x_t =1\,;$$
--- la liste  $(P_i)_{1 \leq i \leq t}$ dans $\DD[X]$
vérifie  $P_i(x_i) = P_{i+1}(x_i)$  pour  $ 1 \leq i \leq t-1$.\\
Le code $f = ((x_i)_{0 \leq i \leq t},(P_i)_{1 \leq i \leq t})$ 
définit le point rationnel  $\wi{f}$: la fonction continue qui coïncide avec $\wi{P_i}$ sur l'intervalle $[x_{i-1},x_i]$. 

\smallskip Ainsi $\csp=(\ysp,\eta,\delta)$ où \smash{$\eta(f)=\wi f$} et le lecteur donnera  $\delta$ conformément à la définition~\ref{f211}.

\smallskip La \pres  $\csp$  de  $\czu$  est clairement de classe  $\p$.

La proposition suivante se démontre comme la proposition \ref{f513}.

\begin{fproposition} \label{f514} ~\\
--- L'identité de  $\czu$  de  $\cw$  vers  $\csp$  est  $\LINT $.\\
--- L'identité de  $\czu$  de  $\csp$  vers  $\cw$  n'est pas de classe  $\p$.
\end{fproposition} 

\subsubsection{Présentation $\csr$ (via des semi-fractions rationnelles 
contrôlées et données en \pres par formule)}\label{fsubsubsec513}
L'ensemble  $\ysr$ des codes des points rationnels  est maintenant un ensemble qui code des fonctions rationnelles par morceaux (ou semi-fractions rationnelles) à coefficients dyadiques et qui sont convenablement contrôlées. 

\smallskip Plus précisément, $\ysr  \subset \lst(\DD) \times \lst(\DD[X]_f \times \DD[X]_f)$,  et les deux listes dans $\DD$ et~$\DD[X]_f \times \DD[X]_f$  sont assujetties aux conditions suivantes:\\
--- la liste  $(x_i)_{0 \leq i \leq t}$ de nombres rationnels dyadiques est 
ordonnée par ordre croissant: 
$$0 = x_0 < x_1 < x_2 <\cdots< x_{t-1} < x_t =1;$$
--- chaque couple   $(P_i,Q_i) \; (1 \leq i \leq t$ de la 2ème liste 
représente une fraction rationnelle  $R_i = P_i/Q_i$    avec le 
dénominateur $Q_i$ minoré par  $1$  sur l'intervalle $[x_{i-1} , x_i]$;\\
--- la liste  $(R_i)_{1 \leq i \leq t}$  vérifie   $R_i(x_i) = R_{i+1}(x_i)$
   pour  $1 \leq i < t$.
\\
Le code  $f = ((x_i)_{0 \leq i \leq t},(P_i,Q_i)_{1 \leq i \leq t})$  
définit le point rationnel \smash{$\wi{f}$}: la fonction continue qui coïncide avec $\wi{R_i}$  sur chaque intervalle  $[x_{i-1} , x_i]$.

\smallskip Ainsi $\csr=(\ysr,\eta,\delta)$ avec $\eta(f)=\wi f$ et la lectrice donnera $\delta$ conformément à la définition~\ref{f211}.

\smallskip La \pres  $\csr=(\ysr,\eta,\delta)$  de  $\czu$  est clairement de classe  $\p$.

\subsection {Résultats concernant la \pres à la Weierstrass}\label{fsubsec52}

Cette section est pour l'essentiel un développement de la section C-c de l'article \cite{fLo89} qui reprend la troisième partie la thèse du deuxième auteur.

Les références classiques de base pour la théorie de l'approximation sont
\cite[Bakhvalov, 1973]{fBa},
\cite[Cheney, 1966]{fCh} et \cite[Rivlin, 1974]{fRi}. Pour la classe de Gevrey on se réfère à \cite[Hörmander, 1983]{fHo}. Nous suivons les notations de \cite{fCh}.  

\smallskip Avant de caractériser les $\p$-points de $\cw$, il nous faut rappeler quelques 
résultats classiques de la théorie de l'approximation uniforme par des 
\pols.

\mni {\bf Attention~!} Vu la manière usuelle dont est formulée la théorie 
de l'approximation, on utilisera l'intervalle     $[-1,1]$ pour donner les 
résultats et les preuves concernant  $\cw$.

\subsubsection{Quelques définitions et résultats de la théorie de 
l'approximation uniforme par des \pols}\label{fsubsubsecf521}

Voir par exemple  \cite{fBa}, \cite{fRi}  et   \cite{fCh}. 

\begin{fnotation} \label{f521} ~ 
\begin{itemize}\itemsep2pt
\item  $\cab$ est l'espace des fonctions réelles continues sur le segment 
$[a,b]$.
\item  
$\C$ est l'espace  $\cuu$,  la norme uniforme sur cet intervalle est notée  
$\norme{ f }_{\infty}$   et  la distance correspondante $d_{\infty}$. 
\item   
$\C^{(k)}$  est l'espace des fonctions $k$ fois continûment dérivables sur 
$[-1,1]$. 
\item  
$\Ci$  est l'espace des fonctions  indéfiniment dérivables sur  $[-1,1]$.  
\item  
$\po_n$ est l'espace des \pols de degré  $\leq n$.  
\item  
$\Tch_n$  est le \pol de Chebyshev de degré $n$:  
\[
\Tch_n\big(\varphi (z)\big)= \varphi(z^n) \; \hbox{avec} \; \varphi(z) = \frac {1} {2} (z 
+ 1/z)
\]
on peut également les définir par    $\Tch_n\big(\cos (x)\big) = \cos (nx)$    ou par  
\[
F(u,x) = \frac {1-u\,x} {1-u^2-2u\,x} = \sum_{n=0}^{\infty} \Tch_n(x)u^n.
\]
\item  
On note: \fbox{$E_n(f) = d_{\infty} (f, \po_n)$}    pour  $f \in \C$.    
\item  On considère   sur $\C$  le produit scalaire  
\[
\left < g,h \right > := \int_{-1}^1 \frac {g(x)\,h(x)} {\sqrt {1-x^2}}\, dx 
= \int_0^{\pi} g\big(\cos (x)\big)\,h \big(\cos (x)\big)dx. 
\]
On notera  $\norme{ f }_2$  la norme au sens de ce produit scalaire.  
Les \pols  $(\Tch_i)_{0 \leq i \leq n}$   forment une base orthogonale de  
$\po_n$   pour ce produit scalaire, avec  
\[
\left < \Tch_0,\Tch_0 \right > = \pi \; \hbox{et} \; \left < \Tch_i,\Tch_i \right > = 
\pi /2 \; \hbox{pour} \; i>0.
\]
\item  On note
\[
A_k = A_k(f):= \frac {2} {\pi} \int_{-1}^{1}  \frac {f(x)\,\Tch_k(x)} {\sqrt {1-x^2}}\, {dx}= \frac {2} {\pi} \int_{0}^{\pi} \cos (kx) f(\cos (x))dx
\]
Les  $A_k$  sont appelés les  {\em coefficients de Chebyshev}  de  $f$.
\item  La fonction 
\[
s_n(f) := A_0/2 + \sum_{i=1}^{n}A_i\Tch_i, \quad   \hbox{ aussi notée } 
 {\sum_{i=0}^n}\,'\, A_i\Tch_i
\]  
est la projection orthogonale de $f$ sur $\po_n$ au sens du produit scalaire considéré.    
\item  
La série correspondante est appelée  {\em la série de Chebyshev}  de 
$f$~{\footnote{Elle converge au sens de  $L^2$  pour le produit scalaire 
considéré.  Les séries de Chebyshev sont aux fonctions continues sur  
$[-1,1]$  ce que les séries de Fourier sont aux fonctions continues 
périodiques, ce qui se comprend bien en considérant le ``changement de 
variable''  $z\mapsto 1/2(z + 1/z)$   qui transforme le cercle unité du plan complexe en le segment   $[-1,1]$  et la fonction  $z\mapsto z^n$  en le \pol~$\Tch_n.$}}. 
\item  
On note  \fbox{$S_n(f) = \norme{ f - s_n(f) }_{\infty}$}.        
On a immédiatement   $\abs{ A_{n+1}(f) }\;  \leq S_n(f) + S_{n+1}(f)$. 
\item  
Les  zéros de  $\Tch_n$   sont les 
\[
\xi_i^{[n]} = \cos \left(\frac{2i-1} {n}\cdot \frac {\pi} {2}\right) \; \; \; i =1,\ldots,n
\]
et l'on a  
\[
\Tch_n(x) = 2^{n-1} \prod_{i=1}^{n} (x - \xi_i^{[n]}) \; \; (  
\hbox{pour}\; n \geq 1).
\]          
\item  
Les  extrema de  $\Tch_n$  sur  $[-1,1]$  sont égaux à $\pm 1$ et obtenus aux points
\[
\eta_i^{[n]} = \cos \left(\frac{i} {n}\cdot \pi\right) \; \; \; i=0,\ldots,n.
\] 
\item  
Des valeurs approchées de  $s_n(f)$ peuvent être calculées au moyen de 
formules d'interpolation:  on pose   
\[
\alpha_k^{[m]} = \frac {2} {m} \;
{\sum_{i=0}^m}\,'f(\xi_i^{[m]})\Tch_k(\xi_i^{[m]}), \; \; \; \; u_n^{[m]} = 
\sum_{k=1}^{n} \alpha_k^{[m]}\Tch_k(x)
\]
et l'on a:  $u_n^{[n+1]}$   est le \pol qui interpole $f$ aux zéros de  
$\Tch_{n}$
\end{itemize}
\end{fnotation}
La théorie de l'approximation uniforme par des \pols établit des liens 
étroits entre ``être suffisamment bien approchable par des \pols''  et  
``être suffisamment régulière''.

\subsubsection{Quelques résultats classiques}\label{fsubsubsec 522}

Ces résultats se trouvent pour l'essentiel dans \cite{fCh}.

\smallskip  Dans ce paragraphe les fonctions sont dans $\C=\cuu$.

\smallskip \noindent {\bf Évaluation d'un \pol} $P=\sum_{k=0}^{n}a_k\Tch_k$\\
 Les formules récurrentes  $\Tch_{m+1}(x) = 2x\Tch_m(x) - \Tch_{m-1}(x)$  conduisent à un algorithme à la Horner:
\[
B_{n+1} = B_{n+2} = 0, \quad  B_k = 2xB_{k+1} - B_{k+2} + a_k, \quad  P(x) = \frac 
{B_0 - B_2} {2}.
\]

\smallskip\noindent {\bf Inégalités de Markov} \\
Si  $g \in \po_n$ alors (A.A. Markov, \cite[page 91]{fCh})    
\begin{equation} \label{fF 5.2.1}
\Norme{ g' }_{\infty} \leq n^2 \norme{ g }_{\infty}
\end{equation}
et pour  $k \geq 2$ (V.A. Markov, \cite[Theorem 2.24]{fRi})   
\begin{equation} \label{fF 5.2.2}
\Norme{ g^{(k)} }_{\infty} \leq \Tch_n^{(k)}(1) \norme{ g 
}_{\infty} = \frac {n^2(n^2-1)\cdots(n^2-(k-1)^2)} {1.3.5\cdots(2k-1)} 
\,\norme{ g }_{\infty}
\end{equation}

\smallskip\noindent {\bf Comparaison de $E_n(f)$ et $S_n(f)$}
\begin{equation} \label{fF 5.2.3}
E_n(f) \leq S_n(f) \leq \left(4+ \frac {4} {\pi^{2}} \log (n)\right)\,E_n(f)
\end{equation}

\smallskip\noindent {\bf  Comparaison de $E_n(f)$ et $A_{n+1}(f)$}\\
Pour  $n \geq  1$  on a	
\[
\int_{-1}^{1} \frac {\abs{ \Tch_n(x) }} {\sqrt {1-x^2}}\,dx = 2
\]
d'où l'on déduit   
\begin{equation} \label{fF 5.2.4}
(\pi /4) \abS{A_{n+1}(f)}   \leq E_n(f)
\end{equation}

\smallskip\noindent{\bf Théorèmes de Jackson}\\
Soit $f \in \C$. Pour tout entier $n \geq 1$ on a 
\begin{equation} \label{fF 5.2.5}
E_n(f) \leq \pi \lambda /(2n+2) \; \; \; \hbox{si} \; \; \; \abs{f(x)-f(y)}\;  \leq \lambda \abs{x-y}
\end{equation}
\begin{equation} \label{fF 5.2.6}
E_n(f) \leq (\pi /2)^k \Norme{ f^{(k)} }_{\infty} \big/ \big((n+1)(n)(n-
1)\cdots(n-k+2)\big) \; \;\; \hbox{si} \; f \in \C^{(k)} \; \hbox{et} \; n\geq k
\end{equation}

\smallskip\noindent{\bf Convergence de la série de Chebyshev d'une fonction} 
\\
La série de Chebyshev d'une fonction  $f \in \C^{(k)}$  converge \uni  vers 
$f$ si  $k \geq 1$,   et elle est absolument convergente (pour la norme  
$\norme{ f }_{\infty}$)  si  $k \geq 2$.   
\begin{equation} \label{fF 5.2.7}
S_n(f) = \norme{ s_n(f)-f }_{\infty}\; \leq \sum_{j=n+1}^{\infty} 
\abS{A_j}
\end{equation}
et (cf. \cite{fRi} Theorem 3.12 p. 182)
\begin{equation} \label{fF 5.2.8}
\NOrme{ s_n(f)- u_n^{[n+1]} }_{\infty}\; \leq \sum_{j=n+2}^{\infty} 
\abS{A_j}
\end{equation}

\smallskip\noindent{\bf Approximation uniforme des fonctions dans
 $\Ci$  par des \pols} \\
Les propriétés suivantes sont équivalentes. 
\begin{itemize}\itemsep2pt
\item [(i)]  $\forall k \; \; \exists M > 0 \; \; \forall n> 0 \; \; E_n(f) \leq M/n^k$; 
\item [(ii)]
 $\forall k \; \; \exists M > 0 \; \; \forall n> 0 \; \; S_n(f) \leq M/n^k$;
\item [(iii)] 
$\forall k \; \; \exists M > 0 \; \; \forall n> 0 \; \; \abS{A_n(f)}  \leq 
M/n^k$;
\item [(iv)] $\forall k \; \exists M > 0 \; \forall n> 0 \; \NOrme{ u_n^{[n+1]} - f }_{\infty} \leq M/n^k$;
\item [(v)] 
La fonction  $f$ est de classe $\ca^{\infty}$ (i.e., $f \in \Ci$).
\end{itemize}

\proof 
(i)  et  (ii)  sont équivalents d'après  (\ref{fF 5.2.3}). \\   
(iv)$\Rightarrow$ (i)  trivialement.  \\  
(ii) $\Rightarrow$ (iii)  parce que  $\abS{A_n(f)}\;  \leq S_n(f) + S_{n-
1}(f)$.  \\
(iii) $\Rightarrow$ (iv)  d'après  (\ref{fF 5.2.7})  et  (\ref{fF 5.2.8}).  \\  
(iii) $\Rightarrow$ (v): la série  $\sum' A_i\Tch_i^{(h)}$  est  absolument 
convergente d'après  (\ref{fF 5.2.2})  et les majorations  (iii)~;  donc on 
peut dériver $h$ fois terme à terme la série de Chebyshev.    \\ 
(v) $\Rightarrow$ (i)  d'après  (\ref{fF 5.2.6}).  \eop

\smallskip\noindent{\bf Analyticité et approximation uniforme par des 
\pols} \\
Les propriétés suivantes sont \equivas:   
\begin{itemize}\itemsep2pt
\item [(i)] $ \exists M > 0, \; \; r < 1 \; \; \forall n> 0, \; \; E_n(f) \leq Mr^n$;

\item [(ii)] $\exists M > 0, \; \; r < 1 \; \; \forall n> 0, \; \; S_n(f) \leq Mr^n$; 

\item [(iii)] $\exists M >0, \; \; r < 1 \; \forall n> 0, \; \; \abS{A_n(f)} \leq Mr^n $;

\item [(iv)] $ \exists M > 0, \;  r < 1 \; \; \forall n> 0,  \; \NOrme{ u_n^{[n+1]} - f }_{\infty}   \leq Mr^n$;

\item [(v)] $\exists r < 1$ 
telle que  $f$ est analytique dans le plan complexe à l'intérieur                      
de l'ellipse  $\sE_\rho$  de foyers  $1$, $-1$  et dont le demi-somme des                      
diamètres principaux est égale  à  $\rho = 1/r$;

\item [(vi)] $ \exists M > 0, \;  R > 1 \; \; \forall n, \;  \NOrme{f^{(n)} }_{\infty}  \leq MR^nn!$\,.

\item [(vii)]  $f$ est analytique sur l'intervalle $[-1,1]$.
\end{itemize}
En outre la limite inférieure des valeurs de  $r$  possibles est la même 
dans les  5 premiers cas{\footnote{Les équivalences   (i) \ldots (iv)  se 
démontrent comme pour la proposition précédente.  Pour l'équivalence avec  
(v)  voir par exemple \cite{fRi}.  La condition  (vi)   représente à très 
peu près l'analyticité dans l'ouvert  $U_R$  formé des points dont la 
distance à l'intervalle est inférieure à  $1/R$.}}.

\begin{figure}[htbp]  
\begin{center}
\includegraphics*[width=10cm]{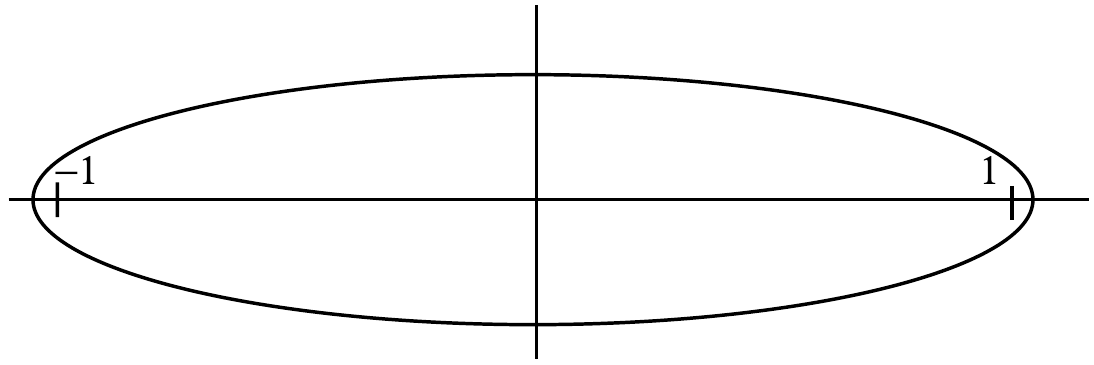}
\end{center}
\caption[L'ellipse $\sE_\rho$]{\label{ffi521}  
L'ellipse $\sE_\rho$}  
\end{figure}  

\begin{fremarks}\label{frem-anal} ~\\
1) L'espace des fonctions analytiques sur un intervalle compact possède donc 
une bonne description constructive, en termes de série de Chebyshev par 
exemple.  Il apparaît comme une réunion dénombrable emboîtée 
d'espaces métriques complets  (ceux obtenus en utilisant la définition  
(iii)  et en fixant  $M$  et  $r$  rationnels par exemple).  
L'espace des fonctions  $\Ci$  est beaucoup plus difficile à décrire 
constructivement, essentiellement parce qu'il n'existe pas de manière 
agréable d'engendrer les suites de rationnels à 
décroissance rapide{\footnote{Cela tient au  $\forall k  \; \exists M $  dans 
la définition de la décroissance rapide.  
Cette alternance de quantificateurs prend une forme explicite lorsqu'on donne  
$M$  en fonction de  $k$  explicitement.  Mais, en vertu de l'argument diagonal 
de Cantor, il n'y a pas de manière effective d'engendrer les fonctions 
effectives de $\NN$   vers $\NN$.}}. \\
2) La condition  (i)  peut être également lue comme suit:  
la fonction $f$ peut être approchée à  $1/2^n$  (pour la norme uniforme)  
par un \pol de degré $\leq  c.n$,  où  $c$  est une constante 
fixée,  c.-à-d. encore: il existe un entier $h$ tel que
  $E_{hn}(f) \leq 1/2^n$. \\
   Même remarque pour les conditions  (ii),  (iii)  et  (iv). Cela implique 
que la fonction $f$ peut être approchée à  $1/2^n$  par un \pol à 
coefficients dyadiques dont la taille (en \pres dense sur la base des  $X^n$  ou 
sur la base des  $\Tch_n$)  est   en  $\Oo(n^2)$.  La taille de la somme des 
valeurs absolues des coefficients est, elle, en  $\Oo(n)$.    
Bakhvalov  (cf.  \cite {fBa}  IV-8  Th. p. 233)  donne une condition 
suffisante du même genre pour qu'une fonction $f$ soit analytique dans une 
lentille d'extrémités   $-1$  et  $1$  du plan complexe  (et non plus dans 
un voisinage du segment):  il suffit que la somme des valeurs absolues des 
coefficients d'un \pol donnant $f$ à  $1/2^n$  près soit majorée par  
$M2^{qn}$  (où~$M$ et $q$  sont des constantes fixées). C.-à-d. encore: 
la taille de la somme des valeurs absolues des coefficients d'un \pol 
approchant $f$ à  $1/2^n$  est en  $\Oo(n)$.
\begin{figure}[htbp]  
\begin{center}
\includegraphics*[width=12cm]{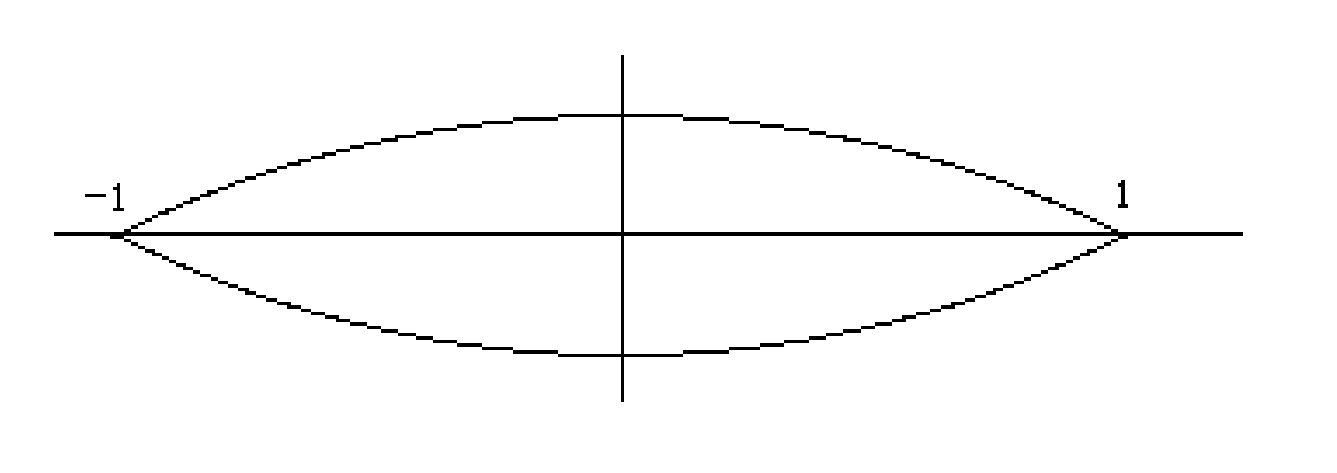}
\end{center}
\caption[La lentille de Bakhvalov]{\label{ffi522}  
La lentille de Bakhvalov}  
\end{figure}  

\end{fremarks}

\smallskip\noindent{\bf Classe de Gevrey et approximation uniforme par des 
\pols} \\
Si $f$ est un  $\p$-point de  $\cw$   donné par une suite 
$\; m \mapsto P_m\; $
 $\p$-calculable (avec $\norme{ f - P_m}_{\infty} \leq 2^m$), alors 
le degré de $P_m$  
est majoré par un \pol en  $m$,  donc il existe un entier $ k $ et une 
constante~$ B $ telles que le degré de $ P_m $ soit majoré par $(Bm)^k$. 
Soit alors $ n $ arbitraire, et considérons le plus grand entier $ m $ tel que  
$(Bm)^k \leq n$,  c.-à-d.  $m := \Flo{\sqrt[k]{n}/B}$. 
On a donc  $m+1 \geq \sqrt[k]{n}/B$.  En posant  $r := 1/2^{1/B}$
  et   $\gamma := 1/k$,  on obtient:  
$$E_n(f) \leq 1/2^m \leq 2.r^{n^{\gamma}}, \; \hbox{avec} \; r \in ]0,1[, \; 
\gamma > 0.$$
En particulier, la suite  $E_n(f)$  est à décroissance rapide et $f\in\Ci$.
Ceci nous amène à étudier les fonctions $f$ pour lesquelles ce genre de 
majoration est obtenu.

\begin{fdefinition}[classe de Gevrey{\footnote{Cf. par exemple  Hörmander \cite{fHo}: The 
Analysis of Linear Partial Differential Operators  I  p 281  (Springer 1983).  
Une fonction est Gevrey d'ordre $ 1 $ si et seulement si elle est 
analytique.}}] \label{f523}

Une fonction $f\in \Ci$ est dite dans la classe de Gevrey d'ordre $\alpha > 0 $ 
si ses dérivées vérifient une majoration:   
\[
\NOrme{f^{(n)} }_{\infty}\leq  MR^n n^{\alpha n}.
\]
La classe de Gevrey est obtenue lorsqu'on ne précise pas l'ordre  $\alpha$. 
\end{fdefinition}

\begin{ftheorem} \label{f524}
Soit $f \in \C$.  Les propriétés suivantes sont équivalentes. 
\begin{itemize}\itemsep2pt
\item [(i)] 
$ \exists M > 0, \; \; r < 1, \; \; \gamma > 0 \; \; 
\forall n> 0, \; \; E_n(f) \leq Mr^{n^{\gamma}}$,
\item [(ii)] 
$ \exists M > 0, \; \; r < 1, \; \; \gamma > 0 \; \; 
\forall n> 0, \; \; S_n(f) \leq Mr^{n^{\gamma}}$,
\item [(iii)] 
$ \exists M > 0, \; \; r < 1, \; \; \gamma > 0 \; \; 
\forall n> 0, \; \; \abS{A_n(f)}\;  
 \leq Mr^{n^{\gamma}}$, 
\item [(iv)] 
$ \exists M > 0, \; r < 1 \; \forall n> 0, \; \NOrme{ u_n^{[n+1]} - f }_{\infty} \leq Mr^{n^{\gamma}}$,
\item [(j)] 
$ \exists c, \beta > 0 \; \forall n> 0,\; \forall m \geq cn^{\beta} 
\; E_m(f) \leq 1/2^n $,
\item [(jj)]
$ \exists c, \beta > 0 \; \forall n> 0,\; \forall m \geq cn^{\beta} \; S_m(f) \leq 1/2^n $,
\item [(jjj)] 
$ \exists c, \beta > 0 \; \forall n> 0,\; \forall m \geq cn^{\beta} \; 
\abS{A_m(f)}\;  \leq 1/2^n $,
\item [(jw)] 
$ \exists c, \beta > 0 \; \forall n> 0,\;\forall m \geq cn^{\beta} \NOrme{ u_m^{[m+1]} - f }_{\infty} \leq 1/2^n $,
\item [(k)]  $f$ est dans la classe de Gevrey.
\end{itemize}
\end{ftheorem}

\proof
(i) $\Leftrightarrow$ (ii)  à partir de l'équation (\ref{fF 5.2.3}). \\  
(i)  $\Rightarrow$ (iii) à partir de l'équation (\ref{fF 5.2.4}). \\ 
(iv)  $\Rightarrow$ (i)  est triviale.\\  
Les  4  équivalences du type (i) $\Leftrightarrow$ (j)  résultent du même 
genre de calcul que celui qui a été fait avant le théorème.\\
L'implication  (jjj) $\Rightarrow$ (jw)  résulte d'un calcul de majoration 
simple utilisant  les inégalités (\ref{fF 5.2.7}) et (\ref{fF 5.2.8}). \\
Supposons  (k), c.-à-d. que $f$ soit Gevrey d'ordre  $\alpha$, et démontrons 
(i). 
Le problème de majoration n'est délicat que pour  $\alpha \geq 1$,
 ce qu'on supposera maintenant. En appliquant le théorème de Jackson, on 
obtient une majoration  
$E_n(f) \leq \pi^{k} \Norme{f^{(k)} }_{\infty} /n^k$
  dès que  $n \geq 2k$,  ce qui donne avec la majoration de Gevrey    
$E_n(f) \leq A(Ck^{\alpha}/n)^k$.  
On peut supposer  $C^{1/ \alpha} \geq 2$  et on prend pour  $k$  un entier 
proche de  $(n/2C)^{1/ \alpha}$ ($ \leq n/2$),
 d'où à très peu près:
$$E_n(f) \leq A(1/2)^{(n/2C)^{1/ \alpha}} = Ar^{n^{\gamma}}, \; \hbox{avec} \; 
\gamma = 1/ \alpha.$$
Supposons maintenant que $f$ vérifie (j)  et démontrons que $f$ est Gevrey.  \\
Le problème de majoration n'est délicat que pour  $\beta \geq 1$, ce qu'on 
supposera maintenant. On écrit  $f^{(k)} = \sum' A_m \Tch_m^{(k)}$.   
D'où   $\Norme{f^{(k)} }_{\infty} \leq \sum' \abS{A_m} m^{2k}$ d'après l'inégalité de V.A. Markov (\ref{fF 5.2.2}).  
On utilise maintenant la majoration  (jjj). On prend $c$ et $\beta$ 
entiers pour simplifier (ce n'est pas une restriction). Dans la somme ci-dessus, 
on regroupe les termes pour $ m $ compris entre $cn^{\beta}$ et 
$c(n+1)^{\beta}$. Dans le paquet obtenu, on majore chaque terme par $ (1/2^n) 
m^2k$, et on majore le nombre de termes par $ c(n+1)b,$ d'où:
\[
\Norme{f^{(k)} }_\infty \leq \sum_n (c(n+1)^{\beta}/2^n)(c(n+1)^{\beta})^{2k} \leq 2c^{2k+1} \sum_n (n+1)^{\beta (2k+1)}/2^n
\]
\[
\leq 4c^{2k+1} \sum_n n^h/2^n, \; \hbox{où} \; h = \beta(2k+1).
\] 
On majore cette série par la série obtenue en dérivant $h$ fois  la 
série  $\sum_n x^n$  (puis en faisant  $x=1/2$)  et on obtient que $f$ est 
Gevrey d'ordre  $2\beta$.  \eop

\begin{fremarks} \label{f525}~\\
1)  L'espace des fonctions Gevrey possède donc une \pres constructive 
agréable.\\
2)  Pour  $\gamma = 1$ on obtient les fonctions analytiques. Pour  $\gamma > 1$, 
on obtient des fonctions entières.\\
3)  Pour $\gamma \leq 1$, la limite supérieure des $ \gamma $ possibles est la 
même dans  (i),  (ii),  (iii)  et  (iv),  la limite inférieure des $\beta$ 
possibles est la même dans  (j),  (jj),  (jjj)  et  (jw),  avec
$\gamma = 1/ \beta$. \\
4)  Si on se base sur le cas des fonctions analytiques  
($\alpha = \beta = \gamma = 1$), 
on peut espérer, pour l'implication  (j) $\Rightarrow$ (k),
 obtenir que $f$ soit  Gevrey d'ordre $\beta$  au moyen d'un calcul de 
majoration plus sophistiqué.\\
5)  Dans  (j),  (jj),  (jw)  on peut supprimer le quantificateur $\forall m$   
si on prend $ c $ et $\beta$ entiers  et $m = cn^{\beta}$.
\end{fremarks}

\subsubsection{Retour aux questions de \com dans l'espace  $\cw$}
\label{fsubsubsec 523}
Nous commençons par une remarque importante.

\smallskip \noindent {\bf  Remarque importante.}  En ce qui concerne  $\DD[X]$, la \pres dense ordinaire  (sur la base des  $X^n$)  et la \pres dense sur la base des \pols de Chebyshev   $\Tch_n$, sont \equivas en temps \poll. Nous utiliserons indifféremment l'une ou l'autre des deux bases, selon la commodité du moment.

\medskip  Rappelons également que la norme 
$\; P \to \norme{ P }_{\infty}\; $ 
 est une fonction $\p$-calculable de $\DD[X]$  vers~$\RR$.

La preuve de la proposition suivante est immédiate. En fait toute 
fonctionnelle définie sur  $\cw$  qui a un \mcu \poll et dont la 
restriction à  $\DD[X]$  est ``facile à calculer'' est elle même ``facile 
à calculer''. Cette proposition prend toute sa valeur au vu du théorème de 
caractérisation \ref{f527}. 

\begin{fproposition}[bon comportement des fonctionnelles usuelles] \label{f526}~ \\ 
Les fonctionnelles:
\[
\cw \to \RR \quad  f \mapsto \norme{f}_{\infty}, \; \norme{f}_2, \; \norme{f}_1
\]
  sont  \uni de classe  $\p$.\\
Les fonctionnelles:
\[
\cw \times [0,1] \times [0,1] \to \RR \quad  (f,a,b) \mapsto \sup_{x 
\in [a,b]} (f(x)), \; \int_{a}^{b} f(x) dx
\]
sont  \uni de classe  $\p$.
\end{fproposition}

\begin{ftheorem}[caractérisation des  $\p$-points de  $\cw$] \label{f527}
~\\
Soit   $f \in \C$.  Les propriétés suivantes sont équivalentes.
\begin{itemize}\itemsep2pt
\item [a)]  La fonction $f$  est  un $\p$-point de  $\ckf$  et est dans la classe de Gevrey.
\item  [b)]  La suite   $A_n(f)$   est une  $\p$-suite dans  $\RR$   et vérifie une majoration     $\abS{A_n(f)}   \leq Mr^{n^{\gamma}}$
avec   $M > 0$, $\gamma > 0$ et $0 < r < 1$.
\item  [c)]  La fonction $f$  est un  $\p$-point de  $\cw$.
\end{itemize}

\end{ftheorem}

\proof  Les implications  (c) $\Rightarrow$ (a)  et   (c)   $\Rightarrow$  (b) sont faciles à partir du théorème \ref{f524}. \\   
(b)  $\Rightarrow$  (c). Un \pol (en \pres dense sur la base des $\Tch_n$)  approchant $f$  avec la précision  $1/2^{n+1}$  
est obtenu avec la somme partielle extraite de la série de Chebyshev de $f$ en s'arrêtant à l'indice  $(Bn)^h$   (où  $B$  et $h$ se calculent à partir 
de  $M$ et $\gamma$).  
Il reste à remplacer chaque coefficient de Chebyshev par un dyadique l'approchant avec la précision: 
\[
1/ \Flo{(Bn)^h2^{n+1}}  = 1/2^{n+1+h \log(Bn)}.
\]
(a)  $\Rightarrow$  (c).  Un \pol  approchant  $f$  avec la précision  
$1/2^{n+1}$  est obtenu avec   $u_m^{[m+1]}$  (où  $m = (Cn)^k$),
$C$ et $k$  se calculent à partir de  $M$ et $\gamma$, en tenant compte des inégalités  (\ref{fF 5.2.7}) et (\ref{fF 5.2.8})).   
La formule définissant  $u_m^{[m+1]}$   fournit ses coefficients sur la base des  $\Tch_n$  et on peut calculer (en temps \poll)  une approximation à  $1/2^{n+1+k \log(Cn)}$  près de ces coefficients en profitant du fait que la suite double $\xi_i^{[n]}$ est une  $\p$-suite de réels et que la fonction $f$ est  un $\p$-point de  $\ckf$.  \eop

\smallskip Une conséquence immédiate du théorème précédent est obtenue dans le 
cas des fonctions analytiques.

\begin{ftheorem} \label{f528}
Soit   $f \in \C$.  Les propriétés suivantes sont équivalentes. 
\begin{itemize}

\item [(a)] La fonction $f$ est une fonction analytique et c'est un $\p$-point 
de  $\ckf$.

\item [(b)] La suite  $A_n(f)$  est une  $\p$-suite dans  $\RR$  et vérifie 
une majoration
$$ \abS{A_n(f)}\;  \leq Mr^n \; (M > 0, \; r < 1).$$

\item [(c)] La fonction $f$ est une fonction analytique et est un  $\p$-point de $\cw$.
\end{itemize}
\end{ftheorem}

\begin{fdefinition}[fonctions $\p$-analytiques]  \label{f529} 
~\\ 
Lorsque ces propriétés sont vérifiées, on dira que la fonction $f$ est  
$\p$-analytique sur l'intervalle  $[-1,1]$.
 \end{fdefinition}

\begin{ftheorem}[assez bon comportement de la dérivation vis à vis de la \com] \label{f5210} ~\\
Soit $\wi f$ un  $\p$-point de  $\cw$.  Alors  la suite $k \mapsto \wi{f^{(k)}}$   est  une  $\p$-suite de  $\cw$.  
Plus généralement, si $\big(\wi{f_n}\big)$ est une $\p$-suite de  $\cw$  alors la suite 
double $\bigg(\wi{f_n^{(k)}}\bigg)$ est une $\p$-suite de  $\cw$.
\end{ftheorem}

\proof Nous donnons la preuve pour la première partie de la proposition. Elle 
s'appliquerait sans changement pour le cas d'une $\p$-suite de $\cw$. \\
La fonction $\wi f$ est un $\p$-point de $\cw$ donné comme limite d'une suite   $\p$-calculable $n \mapsto P_n$. 
La suite double  $P_n^{(k)}$  est  $\p$-calculable  ($n,k\in\NN_1$).  
Il existe deux entiers  $a$  et  $b$  tels que le degré de~$P_n$ soit majoré par  $2^a n^b$.  
Donc,  d'après l'inégalité de V.A. Markov  (\ref{fF 5.2.2})  on a la majoration:
\[
\NOrme{ P_n^{(k)} - P_{n-1}^{(k)} }_\infty \leq (2^an^{2b})^k \Norme{ P_n - P_{n-1} }_{\infty} \leq (2^an^{2b})^k /2^{n-2} = 1/2^{n-(k.(a+2b \log(n))+2)}
\]
On détermine alors aisément une constante  $n_0$  telle que, pour 
$n \geq 2n_0k$, on ait:   $$n \geq 2(k.(a+2b \log(n))+2)$$ 
et donc  
\[
\NOrme{ P_n^{(k)} - P_{n-1}^{(k)} }_\infty \leq 1/2^{n/2}
\] 
de sorte qu'en posant  $\nu(n) := 2 \sup(n_0k,n)$,  on a, pour 
$q \geq \nu(n),$  
\[
\NOrme{ P_{\nu(n)}^{(k)} - P_{\nu(n+1)}^{(k)} }_\infty \leq 1/2^{n-1}
\]
et donc, puisque  
$\nu(n+1) = \nu(n)$ ou  $\nu(n)+2,$ 
\[
\NOrme{ P_{\nu(n)}^{(k)} - P_{\nu(n+1)}^{(k)} }_\infty \leq 1/2^{n-1}
\]
 d'où enfin:   
\[
\NOrme{ P_{\nu(n)}^{(k)} - f^{(k)} }_\infty \leq 1/2^{n-2}
\]
On termine en notant que la suite double  $(n,k) \mapsto P_{\nu(n+2)}^{(k)}$  
est  $\p$-calculable.    \eop

\begin{fcorollary} \label{f5211}
Si $f$ est un  $\p$-point de  $\cw$   et   $a, b$   deux  $\p$-points de  
$[-1,1]$, alors les suites  
\[
\Norme{f^{(n)} }_{\infty}, \; \Norme{f^{(n)} }_2, \; \Norme{f^{(n)} }_1, \; \Norme{f^{(n)}(a)} \;\hbox{et} \; \sup\nolimits_{x \in [a,b]} (f^{(n)}(x))
\]
sont des $\p$-suites dans  $\RR$.\\
Plus généralement, si $(f_p)$ est une $\p$-suite de  $\cw$  alors les suites 
doubles  
\[
\norme{f_p^{(n)}}_\infty, \; \norme{f_p^{(n)} }_2, \; \norme{f_p^{(n)} }_1, \; \norme{f_p^{(n)}(a)} \; \hbox{et} \; \sup\nolimits_{x \in [a,b]} (f_p^{(n)}(x))
\]
sont des $\p$-suites dans  $\RR$.
\end{fcorollary}

La preuve du théorème \ref{f5210} (et donc du corollaire \ref{f5211}) est en 
quelque sorte uniforme et a une signification plus générale. Nous allons 
maintenant définir le cadre naturel dans lequel s'applique ce théorème et donner un nouvel énoncé, plus général et plus satisfaisant.

\begin{fdefinition} \label{f5212}
Pour  $c$ et $\beta > 0$  on note $\Gv_{c, \beta}$ la classe des fonctions 
Gevrey vérifiant la majoration (du type (jjj) dans \ref{f524})
$$ \forall m > cn^{\beta} \quad  \abS{A_m(f)}\;  \leq 1/2^n$$
C'est une partie convexe fermée de $\C$.   
Pour  $c$ et $\beta$  entiers, on note  $\Y_{\Gv_{c, \beta}}$  les éléments de  $\DD[X]$  qui sont dans la classe  $\Gv_{c, \beta}$. 
Cet ensemble $\Y_{\Gv_{c, \beta}}$  peut être pris 
pour ensemble des codes des points rationnels d'une \pres rationnelle $\sC_{\Gv, c, \beta}$  de $\Gv_{c, \beta}$.
\end{fdefinition}

On notera que le test d'appartenance à la partie $\Y_{\Gv_{c, \beta}}$  de  
$\DD[X]$   est en temps \poll, puisque les $A_m(f)$ pour un \pol $f$ 
sont ses coefficients sur la base de Chebyshev.
Dans ce nouveau cadre, le théorème \ref{f528} admet une formulation plus 
uniforme et plus efficace.

\begin{ftheorem} \label{f5213}
Chaque fonctionnelle $f \mapsto f^{(k)}$  est une fonction \uni de classe  $\p$   
de $\sC_{\Gv_{c, \beta}}$ vers  $\cw$.   Plus précisément la suite de 
fonctionnelles 
$$ (k,f) \mapsto f^{(k)} \; : \; \NN_1 \times \sC_{\Gv_{c, \beta}} \rightarrow 
\cw $$
est \uni de classe  $\p$   (au sens de la définition \ref{f229}).
\end{ftheorem}

\proof 
La suite double   $(k,f) \mapsto f^{(k)}$   est de faible \com en tant que 
fonction de  $\NN_1 \times  \DD[X]$   
vers  $\DD[X]$  donc aussi en tant que fonction de 
 $\NN_1 \times \Y_{\Gv, \beta}$
  vers  $\DD[X]$. \\
Tout le problème est donc de démontrer que l'on a un \mcu \poll (au sens de 
\ref{f229}).  
Nous devons calculer une fonction  $\mu(k,h)$  telle que l'on ait pour tous $f$ 
et $g$  dans  $\Y_{\Gv, \beta}$:
\[
\norme{ f - g }_\infty \leq 1/2^{\mu (k,h)} \Rightarrow \NOrme{f^{(k)} - g^{(k)} }_\infty \leq 1/2^h. 
\] 
Ce calcul de majoration est assez proche de celui qui a été fait dans la 
preuve du théorème \ref{f528}. On écrit   
\[
\NOrme{f^{(k)} - g^{(k)} }_\infty \leq \NOrme{f^{(k)} - s_n(f)^{(k)} }_{\infty} + \NOrme{ g^{(k)} - s_n(g)^{(k)} }_{\infty} + \NOrme{ s_n(f-g)^{(k)} }_{\infty}.
\]
Dans la somme du second membre les deux premiers termes sont majorés comme 
suit
\[
\NOrme{f^{(k)} - s_n(f)^{(k)} }_{\infty} \;\leq\;\sum_{q>n} \abS{A_q(f)} \NOrme{ T_q^{(k)} }_\infty \;\leq\; \sum_{q>n} \abS{A_q(f)} q^{2k}.
\]
Comme on a:  $\forall q > cn^{\beta} \; \abS{A_q(f)}\;  \leq 1/2^n$, 
$\sum_{q>n} \abS{A_q(f)} q^{2k}$ 
est ``bien'' convergente et on peut expliciter un  $\alpha(k,h)$  \poll en 
$k,h$  tel que  (voir l'explicitation en fin de preuve),
\[
\hbox{avec} \; n = \alpha(k,h) \; \; \forall f \in \Y_{\Gv,c,\beta} \; : \; \; 
\sum_{q>n} \abS{A_q(f)} q^{2k} \leq 1/2^{h+2}.
\]
Une fois fixé  $n = \alpha(k,h)$ il nous reste à rendre petit le terme 
$\NOrme{ s_n(f-g)^{(k)} }_{\infty}$. 
L'inégalité de Markov  (\ref{fF 5.2.2}) implique que 
\[
\NOrme{ s_n(f-g)^{(k)} }_\infty \leq \norme{ s_n(f-g) }_{\infty} n^{2k}.
\]
Il ne reste plus qu'à obtenir une majoration convenable de  
$\norme{ s_n(f-g) }_{\infty}$ à partir de  
$\norme{ f-g }_{\infty}$.
Par exemple, on peut utiliser la majoration  
$S_n(f) \leq (4+ \log (n)) E_n(f)$ 
  (d'après la formule (\ref{fF 5.2.3})) d'où
\[
\NOrme{f^{(k)} }_\infty \leq \norme{f}_\infty + S_n(f) \leq \norme{f}_\infty + (4+ \log (n)) E_n(f) \leq (5+ \log (n)) \norme{f}_\infty.
\]
Donnons pour terminer l'explicitation de $\alpha(k,h)$.  On a les inégalités
\[
\sum_{q \geq cn_0^{\beta}} \abS{A_q(f)} q^{2k} \leq \sum_{n \geq n_0}\sum_{q \leq c(n+1)^{\beta}}q^{2k}/2^n \leq \sum_{n \geq n_0} c(n+1)^{\beta}(c(n+1)^{\beta})^(2k)/2^n,
\]
donc
\[
\sum_{q \geq cn_0^{\beta}} \abS{A_q(f)} q^{2k} \leq \sum_{n \geq n_0} (c(n+1)^{\beta})^{2k+1}/2^n \leq \sum_{n \geq n_0} 1/2^{\varphi (n, \beta ,c, k)}
\]
avec, si  $2^a \geq c$,
\[
\varphi (n, \beta ,c, k) \geq n- (2k+1)a-(2k+1) \beta \log(n+1).
\]
Si l'on a   
\[
\hbox{pour} \; n \geq n_0 \; \; \varphi (n, \beta ,c, k) \geq h+n/2+4, \eqno (\star)
\]
on obtient
\[
\sum_{q \geq cn_0^{\beta}} \abS{A_q(f)} q^{2k} \leq \sum_{n \geq n_0} 1/2^{\varphi (n, \beta ,c, k)} \leq (1/2^{h+2}(1/4) \sum_{n \geq n_0} 1/2^{n/2} \leq 1/2^{h+2}.
\]
Et l'on peut prendre $\alpha (k,h) = cn_0^{\beta}.$ \\
Il reste à voir comment on peut réaliser la condition $(\star)$. \\
Pour tout entier $b$  on a un  entier  $\nu (b) \leq \max (8,b^2)$  pour lequel  
\[
n > \nu (b) \; \Rightarrow\; n \geq b \log (n+1).
\]
Si donc $n \geq \nu (4(2k+1) \beta)$  on obtient
\[
\varphi (n, \beta ,c, k) \geq n-(2k+1)a-(2k+1)\beta \log(n+1) \geq (3n/4)-(2k+1)a
\]
et la condition $\varphi (n, \beta ,c, k)\geq h+n/2+4$ est réalisée  si    
$n/4 \geq (2k+1)a +h +4.$
D'où finalement  
$\alpha (h,k) = c \max(\nu (4(2k+1)\beta ),4((2k+1)a + h +4))^{\beta}.$    \eop
 
En appliquant la proposition \ref{f526}, on obtient le corollaire suivant.

\begin{fcorollary} \label{f5214}~\\
i) Les trois suites de fonctionnelles 
\[
(n,f) \mapsto \Norme{f^{(n)} }_{\infty} , \;  \Norme{f^{(n)} }_2, \; \Norme{f^{(n)} }_1 \qquad  \NN_1 \times \sC_{\Gv,c, \beta} \to \RR
\]
sont \uni de classe  $\p$   (au sens de la définition \ref{f229}).\\
ii) La suite de fonctionnelles
\[(n,f,x) \mapsto f^{(n)}(x) \qquad 
\NN_1 \times \sC_{\Gv,c, \beta} \times [-1,1] \to \RR  
\] 
est \uni de classe  $\p$.\\
iii) La suite de fonctionnelles
\[  (n,f,a, b) \mapsto \sup_{x \in [a,b]} (f^{(n)}(x)) \qquad 
\NN_1 \times \sC_{\Gv,c, \beta} \times [-1,1] \times [-1,1] \to \RR 
\] 
est \uni de classe  $\p$.  
\end{fcorollary}

\begin{fremark} 
Les théorèmes  \ref{f527}, \ref{f528}, \ref{f5210}, \ref{f5213}, la proposition 
\ref{f526} et les corollaires \ref{f5211} et \ref{f5214}  améliorent sensiblement 
les résultats de  \cite{fKF82}, \cite{fKF88}  et  \cite{fMu87}  sur les fonctions 
analytiques calculables \etpo (au sens de Ko-Friedman).
\end{fremark}

\subsection {Comparaisons de différentes présentations de classe  
$\p$}\label{fsubsec53}
Dans cette section, nous obtenons la chaîne suivante de fonctions \uni de 
classe  $\p$  pour l'identité de $\czu$. 
\[
\cw \to \csp \to \crf \equiv \csr \to \ckf 
\] 
et aucune des flèches  $\rightarrow$   dans la ligne ci-dessus n'est une $\p$-équivalence sauf peut-être  
$\csp  \rightarrow   \crf$  et très éventuellement  $\crf  \rightarrow   
\ckf$  (mais cela impliquerait $\p = \np$).
Tout d'abord, il est clair que l'identité de $\czu$  est de classe  $\LINT $ 
pour les cas suivants:  
\[
\cw \rightarrow \csp; \quad  \csp \to \csr ;  \quad  \crf \rightarrow \csr. 
\]
Par ailleurs l'identité de $\czu$  est de classe  $\p$   dans le cas suivant 
\[
\crf \rightarrow \caf 
\] 
(en fait, seul le calcul de la magnitude n'est pas complètement trivial, et il est sûrement dans $\DTI (\Oo(N^2))$). 

Il nous reste à démontrer que l'identité de $\czu$ de  $\csp$  vers  $\crf$ et celle de  $\csr$  vers  $\crf$  
sont de classe~$\p$.

\begin{ftheorem} \label{f531}
L'identité de  $\czu$  de   $\csp$  vers  $\crf$  est \uni de classe $\p$.
\end{ftheorem} 
\proof 
Soit  $n \in  \NN_1$  et  $f \in  \ysp$. On doit calculer un élément  
$g \in \yrf $  tel que  
\[
\Norme{ \wi{f} - \wi{g} }_{\infty} \leq 1/2^n.
\] 
On a $f = ((x_0,x_1,\ldots,x_t),(P_1,P_2,\ldots,P_t))$ avec $x_0=0$, $x_t=1$ et 
$P_i(x_i) = P_{i+1}(x_i)$   pour  $i = 1,\ldots,t-1$.   
On calcule   $m \in \NN_1$   tel que  $2^m$  majore les 
$\norme{P_i}_{\infty}$ et $\norme{P_i'}_{\infty}$.  \\
On pose $p=m+n+1$ et $z_i = x_i - 1/2^{p+1}$ pour $i=1,\dots,t-1$. Pour  $i = 0$ on pose $h_0 := -C_{p,z_1}$. Pour  $i = t-1$ on pose $h_{t-1} := C_{p,z_{t-1}}$. Pour  $i = 1,\ldots,t-2$ on pose 
 $h_i := C_{p,z_i} - C_{p,z_{i+1}}$   (voir la figure \ref{ffi531}). 

\begin{figure}[htbp]  
\begin{center}
\includegraphics*[width=12cm]{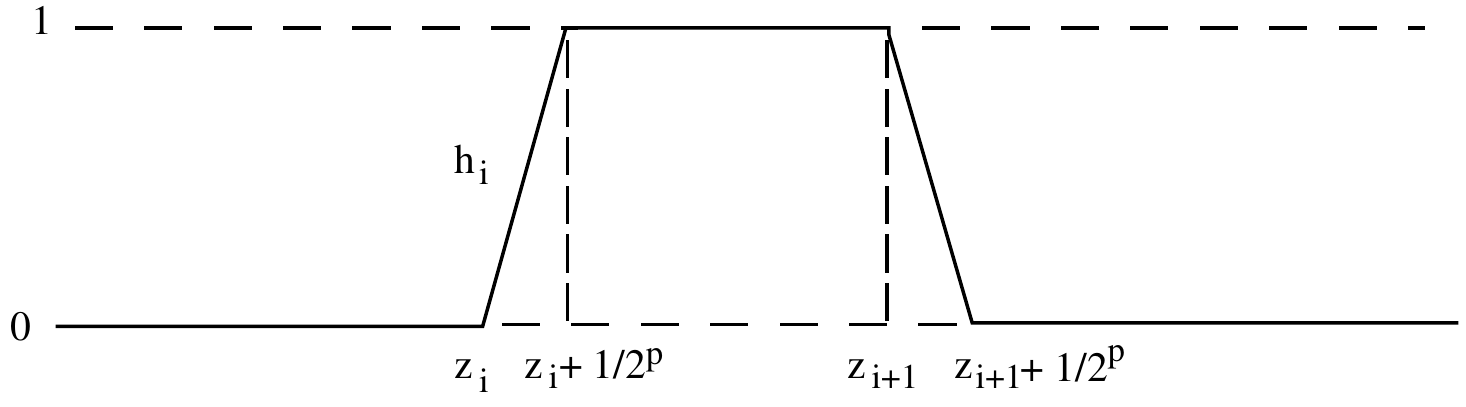}
\end{center}
\caption[La fonction $h_i$]{\label{ffi531}  
La fonction $h_i$}  
\end{figure}  

\noindent Alors la fonction  $\wi{f}$ est à très peu près égale à  
$h=\sum_i h_iP_{i+1}$:  sur les intervalles  
$[z_i+1/2^p,z_{i+1}]$ on a   $h =\wi{f}$, 
tandis que sur un intervalle  
\[
[z_i, z_i +1/2^p] = [x_i - 1/2^{p+1}, x_i +1/2^{p+1}]
\]
on obtient  $h = h_{i-1}P_i + h_i P_{i+1}$  qui est une moyenne pondérée de  
$P_i$ et $P_{i+1}$,   au lieu de  $P_i$   ou  $P_{i+1}$. 
En un point $x$ de cet intervalle, on a  
$\abs { x-x_i }\;  \leq 1/2^{p+1}$, on applique le théorème des 
accroissements finis et en utilisant $P_i(x_i) =P_{i+1}(x_i)$  on obtient
\[
\abs{P_i(x) - P_{i+1}(x)}\; \leq \abs{P_i(x) - P_i(x_i)} + \abs{P_{i+1}(x) - P_{i+1}(x_i)}\; \leq 2^{m+1}/2^{p+1} \leq 1/2^{n+1}
\]
et donc:
\[
\Norme{ \wi{f} - h }_{\infty} \leq 1/2^{n+1}.
\]
Il reste à remplacer chaque  $h_i$  par un élément  $g_i$ de  $\yrf$  
vérifiant   
$\Norme{ \wi{g_i} - h_i }_{\infty} \leq 1/(t2^{p})$
 (de sorte que
$\Norme{\wi{g_i}P_{i+1}-h_iP_{i+1}}_{\infty}\leq 1/(t2^{n+1})$
et donc
\[
\NOrme{ h - \sum\nolimits_i \wi{g_i}P_{i+1} }_\infty \leq 1/2^{n+1}.
\]
Vue la proposition \ref{f337} concernant l'approximation des fonctions  $C_{p,a}$  
par des fractions rationnelles, le calcul des  $g_i$  se fait \etpo à partir 
de la donnée  $(f,n)$. Il reste à exprimer  
$\sum_i\wi{g_i}P_{i+1}$  
 sous forme $\wi{g}$   avec  $ g \in \yrf$, ce qui n'offre pas de 
difficulté, pour obtenir  
$\Norme{ \wi{f} - \wi{g} }_{\infty} \leq 1/2^n.$  
\eop

\begin{ftheorem} \label{f532}
Les représentations   $\crf$  et   $\csr$  de  $\czu$  sont $\p$-
équivalentes.
\end{ftheorem}
\proof
À partir d'un élément  
$f=((x_0,x_1,\ldots,x_t),((P_1,Q_1),\ldots,(P_t,Q_t)))$  arbitraire de  $\ysr$  
on peut calculer \etpo des nombres dyadiques $d_1, d_2,\ldots,d_t \geq 1$   
vérifiant   
$d_iQ_i(x_i) = d_{i+1}Q_{i+1}(x_i)$  (et donc aussi  $d_iP_i(x_i) = 
d_{i+1}P_{i+1}(x_i)$)  pour  $i = 1, \ldots, t-1$ 
Alors  $\wi{f} = \wi{g} / \wi{h}$    où  
$g,h \in \Y_{Sp}$  sont donnés par:
\[
g = ((x_0,x_1,\ldots,x_t),(d_1P_1,\ldots,d_tP_t)) \quad  \hbox{et} \quad  h = ((x_0,x_1,\ldots,x_t),(d_1Q_1,\ldots,d_tQ_t)).
\]
On conclut en utilisant le \thref{f531} qui permet d'approcher 
convenablement  $\wi g$  et $\wi h$ par des fractions rationnelles.  \eop

\section*{Conclusion}\label{fsec Conclusion}
\addcontentsline{toc}{section}{Conclusion}
Hoover a relié de manière intéressante la notion naturelle de \com des 
fonctions réelles continues donnée par Ko et Friedman à une autre notion, 
basée sur les circuits arithmétiques. Il a donné de cette manière une 
certaine version ``en temps \poll'' du théorème d'approximation de 
Weierstrass. 

Dans cet article nous avons généralisé l'approche de Hoover, en 
introduisant un point de vue uniforme grâce à la notion de \rp  d'un espace 
métrique.
Ceci fournit un cadre de travail général satisfaisant pour l'étude de 
très nombreux problèmes de \com algorithmique en analyse. Grâce à cette approche
nous avons également généralisé  les 
résultats de Ko, Friedman et Müller concernant les fonctions analytiques 
calculables en temps \poll.  
La \rp  $\csl$ $\p$-\equiva à  $\ckf$ est la plus naturelle du point de vue de 
l'informatique théorique. Cependant ce n'est pas une $\p$-\pres et elle est 
mal adaptée à l'analyse numérique dés qu'on se pose des problèmes plus 
compliqués que l'évaluation (le calcul de la norme, ou d'une primitive par 
exemple).

Parmi les représentations que nous avons étudiées, la \pres  $\cw$  semble 
être la plus facile à utiliser pour de nombreux problèmes de l'analyse 
numérique. 

Quant à la \pres par les fractions rationnelles, elle mérite 
une étude plus approfondie. On aimerait obtenir pour cette \pres un analogue du théorème (donné pour la \pres~$\cw$) concernant la caractérisation des $\p$-points. 
Il serait également intéressant d'obtenir pour $\crf$ la 
$\p$-calculabilité de certaines opérations usuelles de l'analyse numérique, 
comme le calcul d'une primitive ou plus généralement le calcul de la 
solution d'une équation différentielle ordinaire.

\bni{\bf Remerciements}
Nous remercions Maurice Margenstern et le referee pour leurs remarques 
pertinentes.


\end{document}